\documentclass[reqno]{amsart}
\usepackage{amsmath,amsthm,amssymb,amsfonts}
\usepackage{mathrsfs}
\numberwithin{equation}{section}
\usepackage{color}
\usepackage{enumerate}
\usepackage{multirow} 
\newtheoremstyle{mystyle}
{}
{}
{\normalfont}
{ }
{\bfseries}
{}
{10pt}
{ }
\theoremstyle{mystyle}
\newtheorem{theorem}{Theorem}
\newtheorem{proposition}{Proposition}
\newtheorem{lemma}{Lemma}
\newtheorem{remark}{Remark}

\DeclareMathOperator*{\arginf}{arg\,inf}
\usepackage{mathtools}
\usepackage[pdftex]{graphicx}
\def\diag{\mathop{\rm diag}\nolimits}
\def\Diag{\mathop{\rm Diag}\nolimits}
\def\tr{\mathop{\rm tr}\nolimits}
\def\vec{\mathop{\rm vec}\nolimits}
\def\vech{\mathop{\rm vech}\nolimits}
\def\det{\mathop{\rm det}\nolimits}
\def\rank{\mathop{\rm rank}\nolimits}
\def\Int{\mathop{\rm Int}\nolimits}
\newcommand{\dd}{\mathrm d}
\newcommand{\E}{\mathbb E}
\newcommand{\PP}{\mathbb P}

\allowdisplaybreaks[4]
\usepackage[foot]{amsaddr}
\usepackage[top=3cm, bottom=3cm, left=2.5cm, right=2.5cm]{geometry}
\usepackage{comment}
\large
\title[SEM for diffusion processes based on high-frequency data]{Structural equation modeling with latent variables for diffusion processes and its application to sparse estimation}
\author[S. Kusano]{Shogo Kusano $^{1}$}
\author[M. Uchida]{Masayuki Uchida $^{1,2}$}
\address{$^{1}$Graduate School of Engineering Science, Osaka University}
\address{$^{2}$Center for Mathematical Modeling and Data Science (MMDS), Osaka University and JST CREST}
\begin{document}
\begin{abstract}
\fontsize{8pt}{10pt}\selectfont
 We consider structural equation modeling (SEM) with latent variables for diffusion processes based on high-frequency data. The quasi-likelihood estimators for parameters in the SEM are proposed. The goodness-of-fit test is derived from the quasi-likelihood ratio. We also treat sparse estimation in the SEM. The goodness-of-fit test for the sparse estimation in the SEM is developed. Furthermore, the asymptotic properties of our proposed estimators are examined.
\end{abstract}
\maketitle

\section{Introduction}
\fontsize{10pt}{16pt}\selectfont
We consider structural equation modeling (SEM) with latent variables for diffusion processes. First, a true model is set. The stochastic process $\mathbb{X}_{1,0,t}$ is defined by a true factor model as follows:
\begin{align}
    \mathbb{X}_{1,0,t}={\bf{\Lambda}}_{x_1,0}\xi_{0,t}+\delta_{0,t},
\end{align}
where $\{\mathbb{X}_{1,0,t}\}_{t\geq 0}$ is a $p_1$-dimensional observable vector process, $\{\xi_{0,t}\}_{t\geq 0}$ is a $k_{1}$-dimensional latent common factor vector process,
$\{\delta_{0,t}\}_{t\geq 0}$ is a $p_{1}$-dimensional latent unique factor vector process, ${\bf{\Lambda}}_{x_1,0}\in\mathbb{R}^{p_1\times k_{1}}$ is a constant loading matrix, $p_1$ is not zero, $p_1$ and $k_{1}$ are fixed, and $k_{1}\leq p_1$. The stochastic process $\mathbb{X}_{2,0,t}$ is defined as the following true factor model:
\begin{align}
    \mathbb{X}_{2,0,t}={\bf{\Lambda}}_{x_2,0}\eta_{0,t}+\varepsilon_{0,t},
\end{align}
where $\{\mathbb{X}_{2,0,t}\}_{t\geq 0}$ is a $p_2$-dimensional observable vector process, $\{\eta_{0,t}\}_{t\geq 0}$ is a $k_{2}$-dimensional latent common factor vector process, $\{\varepsilon_{0,t}\}_{t\geq 0}$ is a $p_{2}$-dimensional latent unique factor vector process, ${\bf{\Lambda}}_{x_2,0}\in\mathbb{R}^{p_2\times k_{2}}$ is a constant loading matrix, $p_2$ is not zero, $p_2$ and $k_{2}$ are fixed, and $k_{2}\leq p_2$. Set $p=p_1+p_2$. Moreover, we express the relationship between $\eta_{0,t}$ and  $\xi_{0,t}$ as follows:
\begin{align}
    \eta_{0,t}={\bf{B}}_{0}\eta_{0,t}+{\bf{\Gamma}}_{0}\xi_{0,t}+\zeta_{0,t},
\end{align}
where $\{\zeta_{0,t}\}_{t\geq 0}$ is a $k_{2}$-dimensional latent unique factor vector process, ${\bf{B}}_{0}\in\mathbb{R}^{k_{2}\times k_{2}}$ is a constant loading matrix, whose diagonal elements are zero, and ${\bf{\Gamma}}_0\in\mathbb{R}^{k_{2}\times k_{1}}$ is a constant loading matrix. Set ${\bf{\Psi}}_0=\mathbb{I}_{k_2}-{\bf{B}}_0$, where $\mathbb{I}_{k_2}$ denotes the identity matrix of size $k_2$. It is assumed that ${\bf{\Lambda}}_{x_1,0}$ is a full column rank matrix and ${\bf{\Psi}}_0$ is non-singular. Suppose that $\{\xi_{0,t}\}_{t\geq 0}$ is defined as the following stochastic differential equation:
\begin{align}
    \qquad\dd \xi_{0,t}=B_{1}(\xi_{0,t})\dd t+{\bf{S}}_{1,0}\dd W_{1,t}\ \ (t\in [0,T]),\ \ 
    \xi_{0,0}=c_{1},
\end{align}
where $B_{1}:\mathbb{R}^{k_{1}}\rightarrow\mathbb{R}^{k_{1}}$, ${\bf{S}}_{1,0}\in\mathbb{R}^{k_{1}\times r_{1}}$, $c_{1}\in\mathbb{R}^{k_{1}}$ and $W_{1,t}$ is an $r_{1}$-dimensional standard Wiener process. $\{\delta_{0,t}\}_{t\geq 0}$ satisfies the following stochastic differential equation:
\begin{align}
    \qquad\dd \delta_{0,t}=B_{2}(\delta_{0,t})\dd t+{\bf{S}}_{2,0}\dd W_{2,t}\ \ (t\in [0,T]),\ \ 
    \delta_{0,0}=c_{2},
\end{align}
where $B_{2}:\mathbb{R}^{p_1}\rightarrow\mathbb{R}^{p_1}$, ${\bf{S}}_{2,0}\in\mathbb{R}^{p_1\times r_{2}}$, $c_{2}\in\mathbb{R}^{p_1}$ and $W_{2,t}$ is an $r_{2}$-dimensional standard Wiener process. $\{\varepsilon_{0,t}\}_{t\geq 0}$ 
is defined as the following stochastic differential equation:
\begin{align}
    \qquad\dd \varepsilon_{0,t}=B_{3}(\varepsilon_{0,t})\dd t+{\bf{S}}_{3,0}\dd W_{3,t}\quad (t\in [0,T]),\ \  \varepsilon_{0,0}=c_{3},
\end{align}
where $B_{3}:\mathbb{R}^{p_2}\rightarrow\mathbb{R}^{p_2}$, ${\bf{S}}_{3,0}\in\mathbb{R}^{p_2\times r_{3}}$, $c_{3}\in\mathbb{R}^{p_2}$ and $W_{3,t}$ is an $r_{3}$-dimensional standard Wiener process. $\{\zeta_{0,t}\}_{t\geq 0}$ 
satisfies the following stochastic differential equation:
\begin{align}
    \qquad\dd \zeta_{0,t}=B_{4}(\zeta_{0,t})\dd t+{\bf{S}}_{4,0}\dd W_{4,t}\ \ (t\in [0,T]),\ \ 
    \zeta_{0,0}=c_{4},
\end{align}
where $B_{4}:\mathbb{R}^{k_{2}}\rightarrow\mathbb{R}^{k_{2}}$, ${\bf{S}}_{4,0}\in\mathbb{R}^{k_{2}\times r_{4}}$, $c_{4}\in\mathbb{R}^{k_{2}}$ and $W_{4,t}$ is an $r_{4}$-dimensional standard Wiener process. 
Let ${\bf{\Sigma}}_{\xi\xi,0}={\bf{S}}_{1,0}{\bf{S}}_{1,0}^{\top}$, ${\bf{\Sigma}}_{\delta\delta,0}={\bf{S}}_{2,0}{\bf{S}}_{2,0}^{\top}$, ${\bf{\Sigma}}_{\varepsilon\varepsilon,0}={\bf{S}}_{3,0}{\bf{S}}_{3,0}^{\top}$ and ${\bf{\Sigma}}_{\zeta\zeta,0}={\bf{S}}_{4,0}{\bf{S}}_{4,0}^{\top}$, where $\top$ denotes the transpose. It is assumed that ${\bf{\Sigma}}_{\delta\delta,0}$ and ${\bf{\Sigma}}_{\varepsilon\varepsilon,0}$ are positive definite matrices, and $W_{1,t}$, $W_{2,t}$, $W_{3,t}$ and $W_{4,t}$ are independent.

Next, we set a parametric model. The stochastic process $\mathbb{X}_{1,t}$ is defined as the following factor model:
\begin{align}
    \mathbb{X}_{1,t}={\bf{\Lambda}}_{x_1}\xi_{t}+\delta_{t}\label{X},
\end{align}
where $\{\xi_{t}\}_{t\geq 0}$ is a $k_{1}$-dimensional latent common factor vector process, $\{\delta_{t}\}_{t\geq 0}$ is a $p_{1}$-dimensional latent unique factor vector process and ${\bf{\Lambda}}_{x_1}\in\mathbb{R}^{p_1\times k_{1}}$ is a constant loading matrix. The stochastic process $\mathbb{X}_{2,t}$ is defined by the factor model as follows:
\begin{align}
    \mathbb{X}_{2,t}={\bf{\Lambda}}_{x_2}\eta_{t}+\varepsilon_{t}\label{Y},
\end{align}
where $\{\eta_{t}\}_{t\geq 0}$ is a $k_{2}$-dimensional latent common factor vector process, $\{\varepsilon_{t}\}_{t\geq 0}$ is a $p_{2}$-dimensional latent unique factor vector process and ${\bf{\Lambda}}_{x_2}\in\mathbb{R}^{p_2\times k_{2}}$ is a constant loading matrix. Furthermore, the relationship between $\eta_{t}$ and  $\xi_{t}$ is expressed as follows:
\begin{align}
    \eta_{t}={\bf{B}}\eta_{t}+{\bf{\Gamma}}\xi_{t}+\zeta_{t}\label{eta},
\end{align}
where $\{\zeta_{t}\}_{t\geq 0}$ is a $k_{2}$-dimensional latent unique factor vector process, ${\bf{B}}\in\mathbb{R}^{k_{2}\times k_{2}}$ is a constant loading matrix, whose diagonal elements are zero, and ${\bf{\Gamma}}\in\mathbb{R}^{k_{2}\times k_{1}}$ is a constant loading matrix. It is supposed that ${\bf{\Lambda}}_{x_1}$ is a full column rank matrix and ${\bf{\Psi}}$ is non-singular, where ${\bf{\Psi}}=\mathbb{I}_{k_2}-{\bf{B}}$. Assume that $\{\xi_{t}\}_{t\geq 0}$ satisfies the following  stochastic differential equation:
\begin{align}
    \quad\dd \xi_{t}=B_{1}(\xi_{t})\dd t+{\bf{S}}_1\dd W_{1,t}\ \ (t\in [0,T]),\ \
    \xi_{0}=c_{1},\label{xiP}
\end{align}
where ${\bf{S}}_1\in\mathbb{R}^{k_{1}\times r_{1}}$. $\{\delta_{t}\}_{t\geq 0}$ is defined as the following stochastic differential equation:
\begin{align}
    \quad\dd\delta_{t}=B_{2}(\delta_{t})\dd t+{\bf{S}}_2\dd W_{2,t}\ \ (t\in [0,T]),\ \ 
    \delta_{0}=c_{2},\label{deltaP}
\end{align}
where ${\bf{S}}_2\in\mathbb{R}^{p_1\times r_{2}}$. $\{\varepsilon_{t}\}_{t\geq 0}$ 
satisfies the following  stochastic differential equation:
\begin{align}
    \quad\dd\varepsilon_{t}=B_{3}(\varepsilon_{t})\dd t+{\bf{S}}_3\dd W_{3,t}\ \ (t\in [0,T]),\ \ \varepsilon_{0}=c_{3},\label{epsilonP}
\end{align}
where ${\bf{S}}_3\in\mathbb{R}^{p_2\times r_{3}}$. $\{\zeta_{t}\}_{t\geq 0}$ is defined by the  stochastic differential equation as follows:
\begin{align}
    \quad\dd\zeta_{t}=B_{4}(\zeta_{t})\dd t+{\bf{S}}_4\dd W_{4,t}\ \ (t\in [0,T]),\ \ 
    \zeta_{0}=c_{4},\label{zetaP}
\end{align}
where ${\bf{S}}_4\in\mathbb{R}^{k_{2}\times r_4}$. Let ${\bf{\Sigma}}_{\xi\xi}={\bf{S}}_1{\bf{S}}_1^{\top}$, ${\bf{\Sigma}}_{\delta\delta}={\bf{S}}_2{\bf{S}}_2^{\top}$, ${\bf{\Sigma}}_{\varepsilon\varepsilon}={\bf{S}}_3{\bf{S}}_3^{\top}$ and ${\bf{\Sigma}}_{\zeta\zeta}={\bf{S}}_4{\bf{S}}_4^{\top}$. It is supposed that ${\bf{\Sigma}}_{\delta\delta}$ and ${\bf{\Sigma}}_{\varepsilon\varepsilon}$ are positive definite matrices. Set $\mathbb{X}_t=(\mathbb{X}_{1,t}^{\top},\mathbb{X}_{2,t}^{\top})^{\top}$. ${\bf{\Sigma}}\in\mathbb{R}^{p\times p}$ denotes the covariance structure of $\mathbb{X}_t$. $\{\mathbb{X}_{t_{i}^n}\}_{i=1}^n$ are discrete observations, where $t_{i}^n=ih_n$, $T=nh_n$, and $p_1$, $p_2$, $k_{1}$ and $k_{2}$ are independent of $n$. 

SEM is a method that describes the relationships between latent variables that cannot
be observed. SEM has been used in various fields, e.g., behavioral science, economics, engineering, and medical science. For example, in psychology, SEM is used to investigate the relationship between intelligence and motivation. Note that intelligence and motivation are latent variables. J{\"o}reskog \cite{Joreskog(1970)} proposed this method by combining path analysis and confirmatory factor analysis. For path analysis and confirmatory factor analysis, see, e.g., Mueller \cite{Mueller(1999)}. Several models have been proposed to formulate SEM. In this paper, we consider the model defined by (\ref{X}), (\ref{Y}) and (\ref{eta}), which is called the LInear Structural RELations (LISREL) model (J{\"o}reskog \cite{Joreskog(1972)}). The LISREL model is one of the most well-known models in SEM and can express the complex relationship between latent variables. For more information on the LISREL model, see, e.g., Everitt \cite{Everitt(1984)}. Note that SEM is a confirmatory analysis method rather than an exploratory analysis method. SEM is used to specify the model from a theoretical viewpoint of each research field before conducting the analysis. This is the difference between confirmatory analysis methods and exploratory analysis methods such as exploratory factor analysis. 

Sparse estimation has been applied to many methods, e.g., principal component analysis (Zou et.al. \cite{Zou et.al.(2006)}) and exploratory factor analysis (Choi et.al. \cite{Choi(2010)}). Jacobucci et.al. \cite{Jacobucci(2016)} and Huang et.al. \cite{Huang(2017)} suggested sparse estimation in SEM. In this paper, we call the method ``Sparse estimation in Structural Equation Modeling (SSEM)". In SEM, although some parameters may be set to 0, this assumption may be incorrect, which means that the model is misspecified. SSEM can overcome the problem of model misspecification through sparse estimation. Note that in SSEM, like exploratory factor analysis, 
statisticians make only the minimum assumption to satisfy an identifiability condition. Thus, SSEM may be referred to as ``exploratory structural equation modeling". Huang et.al. \cite{Huang(2017)} showed that the estimator has the oracle property, which means that the true sparsity pattern of the parameters is correctly specified asymptotically. See Fan and Li \cite{Fan(2001)} for the oracle property. In addition, they proposed a test statistic for the goodness-of-fit test in SSEM and examined its asymptotic properties.
 
In behavioral science, factor analysis for time series data has been actively studied; see, e.g., Molenaar \cite{Molenaar(1985)} and Pena and box \cite{Pena(1987)}. 
Moreover, Czi{\'a}ky \cite{Cziraky(2004)} proposed SEM for time series data called dynamic structural equation model with latent variables (DSEM). Asparouhov et.al. \cite{Asparouhov} studied a more general DSEM model. Recently, we can easily obtain high-frequency data such as stock price data and life-log data (blood pressure and EEG, etc.) thanks to the development of measuring devices, and statistical inference for stochastic differential equations based on high-frequency data has been developed. For parametric estimation of diffusion processes based on high-frequency data, see for example, Yoshida \cite{Yoshida(1992)}, Genon-Catalot and Jacod \cite{Genon(1993)}, Kessler \cite{kessler(1997)}, Uchida and Yoshida \cite{Uchi-Yoshi(2012)} and references therein. De Gregorio and Iacus \cite{De Gregorio(2012)} and Masuda and Shimizu \cite{Masuda(2017)} studied sparse estimation of diffusion processes based on high-frequency data. Suzuki and Yoshida \cite{Suzuki(2020)} considered a more general situation and proposed a new method that has a computational advantage. In financial econometrics, the factor model for high-frequency data has been extensively researched. In this field, parameters and the number of factors are estimated by using principal component analysis for high-frequency data (A{\"i}t-Sahalia and Xiu \cite{Ait(2019)}) when the factors are latent; see, e.g., A{\"i}t-Sahalia and Xiu \cite{Ait(2017)}. However, these studies are based on high dimensionality. 
For a low-dimensional model, the estimator does not have consistency; see Bai \cite{Bai(2003)}. On the other hand, Kusano and Uchida \cite{Kusano(2022)} proposed classical factor analysis for diffusion processes. Their method works well for a low-dimensional model. However, to the best of our knowledge, there have been few studies of SEM and SSEM for high-frequency data.
Oud and Jansen \cite{Oud(2000)} and Driver et.al. \cite{Driver(2017)} considered SEM for stochastic differential equations. 
Note that their model differs from the model in this paper.
In the field of causal inference, Hansen and Sokol \cite{Hansen(2014)} studied SEM for stochastic differential equations. 
However, since their model is the path analysis model, it cannot describe the relationship between latent variables. Note that these studies do not assume that the data is sampled with high-frequency. On the other hand, we propose SEM and SSEM for diffusion processes based on high-frequency data.

In this paper, we assume that the volatilities for diffusion processes and loading matrices are not time-variant but constant to simplify the discussion.
We leave for future work the discussion on the model 
where the volatilities for diffusion processes and loading matrices are time-varying. Furthermore, we do not discuss a high-dimensional case. Bai \cite{bai(2012)} studied the asymptotic properties of factor analysis based on the maximum likelihood estimation for a high-dimension model. We expect that our quasi-likelihood method will also work well for a high-dimension model. The investigation is future work.

The paper is organized as follows. In Section 2, notation and assumptions are introduced. 
In Section 3, we study SEM for diffusion processes in the ergodic and non-ergodic cases. 
First, the asymptotic properties of the realized covariance are examined. 
Next, we obtain the quasi-likelihood estimators for parameters in the SEM. 
Asymptotic properties of the estimators are shown. Furthermore, we propose a goodness-of-fit test based on the quasi-likelihood ratio and investigate its asymptotic properties. In section 4, SSEM for diffusion processes in the ergodic and non-ergodic cases is discussed. 
We propose a goodness-of-fit test for SSEM and study its asymptotic properties. In Sections 5 and 6, we give examples and simulation studies to investigate the asymptotic performance of the results described in Sections 3 and 4. Section 7 is devoted to the proofs of the theorems given in Sections 3 and 4.
\section{Preliminaries}
First, we prepare the following notations and definitions.
For any vector $v$, $|v|=\sqrt{\tr{vv^\top}}$, $v^{(i)}$ is the $i$-th element of $v$ and $\Diag v$ is a diagonal matrix, whose $i$-th diagonal element is $v^{(i)}$. For any matrix A, $\|A\|=\sqrt{\tr{AA^\top}}$, and $A_{ij}$ is the $(i,j)$-th element of $A$. For any matrix $A\in\mathbb{R}^{p\times p}$, $\diag A$ is a $p$-dimensional vector, whose $i$-th element is $A_{ii}$. For any set $S$, $|S|$ is the number of elements in $S$.
Define $O_{p\times q}$ as the $p\times q$ zero matrix. For any symmetric matrix $A\in\mathbb{R}^{p\times p}$, $\vec A$,  $\vech A$ and $\mathbb{D}_{p}$ denote the vectorization of $A$, the half-vectorization of $A$ and the $p^2\times\bar{p}$ duplication matrix, respectively. Here, $\vec{A}=\mathbb{D}_{p}\vech{A}$ and $\bar{p}=p(p+1)/2$; see, e.g., Harville \cite{Harville(1998)}. For any matrix $A$, $A^{+}$ denotes the Moore-Penrose inverse of $A$. If $A$ is a positive definite matrix, we write $A>0$. For any positive sequence $u_{n}$, 
$R:[0,\infty)\times \mathbb{R}^d\rightarrow \mathbb{R}$ denotes the short notation for functions which satisfy $|R({u_{n}},x)|\leq u_{n}C(1+|x|)^C$ for some $C>0$. Let $ C^{k}_{\uparrow}(\mathbb R^{d})$ be the space of all functions $f$ satisfying the following conditions:
\begin{itemize}
    \item[(i)] $f$ is continuously differentiable with respect to $x\in \mathbb{R}^d$ up to order $k$. 
    \item[(ii)] $f$ and all its derivatives are of polynomial growth in $x\in \mathbb{R}^d$, i.e., 
    $g$ is
of polynomial growth in $x\in \mathbb{R}^d$ if $\displaystyle g(x)=R(1,x)$. 
\end{itemize}
$N_{p}(\mu,\Sigma)$ represents the $p$-dimensional normal random variable with mean $\mu\in\mathbb{R}^p$ and covariance matrix $\Sigma\in\mathbb{R}^{p\times p}$. 
Let $\chi^2_{r}$ be  the random variable which has 
the chi-squared distribution with $r$ degrees of freedom. $\chi^2_{r}(\alpha)$ denotes an upper $\alpha$ point of the chi-squared distribution with $r$ degrees of freedom, where $0\leq\alpha\leq 1$. The symbols $\stackrel{P}{\longrightarrow}$ and $\stackrel{d}{\longrightarrow}$ denote convergence in probability and convergence in distribution, respectively.  Set the true value of the covariance structure ${\bf{\Sigma}}$ as
\begin{align}
    {\bf{\Sigma}}_0=\begin{pmatrix}
    {\bf{\Sigma}}_0^{11} & {\bf{\Sigma}}_0^{12}\\
    {\bf{\Sigma}}_0^{12\top} & {\bf{\Sigma}}_0^{22}
    \end{pmatrix},\label{sigma0}
\end{align}
where 
\begin{align*}                      
    \qquad\qquad{\bf{\Sigma}}_0^{11}&
    ={\bf{\Lambda}}_{x_1,0}{\bf{\Sigma}}_{\xi\xi,0}{\bf{\Lambda}}_{x_1,0}^{\top}
    +{\bf{\Sigma}}_{\delta\delta,0},\\  {\bf{\Sigma}}_0^{12}&={\bf{\Lambda}}_{x_1,0}{\bf{\Sigma}}_{\xi\xi,0}{\bf{\Gamma}}_{0}^{\top}{\bf{\Psi}}_0^{-1\top}{\bf{\Lambda}}_{x_2,0}^{\top},\\
    {\bf{\Sigma}}_0^{22}&={\bf{\Lambda}}_{x_2,0}{\bf{\Psi}}_0^{-1}({\bf{\Gamma}}_0{\bf{\Sigma}}_{\xi\xi,0}{\bf{\Gamma}}_0^{\top}+{\bf{\Sigma}}_{\zeta\zeta,0}){\bf{\Psi}}_0^{-1\top}{\bf{\Lambda}}_{x_2,0}^{\top}+{\bf{\Sigma}}_{\varepsilon\varepsilon,0}.
\end{align*}
Next, we make the following assumptions.
\begin{enumerate}
    \vspace{2mm}
    \item[\bf{[A1]}]
    \begin{enumerate}
    \item[(i)] There exists a constant $C>0$ such that for any $x,y\in\mathbb R^{k_{1}}$, \begin{align*}
    |B_{1}(x)-B_{1}(y)|\leq C|x-y|.
    \end{align*}
    \item[(ii)] For all $\ell\geq 0$,
    $\displaystyle\sup_t\E\bigl[|\xi_{t}|^{\ell}\bigr]<\infty$.
    \item[(iii)] $B_{1}\in C^{4}_{\uparrow}(\mathbb R^{k_{1}})$.
    \end{enumerate}
    \vspace{2mm}
    \item[\bf{[A2]}]
    The diffusion process  $\xi_{t}$ is ergodic with its invariant measure $\pi_{\xi}$: For any $\pi_{\xi}$-integrable function $g$, it holds that
    \begin{align*}
    \frac{1}{T}\int_{0}^{T}{g(\xi_{t})dt}\overset{P}{\longrightarrow}\int g(x)\pi_{\xi}(dx)
    \end{align*}
    as $T\longrightarrow\infty$. 
    \vspace{2mm}
    \item[\bf{[B1]}]
    \begin{enumerate}
    \item[(i)] There exists a constant $C>0$ such that for any $x,y\in\mathbb R^{p_1}$, 
    \begin{align*}
    |B_{2}(x)-B_{2}(y)|\leq C|x-y|.
    \end{align*}
    \item[(ii)] For all $\ell\geq 0$, 
    $\displaystyle\sup_t\E\bigl[|\delta_{t}|^{\ell}\bigr]<\infty$.
    \item[(iii)] $B_{2}\in C^{4}_{\uparrow}(\mathbb R^{p_1})$.
    \end{enumerate}
    \vspace{2mm}
    \item[\bf{[B2]}] The diffusion process  $\delta_{t}$ is ergodic with its invariant measure $\pi_{\delta}$: For any $\pi_{\delta}$-integrable function $g$, it holds that
    \begin{align*}
    \frac{1}{T}\int_{0}^{T}{g(\delta_{t})dt}\overset{P}{\longrightarrow}\int g(x)\pi_{\delta}(dx)
    \end{align*}
    as $T\longrightarrow\infty$.
    \vspace{2mm}
    \item[\bf{[C1]}]\begin{enumerate}
    \item[(i)] There exists a constant $C>0$ such that for any $x,y\in\mathbb R^{p_2}$, 
    \begin{align*}
    |B_{3}(x)-B_{3}(y)|\leq C|x-y|.
    \end{align*}
    \item[(ii)] For all $\ell\geq 0$, 
    $\displaystyle\sup_t\E\bigl[|\varepsilon_{t}|^{\ell}\bigr]<\infty$.
    \item[(iii)] $B_{3}\in C^{4}_{\uparrow}(\mathbb R^{p_2})$.
    \end{enumerate}
    \vspace{2mm}
    \item[\bf{[C2]}]
    The diffusion process  $\varepsilon_{t}$ is ergodic with its invariant measure $\pi_{\varepsilon}$: For any $\pi_{\varepsilon}$-integrable function $g$, it holds that
    \begin{align*}
    \frac{1}{T}\int_{0}^{T}{g(\varepsilon_{t})dt}\overset{P}{\longrightarrow}\int g(x)\pi_{\varepsilon}(dx)
    \end{align*}
    as $T\longrightarrow\infty$.
    \vspace{2mm}
    \item[\bf{[D1]}]
    \begin{enumerate}
    \item[(i)] There exists a constant $C>0$ such that for any $x,y\in\mathbb R^{k_{2}}$, 
    \begin{align*}
    |B_{4}(x)-B_{4}(y)|\leq C|x-y|.
    \end{align*}
    \item[(ii)] For all $\ell\geq 0$, $\displaystyle\sup_t\E\bigl[|\zeta_{t}|^{\ell}\bigr]<\infty$.
    \item[(iii)] $B_{4}\in C^{4}_{\uparrow}(\mathbb R^{k_{2}})$.
    \end{enumerate}
    \vspace{2mm}
    \item[\bf{[D2]}]
    The diffusion process  $\zeta_{t}$ is ergodic with its invariant measure $\pi_{\zeta}$: For any $\pi_{\zeta}$-integrable function $g$, it holds that
    \begin{align*}
    \frac{1}{T}\int_{0}^{T}{g(\zeta_{t})dt}\overset{P}{\longrightarrow}\int g(x)\pi_{\zeta}(dx)
    \end{align*}
    as $T\longrightarrow\infty$.
    \vspace{2mm}
\end{enumerate}
\begin{remark}
For diffusion processes, Assumptions $[{\bf{A1}}]$, $[{\bf{B1}}]$, $[{\bf{C1}}]$ and $[{\bf{D1}}]$ are the standard assumptions; see, e.g., Kessler \cite{kessler(1997)}.
\end{remark}
\section{SEM for diffusion processes}
First, we investigate the non-ergodic case, where $[{\bf{A2}}]$, $[{\bf{B2}}]$, $[{\bf{C2}}]$ and $[{\bf{D2}}]$ are not assumed and $T$ is fixed. To estimate ${\bf{\Sigma}}_0$, we use the realized covariance as follows:
\begin{align}
    \mathbb{Q}_{\mathbb{XX}}=\frac{1}{T}\sum_{i=1}^{n}(\mathbb{X}_{t_{i}^n}-\mathbb{X}_{t_{i-1}^n})(\mathbb{X}_{t_{i}^n}-\mathbb{X}_{t_{i-1}^n})^{\top}.
    \label{Qx}
\end{align}
For the realized covariance, the following theorem holds.
\begin{theorem}\label{Qtheoremnon}
Under $[{\bf{A1}}]$, $[{\bf{B1}}]$, $[{\bf{C1}}]$ and $[{\bf{D1}}]$, as $h_n\longrightarrow0$,
\begin{align*}
    \mathbb{Q}_{\mathbb{XX}}\stackrel{P}{\longrightarrow} {\bf{\Sigma}}_0
\end{align*}
and
\begin{align*}
    \sqrt{n}(\vech{\mathbb{Q}_{\mathbb{XX}}}-\vech{{\bf{\Sigma}}_0})\stackrel{d}{\longrightarrow} N_{\bar{p}}\Bigl(0,2\mathbb{D}_{p}^{+}({\bf{\Sigma}}_0\otimes{\bf{\Sigma}}_0)\mathbb{D}_{p}^{+\top}\Bigr).
\end{align*}
\end{theorem}
Next, we consider the parameter estimation. The parameters ${\bf{\Lambda}}_{x_1}$, ${\bf{\Lambda}}_{x_2}$, ${\bf{\Gamma}}$, ${\bf{\Psi}}$, ${\bf{\Sigma}}_{\xi\xi}$, ${\bf{\Sigma}}_{\delta\delta}$, ${\bf{\Sigma}}_{\varepsilon\varepsilon}$ and ${\bf{\Sigma}}_{\zeta\zeta}$ are estimated. Note that some of these elements are assumed to be known in order to satisfy an identifiability condition for parameter estimation. See Remark \ref{identification} for constraints on the parameter and the identifiability condition. Set the parameter as $\theta\in\Theta$, where $\Theta\subset\mathbb{R}^{q}$ is a convex compact space. $\theta$ includes only unknown and non-duplicated elements of ${\bf{\Lambda}}_{x_1}$, ${\bf{\Lambda}}_{x_2}$, ${\bf{\Gamma}}$, ${\bf{\Psi}}$, ${\bf{\Sigma}}_{\xi\xi}$, ${\bf{\Sigma}}_{\delta\delta}$, ${\bf{\Sigma}}_{\varepsilon\varepsilon}$ and ${\bf{\Sigma}}_{\zeta\zeta}$. Set the covariance structure of the parametric model as
\begin{align}
    {\bf{\Sigma}}(\theta)=\begin{pmatrix}
    {\bf{\Sigma}}^{11}(\theta) & {\bf{\Sigma}}^{12}(\theta)\\
    {\bf{\Sigma}}^{12}(\theta)^{\top} & {\bf{\Sigma}}^{22}(\theta)
    \end{pmatrix},\label{sigmatheta}
\end{align}
where 
\begin{align*}
    \qquad\quad{\bf{\Sigma}}^{11}(\theta)&={\bf{\Lambda}}_{x_1}{\bf{\Sigma}}_{\xi\xi}{\bf{\Lambda}}_{x_1}^{\top}+{\bf{\Sigma}}_{\delta\delta},\\
    {\bf{\Sigma}}^{12}(\theta)&={\bf{\Lambda}}_{x_1}{\bf{\Sigma}}_{\xi\xi}{\bf{\Gamma}}^{\top}{\bf{\Psi}}^{-1\top}{\bf{\Lambda}}_{x_2}^{\top},\\
    {\bf{\Sigma}}^{22}(\theta)&={\bf{\Lambda}}_{x_2}{\bf{\Psi}}^{-1}({\bf{\Gamma}}{\bf{\Sigma}}_{\xi\xi}{\bf{\Gamma}}^{\top}+{\bf{\Sigma}}_{\zeta\zeta}){\bf{\Psi}}^{-1\top}{\bf{\Lambda}}_{x_2}^{\top}+{\bf{\Sigma}}_{\varepsilon\varepsilon}.
\end{align*}
Suppose that there exists $\theta_{0}\in\Int \Theta$ such that
\begin{align*}
    {\bf{\Sigma}}_0={\bf{\Sigma}}(\theta_{0}).
\end{align*}
Note that ${\bf{\Sigma}}_0$ and ${\bf{\Sigma}}(\theta)$ are positive definite matrices; see Lemma \ref{Sigmaposlemma}. Define the quasi-likelihood function of the parametric model as
\begin{align}
    \mathbb{L}_{n}(\theta)=\prod_{i=1}^{n}\frac{1}{(2\pi)^{\frac{p}{2}}\det{(h_n{\bf{\Sigma}}(\theta)})^{\frac{1}{2}}}\exp{\left\{-\frac{1}{2h_n}(\mathbb{X}_{t_{i}^n}-\mathbb{X}_{t_{i-1}^n})^{\top}{\bf{\Sigma}}(\theta)^{-1}(\mathbb{X}_{t_{i}^n}-\mathbb{X}_{t_{i-1}^n})\right\}}.\label{quasi}
\end{align}
It holds
\begin{align}
    \log\mathbb{L}_{n}(\theta)=-\frac{pn}{2}\log(2\pi)-\frac{pn}{2}\log h_n-\frac{n}{2}\log\det{\bf{\Sigma}}(\theta)-\frac{n}{2}\tr\bigl\{{\bf{\Sigma}}(\theta)^{-1}\mathbb{Q}_{\mathbb{XX}}\bigr\}.\label{quasilog}
\end{align}
See Appendix \ref{quasi-likelihood} for details of (\ref{quasi}) and (\ref{quasilog}). Let
\begin{align*}
    \ell_{n}({\bf{\Sigma}})=-\frac{pn}{2}\log(2\pi)-\frac{pn}{2}\log h_n-\frac{n}{2}\log\det{\bf{\Sigma}}-\frac{n}{2}\tr\bigl\{{\bf{\Sigma}}^{-1}\mathbb{Q}_{\mathbb{XX}}\bigr\}.
\end{align*}
Note that $\ell_{n}({\bf{\Sigma}})$ has a maximum value
\begin{align*}
    -\frac{pn}{2}\log(2\pi)-\frac{pn}{2}\log h_n-\frac{n}{2}\log\det \mathbb{Q}_{\mathbb{XX}}-\frac{np}{2}
\end{align*}
at ${\bf{\Sigma}}=\mathbb{Q}_{\mathbb{XX}}$ as $\mathbb{Q}_{\mathbb{XX}}>0$. Define the following function:
\begin{align}  
    \begin{split}
    \rm{F}(\mathbb{Q}_{\mathbb{XX}},{\bf{\Sigma}}(\theta))&=-\frac{2}{n}\log\mathbb{L}_{n}(\theta)+\frac{2}{n}\left\{-\frac{pn}{2}\log(2\pi)-\frac{pn}{2}\log h_n-\frac{n}{2}\log\det \mathbb{Q}_{\mathbb{XX}}-\frac{np}{2}\right\}\\
    &=\log\det{\bf{\Sigma}}(\theta)-\log\det \mathbb{Q}_{\mathbb{XX}}+\tr{\bigl\{{\bf{\Sigma}}(\theta)^{-1}\mathbb{Q}_{\mathbb{XX}}\bigr\}}-p.\label{F1}
    \end{split}
\end{align}
From Theorem 1 in Shapiro \cite{Shapiro(1985)}, (\ref{F1}) is rewritten as
\begin{align}
    \rm{F}(\mathbb{Q}_{\mathbb{XX}},{\bf{\Sigma}}(\theta))=(\vech{\mathbb{Q}_{\mathbb{XX}}}-\vech{{\bf{\Sigma}}(\theta)})^{\top}\rm{V}(\mathbb{Q}_{\mathbb{XX}},{\bf{\Sigma}}(\theta))(\vech{\mathbb{Q}_{\mathbb{XX}}}-\vech{{\bf{\Sigma}}(\theta})) \label{FQ}
\end{align}
as $\mathbb{Q}_{\mathbb{XX}}>0$, where 
\begin{align*}
    \rm{V}(\mathbb{Q}_{\mathbb{XX}},{\bf{\Sigma}}(\theta))
    &=\mathbb{D}_{p}^{\top}\int_{0}^{1}\int_{0}^{1}\lambda_2({\bf{\Sigma}}(\theta)+\lambda_1\lambda_2(\mathbb{Q}_{\mathbb{XX}}-{\bf{\Sigma}}(\theta)))^{-1}\\
    &\qquad\qquad\qquad\qquad\otimes({\bf{\Sigma}}(\theta)+\lambda_1\lambda_2(\mathbb{Q}_{\mathbb{XX}}-{\bf{\Sigma}}(\theta)))^{-1}d\lambda_{1}d\lambda_2\mathbb{D}_{p}
\end{align*}
as $\mathbb{Q}_{\mathbb{XX}}>0$. Moreover, set the following function:
\begin{align*}
    \tilde{\rm{F}}(\mathbb{Q}_{\mathbb{XX}},{\bf{\Sigma}}(\theta))=(\vech{\mathbb{Q}_{\mathbb{XX}}}-\vech{{\bf{\Sigma}}(\theta)})^{\top}\tilde{\rm{V}}(\mathbb{Q}_{\mathbb{XX}},{\bf{\Sigma}}(\theta))(\vech{\mathbb{Q}_{\mathbb{XX}}}-\vech{{\bf{\Sigma}}(\theta})),
\end{align*}
where 
\begin{align*}
    \tilde{\rm{V}}(\mathbb{Q}_{\mathbb{XX}},{\bf{\Sigma}}(\theta))=\left\{
    \begin{array}{ll}
    \rm{V}(\mathbb{Q}_{\mathbb{XX}},{\bf{\Sigma}}(\theta)),
    & (\mathbb{Q}_{\mathbb{XX}} \mbox{ is non-singular}), \\
    \mathbb{I}_{\bar{p}},  & (\mathbb{Q}_{\mathbb{XX}}\mbox{ is singular}).
    \end{array}\right. 
\end{align*}
The contrast function is given by
\begin{align}
    \mathbb{F}_{n}(\theta)=\tilde{\rm{F}}(\mathbb{Q}_{\mathbb{XX}},{\bf{\Sigma}}(\theta)). \label{F}
\end{align}
The minimum contrast estimator $\hat{\theta}_{n}$ is defined as
\begin{align}
    \mathbb{F}_{n}(\hat{\theta}_n)=\inf_{\theta\in\Theta}\mathbb{F}_{n}(\theta). \label{thetahat}
\end{align}
Set 
\begin{align}
    \Delta&=\left.\partial_{\theta}\vech{{\bf{\Sigma}}(\theta)}\right|_{\theta=\theta_0}, \label{Delta}
\end{align}
where $\partial_{\theta}=\partial/\partial\theta$. Let
\begin{align}
    {\bf{W}}(\theta)=2\mathbb{D}_{p}^{+}({\bf{\Sigma}}(\theta)\otimes{\bf{\Sigma}}(\theta))\mathbb{D}_{p}^{+\top}. \label{Wtheta}
\end{align}
Furthermore, we make the following assumptions.
\begin{enumerate}
    \vspace{2mm}
    \item[\bf{[E1]}] 
    \begin{enumerate}
    \item[(i)] ${\bf{\Sigma}}(\theta_1)={\bf{\Sigma}}(\theta_2)\Longrightarrow
    \theta_1=\theta_2$.
    \vspace{2mm}
    \item[(ii)] $\rank{\Delta}=q$.
    \end{enumerate}
    \vspace{2mm}
\end{enumerate}
\begin{remark}\label{identification}
Assumption $[{\bf{E1}}]$ (i) is an identifiability condition for parameter estimation and implies the consistency of the minimum contrast estimator $\hat{\theta}_{n}$. 
Like the factor model, the LISREL model does not have the identifiability condition for parameter estimation when the parameters are unconstrained.
To satisfy $[{\bf{E1}}]$ (i), some parameters may be fixed to 0 or 1, or some parameters are assumed to be the same value as other parameters. These constraints are determined from the theoretical viewpoint of each research field, see Section 5 for an example of a model that satisfies $[{\bf{E1}}]$ (i). Unfortunately, in the LISREL model, simple sufficient conditions for $[{\bf{E1}}]$ (i) are not known. For the identification problem, e.g., see Everitt \cite{Everitt(1984)}. Assumption $[{\bf{E1}}]$ (ii) implies that $\Delta^{\top}{\bf{W}}(\theta_0)^{-1}\Delta$ is non-singular, see Lemma \ref{Aposlemma}.
\end{remark}
For the minimum contrast estimator, we obtain the following theorem.
\begin{theorem}\label{thetatheoremnon}
Under $[{\bf{A1}}]$, $[{\bf{B1}}]$, $[{\bf{C1}}]$, $[{\bf{D1}}]$ and $[{\bf{E1}}]$, as $h_n\longrightarrow0$, 
\begin{align*}
    \hat{\theta}_{n}\stackrel{P}{\longrightarrow} \theta_{0}
\end{align*}
and
\begin{align*}
    \sqrt{n}(\hat{\theta}_{n}-\theta_{0})\stackrel{d}{\longrightarrow}N_{q}\Bigl(0,\bigl(\Delta^{\top}{\bf{W}}(\theta_0)^{-1}\Delta\bigr)^{-1}\Bigr).
\end{align*}
\end{theorem}
Next, we consider the goodness-of-fit test. The statistical hypothesis test is as follows:
\begin{align}
    \left\{
    \begin{array}{ll}
    H_0: {\bf{\Sigma}}={\bf{\Sigma}}(\theta),\\
    H_1: {\bf{\Sigma}}\neq{\bf{\Sigma}}(\theta).
    \end{array}
    \right.\label{hypothesis}
\end{align}
The quasi-likelihood ratio $\Lambda_{n}$ is defined as
\begin{align*}
    \Lambda_{n}=\frac{\max_{\theta\in\Theta}\rm{L}_{n}({\bf{\Sigma}}(\theta))}{\max_{{\bf{\Sigma}}>0}\rm{L}_{n}({\bf{\Sigma}})},
\end{align*}
where 
\begin{align*}
    {\rm{L}}_{n}({\bf{\Sigma}})=\prod_{i=1}^{n}\frac{1}{(2\pi )^{\frac{p}{2}}\det{(h_n{\bf{\Sigma}}})^{\frac{1}{2}}}\exp{\left\{-\frac{1}{2h_n}(\mathbb{X}_{t_{i}^n}-\mathbb{X}_{t_{i-1}^n})^{\top}{\bf{\Sigma}}^{-1}(\mathbb{X}_{t_{i}^n}-\mathbb{X}_{t_{i-1}^n})\right\}}.
\end{align*}
It follows that
\begin{align}
\begin{split}
    -2\log\Lambda_{n}&=-2\max_{\theta\in\Theta}\log \rm{L}_{n}({\bf{\Sigma}}(\theta))+2\max_{{\bf{\Sigma}}>0}\log \rm{L}_{n}({\bf{\Sigma}})\\
    &=-2\left\{-\frac{pn}{2}\log(2\pi)-\frac{pn}{2}\log h_n-\frac{n}{2}\log\det{{\bf{\Sigma}}(\hat{\theta}_{n})}
    -\frac{n}{2}\tr{\bigl\{{\bf{\Sigma}}(\hat{\theta}_{n})^{-1}\mathbb{Q}_{\mathbb{XX}}\bigr\}}\right\}\\
    &\quad +2\left\{-\frac{pn}{2}\log(2\pi)-\frac{pn}{2}\log h_n-\frac{n}{2}\log\det{\mathbb{Q}_{\mathbb{XX}}}-\frac{np}{2}\right\}\\
    &=n\left\{\log\det{{\bf{\Sigma}}(\hat{\theta}_{n})}-\log\det{\mathbb{Q}_{\mathbb{XX}}+\tr{\bigl\{{\bf{\Sigma}}(\hat{\theta}_{n})^{-1}\mathbb{Q}_{\mathbb{XX}}\bigr\}}-p}\right\}\\
    &=n\rm{F}(\mathbb{Q}_{\mathbb{XX}},{\bf{\Sigma}}(\hat{\theta}_n))
\end{split}\label{ratio}
\end{align}
as $\mathbb{Q}_{\mathbb{XX}}>0$. The quasi-likelihood ratio test statistic is given by
\begin{align}
    \mathbb{T}_{n}=n\mathbb{F}_{n}(\hat{\theta}_{n}). \label{T}
\end{align}
We have the following asymptotic result of the test statistic $\mathbb{T}_{n}$.
\begin{theorem}\label{chitheoremnon}
Under $[{\bf{A1}}]$, $[{\bf{B1}}]$, $[{\bf{C1}}]$, $[{\bf{D1}}]$ and $[{\bf{E1}}]$, as $h_n\longrightarrow0$, 
\begin{align*}
    \mathbb{T}_{n}\stackrel{d}{\longrightarrow}\chi^2_{\bar{p}-q}
\end{align*}
under $H_0$.
\end{theorem}
From Theorem \ref{chitheoremnon}, the test of asymptotic significance level $\alpha\in(0,1)$ is constructed. The rejection region is set to
\begin{align*}
    \Bigl\{\mathbb{T}_{n}>\chi^2_{\bar{p}-q}(\alpha)\Bigr\}.
\end{align*}

Finally, we investigate the consistency of the test.
Let $\mathbb{U}(\theta)=\rm{F}({\bf{\Sigma}}_0,{\bf{\Sigma}}(\theta))$. $\bar{\theta}$ is defined as
\begin{align}
    \mathbb{U}(\bar{\theta})=\inf_{\theta\in\Theta} \mathbb{U}(\theta). \label{U}
\end{align}
In addition, we make the following assumption:
\begin{enumerate}
    \vspace{2mm}
    \item[\bf{[E2]}]  $\mathbb{U}(\theta_1)=\mathbb{U}(\theta_2)\Longrightarrow \theta_1=\theta_2$.
    \vspace{2mm}
\end{enumerate}
\begin{remark}
Assumption $[{\bf{E2}}]$ implies that $\hat{\theta}_{n}\stackrel{P}{\longrightarrow}\bar{\theta}$ under $H_1$, see Lemma \ref{starconslemma}. 
\end{remark}

We have the following theorem.
\begin{theorem}\label{testtheoremnon}
Under $[{\bf{A1}}]$, $[{\bf{B1}}]$, $[{\bf{C1}}]$, $[{\bf{D1}}]$ and $[{\bf{E2}}]$, as $h_n\longrightarrow0$, 
\begin{align*}
    \PP\Bigl(\mathbb{T}_{n}>\chi^2_{\bar{p}-q}(\alpha)\Bigr)\stackrel{}{\longrightarrow}1
\end{align*}
under $H_1$.
\end{theorem}
\begin{remark}
The goodness-of-fit test has several problems. 
See, e.g., Bentler and Bonett \cite{Bentler(1980)} for problems with 
the goodness-of-fit test. 
However, the goodness-of-fit test is one of 
the most popular methods for model evaluation in SEM; 
see, e.g., Mcdonald \cite{McDonald(2002)}. Thus, we consider only the goodness-of-fit test as a model evaluation method in this paper and leave the other methods for future work.
\end{remark}
In the ergodic case, the following results similar to the non-ergodic case hold.
\begin{theorem}\label{Qtheorem}
Under $[{\bf{A1}}]$-$[{\bf{A2}}]$, $[{\bf{B1}}]$-$[{\bf{B2}}]$, $[{\bf{C1}}]$-$[{\bf{C2}}]$ and $[{\bf{D1}}]$-$[{\bf{D2}}]$, as $h_n\longrightarrow0$ and $nh_n\longrightarrow\infty$,
\begin{align*}
    \mathbb{Q}_{\mathbb{XX}}\stackrel{P}{\longrightarrow} {\bf{\Sigma}}_0.
\end{align*}
In addition, as $nh_n^2\longrightarrow0$,
\begin{align*}
    \sqrt{n}(\vech{\mathbb{Q}_{\mathbb{XX}}}-\vech{{\bf{\Sigma}}_0})\stackrel{d}{\longrightarrow} 
    N_{\bar{p}}\Bigl(0,2\mathbb{D}_{p}^{+}({\bf{\Sigma}}_0\otimes{\bf{\Sigma}}_0)\mathbb{D}_{p}^{+\top}\Bigr). 
\end{align*}
\end{theorem}
\begin{theorem}\label{thetatheorem}
Under $[{\bf{A1}}]$-$[{\bf{A2}}]$, $[{\bf{B1}}]$-$[{\bf{B2}}]$, $[{\bf{C1}}]$-$[{\bf{C2}}]$, $[{\bf{D1}}]$-$[{\bf{D2}}]$ and $[{\bf{E1}}]$, as $h_n\longrightarrow0$ and $nh_n\longrightarrow\infty$, 
\begin{align*}
    \hat{\theta}_{n}\stackrel{P}{\longrightarrow}\theta_{0}.
\end{align*}
In addition, as $nh_n^2\longrightarrow0$,
\begin{align*}
    \sqrt{n}(\hat{\theta}_{n}-\theta_{0})\stackrel{d}{\longrightarrow}N_{q}\Bigl(0,\bigl(\Delta^{\top}{\bf{W}}(\theta_0)^{-1}\Delta\bigr)^{-1}\Bigr).
\end{align*}
\end{theorem}
\begin{theorem}\label{testtheorem}
Under $[{\bf{A1}}]$-$[{\bf{A2}}]$, $[{\bf{B1}}]$-$[{\bf{B2}}]$, $[{\bf{C1}}]$-$[{\bf{C2}}]$, $[{\bf{D1}}]$-$[{\bf{D2}}]$ and $[{\bf{E1}}]$, as $h_n\longrightarrow0$, $nh_n\longrightarrow\infty$ and $nh_n^2\longrightarrow 0$,
\begin{align*}
    \mathbb{T}_{n}\stackrel{d}{\longrightarrow}\chi^2_{\bar{p}-q}
\end{align*}
under $H_0$.
\end{theorem}
\begin{theorem}\label{testconstheorem}
Under $[{\bf{A1}}]$-$[{\bf{A2}}]$, $[{\bf{B1}}]$-$[{\bf{B2}}]$, $[{\bf{C1}}]$-$[{\bf{C2}}]$, $[{\bf{D1}}]$-$[{\bf{D2}}]$ and $[{\bf{E2}}]$, as $h_n\longrightarrow0$ and $nh_n\longrightarrow\infty$, 
\begin{align*}
    \PP\Bigl(\mathbb{T}_{n}>\chi^2_{\bar{p}-q}(\alpha)\Bigr)\stackrel{}{\longrightarrow}1
\end{align*}
under $H_1$.
\end{theorem}
\begin{remark}
Generally, in the ergodic and non-ergodic cases, the limit distributions of $\mathbb{Q}_{\mathbb{XX}}$ and $\hat{\theta}_n$ are different. In our setting, they converge to the same distribution for both cases. It is because the volatility is constant in our setting.
\end{remark}
\section{SSEM for diffusion processes}
In this section, we study sparse estimation in SEM for diffusion processes.  Define $\mathbb{F}_{n}(\theta)$, $\hat{\theta}_n$ and $\mathbb{U}(\bar{\theta})$ by (\ref{F}), (\ref{thetahat}) and (\ref{U}) respectively. First, we consider the non-ergodic case.
The lasso estimator (Tibshirani \cite{Tibshirani(2006)}) is defined as follows:
\begin{align*}
    \hat{\theta}_{lasso,n}\in\arginf_{\theta\in\Theta}\Bigl\{\mathbb{F}_n(\theta)+\lambda_n\sum_{j=1}^{q}|\theta^{(j)}|\Bigr\},
\end{align*}
where $\lambda_n>0$. Lasso does not have the oracle property; see, e.g., Zou \cite{Zou(2006)}. Set
\begin{align*}
    \qquad\quad\kappa_{n}^{(j)}&=\lambda_{1,n}|\hat{\theta}_{n}^{(j)}|^{-\gamma}1_{\bigl\{|\hat{\theta}_{n}^{(j)}|\geq\delta\bigr\}}+\lambda_{2,n}1_{\bigl\{|\hat{\theta}_{n}^{(j)}|<\delta\bigr\}}
\end{align*}
for $j=1,\cdots, q$, where $\lambda_{1,n}>0$, $\lambda_{2,n}>0$, $\gamma>0$ and $\delta>0$. The adaptive lasso estimator (Zou \cite{Zou(2006)}) is defined by
\begin{align}
    \hat{\theta}_{adaptive,n}\in\arginf_{\theta\in\Theta}\Bigl\{\mathbb{F}_n(\theta)+\sum_{j=1}^{q}\kappa_{n}^{(j)}|\theta^{(j)}|\Bigr\}.\label{adalasso}
\end{align}
Although it needs the initial estimator $\hat{\theta}_n$, the adaptive lasso has the oracle property. Wang and  Leng \cite {Wang(2007)} showed that the contrast function of the adaptive lasso is rewritten as
\begin{align}
    (\theta-\hat{\theta}_{n})^{\top}\partial^2_{\theta}\mathbb{F}_n(\theta)(\theta-\hat{\theta}_{n})+\sum_{j=1}^{q}\kappa_{n}^{(j)}|\theta^{(j)}|,\label{Qadalasso}
\end{align}
where $\partial^2_{\theta}=\partial_{\theta}\partial_{\theta}^{\top}$.
(\ref{Qadalasso}) is much easier to solve numerically than (\ref{adalasso}). Suzuki and Yoshida \cite{Suzuki(2020)} studied the more general setting than Wang and Leng \cite {Wang(2007)}. Let $G_{n}\in\mathbb{R}^{q\times q}$ and the penalized contrast function is given by
\begin{align}
    \mathbb{Q}_{G,n}(\theta)=(\theta-\hat{\theta}_{n})^{\top}\tilde{G}_{n}(\theta-\hat{\theta}_{n})+\sum_{j=1}^{q}\kappa_{n}^{(j)}|\theta^{(j)}|,\label{Q}
\end{align} 
where
\begin{align*}
    \qquad\tilde{G}_{n}&=\begin{cases}
    G_{n},\quad (G_{n}\ \mbox{is a positive definite matrix}),\\
    \mathbb{I}_{q},\quad (G_{n}\ \mbox{is not a positive definite matrix}).
    \end{cases}
\end{align*}
The minimum penalized contrast estimator is defined by
\begin{align*}
    \mathbb{Q}_{G,n}(\tilde{\theta}_{G,n})=\inf_{\theta\in\Theta}\mathbb{Q}_{G,n}(\theta).
\end{align*}
The estimator $\tilde{\theta}_{G,n}$ is called the Penalized Least Squares Approximation (PLSA) estimator. See Suzuki and Yoshida \cite{Suzuki(2020)} for details of the PLSA estimator. In addition, we make the following assumption.
\begin{enumerate}
    \vspace{2mm}
    \item[\bf{[F1]}] There exists a positive definite matrix $G\in\mathbb{R}^{q\times q}$ such that $G_{n}\stackrel{P}{\longrightarrow}G$.
    \vspace{2mm}
\end{enumerate}
Without loss of generality, we suppose that $\theta_0^{(j)}\neq 0$ for $j=1,\cdots,q_0$ and $\theta_0^{(j)}=0$ for $j=q_0+1,\cdots,q$ in this section. Set
\begin{align*}
    \mathcal{F}_{1}&=\Bigl\{j\in\{1,\cdots,q\}\ \big|\ \theta^{(j)}_{0}\neq0\Bigr\}=\Bigl\{1,\cdots, q_0\Bigr\}.
\end{align*}
Assume $0<\delta<\min_{j\in\mathcal{F}_1}|\theta_0^{(j)}|$. For any vector $v\in\mathbb{R}^{q}$ and any matrix $M\in\mathbb{R}^{q\times q}$, 
let $v_{\mathcal{F}_{1}}=(v^{(i)})_{1\leq j\leq q_0}$ and
\begin{align*}
    M=\begin{pmatrix}
    M_{\mathcal{F}}^{11} & M_{\mathcal{F}}^{10}\\
    M_{\mathcal{F}}^{10\top} & M_{\mathcal{F}}^{00}
    \end{pmatrix},
\end{align*}
where $M_{\mathcal{F}}^{11}=(M_{ij})_{1\leq i,j\leq q_0}$, $M_{\mathcal{F}}^{10}=(M_{ij})_{1\leq i\leq q_0,q_0+1\leq j\leq q}$ and $M_{\mathcal{F}}^{00}=(M_{ij})_{q_0+1\leq i,j\leq q}$. Set
\begin{align*}
    \mathfrak{G}_{G}=\begin{pmatrix}
    \mathbb{I}_{q_0} & (G_{\mathcal{F}}^{11})^{-1}G_{\mathcal{F}}^{10}
    \end{pmatrix}
\end{align*}
and
\begin{align*}
    \tilde{\mathcal{F}}_{G,n,1}
    =\Bigl\{j\in\{1,\cdots,q\}\ \big|\  \tilde{\theta}^{(j)}_{G,n}\neq0\Bigr\}.
\end{align*}
Write ${\bf{A}}(\theta_0)=\Delta^{\top}{\bf{W}}(\theta_0)^{-1}\Delta$, where $\Delta$ and ${\bf{W}}(\theta_0)$ are defined by (\ref{Delta}) and (\ref{Wtheta}) respectively. Note that it holds from Lemma \ref{Aposlemma} and Corollary 14.2.12 in Harville \cite{Harville(1998)} that ${\bf{A}}_{\mathcal{F}}^{11}(\theta_{0})$ is a positive definite matrix. PLSA has the oracle property as follows.
\begin{lemma}\label{Aeoracle}
Under \textrm{\textbf{[A1]}}, \textrm{\textbf{[B1]}},
\textrm{\textbf{[C1]}},
\textrm{\textbf{[D1]}},
\textrm{\textbf{[E1]}} and \textrm{\textbf{[F1]}},
as $h_n\longrightarrow0$,
$\sqrt{n}{\lambda}_{1,n}\longrightarrow0$ and
$\sqrt{n}{\lambda}_{2,n}\longrightarrow\infty$,
\begin{align*}
    \PP\Bigl(\tilde{\mathcal{F}}_{G,n,1}=\mathcal{F}_{1}\Bigr)\stackrel{}{\longrightarrow}1.
\end{align*}
\end{lemma}
\vspace{1mm}
\begin{lemma}\label{Aeasym}
Under \textrm{\textbf{[A1]}}, \textrm{\textbf{[B1]}},
\textrm{\textbf{[C1]}},
\textrm{\textbf{[D1]}},
\textrm{\textbf{[E1]}} and \textrm{\textbf{[F1]}},
as $h_n\longrightarrow0$, 
$\sqrt{n}\lambda_{1,n}\longrightarrow0$ and
$\sqrt{n}\lambda_{2,n}\longrightarrow\infty$,
\begin{align*}
    \sqrt{n}(\tilde{\theta}_{G,n}-\theta_{0})_{\mathcal{F}_{1}}-\mathfrak{G}_{G}\bigl\{\sqrt{n}(\hat{\theta}_{n}-\theta_{0})\bigr\}\stackrel{P}{\longrightarrow}0.
\end{align*}
In addition, as $G={\bf{A}}(\theta_0)$,
\begin{align*}
    \sqrt{n}(\tilde{\theta}_{G,n}-\theta_{0})_{\mathcal{F}_{1}}\stackrel{d}{\longrightarrow} N_{|\mathcal{F}_{1}|}\Bigl(0,{\bf{A}}_{\mathcal{F}}^{11}(\theta_{0})^{-1}\Bigr).
\end{align*}
\end{lemma}
\vspace{1mm}
\begin{remark}
Suppose that $G_{n}=2^{-1}\partial^2_{\theta}\mathbb{F}_{n}(\hat{\theta}_{n})$.
In a similar way to the proof of Theorem \ref{thetatheoremnon}, it holds $G_{n}\stackrel{P}{\longrightarrow}{\bf{A}}(\theta_0)$, which yields $G={\bf{A}}(\theta_0)$.
\end{remark}
\begin{remark}
    Theorem 2 yields 
    \begin{align*}
    \sqrt{n}(\hat{\theta}_{n}-\theta_{0})_{\mathcal{F}_{1}}\stackrel{d}{\longrightarrow} N_{|\mathcal{F}_{1}|}\Bigl(0,({\bf{A}}(\theta_0)^{-1})_{\mathcal{F}}^{11}\Bigr).   
    \end{align*}
    It follows from Theorem 8.5.11 and Theorem 14.8.4 in Harville \cite{Harville(1998)} that
    \begin{align*}
    ({\bf{A}}(\theta_0)^{-1})_{\mathcal{F}}^{11}\geq {\bf{A}}_{\mathcal{F}}^{11}(\theta_{0})^{-1}. 
    \end{align*}
\end{remark}
In (\ref{Q}), as $G_{n}=\mathbb{I}_{q}$, the penalized contrast function is
\begin{align*}
    \mathbb{Q}_{\mathbb{I},n}(\theta)=\sum_{j=1}^{q}\bigl(\theta^{(j)}-\hat{\theta}_{n}^{(j)}\bigr)^2+\sum_{j=1}^{q}\kappa_{n}^{(j)}|\theta^{(j)}|
\end{align*}
and the LSA estimator is defined as
\begin{align*}
    \mathbb{Q}_{\mathbb{I},n}(\tilde{\theta}_{\mathbb{I},n})=\inf_{\theta\in\Theta}\mathbb{Q}_{\mathbb{I},n}(\theta).
\end{align*}
Minimizing $\mathbb{Q}_{\mathbb{I},n}(\theta)$ is easier than doing $\mathbb{Q}_{G,n}(\theta)$ from a computational viewpoint. Note that the asymptotic
covariance matrix of $\tilde{\theta}_{\mathbb{I},n}$ is different from that of $\tilde{\theta}_{G,n}$. Set
\begin{align*}
    \tilde{\Theta}_n=\Bigl\{\theta\in\Theta\ \big|\ \theta^{(j)}=0\  (j\in \tilde{\mathcal{F}}_{\mathbb{I},n,0})\Bigr\},
\end{align*}
where
\begin{align*}
    \tilde{\mathcal{F}}_{\mathbb{I},n,0}=\Bigl\{j\in\{1,\cdots,q\}\ \big|\ \tilde{\theta}^{(j)}_{\mathbb{I},n}=0\Bigr\}.
\end{align*}
Define the Penalized method to Ordinary method (P-O) estimator as follows:
\begin{align*}
    \mathbb{F}_{n}(\check{\theta}_{n})=\inf_{\theta\in\tilde{\Theta}_{n}}\mathbb{F}_{n}(\theta).
\end{align*}
See Suzuki and Yoshida \cite{Suzuki(2020)} for details of the P-O estimator. The P-O estimator has the following asymptotic property.
\begin{lemma}\label{POasym}
Under \textrm{\textbf{[A1]}}, \textrm{\textbf{[B1]}},
\textrm{\textbf{[C1]}},
\textrm{\textbf{[D1]}} and
\textrm{\textbf{[E1]}}, as $h_n\longrightarrow0$, 
$\sqrt{n}\lambda_{1,n}\longrightarrow0$ and
$\sqrt{n}\lambda_{2,n}\longrightarrow\infty$,
\begin{align*}
    \sqrt{n}(\check{\theta}_{n}-\theta_{0})_{\mathcal{F}_{1}}\stackrel{d}{\longrightarrow}N_{|\mathcal{F}_{1}|}\Bigl(0,{\bf{A}}_{\mathcal{F}}^{11}(\theta_{0})^{-1}\Bigr).
\end{align*}
\end{lemma}
Next, we consider the goodness-of-fit test (\ref{hypothesis}) again. Recall that
\begin{align*}
    {\rm{L}}_{n}({\bf{\Sigma}})=\prod_{i=1}^{n}\frac{1}{(2\pi )^{\frac{p}{2}}\det{(h_n{\bf{\Sigma}}})^{\frac{1}{2}}}\exp{\left\{-\frac{1}{2h_n}(\mathbb{X}_{t_{i}^n}-\mathbb{X}_{t_{i-1}^n})^{\top}{\bf{\Sigma}}^{-1}(\mathbb{X}_{t_{i}^n}-\mathbb{X}_{t_{i-1}^n})\right\}}.
\end{align*}
The restricted quasi-likelihood ratio is defined as 
\begin{align*}
    \check{\Lambda}_{n}=\frac{\max_{\theta\in\tilde{\Theta}_{n}}{\rm{L}}_{n}({\bf{\Sigma}}(\theta))}{\max_{{\bf{\Sigma}}>0}{\rm{L}}_{n}({\bf{\Sigma}})}.
\end{align*}
In an analogous manner to (\ref{ratio}), we see that as $\mathbb{Q}_{\mathbb{XX}}>0$,
\begin{align*}
    -2\log\check{\Lambda}_{n}&=n\rm{F}(\mathbb{Q}_{\mathbb{XX}},{\bf{\Sigma}}(\check{\theta}_n)),
\end{align*}
where $\rm{F}$ is defined by (\ref{FQ}). Hence, we define the penalized quasi-likelihood ratio test statistic as
\begin{align}
    \check{\mathbb{T}}_{n}=n\mathbb{F}_{n}(\check{\theta}_{n}). \label{checkT}
\end{align}
The asymptotic result of the test statistic $\check{\mathbb{T}}_n$ is as follows.
\begin{theorem}\label{tildeT}
Under \textrm{\textbf{[A1]}}, \textrm{\textbf{[B1]}},
\textrm{\textbf{[C1]}},
\textrm{\textbf{[D1]}} and
\textrm{\textbf{[E1]}}, as $h_n\longrightarrow0$, 
$\sqrt{n}\lambda_{1,n}\longrightarrow0$ and
$\sqrt{n}\lambda_{2,n}\longrightarrow\infty$,
\begin{align*}
    \check{\mathbb{T}}_{n}\stackrel{d}{\longrightarrow}\chi^2_{\bar{p}-|\mathcal{F}_{1}|}
\end{align*}
under $H_0$.
\end{theorem}
Since $|\mathcal{F}_{1}|$ is unknown, we need to estimate $|\mathcal{F}_{1}|$. The following proposition is obtained.
\begin{proposition}\label{Tprop}
Under \textrm{\textbf{[A1]}}, \textrm{\textbf{[B1]}},
\textrm{\textbf{[C1]}},
\textrm{\textbf{[D1]}} and
\textrm{\textbf{[E1]}}, as $h_n\longrightarrow0$, 
$\sqrt{n}\lambda_{1,n}\longrightarrow0$ and
$\sqrt{n}\lambda_{2,n}\longrightarrow\infty$,
\begin{align*}
    \mathbb{P}\left(\check{\mathbb{T}}_{n}>\chi^2_{\bar{p}-|\mathcal{\tilde{F}}_{\mathbb{I},n,1}|}(\alpha)\right)-\mathbb{P}\left(\check{\mathbb{T}}_{n}>\chi^2_{\bar{p}-|\mathcal{F}_{1}|}(\alpha)\right)\longrightarrow 0
\end{align*}
under $H_0$.
\end{proposition}
From Theorem \ref{tildeT} and Proposition \ref{Tprop}, we can construct the test of asymptotic significance level $\alpha\in(0,1)$. 
The rejection region is defined by
\begin{align*}
    \Bigl\{\check{\mathbb{T}}_{n}>\chi^2_{\bar{p}-|\mathcal{\tilde{F}}_{\mathbb{I},n,1}|}(\alpha)\Bigr\}.
\end{align*}

Finally, we consider the consistency of the test. Let
\begin{align*}
    \mathcal{\bar{F}}_{1}=\Bigl\{j\in\{1,\cdots,q\}\ \big|\ \bar{\theta}^{(j)}\neq 0\Bigr\}.
\end{align*}
Suppose that $0<\delta<\min_{j\in \bar{\mathcal{F}}_{1}}|\bar{\theta}^{(j)}|$. Additionally, we make the assumption as follows:
\vspace{2mm}
\begin{enumerate}
    \item[\textbf{[F2]}] $\partial^2_{\theta}\mathbb{U}(\bar{\theta})$ is non-singular.
\end{enumerate}
\vspace{2mm}
The following Lemma holds.
\begin{lemma}\label{AeoracleH1}
Under \textrm{\textbf{[A1]}}, \textrm{\textbf{[B1]}},
\textrm{\textbf{[C1]}}, \textrm{\textbf{[D1]}}, \textrm{\textbf{[E2]}}
and \textrm{\textbf{[F2]}}, as $h_n\longrightarrow0$,
$\sqrt{n}\lambda_{1,n}\longrightarrow0$ and
$\sqrt{n}\lambda_{2,n}\longrightarrow\infty$,
\begin{align*}
    \mathbb{P}\Bigl(\tilde{\mathcal{F}}_{\mathbb{I},n}
    =\bar{\mathcal{F}}_{1}\Bigr)\stackrel{}{\longrightarrow}1
\end{align*}
under $H_1$.
\end{lemma}
Moreover, the following result implies that the test has consistency.
\begin{theorem}\label{checktcons}
Under \textrm{\textbf{[A1]}}, \textrm{\textbf{[B1]}},
\textrm{\textbf{[C1]}}, \textrm{\textbf{[D1]}}, \textrm{\textbf{[E2]}}
and \textrm{\textbf{[F2]}}, as $h_n\longrightarrow0$,
$\sqrt{n}\lambda_{1,n}\longrightarrow0$ and
$\sqrt{n}\lambda_{2,n}\longrightarrow\infty$,
\begin{align*}
    \mathbb{P}\left(\check{\mathbb{T}}_{n}>\chi^2_{\bar{p}-|\mathcal{\tilde{F}}_{\mathbb{I},n,1}|}(\alpha)\right)\stackrel{}{\longrightarrow}1
\end{align*}
under $H_1$.
\end{theorem}
In the ergodic case, we obtain results similar to the non-ergodic case as follows.
\begin{theorem}
Under \textrm{\textbf{[A1]}}-\textrm{\textbf{[A2]}}, \textrm{\textbf{[B1]}}-\textrm{\textbf{[B2]}},
\textrm{\textbf{[C1]}}-\textrm{\textbf{[C2]}},
\textrm{\textbf{[D1]}}-\textrm{\textbf{[D2]}} and
\textrm{\textbf{[E1]}}, as $h_n\longrightarrow0$, $nh_n\longrightarrow\infty$, $nh_n^2\longrightarrow 0$,
$\sqrt{n}\lambda_{1,n}\longrightarrow0$ and
$\sqrt{n}\lambda_{2,n}\longrightarrow\infty$,
\begin{align*}
    \check{\mathbb{T}}_{n}\stackrel{d}{\longrightarrow}\chi^2_{\bar{p}-|\mathcal{F}_{1}|}
\end{align*}
under $H_0$.
\end{theorem}
\begin{theorem}
Under \textrm{\textbf{[A1]}}-\textrm{\textbf{[A2]}}, \textrm{\textbf{[B1]}}-\textrm{\textbf{[B2]}},
\textrm{\textbf{[C1]}}-\textrm{\textbf{[C2]}},
\textrm{\textbf{[D1]}}-\textrm{\textbf{[D2]}},
\textrm{\textbf{[E2]}} and \textrm{\textbf{[F2]}}, as $h_n\longrightarrow0$, $nh_n\longrightarrow\infty$, $nh_n^2\longrightarrow 0$,
$\sqrt{n}\lambda_{1,n}\longrightarrow0$ and
$\sqrt{n}\lambda_{2,n}\longrightarrow\infty$,
\begin{align*}
    \mathbb{P}\left(\check{\mathbb{T}}_{n}>\chi^2_{\bar{p}-|\mathcal{\tilde{F}}_{\mathbb{I},n,1}|}(\alpha)\right)\stackrel{}{\longrightarrow}1
\end{align*}
under $H_1$.
\end{theorem}
\section{Example and simulation results for sem}
\subsection{True model}
The stochastic process $\mathbb{X}_{1,0,t}$ is defined as the following true factor model:
\begin{align*}
    \mathbb{X}_{1,0,t}=\begin{pmatrix}
    1 & 2 & 0 & 0\\
    0 & 0 & 1 & 3
    \end{pmatrix}^{\top}\xi_{0,t}+\delta_{0,t},
\end{align*}
where $\{\mathbb{X}_{1,0,t}\}_{t\geq 0}$ is a four-dimensional observable vector process, 
$\{\xi_{0,t}\}_{t\geq 0}$ is a two-dimensional latent common factor vector process, and $\{\delta_{0,t}\}_{t\geq 0}$ is a four-dimensional latent unique factor vector process. The stochastic process $\mathbb{X}_{2,0,t}$ is defined by the true factor model as follows:
\begin{align*}
    \mathbb{X}_{2,0,t}=\begin{pmatrix}
    1\\
    3
    \end{pmatrix}\eta_{0,t}+\varepsilon_{0,t},
\end{align*}
where $\{\mathbb{X}_{2,0,t}\}_{t\geq 0}$ is a two-dimensional observable vector process, 
$\{\eta_{0,t}\}_{t\geq 0}$ is a one-dimensional latent common factor vector process, and
$\{\varepsilon_{0,t}\}_{t\geq 0}$ is a two-dimensional latent unique factor vector process.
Furthermore, the relationship between $\eta_{0,t}$ and $\xi_{0,t}$ is expressed as follows:
\begin{align*}
    \eta_{0,t}=\begin{pmatrix}
    1 & 2
    \end{pmatrix}\xi_{0,t}+\zeta_{0,t},
\end{align*}
where $\{\zeta_{0,t}\}_{t\geq 0}$ is a one-dimensional latent unique factor vector process. $\{\xi_{0,t}\}_{t\geq 0}$ satisfies the following two-dimensional OU process:
\begin{align*}
    \dd \xi_{0,t}=-\left\{\begin{pmatrix}
    0.5 & 0.3\\
    0.2 & 0.4
    \end{pmatrix}\xi_{0,t}-\begin{pmatrix}
    2\\
    4
    \end{pmatrix}\right\}\dd t+
    \begin{pmatrix}
    1 & 1\\
    0 & 2
    \end{pmatrix}\dd W_{1,t},\ \ (t\in [0,T]),\ \ 
    c_1=\begin{pmatrix}
    3\\
    5
    \end{pmatrix},
\end{align*}
where $W_{1,t}$ is a two-dimensional standard Wiener process. $\{\delta_{0,t}\}_{t\geq 0}$ is defined as the following four-dimensional-OU process:
\begin{align*}
    \dd \delta_{0,t}=-\begin{pmatrix}
    3 & 0 & 0 & 0\\
    0 & 2 & 0 & 0\\
    0 & 0 & 3 & 0\\
    0 & 0 & 0 & 2
    \end{pmatrix}\delta_{0,t}\dd t+
    \begin{pmatrix}
    1 & 0 & 0 & 0\\
    0 & 2 & 0 & 0\\
    0 & 0 & 2 & 0\\
    0 & 0 & 0 & 1
    \end{pmatrix}
    \dd W_{2,t},\ \ (t\in [0,T]),\ \ 
    c_2=0,
\end{align*}
where $W_{2,t}$ is a four-dimensional standard Wiener process. $\{\varepsilon_{0,t}\}_{t\geq 0}$ is defined by the two-dimensional OU process as follows:
\begin{align*}
    \dd \varepsilon_{0,t}=-\begin{pmatrix}
    2 & 0 \\
    0 & 3
    \end{pmatrix}\varepsilon_{0,t}\dd t+
    \begin{pmatrix}
    1 & 0\\
    0 & 3
    \end{pmatrix}\dd W_{3,t},\ \ (t\in [0,T]),\ \ 
    c_3=0,
\end{align*}
where $W_{3,t}$ is a two-dimensional standard Wiener process. $\{\zeta_{0,t}\}_{t\geq 0}$ satisfies the following one-dimensional OU process:
\begin{align*}
    \dd \zeta_{0,t}=-\zeta_{0,t}\dd t+2
    \dd W_{4,t}\ \ (t\in [0,T]),\ \ c_4=0,
\end{align*}
where $W_{4,t}$ is the one-dimensional standard Wiener process. We assume that $W_{1,t}$, $W_{2,t}$, $W_{3,t}$ and $W_{4,t}$ are independent. Figure \ref{truefigure} shows the path diagram of the true model.
\begin{figure}[h]
    \centering
    \includegraphics[width=0.8\columnwidth]{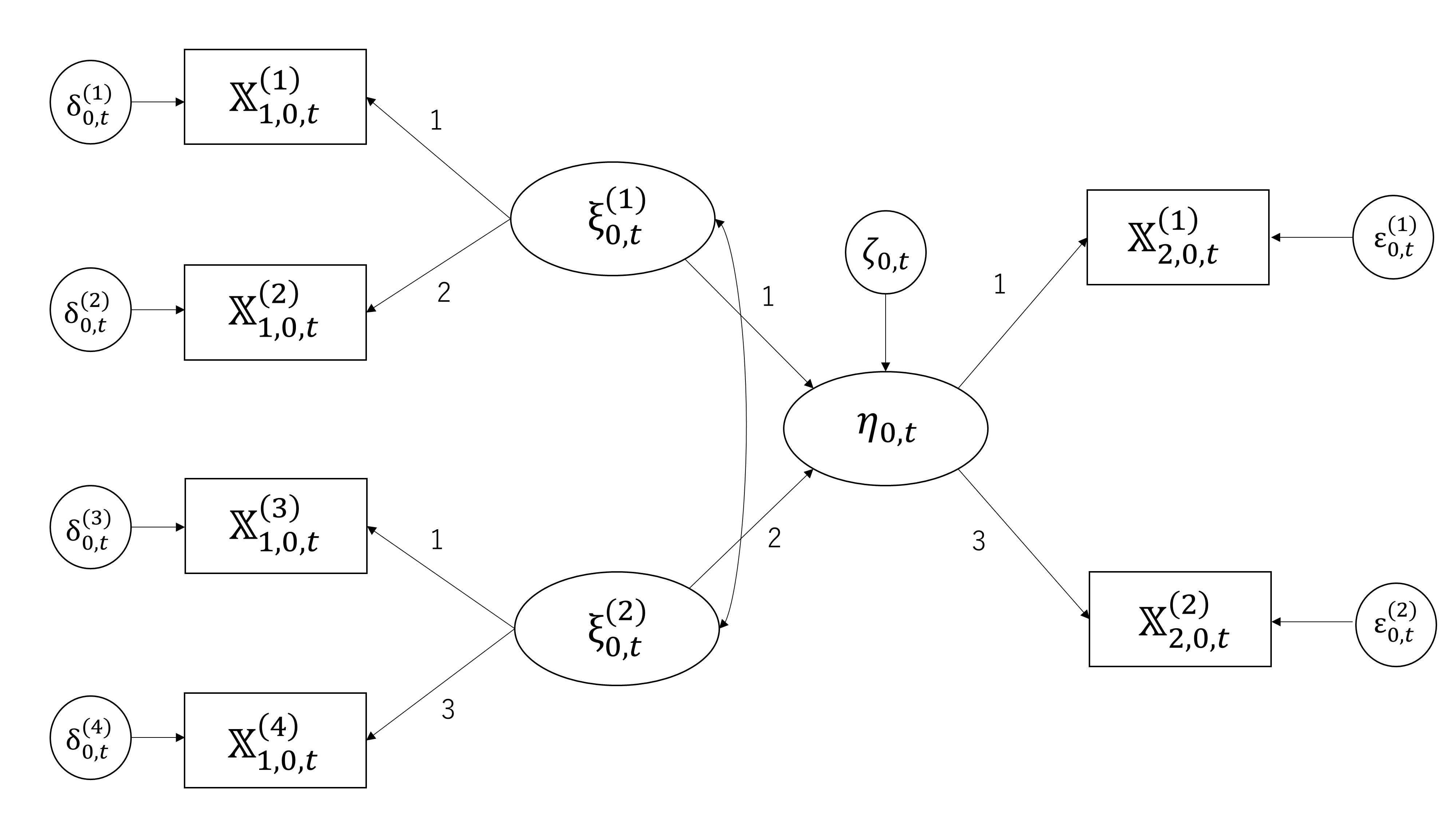}
    \caption{Path diagram of the true model.}\label{truefigure}
\end{figure}
\subsection{Correctly specified parametric model}
Set $p_1=4$, $p_2=2$, $k_1=2$ and $k_2=1$ in the parametric model (\ref{X})-(\ref{zetaP}). Suppose
\begin{align*}
    {\bf{\Lambda}}_{x_1}=
    \begin{pmatrix}
    1 & ({\bf{\Lambda}}_{x_1})_{21} & 0 & 0\\
    0 & 0 & 1 &({\bf{\Lambda}}_{x_1})_{42}
    \end{pmatrix}^{\top}\in\mathbb{R}^{4\times 2},
\end{align*}
where $({\bf{\Lambda}}_{x_1})_{21}$ and $({\bf{\Lambda}}_{x_1})_{42}$ are not zero,
\begin{align*}
    {\bf{\Lambda}}_{x_2}=\begin{pmatrix}
    1 & ({\bf{\Lambda}}_{x_{2}})_{21}
    \end{pmatrix}^{\top}\in\mathbb{R}^{2\times 1},
\end{align*}
where $({\bf{\Lambda}}_{x_2})_{21}$ is not zero, ${\bf{\Gamma}}_{11}$ and ${\bf{\Gamma}}_{12}$ are not zero, ${\bf{\Sigma}}_{\xi\xi}\in\mathbb{R}^{2\times 2}$ is a positive definite matrix, $({\bf{\Sigma}}_{\xi\xi})_{12}$ is not zero, ${\bf{\Sigma}}_{\delta\delta}\in\mathbb{R}^{4\times 4}$ and ${\bf{\Sigma}}_{\varepsilon\varepsilon}\in\mathbb{R}^{2\times 2}$ are positive definite diagonal matrices, and ${\bf{\Sigma}}_{\zeta\zeta}>0$. The parameter is expressed as 
\begin{align*}
    \theta&=\Bigl(({\bf{\Lambda}}_{x_1})_{21},({\bf{\Lambda}}_{x_1})_{42},({\bf{\Lambda}}_{x_2})_{21},{\bf{\Gamma}}_{11},{\bf{\Gamma}}_{12},({\bf{\Sigma}}_{\xi\xi})_{11},({\bf{\Sigma}}_{\xi\xi})_{12},
    ({\bf{\Sigma}}_{\xi\xi})_{22},\\
    &\qquad\qquad({\bf{\Sigma}}_{\delta\delta})_{11},({\bf{\Sigma}}_{\delta\delta})_{22},({\bf{\Sigma}}_{\delta\delta})_{33},({\bf{\Sigma}}_{\delta\delta})_{44},({\bf{\Sigma}}_{\varepsilon\varepsilon})_{11},({\bf{\Sigma}}_{\varepsilon\varepsilon})_{22},{\bf{\Sigma}}_{\zeta\zeta}\Bigr)^{\top}\in\Theta,
\end{align*}
where $\Theta=[-100,100]^5\times[0.1,100]\times[-100,100]\times[0.1,100]^8$. Let
\begin{align*}
    \theta_{0}=\Bigl(
    2, 3, 3, 1, 2, 2, 2, 4, 1, 4, 4, 1, 1, 9, 4
    \Bigr)^{\top}\in\Theta.
\end{align*}
It holds ${\bf{\Sigma}}_0={\bf{\Sigma}}(\theta_0)$, which implies that the model is a correctly specified parametric model. In addition, we have
\begin{align}
    {\bf{\Sigma}}(\theta_1)
    ={\bf{\Sigma}}(\theta_2)\Longrightarrow\theta_1=\theta_2.\label{model1iden}
\end{align}
For details of (\ref{model1iden}), see Appendix \ref{idenap}. Figure \ref{corfigure} shows the path diagram of the correctly specified parametric model.
\begin{figure}[h]
    \centering
    \includegraphics[width=0.8\columnwidth]{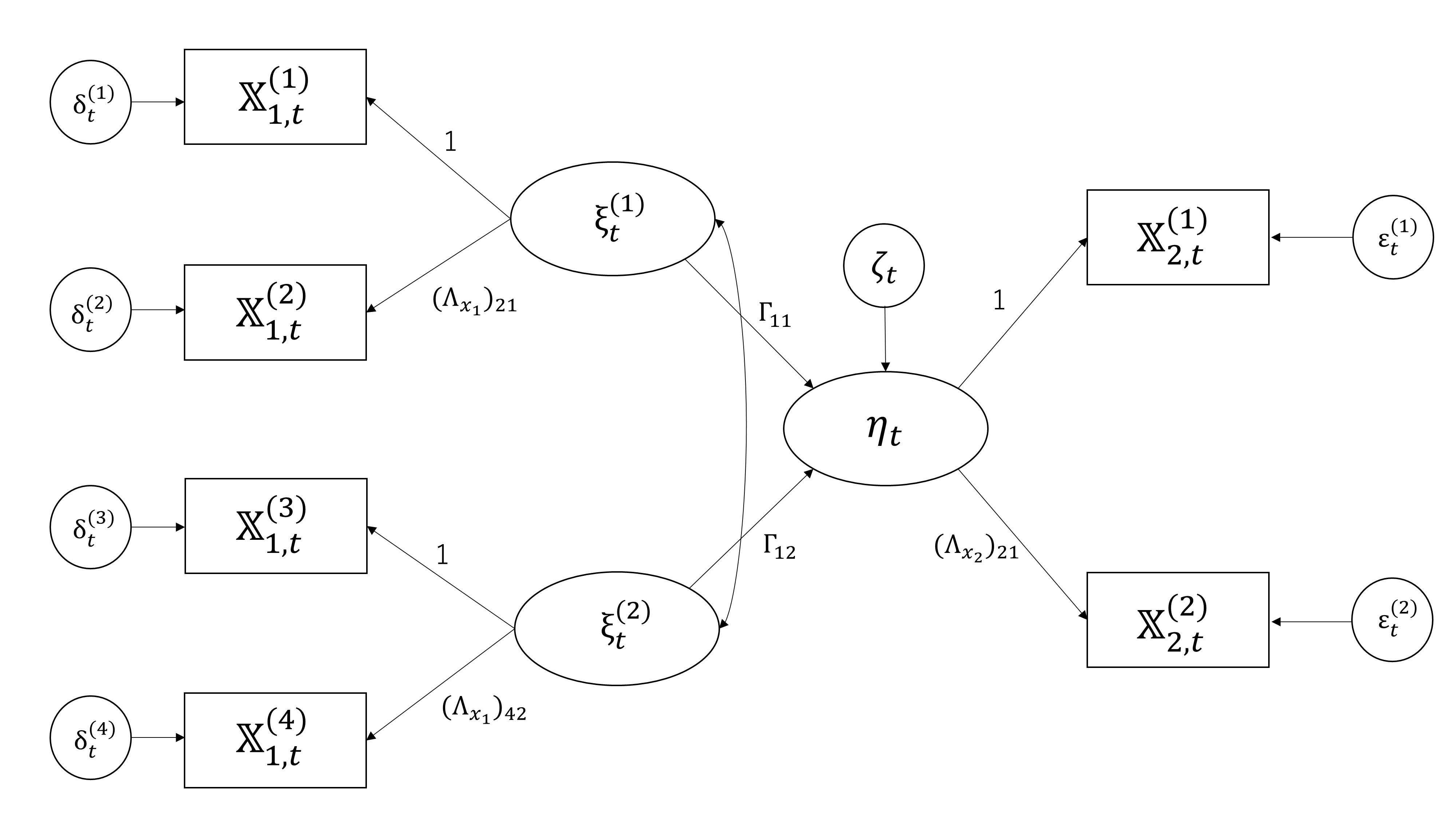}
    \caption{Path diagram of the correctly specified parametric model.}\label{corfigure}
\end{figure}
\subsection{Missspecified parametric model}
\subsubsection{Model A}
Let $p_1=4$, $p_2=2$, $k_1=1$ and $k_2=1$ in the parametric model (\ref{X})-(\ref{zetaP}). Assume 
\begin{align*}
    {\bf{\Lambda}}_{x_1}=\begin{pmatrix}
    1 & ({\bf{\Lambda}}_{x_1})_{21} & ({\bf{\Lambda}}_{x_1})_{31} & ({\bf{\Lambda}}_{x_1})_{41}    
    \end{pmatrix}^{\top}\in\mathbb{R}^{4\times 1},
\end{align*}
where $({\bf{\Lambda}}_{x_1})_{21}$, $({\bf{\Lambda}}_{x_1})_{31}$ and $({\bf{\Lambda}}_{x_1})_{41}$ are not zero,
\begin{align*}
    {\bf{\Lambda}}_{x_2}=\begin{pmatrix}
    1 &  ({\bf{\Lambda}}_{x_2})_{21}    
    \end{pmatrix}^{\top}\in\mathbb{R}^{2\times 1},
\end{align*}
where $({\bf{\Lambda}}_{x_2})_{21}$ is not zero, ${\bf{\Gamma}}\in\mathbb{R}$ is not zero, ${\bf{\Sigma}}_{\xi\xi}>0$, ${\bf{\Sigma}}_{\zeta\zeta}>0$, and ${\bf{\Sigma}}_{\delta\delta}\in\mathbb{R}^{4\times 4}$
and ${\bf{\Sigma}}_{\varepsilon\varepsilon}\in\mathbb{R}^{2\times 2}$ are positive definite diagonal matrices. The parameter is expressed as follows:
\begin{align*}
        \theta&=\Bigl(({\bf{\Lambda}}_{x_1})_{21}, ({\bf{\Lambda}}_{x_1})_{31}, ({\bf{\Lambda}}_{x_1})_{41}, ({\bf{\Lambda}}_{x_2})_{21}, {\bf{\Gamma}}, {\bf{\Sigma}}_{\xi\xi},({\bf{\Sigma}}_{\delta\delta})_{11},\\
    &\qquad\qquad\quad({\bf{\Sigma}}_{\delta\delta})_{22},({\bf{\Sigma}}_{\delta\delta})_{33},({\bf{\Sigma}}_{\delta\delta})_{44},({\bf{\Sigma}}_{\varepsilon\varepsilon})_{11},({\bf{\Sigma}}_{\varepsilon\varepsilon})_{22},{\bf{\Sigma}}_{\zeta\zeta}\Bigr)^{\top}\in\Theta,
\end{align*}
where $\Theta=[-100,100]^5\times[0.1,100]^8$. Figure \ref{missA} shows the path diagram of Model A.
\begin{figure}[h]
    \centering
    \includegraphics[width=0.8\columnwidth]{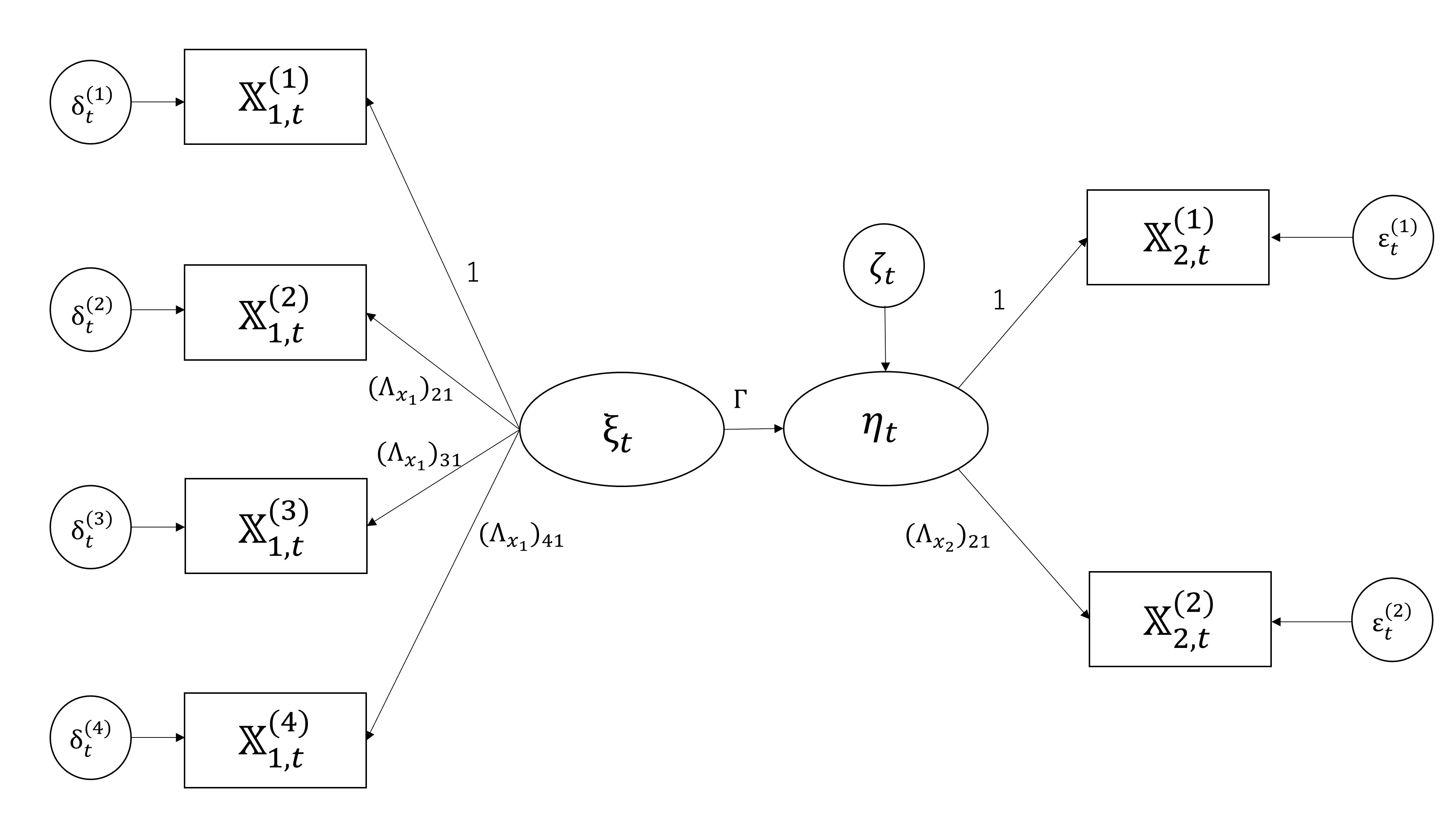}
    \caption{Path diagram of Model A.}\label{missA}
\end{figure}
\subsubsection{Model B}
Set $p_1=4$, $p_2=2$, $k_1=2$ and $k_2=1$ in the parametric model (\ref{X})-(\ref{zetaP}). Suppose
\begin{align*}
    {\bf{\Lambda}}_{x_1}=
    \begin{pmatrix}
    1 & 0 & 0 & ({\bf{\Lambda}}_{x_1})_{41}\\
    0 & ({\bf{\Lambda}}_{x_1})_{22} & 1 & 0
    \end{pmatrix}^{\top}\in\mathbb{R}^{4\times 2},
\end{align*}
where $({\bf{\Lambda}}_{x_{1}})_{22}$ and $({\bf{\Lambda}}_{x_1})_{41}$ are not zero, 
\begin{align*}
    {\bf{\Lambda}}_{x_2}=\begin{pmatrix}
    1 & ({\bf{\Lambda}}_{x_{2}})_{21}
    \end{pmatrix}^{\top}\in\mathbb{R}^{2\times 1},
\end{align*}
where $({\bf{\Lambda}}_{x_2})_{21}$ is not zero,
${\bf{\Gamma}}_{11}$ and ${\bf{\Gamma}}_{12}$ are not zero,
${\bf{\Sigma}}_{\xi\xi}\in\mathbb{R}^{2\times 2}$ is a positive definite matrix, $({\bf{\Sigma}}_{\xi\xi})_{12}$ is not zero, ${\bf{\Sigma}}_{\delta\delta}\in\mathbb{R}^{4\times 4}$ and ${\bf{\Sigma}}_{\varepsilon\varepsilon}\in\mathbb{R}^{2\times 2}$ 
are positive definite diagonal matrices, 
and ${\bf{\Sigma}}_{\zeta\zeta}>0$. The parameter is expressed as 
\begin{align*}
    \theta&=\Bigl(({\bf{\Lambda}}_{x_1})_{22},({\bf{\Lambda}}_{x_1})_{41},({\bf{\Lambda}}_{x_2})_{21},{\bf{\Gamma}}_{11},{\bf{\Gamma}}_{12},({\bf{\Sigma}}_{\xi\xi})_{11},({\bf{\Sigma}}_{\xi\xi})_{12},
    ({\bf{\Sigma}}_{\xi\xi})_{22},\\
    &\qquad\qquad({\bf{\Sigma}}_{\delta\delta})_{11},({\bf{\Sigma}}_{\delta\delta})_{22},({\bf{\Sigma}}_{\delta\delta})_{33},({\bf{\Sigma}}_{\delta\delta})_{44},({\bf{\Sigma}}_{\varepsilon\varepsilon})_{11},({\bf{\Sigma}}_{\varepsilon\varepsilon})_{22},{\bf{\Sigma}}_{\zeta\zeta}\Bigr)^{\top}\in\Theta,
\end{align*}
where $\Theta=[-100,100]^5\times[0.1,100]\times[-100,100]\times[0.1,100]^8$. Figure \ref{missB} shows the path diagram of Model B.
\begin{figure}[h]
    \centering
    \includegraphics[width=0.8\columnwidth]{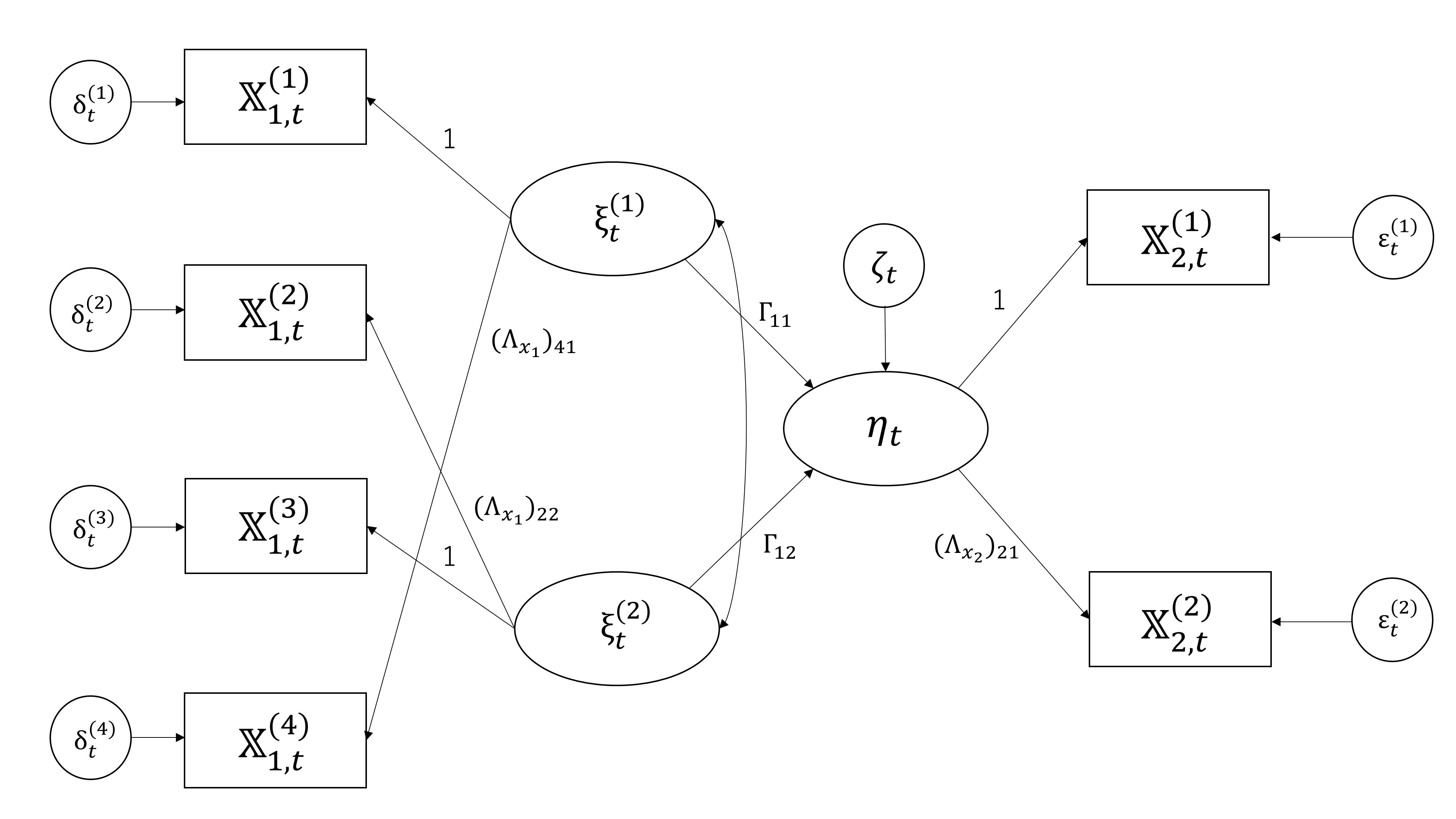}
    \caption{Path diagram of Model B.}\label{missB}
\end{figure}
\subsection{Simulation results}
Let $(n,h_n,T)=(10^4,10^{-3},10^{1})$. We generated 10,000 independent sample paths from the true model. To optimize $\mathbb{F}_{n}(\theta)$, we use optim() with the BFGS method in R language. The initial value of the optimization is set to $\theta_{0}$. See Appendix \ref{ergodic} for simulation results of the ergodic case.
	
\subsubsection{Correctly specified parametric model}
First, we check the asymptotic performance of $\mathbb{Q}_{\mathbb{XX}}$. Table \ref{QQtable} shows the sample mean and the sample standard deviation (SD) of $\mathbb{Q}_{\mathbb{XX}}$. Figure \ref{QQfigure} shows the histogram, the Q-Q plot and the empirical distribution of $\sqrt{n}((\mathbb{Q}_{\mathbb{XX}})_{11}-({\bf{\Sigma}}_0)_{11})$. It seems from Table \ref{QQtable} and Figure \ref{QQfigure} that Theorem 1 holds true for this example. Next, we investigate the asymptotic performance of $\hat{\theta}_{n}$. Table \ref{thetatable} shows the sample mean and the sample SD of $\hat{\theta}_{n}$ and we deduce that $\hat{\theta}_{n}$ has consistency. Figure \ref{thetafigure} shows the histogram, the Q-Q plot and the empirical distribution of $\sqrt{n}(\hat{\theta}_{n}^{(1)}-\theta_{0}^{(1)})$. These simulation results show that Theorem 2 seems to be correct in this example. Table \ref{testtable1} shows the sample mean and the sample SD of the test statistic $\mathbb{T}_{n}$. Figure \ref{testfigure} shows the histogram, the Q-Q plot and the empirical distribution of the test statistic $\mathbb{T}_{n}$. 
Table \ref{testtable1} and  Figure \ref{testfigure} indicate that the test statistic $\mathbb{T}_{n}$ seems to converge $\chi^2_6$ under $H_0$. See Appendix \ref{simulation} for details of simulation results.
\subsubsection{Missspecified parametric model}
Table \ref{testtable2} shows the number of rejections of the quasi-likelihood ratio test in Model A and Model B, which implies that the null hypothesis is rejected in both Model A and Model B tests all 10000 times. Table \ref{testtable3} shows the quartiles of  the test statistics $\mathbb{T}_{n}$ in Model A and Model B. From Table \ref{testtable3}, we deduce that Model B is closer to the true model than Model A.
\newpage
\begin{table}[h]
    \ \\ \ \\ \ \\ \ \\ \ \\ \ \\
    \centering
    \begin{tabular}{ccccc}
    & $(\mathbb{Q}_{\mathbb{XX}})_{11}$ & $(\mathbb{Q}_{\mathbb{XX}})_{12}$ & $(\mathbb{Q}_{\mathbb{XX}})_{13}$ & $(\mathbb{Q}_{\mathbb{XX}})_{14}$ \\\hline
    Mean (True value) & 3.002 (3.000) & 4.002 (4.000) & 2.000 (2.000)& 6.001 (6.000) \\
    SD (Theoretical value) & 0.042 (0.042) & 0.072 (0.072)& 0.052 (0.053) & 0.121 (0.121)\\
    & $(\mathbb{Q}_{\mathbb{XX}})_{15}$ & $(\mathbb{Q}_{\mathbb{XX}})_{16}$ & $(\mathbb{Q}_{\mathbb{XX}})_{22}$ & $(\mathbb{Q}_{\mathbb{XX}})_{23}$\\\hline
    Mean (True value) & 6.001 (6.000) & 18.003 (18.000) & 12.007 (12.000) & 4.001 (4.000)\\
    SD (Theoretical value) & 0.113 (0.114) & 0.337 (0.341)& 0.170 (0.170) & 0.106 (0.106)  \\
    & $(\mathbb{Q}_{\mathbb{XX}})_{24}$ & $(\mathbb{Q}_{\mathbb{XX}})_{25}$ & $(\mathbb{Q}_{\mathbb{XX}})_{26}$ & $(\mathbb{Q}_{\mathbb{XX}})_{33}$\\\hline
    Mean (True value) & 12.002 (12.000) & 12.002 (12.000) & 36.007 (36.000) & 8.006 (8.000)\\
    SD (Theoretical value) & 0.245 (0.242) & 0.229 (0.227) & 0.686 (0.681) & 0.112 (0.113) \\
    & $(\mathbb{Q}_{\mathbb{XX}})_{34}$ & $(\mathbb{Q}_{\mathbb{XX}})_{35}$ & $(\mathbb{Q}_{\mathbb{XX}})_{36}$ & $(\mathbb{Q}_{\mathbb{XX}})_{44}$\\\hline
    Mean (True value) & 12.003 (12.000) & 10.003 (10.000) & 30.009 (30.000) & 37.010 (37.000)\\
    SD (Theoretical value) & 0.210 (0.210) & 0.187 (0.187)& 0.561 (0.560) & 0.530 (0.523) \\
    & $(\mathbb{Q}_{\mathbb{XX}})_{45}$ & $(\mathbb{Q}_{\mathbb{XX}})_{46}$ & $(\mathbb{Q}_{\mathbb{XX}})_{55}$ & $(\mathbb{Q}_{\mathbb{XX}})_{56}$\\\hline
    Mean (True value) & 30.006 (30.000) & 90.021 (90.000) & 31.008 (31.000) & 90.023 (90.000)\\
    SD (Theoretical value) & 0.457 (0.452) & 1.369 (1.357) & 0.441 (0.438) & 1.302 (1.294) \\
    & $(\mathbb{Q}_{\mathbb{XX}})_{66}$ &&& \\\hline
    Mean (True value) & 279.081 (279.000) & & &\\
    SD (Theoretical value) & 3.961 (3.946) & &  \\ \\
\end{tabular}
\caption{Sample mean and sample standard deviation (SD) of $\mathbb{Q}_{\mathbb{XX}}$.}
\label{QQtable}
\end{table}
\ \\
\begin{figure}[h]
    \centering
    \includegraphics[width=0.32\columnwidth]{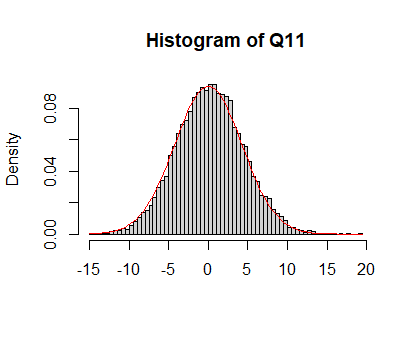}
    \includegraphics[width=0.32\columnwidth]{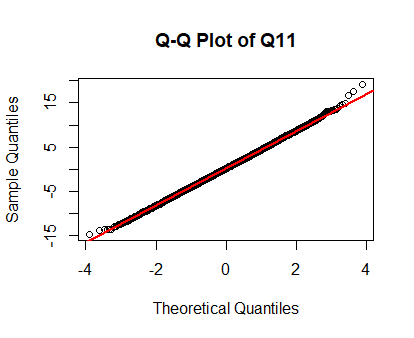}
    \includegraphics[width=0.32\columnwidth]{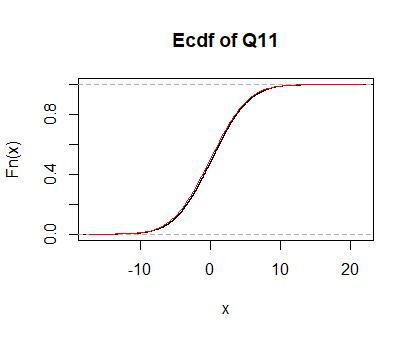}
    \caption{Histogram (left), Q-Q plot (middle) and empirical distribution (right) of $\sqrt{n}((\mathbb{Q}_{\mathbb{XX}})_{11}-({\bf{\Sigma}}_0)_{11})$. The red lines are theoretical curves.}
    \label{QQfigure}
\end{figure}
\newpage
\begin{table}[h]
\ \\ \ \\ \ \\ \ \\ \ \\ \ \\
\centering
\begin{tabular}{ccccc}
    & $\hat{\theta}_{n}^{(1)}$ & $\hat{\theta}_{n}^{(2)}$ & $\hat{\theta}_{n}^{(3)}$ & $\hat{\theta}_{n}^{(4)}$ \\\hline
    Mean (True value) & 2.000 (2.000) & 3.000 (3.000) & 3.000 (3.000)& 0.999 (1.000) \\
    SD (Theoretical value) & 0.026 (0.026) & 0.336 (0.336) & 0.009 (0.008) & 0.036 (0.036) \\
    & $\hat{\theta}_{n}^{(5)}$ & $\hat{\theta}_{n}^{(6)}$ & $\hat{\theta}_{n}^{(7)}$ & $\hat{\theta}_{n}^{(8)}$\\\hline
    Mean (True value) & 2.001 (2.000) & 2.001 (2.000) & 2.000 (2.000) & 4.002 (4.000)\\
    SD (Theoretical value) & 0.030 (0.030) & 0.044 (0.044)& 0.045 (0.046) & 0.100 (0.100) \\
    & $\hat{\theta}_{n}^{(9)}$ & $\hat{\theta}_{n}^{(10)}$ & $\hat{\theta}_{n}^{(11)}$ & $\hat{\theta}_{n}^{(12)}$\\\hline
    Mean (True value) & 1.001 (1.000)  &  4.003 (4.000) & 4.004 (4.000)& 1.004 (1.000)\\
    SD (Theoretical value) & 0.024 (0.024) & 0.096 (0.096) & 0.059 (0.060) & 0.183 (0.182) \\
    & $\hat{\theta}_{n}^{(13)}$ & $\hat{\theta}_{n}^{(14)}$ & $\hat{\theta}_{n}^{(15)}$ & \\\hline
    Mean (True value) & 1.001 (1.000) & 9.007 (9.000) & 3.999 (4.000)&\\
    SD (Theoretical value) & 0.038 (0.038) & 0.341 (0.343) & 0.110 (0.109)\\ \\
\end{tabular}
\caption{Sample mean and sample standard deviation (SD) of $\hat{\theta}_{n}$.}
\label{thetatable}
\end{table}
\ \\
\begin{figure}[h]
    \centering
    \includegraphics[width=0.32\columnwidth]{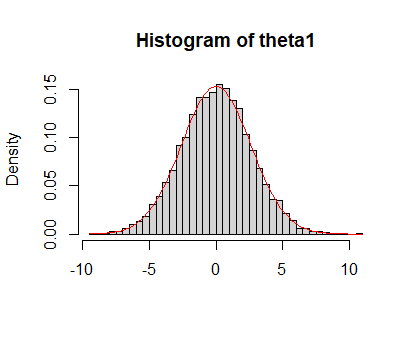}
    \includegraphics[width=0.32\columnwidth]{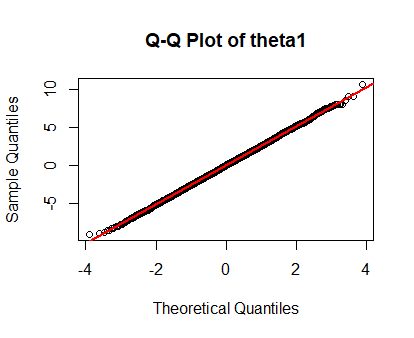}
    \includegraphics[width=0.32\columnwidth]{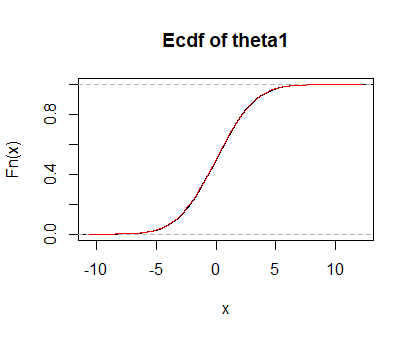}
    \caption{Histogram (left), Q-Q plot (middle) and empirical distribution (right) of $\sqrt{n}(\hat{\theta}_{n}^{(1)}-\theta_{0}^{(1)})$. The red lines are theoretical curves.}
\label{thetafigure}
\end{figure}
\newpage
\begin{table}[h]
    \ \\ \ \\ \ \\
    \centering
    \begin{tabular}{cc}
    \hline
    Mean\ \  (True value) &\quad 5.980\ \ (6.000)\\
    SD\ \ (Theoretical value) &\quad 3.400\ \ (3.464)\\ \\
    \end{tabular}
\caption{Sample mean and sample standard deviation (SD) of the test statistic $\mathbb{T}_{n}$.}
\label{testtable1}
\end{table}
\ \\ \ \\
\begin{figure}[h]
    \centering
    \includegraphics[width=0.32\columnwidth]{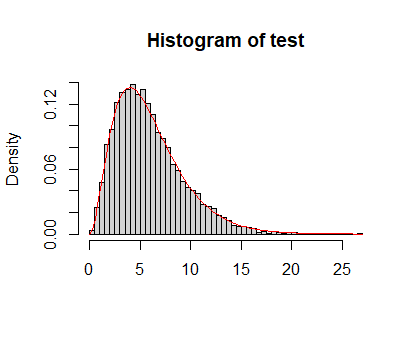}
    \includegraphics[width=0.32\columnwidth]{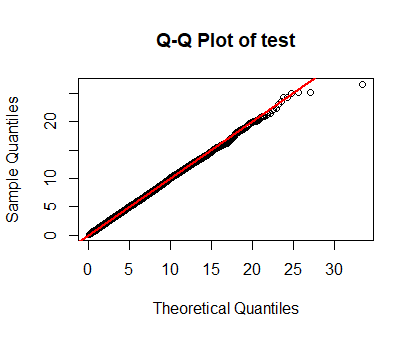}
    \includegraphics[width=0.32\columnwidth]{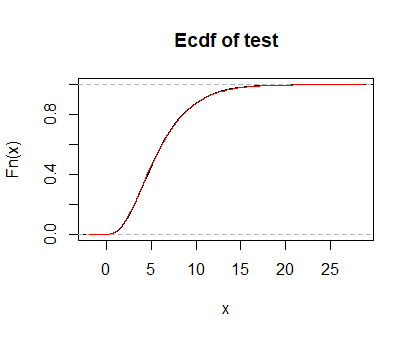}
\caption{Histogram (left), Q-Q plot (middle) and empirical distribution (right) of the test statistic $\mathbb{T}_{n}$. The red lines are theoretical curves.}
\label{testfigure}
\end{figure}
\ \\
\ \\
\begin{table}[h]
    \centering
    \begin{tabular}{cc}
    \hline
    Model A &\quad 10000\\
    Model B &\quad 10000\\ \ \\
    \end{tabular}
\caption{The number of rejections of the quasi-likelihood ratio test in Model A and Model B.}
\label{testtable2}
\end{table}
\ \\ 
\ \\
\begin{table}[h]
    \centering
    \begin{tabular}{cccccc}
    &\quad Min &\  $Q1$ &\  Median &\  $Q3$ &\  Max\\\hline
    Model A &\quad 2391 &\ 2720 &\ 2791 &\ 2862 &\ 3205 \\
    Model B &\quad 1793 &\ 2064 &\ 2123 &\ 2183 &\ 2442 \\ \\
    \end{tabular}
\caption{Quartile of the test statistic $\mathbb{T}_{n}$ in Model A and Model B.}
\label{testtable3}
\end{table}
\newpage
\section{Example and simulation results for ssem}
\subsection{True model}
The stochastic process $\mathbb{X}_{1,0,t}$ is defined by the true factor model as follows:
\begin{align*}
    \mathbb{X}_{1,0,t}=\begin{pmatrix}
    1 & 0 & 0 & 3 & 0 & 0 & 4 & 0 & 0\\
    0 & 1 & 0 & 0 & 2 & 0 & 0 & 6 & 0\\
    0 & 0 & 1 & 0 & 0 & 5 & 0 & 0 & 3
    \end{pmatrix}^{\top}\xi_{0,t}+\delta_{0,t},
\end{align*}
where $\{\mathbb{X}_{1,0,t}\}_{t\geq 0}$ is a nine-dimensional observable vector process, 
$\{\xi_{0,t}\}_{t\geq 0}$ is a three-dimensional latent common factor vector process, 
$\{\delta_{0,t}\}_{t\geq 0}$ is a nine-dimensional latent unique
factor vector process. The stochastic process $\mathbb{X}_{2,0,t}$ is defined as the following true factor model:
\begin{align*}
    \mathbb{X}_{2,0,t}=\begin{pmatrix}
    1 & 0 & 5 & 0 & 7 & 0 \\
    0 & 1 & 0 & 3 & 0 & 2
    \end{pmatrix}^{\top}\eta_{0,t}+\varepsilon_{0,t},
\end{align*}
where $\{\mathbb{X}_{2,0,t}\}_{t\geq 0}$ is a six-dimensional observable vector process.
$\{\eta_{0,t}\}_{t\geq 0}$ is a two-dimensional latent common factor vector process and 
$\{\varepsilon_{0,t}\}_{t\geq 0}$ is a six-dimensional latent unique factor vector process.
Moreover, we express the relationship between $\eta_{0,t}$ and $\xi_{0,t}$ as follows:
\begin{align*}
    \eta_{0,t}=\begin{pmatrix}
    5 & 2 & 0\\
    0 & 0 & 2
    \end{pmatrix}\xi_{0,t}+\zeta_{0,t},
\end{align*}
where $\{\zeta_{0,t}\}_{t\geq 0}$ is a two-dimensional latent unique factor vector process. Suppose that $\{\xi_{0,t}\}_{t\geq 0}$ is defined as the following three-dimensional OU process:
\begin{align*}
    \dd \xi_{0,t}=-\left\{\begin{pmatrix}
    0.5 & 0.4 & 0.1\\
    0.2 & 0.2 & 0.6\\
    0.3 & 0.4 & 0.2 \\
    \end{pmatrix}\xi_{0,t}-\begin{pmatrix}
    2\\
    4\\
    2
    \end{pmatrix}\right\}\dd t+
    \begin{pmatrix}
    2 & 0 & 0 \\
    0 & 1 & 0 \\
    0 & 0 & 3 
    \end{pmatrix}
    \dd W_{1,t},\ (t\in [0,T]),\ \ 
    \xi_{0,0}=\begin{pmatrix}
    3\\
    5\\
    2\\
    \end{pmatrix},
\end{align*}
where $W_{1,t}$ is a three-dimensional standard Wiener process. $\{\delta_{0,t}\}_{t\geq 0}$ satisfies the following nine-dimensional OU process:
\begin{align*}
    \dd \delta_{0,t}&=-B_2\xi_{0,t}\dd t+{\bf{S}}_{2,0}\dd W_{2,t},\ (t\in [0,T]),\ \delta_{0,0}=0,
\end{align*}
where $B_2=\Diag(3,2,3,2,2,3,1,3,1)$, ${\bf{S}}_{2,0}=\Diag(1,2,1,5,2,3,1,2,3)$ and $W_{2,t}$ is a nine-dimensional standard Wiener process. $\{\varepsilon_{0,t}\}_{t\geq 0}$ is defined as the following six-dimensional OU process:
\begin{align*}
    \dd \varepsilon_{0,t}&=-B_3\varepsilon_{0,t}\dd t+{\bf{S}}_{3,0}\dd W_{3,t},\ (t\in [0,T]),\ \varepsilon_{0,0}=0,
\end{align*}
where $B_3=\Diag(1,3,2,3,2,2)$, ${\bf{S}}_{3,0}=\Diag(3,1,2,1,5,2)$ and $W_{3,t}$ is a six-dimensional standard Wiener process. $\{\zeta_{0,t}\}_{t\geq 0}$ satisfies the following two-dimensional OU process:
\begin{align*}
    \dd \zeta_{0,t}&=-\begin{pmatrix}
    3 & 0\\
    0 & 1
    \end{pmatrix}\zeta_{0,t}\dd t+\begin{pmatrix}
    3 & 0\\
    0 & 1
    \end{pmatrix}\dd W_{4,t},\ (t\in [0,T]),\ \zeta_{0,0}=0,
\end{align*}
where $W_{4,t}$ is a two-dimensional standard Wiener process. Figure \ref{truesparse} shows the path diagram of the true model.
\begin{figure}[h]
    \centering
    \includegraphics[width=0.8\columnwidth]{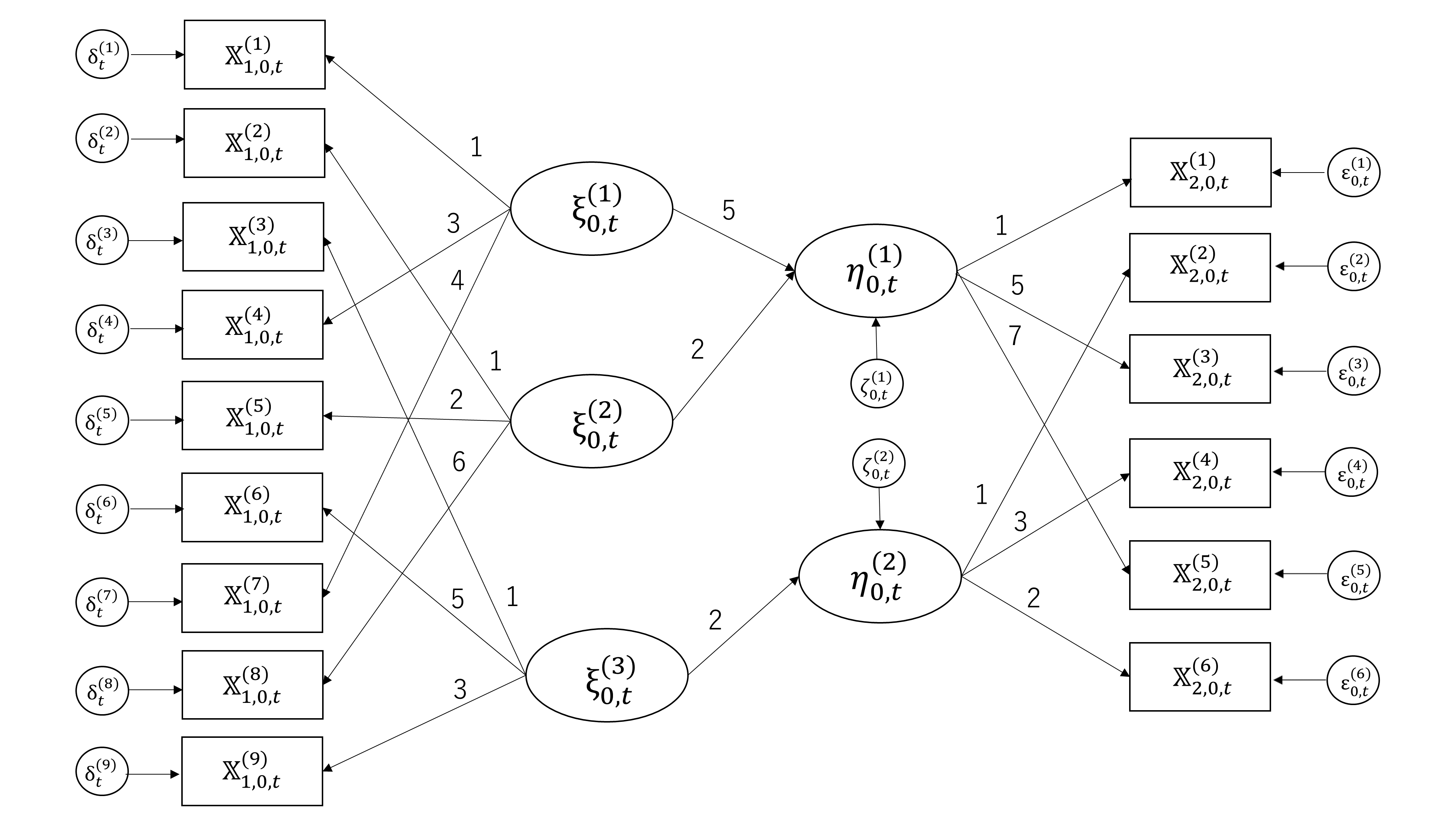}
    \caption{Path diagram of the true model.}\label{truesparse}
\end{figure}
\subsection{Correctly specified parametric model}
Let $p_1=9$, $p_2=6$, $k_1=3$ and $k_2=2$ in the parametric model (\ref{X})-(\ref{zetaP}). Assume
\begin{align*}
    {\bf{\Lambda}}_{x_1}=\Bigl(
    \mathbb{I}_3, {\bf{A}}_{x_1}^{\top}\Bigr)^{\top},\quad
    {\bf{\Lambda}}_{x_2}=\Bigl(
    \mathbb{I}_2, {\bf{A}}_{x_2}^{\top}\Bigr)^{\top}
\end{align*}
where ${\bf{A}}_{x_1}\in\mathbb{R}^{6\times 3}$ and ${\bf{A}}_{x_2}\in\mathbb{R}^{4\times 2}$, ${\bf{A}}_{x_2}$ is a full column rank matrix. Moreover, we suppose that ${\bf{\Lambda}}_{x_1}$ meets the identifiability condition of Theorem 5.1 in Anderson and Rubin \cite{Anderson(1956)}. The parameter is expressed as
\begin{align*}
    \theta=\Bigl(\vec {\bf{A}}_{x_1}^{\top}, \vec {\bf{A}}_{x_2}^{\top}, \vec{{\bf{\Gamma}}}^{\top}, \vech{\bf{\Sigma}}_{\xi\xi}^{\top},\diag{\bf{\Sigma}}_{\delta\delta}^{\top},
    \diag{\bf{\Sigma}}_{\varepsilon\varepsilon}^{\top},
    \vech{\bf{\Sigma}}_{\zeta\zeta}^{\top}\Bigr)^{\top}\in\Theta
\end{align*}
where 
\begin{align*}
    \Theta&=[-100,100]^{32}\times[0.1,100]\times[-100,100]^2\times[0.1,100]\\
    &\qquad\times[-100,100]\times[0.1,100]^{17}\times[-100,100]\times[0.1,100],
\end{align*}${\bf{\Gamma}}\in\mathbb{R}^{2\times 3}$ is a full row rank matrix, ${\bf{\Sigma}}_{\xi\xi}\in\mathbb{R}^{3\times 3}$ and  
${\bf{\Sigma}}_{\zeta\zeta}\in\mathbb{R}^{2\times 2}$
are positive definite matrices,
${\bf{\Sigma}}_{\delta\delta}\in\mathbb{R}^{9\times 9}$ and
${\bf{\Sigma}}_{\varepsilon\varepsilon}\in\mathbb{R}^{6\times 6}$ 
are positive definite diagonal matrices.
Set
\begin{align*}
    \theta_0&=\Bigl(3,0,0,0,2,0,0,0,5,4,0,0,0,6,0,0,0,3,
    5,0,0,3,7,0,0,2,5,2,0,\\
    &\qquad\qquad 0,0,2,4,0,0,1,0,9,1,4,1,25,4,9,1,4,9,9,1,4,1,25,4,9,0,1\Bigr)^{\top}\in\Theta
\end{align*}
and we note that ${\bf{\Sigma}}_0={\bf{\Sigma}}(\theta_0)$, which implies that the model is a correctly specified parametric model. The model satisfies
\begin{align}
    {\bf{\Sigma}}(\theta_1)
    ={\bf{\Sigma}}(\theta_2)\Longrightarrow\theta_1=\theta_2. \label{model2iden}
\end{align}
For the proof of (\ref{model2iden}), see Appendix \ref{idenap2}. Figure \ref{corsparse} shows the path diagram of the correctly specified parametric model.
\begin{figure}[t]
    \centering
    \includegraphics[width=0.8\columnwidth]{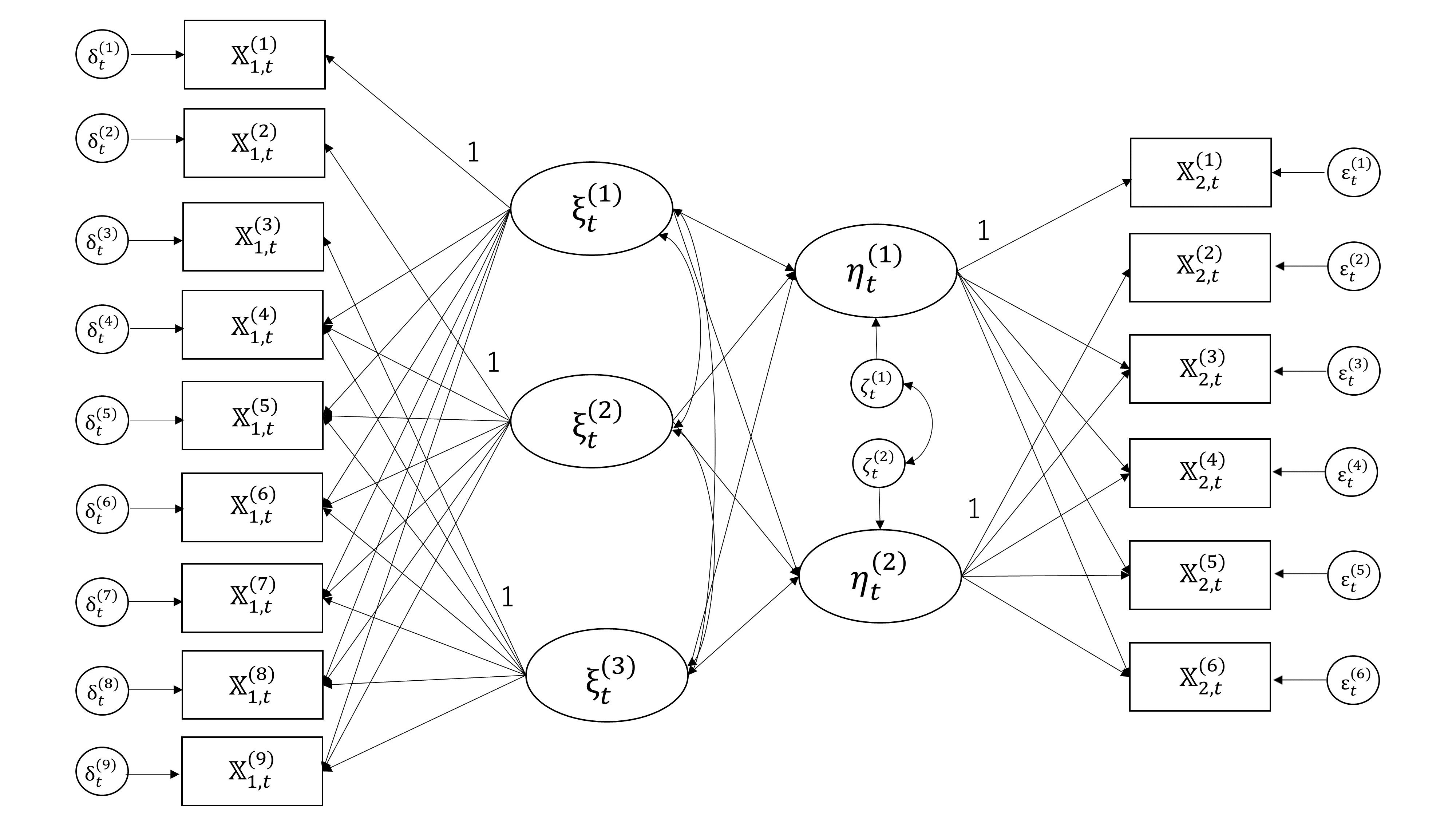}
    \caption{Path diagram of the correctly specified parametric model.}\label{corsparse}
\end{figure}
\subsection{Missspecified parametric model}
Set $p_1=9$, $p_2=6$, $k_1=2$ and $k_2=2$ in the parametric model (\ref{X})-(\ref{zetaP}). Suppose
\begin{align*}
    {\bf{\Lambda}}_{x_1}=\Bigl(\mathbb{I}_2, {\bf{A}}_{x_1}^{\top}\Bigr)^{\top},\quad
    {\bf{\Lambda}}_{x_2}=\Bigl(
    \mathbb{I}_2, {\bf{A}}_{x_2}^{\top}\Bigr)^{\top}
\end{align*}
where ${\bf{A}}_{x_1}\in\mathbb{R}^{6\times 2}$, ${\bf{A}}_{x_2}\in\mathbb{R}^{4\times 2}$, ${\bf{A}}_{x_2}$ is a full column rank matrix. Furthermore, it is assumed that ${\bf{\Lambda}}_{x_1}$ satisfies the identifiability condition of Theorem 5.1 in Anderson and Rubin \cite{Anderson(1956)}. The parameter is
\begin{align*}
    \theta=\Bigl(\vec {\bf{A}}_{x_1}^{\top}, \vec {\bf{A}}_{x_2}^{\top}, \vec{{\bf{\Gamma}}}^{\top}, \vech{\bf{\Sigma}}_{\xi\xi}^{\top},\diag{\bf{\Sigma}}_{\delta\delta}^{\top},
    \diag{\bf{\Sigma}}_{\varepsilon\varepsilon}^{\top},\vech{\bf{\Sigma}}_{\zeta\zeta}^{\top}\Bigr)^{\top}\in\Theta
\end{align*}
where 
\begin{align*}
    \Theta&=[-100,100]^{26}\times [0.1,100]\times[-100,100]\times [0.1,100]^{17}\times[-100,100]\times[0.1,100],
\end{align*}
${\bf{\Gamma}}\in\mathbb{R}^{2\times 2}$ is non-singular,
${\bf{\Sigma}}_{\xi\xi}\in\mathbb{R}^{2\times 2}$ and
${\bf{\Sigma}}_{\zeta\zeta}\in\mathbb{R}^{2\times 2}$
are positive definite matrices, 
${\bf{\Sigma}}_{\delta\delta}\in\mathbb{R}^{9\times 9}$ and  
${\bf{\Sigma}}_{\varepsilon\varepsilon}\in\mathbb{R}^{6\times 6}$ 
are positive definite diagonal matrices.
Figure \ref{misssparse} shows the path diagram of the missspecified parametric model.
\subsection{Simulation results}
Let $(n,h_n,T)=(10^4,10^{-4},1)$. 10,000 independent sample paths are generated from the true model. Let $\delta=0.1$, $\lambda_{1,n}=n^{-0.6}$, $\lambda_{2,n}=\delta^{-1}$ and $\gamma=4$. To optimize $\mathbb{F}_n$, we use optim() with the BFGS method in R language. The subgradient method is used in order to optimize $\mathbb{Q}_{\mathbb{I},n}(\theta)$. See Shor \cite{Shor(2012)} for the subgradient method. The initial value of the optimization is set to $\theta_0$.

\subsubsection{Correctly specified parametric model}
Table \ref{POtesttable1} shows the sample mean and sample standard deviation of $\check{\mathbb{T}}_n$. Figure \ref{testPOfigure1} shows Histogram, Q-Q plot and empirical distribution of $\check{\mathbb{T}}_n$. From Table \ref{POtesttable1} and Figure \ref{testPOfigure1}, we can see that $\check{\mathbb{T}}_n$ converges in distribution to $\chi^2_{87}$ under $H_0$.
\subsubsection{Missspecified parametric model}
Table \ref{POtesttable2} shows the number of rejections of the penalized quasi-likelihood ratio test. It seems from Table \ref{POtesttable2} that Theorem 10 holds true for this example.
\clearpage
\begin{figure}[h]
    \ \\ \ \\ \ \\\ \\ \ \\ \ \\ \ \\ \ \\
    \includegraphics[width=0.8\columnwidth]{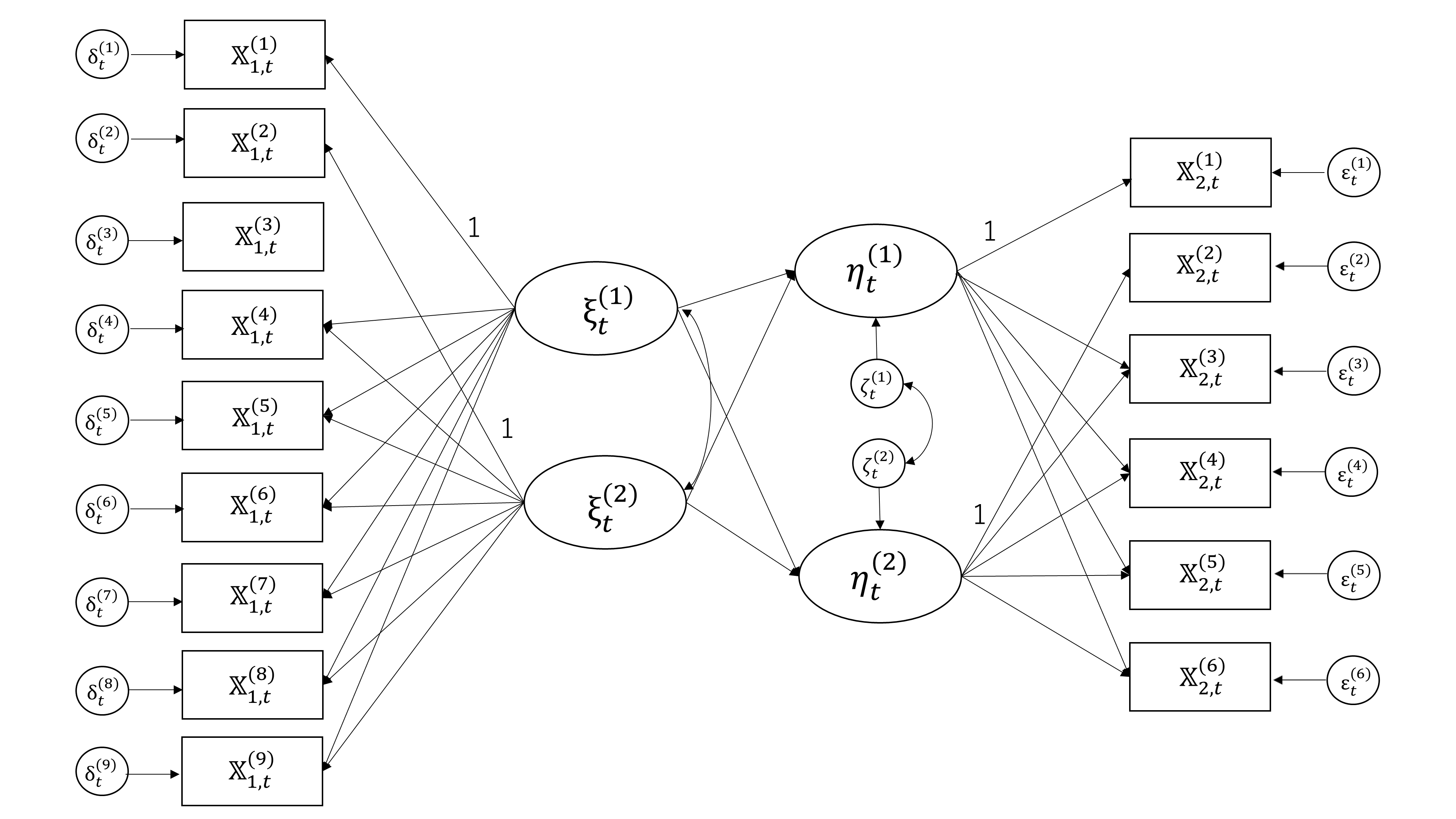}
    \caption{Path diagram of the missspecified parametric model.}\label{misssparse}
\end{figure}
\begin{table}[h]
    \ \\ \ \\ \ \\\ \\ \ \\ \ \\ \ \\ \ \\
    \begin{tabular}{cc}
    \hline
    Mean\ \  (True value) &\quad  86.684\ \ (87.000)\\
    SD\ \ (Theoretical value) &\quad 13.096\ \ (13.191)\\ \\
    \end{tabular}
\caption{Sample mean and sample standard deviation (SD) of the test statistic $\check{\mathbb{T}}_{n}$.} \label{POtesttable1}
\end{table}
\clearpage
\begin{figure}[h]
    \ \\ \ \\ \ \\\ \\ \ \\ \ \\ \ \\ \ \\ \ \\ \ \\ \ \\ \ \\ \ \\
    \centering
    \includegraphics[width=0.32\columnwidth]{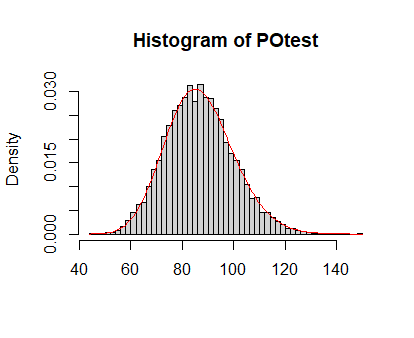}
    \includegraphics[width=0.32\columnwidth]{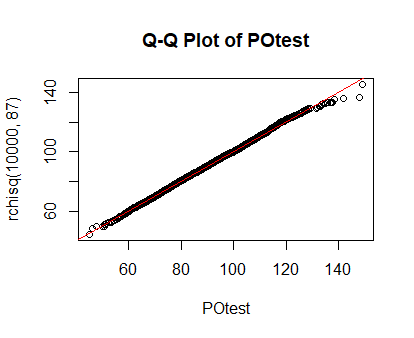}
    \includegraphics[width=0.32\columnwidth]{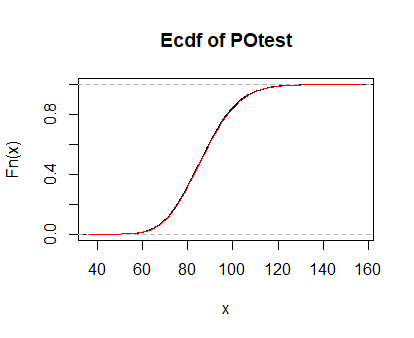}
\caption{Histogram (left), Q-Q plot (middle) and empirical distribution (right) of the test statistic $\check{\mathbb{T}}_{n}$. The red lines are theoretical curve.}\label{testPOfigure1}
\end{figure}
\begin{table}[h]
    \ \\ \ \\ \ \\ \ \\
    \centering
    \begin{tabular}{cc}
    \hline
    Missspecified parametric model &\quad 10000\\ \ \\
    \end{tabular}
\caption{The number of rejections of the penalized quasi-likelihood ratio test.}\label{POtesttable2}
\end{table}
\clearpage
\section{Proofs}
Let $\Delta Y_i=Y_{t_{i}^n}-Y_{t_{i-1}^n}$ for any stochastic process $Y_t$. Set
\begin{align*}
    \mathbb{Q}_{\xi\xi,0}&=\frac{1}{T}\sum_{i=1}^n(\Delta\xi_{0,i})(\Delta\xi_{0,i})^\top, \ \mathbb{Q}_{\delta\delta,0}=\frac{1}{T}\sum_{i=1}^n(\Delta\delta_{0,i})(\Delta\delta_{0,i})^\top, \ \mathbb{Q}_{\varepsilon\varepsilon,0}=\frac{1}{T}\sum_{i=1}^n(\Delta\varepsilon_{0,i})(\Delta\varepsilon_{0,i})^\top,\\
    \mathbb{Q}_{\zeta\zeta,0}&=\frac{1}{T}\sum_{i=1}^n(\Delta\zeta_{0,i})(\Delta\zeta_{0,i})^\top, \
    \mathbb{Q}_{\xi\delta,0}=\frac{1}{T}\sum_{i=1}^n(\Delta\xi_{0,i})(\Delta\delta_{0,i})^\top, \ \mathbb{Q}_{\xi\varepsilon,0}=\frac{1}{T}\sum_{i=1}^n(\Delta\xi_{0,i})(\Delta\varepsilon_{0,i})^\top,\\
    \mathbb{Q}_{\xi\zeta,0}&=\frac{1}{T}\sum_{i=1}^n(\Delta\xi_{0,i})(\Delta\zeta_{0,i})^\top, \  \mathbb{Q}_{\delta\varepsilon,0}=\frac{1}{T}\sum_{i=1}^n(\Delta\delta_{0,i})(\Delta\varepsilon_{0,i})^\top, \  \mathbb{Q}_{\delta\zeta,0}=\frac{1}{T}\sum_{i=1}^n(\Delta\delta_{0,i})(\Delta\zeta_{0,i})^\top,\\
    \mathbb{Q}_{\varepsilon\zeta,0}&=\frac{1}{T}\sum_{i=1}^n(\Delta\varepsilon_{0,i})(\Delta\zeta_{0,i})^{\top}.
\end{align*}
Define $\bar{r}=\sum_{i=1}^4 r_i$. Let 
\begin{align*}
    \mathscr{F}^{n}_{i}
    =\sigma(W_{1,s},W_{2,s},W_{3,s},W_{4,s},s\leq t_{i}^n)
\end{align*}
for $i=1,\cdots n$. Without loss of generality, $T=1$ is supposed. We then have $h_n=n^{-1}$.
\begin{lemma}\label{Qnonlemma}
Under \textrm{\textbf{[A1]}}, \textrm{\textbf{[B1]}},
\textrm{\textbf{[C1]}} and \textrm{\textbf{[D1]}}, as $h_n\longrightarrow 0$,
\begin{align*}
    \mathbb{Q}_{\xi\xi,0}\stackrel{P}{\longrightarrow}{\bf{\Sigma}}_{\xi\xi,0}, \  
    \mathbb{Q}_{\delta\delta,0}\stackrel{P}{\longrightarrow}{\bf{\Sigma}}_{\delta\delta,0}, \ 
    \mathbb{Q}_{\varepsilon\varepsilon,0}\stackrel{P}{\longrightarrow}
    {\bf{\Sigma}}_{\varepsilon\varepsilon,0},\qquad\qquad\quad\\
    \mathbb{Q}_{\zeta\zeta,0}\stackrel{P}{\longrightarrow} {\bf{\Sigma}}_{\zeta\zeta,0},\ \mathbb{Q}_{\xi\delta,0}\stackrel{P}{\longrightarrow} O_{k_1\times p_1},\
    \mathbb{Q}_{\xi\varepsilon,0}\stackrel{P}{\longrightarrow} O_{k_1\times p_2},\qquad\qquad\\
    \mathbb{Q}_{\xi\zeta,0}\stackrel{P}{\longrightarrow} O_{k_1\times k_2},\  \mathbb{Q}_{\delta\varepsilon,0}\stackrel{P}{\longrightarrow} O_{p_1\times p_2},\ \mathbb{Q}_{\delta\zeta} \stackrel{P}{\longrightarrow}O_{p_1\times k_2}, \  \mathbb{Q}_{\varepsilon\zeta,0}\stackrel{P}{\longrightarrow} O_{p_2\times k_2}.
\end{align*}
\end{lemma}
\begin{proof}
It holds from Lemma 7 in Kessler \cite{kessler(1997)} that
\begin{align*}
    \sum_{i=1}^n \E\Bigl[\Delta\xi_{0,i}^{(j_1)}\Delta\xi_{0,i}^{(j_2)}\big|\mathscr{F}^{n}_{i-1}\Bigl]
    &=\sum_{i=1}^n\bigl\{h_n({\bf{S}}_{1,0}{\bf{S}}_{1,0}^\top)_{j_1j_2}+R_i(h_n^2,\xi)\bigr\}\\
    &=({\bf{\Sigma}}_{\xi\xi,0})_{j_1j_2}+\frac{1}{n}\times\frac{1}{n}\sum_{i=1}^n R_i(1,\xi)
    \stackrel{P}{\longrightarrow}({\bf{\Sigma}}_{\xi\xi,0})_{j_1j_2}
\end{align*}
for $j_1,j_2=1,\cdots,k_1$, which implies
\begin{align*}
    \sum_{i=1}^n\E\Bigl[\Delta\xi_{0,i}\Delta\xi_{0,i}^{\top}
    \big|\mathscr{F}^{n}_{i-1}\Bigl]\stackrel{P}{\longrightarrow}{\bf{\Sigma}}_{\xi\xi,0}.
\end{align*}
Since 
\begin{align*}
    &\quad\ \sum_{i=1}^n \E\Bigl[\Delta\xi_{0,i}^{(j_1)}\Delta\xi_{0,i}^{(j_2)}
    \Delta\xi_{0,i}^{(j_3)}\Delta\xi_{0,i}^{(j_4)}\big|\mathscr{F}^{n}_{i-1}\Bigl]
    \stackrel{P}{\longrightarrow} 0
\end{align*}
for $j_1,j_2,j_3,j_4=1,\cdots,k_1$ in an analogous manner, we obtain
\begin{align*}
    \sum_{i=1}^n\E\Bigl[\Delta\xi_{0,i}\Delta\xi_{0,i}^{\top}
    \Delta\xi_{0,i}\Delta\xi_{0,i}^{\top}\big|\mathscr{F}^{n}_{i-1}\Bigl]
    \stackrel{P}{\longrightarrow}O_{k_1\times k_1}.
\end{align*}
Therefore, it follows from Lemma 9 in Genon-Catalot and Jacod \cite{Genon(1993)} that
\begin{align*}
    \mathbb{Q}_{\xi\xi,0}=\sum_{i=1}^n\Delta\xi_{0,i}\Delta\xi_{0,i}^{\top}
    \stackrel{P}{\longrightarrow}{\bf{\Sigma}}_{\xi\xi,0}.
\end{align*}
In a similar way, the other results can be shown.
\end{proof}
\begin{lemma}\label{EXlemmanon}
Under \textrm{\textbf{[A1]}}, \textrm{\textbf{[B1]}},
\textrm{\textbf{[C1]}} and
\textrm{\textbf{[D1]}}, as $h_n\longrightarrow 0$, 
\begin{align}
    \sum_{i=1}^{n}\left|\E\left[\sqrt{n}\Delta \mathbb{X}_{i}^{(j_1)}\Delta \mathbb{X}_{i}^{(j_2)}-\frac{1}{\sqrt{n}}({\bf{\Sigma}}_0)_{j_1j_2}\Big|\mathscr{F}^{n}_{i-1}\right]\right|\stackrel{P}{\longrightarrow}0 \label{EXXnonp}
\end{align}
for $j_1,j_2=1,\cdots,p$,
\begin{align}
    \begin{split}
    &\quad\ \ \sum_{i=1}^{[nt]}\E\left[\left\{\sqrt{n}\Delta \mathbb{X}_{i}^{(j_1)}\Delta \mathbb{X}_{i}^{(j_2)}-\frac{1}{\sqrt{n}}({\bf{\Sigma}}_0)_{j_1j_2}\right\}\right.\\
    &\left.\qquad\qquad\qquad\qquad\times\left\{\sqrt{n}\Delta \mathbb{X}_{i}^{(j_3)}\Delta \mathbb{X}_{i}^{(j_4)}-\frac{1}{\sqrt{n}}({\bf{\Sigma}}_0)_{j_3j_4}\right\}\Big|\mathscr{F}^{n}_{i-1}\right]\\
    &\qquad\qquad-\sum_{i=1}^{[nt]}\E\left[\sqrt{n}\Delta \mathbb{X}_{i}^{(j_1)}\Delta \mathbb{X}_{i}^{(j_2)}-\frac{1}{\sqrt{n}}({\bf{\Sigma}}_0)_{j_1j_2}\Big|\mathscr{F}^{n}_{i-1}\right]\\
    &\qquad\qquad\qquad\qquad\qquad\qquad\times\E\left[\sqrt{n}\Delta \mathbb{X}_{i}^{(j_3)}\Delta \mathbb{X}_{i}^{(j_4)}-\frac{1}{\sqrt{n}}({\bf{\Sigma}}_0)_{j_3j_4}\Big|\mathscr{F}^{n}_{i-1}\right]\\
    &\stackrel{P}{\longrightarrow}t\bigl\{({\bf{\Sigma}}_0)_{j_{1}j_{3}}({\bf{\Sigma}}_0)_{j_{2}j_{4}}+({\bf{\Sigma}}_0)_{j_{1}j_{4}}({\bf{\Sigma}}_0)_{j_{2}j_{3}}\bigr\},\quad\forall t\in [0,1]\label{EXXXXnonp}
    \end{split}
\end{align}
for $j_1,j_2,j_3,j_4=1,\cdots,p$, 
\begin{align}
    \sum_{i=1}^{[nt]}\E\left[\sqrt{n}\Delta \mathbb{X}_{i}^{(j_1)}\Delta \mathbb{X}_{i}^{(j_2)}\Delta \bar{W}_i^{(j_3)}\big|\mathscr{F}^{n}_{i-1}\right]\stackrel{P}{\longrightarrow}0,\quad\forall t\in [0,1]\label{EXWnonp}
\end{align}
for $j_1,j_2=1,\cdots, p, j_3=1,\cdots,\bar{r}$, and
\begin{align}
    \begin{split}
    &\sum_{i=1}^n\E\left[\Bigl|\sqrt{n}\Delta \mathbb{X}_i^{(j_1)}\Delta \mathbb{X}_i^{(j_2)}-\frac{1}{\sqrt{n}}({\bf{\Sigma}}_0)_{j_1j_2}\Bigr|^4\Big|\mathscr{F}^{n}_{i-1}\right]\stackrel{P}{\longrightarrow}0\label{EXX4nonp}
    \end{split}
\end{align}
for $j_1,j_2=1,\cdots, p$.
\end{lemma}
\begin{proof}
    See Appendix \ref{EXlemmanonproof}.
\end{proof}
\begin{proof}[\textbf{Proof of Theorem 1}]
We first show
\begin{align}
    \mathbb{Q}_{\mathbb{XX}}\stackrel{P}{\longrightarrow}{\bf{\Sigma}}_0.\label{Qxxprob}
\end{align}
In order to show (\ref{Qxxprob}), it is sufficient to show that
\begin{align}
    \mathbb{Q}_{\mathbb{X}_1\mathbb{X}_1}&\stackrel{P}{\longrightarrow}{\bf{\Lambda}}_{x_1,0}{\bf{\Sigma}}_{\xi\xi,0}{\bf{\Lambda}}_{x_1,0}^{\top}+{\bf{\Sigma}}_{\delta\delta,0}\ (={\bf{\Sigma}}_0^{11}),\label{Qx1x1cons}\\
    \mathbb{Q}_{\mathbb{X}_1\mathbb{X}_2}&\stackrel{P}{\longrightarrow}
    {\bf{\Lambda}}_{x_1,0}{\bf{\Sigma}}_{\xi\xi,0}{\bf{\Gamma}}_0^{\top}
    {\bf{\Psi}}_0^{-1\top}{\bf{\Lambda}}_{x_2,0}^{\top}\ (={\bf{\Sigma}}_0^{12}),\label{Qx1x2cons}\\
    \mathbb{Q}_{\mathbb{X}_2\mathbb{X}_2}&\stackrel{P}{\longrightarrow}{\bf{\Lambda}}_{x_2,0}{\bf{\Psi}}_0^{-1}({\bf{\Gamma}}_0{\bf{\Sigma}}_{\xi\xi,0}{\bf{\Gamma}}_0^{\top}+{\bf{\Sigma}}_{\zeta\zeta,0}){\bf{\Psi}}_0^{-1\top}{\bf{\Lambda}}_{x_2,0}^{\top}+{\bf{\Sigma}}_{\varepsilon\varepsilon,0}\ (={\bf{\Sigma}}_0^{22}). \label{Qx2x2cons}
\end{align}
Using Lemma \ref{Qnonlemma} and Slutsky's theorem, one gets
\begin{align*}
    \mathbb{Q}_{\mathbb{X}_1\mathbb{X}_1}
    &=\sum_{i=1}^n({\bf{\Lambda}}_{x_1,0}\Delta\xi_{0,i}+\Delta\delta_{0,i})({\bf{\Lambda}}_{x_1,0}\Delta\xi_{0,i}+\Delta\delta_{0,i})^{\top}\\
    &={\bf{\Lambda}}_{x_1,0}\mathbb{Q}_{\xi\xi,0}{\bf{\Lambda}}_{x_1,0}^{\top}
    +{\bf{\Lambda}}_{x_1,0}\mathbb{Q}_{\xi\delta,0}
    +\mathbb{Q}_{\xi\delta,0}^{\top}{\bf{\Lambda}}_{x_1,0}^{\top}
    +\mathbb{Q}_{\delta\delta,0}\\
    &\stackrel{P}{\longrightarrow}{\bf{\Lambda}}_{x_1,0}{\bf{\Sigma}}_{\xi\xi,0}{\bf{\Lambda}}_{x_1,0}^{\top}+{\bf{\Sigma}}_{\delta\delta,0},
\end{align*}
which yields (\ref{Qx1x1cons}). In a similar way, since
\begin{align*}	 
    \qquad \mathbb{Q}_{\mathbb{X}_1\mathbb{X}_2}&={\bf{\Lambda}}_{x_1,0}\mathbb{Q}_{\xi\xi,0}{\bf{\Gamma}}_0^{\top}
    {\bf{\Psi}}_0^{-1\top}{\bf{\Lambda}}_{x_2,0}^{\top}
    +{\bf{\Lambda}}_{x_1,0}\mathbb{Q}_{\xi\zeta,0}
    {\bf{\Psi}}_0^{-1\top}{\bf{\Lambda}}_{x_2,0}^{\top}
    +{\bf{\Lambda}}_{x_1,0}\mathbb{Q}_{\xi\varepsilon,0}\qquad\quad\\
    &\quad+\mathbb{Q}_{\xi\delta,0}^{\top}{\bf{\Gamma}}_0^{\top}{\bf{\Psi}}_0^{-1\top}
    {\bf{\Lambda}}_{x_2,0}^{\top}+\mathbb{Q}_{\delta\zeta,0}{\bf{\Psi}}_0^{-1\top}
    {\bf{\Lambda}}_{x_2,0}^{\top}+\mathbb{Q}_{\delta\varepsilon,0}\\
    &\stackrel{P}{\longrightarrow}
    {\bf{\Lambda}}_{x_1,0}{\bf{\Sigma}}_{\xi\xi,0}
    {\bf{\Gamma}}_0^{\top}{\bf{\Psi}}_0^{-1\top}{\bf{\Lambda}}_{x_2,0}^{\top}
\end{align*}
and
\begin{align*}
    \mathbb{Q}_{\mathbb{X}_2\mathbb{X}_2}&={\bf{\Lambda}}_{x_2,0}{\bf{\Psi}}_0^{-1}{\bf{\Gamma}}_0
    \mathbb{Q}_{\xi\xi,0}{\bf{\Gamma}}_0^{\top}{\bf{\Psi}}_0^{-1\top}
    {\bf{\Lambda}}_{x_2,0}^{\top}
    +{\bf{\Lambda}}_{x_2,0}{\bf{\Psi}}_0^{-1}{\bf{\Gamma}}_0 \mathbb{Q}_{\xi\zeta,0}{\bf{\Psi}}_0^{-1\top}{\bf{\Lambda}}_{x_2,0}^{\top}\\
    &\quad+{\bf{\Lambda}}_{x_2,0}{\bf{\Psi}}_0^{-1}{\bf{\Gamma}}_0 \mathbb{Q}_{\xi\varepsilon,0}+{\bf{\Lambda}}_{x_2,0}{\bf{\Psi}}_0^{-1}\mathbb{Q}_{\xi\zeta,0}^{\top}{\bf{\Gamma}}_0^{\top}{\bf{\Psi}}_0^{-1\top}{\bf{\Lambda}}_{x_2,0}^{\top}\\
    &\quad+{\bf{\Lambda}}_{x_2,0}{\bf{\Psi}}_0^{-1}\mathbb{Q}_{\zeta\zeta,0}
    {\bf{\Psi}}_0^{-1\top}{\bf{\Lambda}}_{x_2,0}^{\top}
    +{\bf{\Lambda}}_{x_2,0}{\bf{\Psi}}_0^{-1}\mathbb{Q}_{\varepsilon\zeta,0}^{\top}
    \\
    &\quad+\mathbb{Q}_{\xi\varepsilon,0}^{\top}{\bf{\Gamma}}_0^{\top}
    {\bf{\Psi}}_0^{-1\top}{\bf{\Lambda}}_{x_2,0}^{\top}
    +\mathbb{Q}_{\varepsilon\zeta,0}{\bf{\Psi}}_0^{-1\top}{\bf{\Lambda}}_{x_2,0}^{\top}
    +\mathbb{Q}_{\varepsilon\varepsilon,0}\\
    &\stackrel{P}{\longrightarrow}{\bf{\Lambda}}_{x_2,0}{\bf{\Psi}}_0^{-1}({\bf{\Gamma}}_0{\bf{\Sigma}}_{\xi\xi,0}{\bf{\Gamma}}_0^{\top}+{\bf{\Sigma}}_{\zeta\zeta,0}){\bf{\Psi}}_0^{-1\top}{\bf{\Lambda}}_{x_2,0}^{\top}+{\bf{\Sigma}}_{\varepsilon\varepsilon,0},
\end{align*}
we obtain (\ref{Qx1x2cons}) and (\ref{Qx2x2cons}). Next, we prove
\begin{align}
    \sqrt{n}(\vech{\mathbb{Q}_{\mathbb{XX}}}-\vech{{\bf{\Sigma}}_0})\stackrel{d}{\longrightarrow} N_{\bar{p}}(0,{\bf{W}}_0),\label{vechasymnon}
\end{align}
where 
\begin{align*}
    {\bf{W}}_0=2\mathbb{D}_{p}^{+}({\bf{\Sigma}}_0\otimes{\bf{\Sigma}}_0)\mathbb{D}_{p}^{+\top}.
\end{align*}
Consider the following convergence:
\begin{align}
    \sqrt{n}(\vec{\mathbb{Q}_{\mathbb{XX}}}-\vec{{\bf{\Sigma}}_0})\stackrel{d}{\longrightarrow} N_{p^2}(0,{\bar{\bf{W}}}_0),\label{vecasymnon}
\end{align}
where
\begin{align*}
    (\bar{{\bf{W}}}_{0})_{p(j_1-1)+j_2,\ p(j_3-1)+j_4}=({\bf{\Sigma}}_0)_{j_1j_3}({\bf{\Sigma}}_0)_{j_2j_4}+({\bf{\Sigma}}_0)_{j_1j_4}({\bf{\Sigma}}_0)_{j_2j_3}
\end{align*}
for $j_1,j_2,j_3,j_4=1,\cdots,p$. If (\ref{vecasymnon}) holds, then it follows from the continuous mapping theorem that
\begin{align}
    \begin{split}
    \sqrt{n}(\vech{\mathbb{Q}_{\mathbb{XX}}}-\vech{{\bf{\Sigma}}_0})&=f(\sqrt{n}(\vec{\mathbb{Q}_{\mathbb{XX}}}-\vec{{\bf{\Sigma}}_0}))\\
    &\stackrel{d}{\longrightarrow} f(N_{p^2}(0,{\bf{\bar{W}}}_0))\sim N_{\bar{p}}\bigl(0,\mathbb{D}_{p}^{+}{\bf{\bar{W}}}_0\mathbb{D}_{p}^{+\top}\bigr), \label{veccon}
    \end{split}
\end{align}
where $f(x)=\mathbb{D}_{p}^{+}x$ for $x\in\mathbb{R}^{\bar{p}}$. In an analogous manner to Lemma 6 in Kusano and Uchida \cite{Kusano(2022)},
\begin{align*}
    \mathbb{D}_{p}^{+}\bar{{\bf{W}}}_0\mathbb{D}_{p}^{+\top}={\bf{W}}_0,
\end{align*}
so that we obtain (\ref{vechasymnon}) from (\ref{veccon}). Consequently, it is sufficient to prove (\ref{vecasymnon}) in order to prove (\ref{vechasymnon}). Let
\begin{align*}
    L_{i,n}&=\sqrt{n}\vec{\Delta \mathbb{X}_{i}\Delta \mathbb{X}_{i}^{\top}}-\frac{1}{\sqrt{n}}\vec{{\bf{\Sigma}}_0}.
\end{align*}
The left side of (\ref{vecasymnon}) is expressed as
\begin{align*}
    \sqrt{n}(\vec \mathbb{Q}_{\mathbb{XX}}-\vec{\bf{\Sigma}}_0)=\sum_{i=1}^n L_{i,n}.
\end{align*}
From Theorem 3.2 in Jacod \cite{Jacod(1997)}, it is sufficient to prove the following convergences in order to show (\ref{vecasymnon}):
\begin{align}
    \sup_{t\in[0,1]}\Bigl|\sum_{i=1}^{[nt]}\E\left[L_{i,n}
    |\mathscr{F}^{n}_{i-1}\right]\Bigr|&\stackrel{P\ }{\longrightarrow}0,\label{sup}\\
    \sum_{i=1}^{[nt]}\E\left[L_{i,n}L_{i,n}^{\top}|\mathscr{F}^{n}_{i-1}\right]-\E\left[L_{i,n}|\mathscr{F}^{n}_{i-1}\right]\E\left[L_{i,n}|\mathscr{F}^{n}_{i-1}\right]^{\top}&\stackrel{P\ }{\longrightarrow}t\bar{{\bf{W}}}_0,\quad\forall t\in [0,1],\label{LLnon}\\
     \sum_{i=1}^{[nt]}\E\left[L_{i,n}(\Delta\bar{W}_i)^{\top}
    |\mathscr{F}^{n}_{i-1}\right]&\stackrel{P}{\longrightarrow}O_{p^2\times \bar{r}},\quad\forall t\in [0,1],\label{W}\\
    \sum_{i=1}^{n}\E\left[|L_{i,n}|^4
    |\mathscr{F}^{n}_{i-1}\right]&\stackrel{P}{\longrightarrow}0,\label{E4}\\
     \sum_{i=1}^{[nt]}\E\left[L_{i,n}\Delta N_i
    |\mathscr{F}^{n}_{i-1}\right]&\stackrel{P}{\longrightarrow}0,\quad\forall t\in [0,1],\ \forall N_t\in \mathcal{M}_b(\bar{W}^{\perp}),\label{N}
\end{align}
where $\bar{W}_t=(W_{1,t}^{\top},W_{2,t}^{\top},W_{3,t}^{\top},W_{4,t}^{\top})^{\top}$ and $\mathcal{M}_b(\bar{W}^{\perp})$ is the class of all bounded martingales which is orthogonal to $\bar{W}_t$. First, we will prove (\ref{sup}). Since 
\begin{align*}
    \sup_{t\in[0,1]}\Bigl|\sum_{i=1}^{[nt]}\E\left[L_{i,n}
    |\mathscr{F}^{n}_{i-1}\right]\Bigr|&\leq\sup_{t\in[0,1]}\sum_{i=1}^{[nt]}\Bigl|\E\left[L_{i,n}
    |\mathscr{F}^{n}_{i-1}\right]\Bigr|\\
    &\leq\sum_{i=1}^{n}\Bigl|\E\left[L_{i,n}
    |\mathscr{F}^{n}_{i-1}\right]\Bigr|\leq C_p\sum_{i=1}^{n}\sum_{j=1}^{p^2}\Bigl|\E\left[L_{i,n}^{(j)}
    \big|\mathscr{F}^{n}_{i-1}\right]\Bigr|,
\end{align*}
it is sufficient to prove 
\begin{align}
    \sum_{i=1}^{n}\Bigl|\E\left[L_{i,n}^{(j)}
    \big|\mathscr{F}^{n}_{i-1}\right]\Bigr|\stackrel{P}{\longrightarrow}0 \label{ELj}
\end{align}
for $j=1,\cdots,p^2$ in order to prove (\ref{sup}). It holds from (\ref{EXXnonp}) that we obtain (\ref{ELj}). Furthermore, (\ref{EXXXXnonp}) and (\ref{EXWnonp}) yield (\ref{LLnon}) and (\ref{W}) respectively. Next, we prove (\ref{E4}). Note that
\begin{align}
    \begin{split}
    0&\leq \sum_{i=1}^n\E\left[|L_{i,n}|^4|\mathscr{F}^{n}_{i-1}\right]\\
    &\leq\sum_{i=1}^n\E\left[\Bigl|
    \sum_{u=1}^{p^2}L_{i,n}^{(u)2}\Bigr|^2\Big|\mathscr{F}^{n}_{i-1}\right]\leq C_{p}\sum_{u=1}^{p^2}\sum_{i=1}^n\E\left[|L_{i,n}^{(u)}|^4\big|\mathscr{F}^{n}_{i-1}\right]. 
    \label{L4ine}
    \end{split}
\end{align}
From (\ref{EXX4nonp}), we have
\begin{align*}
    \sum_{i=1}^n\E\left[|L_{i,n}^{(u)}|^4\big|\mathscr{F}^{n}_{i-1}\right]\stackrel{P}{\longrightarrow} 0
\end{align*}
for $u=1,\cdots,p^2$, so that 
(\ref{E4}) holds from (\ref{L4ine}). Finally, we prove (\ref{N}). Since $L_{i,n}-\mathbb{E}[L_{i,n}|\mathscr{F}^{n}_{i-1}]$ is the zero-mean martingale, it holds from the martingale representation theorem that there exists a stochastic process $\eta_{t,n}\in\mathbb{R}^{p^2\times \bar{r}}$ such that
\begin{align*}
    L_{i,n}-\mathbb{E}[L_{i,n}|\mathscr{F}^{n}_{i-1}]=\int_{t_{i-1}^n}^{t_{i}^n}\eta_{t,n}d\bar{W}_t.
\end{align*}
Note that $[\bar{W},N]_t=0$ since $N_t\in \mathcal{M}_b(\bar{W}^{\perp})$, where $[\bar{W},N]_t$ is the covariation process of $\bar{W}_t$ and $N_t$. It follows
\begin{align*}
    \E\Bigl[L_{i,n}\Delta N_i|\mathscr{F}^{n}_{i-1}\Bigr]&=\E\Bigl[\bigl(L_{i,n}-\mathbb{E}[L_{i,n}|\mathscr{F}^{n}_{i-1}]\bigr)\Delta N_i\big|\mathscr{F}^{n}_{i-1}\Bigr]\\
    &=\E\left[\left(\int_{t_{i-1}^n}^{t_{i}^n}\eta_{t,n}d\bar{W}_t\right)\left(\int_{t_{i-1}^n}^{t_{i}^n}d N_t\right)\Big|\mathscr{F}^{n}_{i-1}\right]\\
    &=\E\left[\int_{t_{i-1}^n}^{t_{i}^n}\eta_{t,n}d[\bar{W},N]_t\Big|\mathscr{F}^{n}_{i-1}\right]=0,
\end{align*}
which yields (\ref{N}). Therefore, we obtain (\ref{vecasymnon}).
\end{proof}
\begin{proof}[\textbf{Proofs of Theorems 2-4}]
We can prove the results in the same way as the proofs of Theorems 2-4 in Kusano and Uchida \cite{Kusano(2022)}. See also Appendix \ref{Proof2-4}.
\end{proof}
\begin{proof}[\textbf{Proof of Theorem 5}]
See Appendix \ref{Qproof}.
\end{proof}
\begin{proof}[\textbf{Proofs of Theorems 6-8}]
Using Theorem 5, we can show Theorems 6-8 in a similar manner to Theorems 2-4.
\end{proof}
Set
\begin{align*}
    \mathcal{F}_{0}=\Bigl\{j\in\{1,\cdots,q\}\ \big|\ \theta^{(j)}_{0}=0\Bigr\}.
\end{align*}
Let $a_{n}=\max_{j\in\mathcal{F}_{1}}\kappa_{n}^{(j)}$ and $b_{n}=\min_{j\in\mathcal{F}_{0}}\kappa_{n}^{(j)}$.
\begin{lemma}\label{abprob}
Under \textrm{\textbf{[A1]}}, \textrm{\textbf{[B1]}},
\textrm{\textbf{[C1]}},
\textrm{\textbf{[D1]}} and \textrm{\textbf{[E1]}} (i),
as $h_n\longrightarrow0$,
$\sqrt{n}\lambda_{1,n}\longrightarrow0$ and
$\sqrt{n}\lambda_{2,n}\longrightarrow\infty$,
\begin{align*}
    \sqrt{n}a_{n}\stackrel{P}{\longrightarrow} 0,\quad \frac{1}{\sqrt{n}b_{n}}\stackrel{P}{\longrightarrow}0.
\end{align*}
\end{lemma}
\begin{proof}
Note that $\displaystyle 0<\min_{j\in\mathcal{F}_{1}}|\theta_{0}^{(j)}|-\delta$. On
\begin{align*}
    \Bigl\{|\hat{\theta}_{n}^{(j)}-\theta_{0}^{(j)}|<\min_{j\in\mathcal{F}_{1}}|\theta_{0}^{(j)}|-\delta\Bigr\},
\end{align*}
it holds that for $j\in{\mathcal{F}_{1}}$,
\begin{align*}
    |\hat{\theta}_{n}^{(j)}|=|(\theta_{0}^{(j)}-\hat{\theta}_{n}^{(j)})-\theta_{0}^{(j)}|&\geq|\theta_{0}^{(j)}|-|\theta_{0}^{(j)}-\hat{\theta}_{n}^{(j)}|\\
    &>\min_{j\in\mathcal{F}_{1}}|\theta_{0}^{(j)}|-\Bigl\{\min_{j\in\mathcal{F}_{1}}|\theta_{0}^{(j)}|-\delta\Bigr\}\geq \delta,
\end{align*}
so that 
\begin{align}
    \Bigl\{|\hat{\theta}_{n}^{(j)}-\theta_{0}^{(j)}|<\min_{j\in\mathcal{F}_{1}}|\theta_{0}^{(j)}|-\delta\Bigr\}\subset\Bigl\{|\hat{\theta}_{n}^{(j)}|>\delta\Bigr\}\label{thetaine}
\end{align}
for $j\in{\mathcal{F}_{1}}$. Set
\begin{align*}
    B_{n}=\bigcap_{j\in{\mathcal{F}_{1}}}\Bigl\{|\hat{\theta}_{n}^{(j)}|\geq\delta\Bigr\}.
\end{align*}
Theorem \ref{thetatheoremnon} and (\ref{thetaine}) yield
\begin{align}
    \mathbb{P}\bigl(B_{n}^{c}\bigr)
    &\leq\sum_{j\in{\mathcal{F}_{1}}}\mathbb{P}
    \Bigl(|\hat{\theta}_{n}^{(j)}|\leq\delta\Bigr)\leq\sum_{j\in{\mathcal{F}_{1}}}
    \mathbb{P}\Bigl(|\hat{\theta}_{n}^{(j)}-\theta_{0}^{(j)}|\geq\min_{j\in\mathcal{F}_{1}}|\theta_{0}^{(j)}|-\delta\Bigr)\longrightarrow 0\label{Cprob}
\end{align}
as $n\longrightarrow\infty$. Since
\begin{align*}
    \sqrt{n}a_n&=\sqrt{n}\max_{j\in\mathcal{F}_{1}}
    \kappa_{n}^{(j)}\\
    &=\sqrt{n}\lambda_{1,n}
    \max_{j\in\mathcal{F}_{1}}|\hat{\theta}_{n}^{(j)}|^{-\gamma}\leq \sqrt{n}\lambda_{1,n}\delta^{-\gamma}
\end{align*}
on $B_n$, we see from $\sqrt{n}\lambda_{1,n}\longrightarrow 0$ and (\ref{Cprob}) that for all $\varepsilon>0$,
\begin{align*}
    \mathbb{P}\Bigl(\sqrt{n}a_{n}>\varepsilon\Bigr)
    &\leq\mathbb{P}\Bigl(\bigl\{\sqrt{n}a_{n}>\varepsilon\bigr\}\cap B_{n}\Bigr)
    +\mathbb{P}\Bigl(\bigl\{\sqrt{n}a_{n}>\varepsilon\bigr\}\cap B_{n}^c\Bigr)\\
    &\leq\mathbb{P}\Bigl(\bigl\{\sqrt{n}\lambda_{1,n}\delta^{-\gamma} >\varepsilon\bigr\}\cap B_{n}\Bigr)+\mathbb{P}\bigl(B_{n}^c\bigr)\\
    &\leq\mathbb{P}\bigl(\sqrt{n}\lambda_{1,n} >\varepsilon\delta^{\gamma}\bigr)+\mathbb{P}\bigl(B_{n}^c\bigr)\longrightarrow 0
\end{align*}
as $n\longrightarrow\infty$, which implies $\sqrt{n}a_{n}\stackrel{P}{\longrightarrow}0$. Let
\begin{align*}
    C_{n}=\bigcap_{j\in{\mathcal{F}_{0}}}\Bigl\{|\hat{\theta}_{n}^{(j)}|<\delta\Bigr\}.
\end{align*}
From Theorem \ref{thetatheoremnon}, one gets
\begin{align}                  
    \mathbb{P}\bigl(C_{n}^{c}\bigr)
    \leq\sum_{j\in{\mathcal{F}_{0}}}
    \mathbb{P}\Bigl(|\hat{\theta}_{n}^{(j)}|\geq\delta\Bigr)
    =\sum_{j\in{\mathcal{F}_{0}}}\mathbb{P}
    \Bigl(|\hat{\theta}_{n}^{(j)}-\theta_0^{(j)}|\geq\delta\Bigr)
    \longrightarrow 0\label{Dprob}
\end{align}
as $n\longrightarrow\infty$. Note that on $C_n$,
\begin{align*}
    \sqrt{n}b_{n}
    =\sqrt{n}\min_{j\in\mathcal{F}_{0}}\kappa_{n}^{(j)}
    =\sqrt{n}\lambda_{2,n}.
\end{align*}
As it follows from $\sqrt{n}\lambda_{2,n}\longrightarrow\infty$ and (\ref{Dprob}) that 
\begin{align*}
    \mathbb{P}\left(\frac{1}{\sqrt{n}b_{n}}>\varepsilon\right)&\leq\mathbb{P}\left(\Bigl\{\frac{1}{\sqrt{n}b_{n}}>\varepsilon\Bigr\}\cap C_{n}\right)+\mathbb{P}\left(\Bigl\{\frac{1}{\sqrt{n}b_{n}}>\varepsilon\Bigr\}\cap C_{n}^{c}\right)\\
    &\leq\mathbb{P}\left(\Bigl\{\frac{1}{\sqrt{n}\lambda_{2,n}}>\varepsilon\Bigr\}\cap C_{n}\right)+\mathbb{P}\bigl(C_{n}^{c}\bigr)\\
    &\leq\mathbb{P}\left(\frac{1}{\sqrt{n}\lambda_{2,n}}>\varepsilon\right)+\mathbb{P}\bigl(C_{n}^{c}\bigr)\longrightarrow 0
\end{align*}
for all $\varepsilon>0$ as $n\longrightarrow\infty$, we have $\frac{1}{\sqrt{n}b_{n}}\stackrel{P}{\longrightarrow}0$.
\end{proof}
\begin{proof}[\textbf{Proofs of Lemmas \ref{Aeoracle}-\ref{Aeasym}}]
Since it holds from \textbf{[F1]} that $\tilde{G}_n\stackrel{P}{\longrightarrow} G$,
Assumption 1 in Suzuki and Yoshida \cite{Suzuki(2020)} is satisfied. Theorem 2 yields 
\begin{align*}
    \sqrt{n}(\hat{\theta}_n-\theta_0)=O_p(1),
\end{align*}
which satisfies Assumption 2 in Suzuki and Yoshida \cite{Suzuki(2020)}. From Lemma \ref{abprob}, one gets $\sqrt{n}a_n=o_p(1)$ and $\sqrt{n}b_{n}\stackrel{P}{\longrightarrow}\infty$. Hence, Theorems $2$ and $3$ in Suzuki and Yoshida \cite{Suzuki(2020)} imply
\begin{align*}
    \mathbb{P}\Bigl(\tilde{\mathcal{F}}_{G,n}
    =\mathcal{F}_{1}\Bigr)\stackrel{}{\longrightarrow}1
\end{align*}
and
\begin{align*}
    \sqrt{n}(\tilde{\theta}_{G,n}-\theta_{0})_{\mathcal{F}_{1}}-\mathfrak{G}_{G}\bigl\{\sqrt{n}(\hat{\theta}_{n}-\theta_{0})\bigr\}\stackrel{P}{\longrightarrow}0.
\end{align*}
Note that it follows from Lemma \ref{Aposlemma} that ${\bf{A}}(\theta_0)$ is a positive definite matrix. Theorem 2 yields
\begin{align*}
    \sqrt{n}(\hat{\theta}_n-\theta_0)\stackrel{d}{\longrightarrow}{\bf{A}}(\theta_0)^{-\frac{1}{2}}\zeta,
\end{align*}
where $\zeta$ is a $q$-dimensional standard normal random vector. Consequently, Assumption 3 in Suzuki and Yoshida \cite{Suzuki(2020)} is satisfied. Therefore, we see from  Theorem $3$ in Suzuki and Yoshida \cite{Suzuki(2020)} that as $G={\bf{A}}(\theta_0)$,
\begin{align*}
    \sqrt{n}(\tilde{\theta}_{G,n}-\theta_{0})_{\mathcal{F}_{1}}\stackrel{d}{\longrightarrow} N_{|\mathcal{F}_{1}|}\Bigl(0,{\bf{A}}_{\mathcal{F}}^{11}(\theta_0)^{-1}\Bigr). &\qedhere
\end{align*}
\end{proof}
The restricted parameter space is set to
\begin{align*}
    \underline{\Theta}=\Bigl\{\theta\in\Theta\ \big|\ \theta^{(j)}=0\ \ (j\in \mathcal{F}_{0})\Bigr\}.
\end{align*}
The restricted quasi-likelihood estimator $\underline{\theta}_{n}$ is defined as follows:
\begin{align*}
    \mathbb{F}_{n}(\underline{\theta}_{n})=\inf_{\theta\in\underline{\Theta}}\mathbb{F}_{n}(\theta).
\end{align*}
\begin{lemma}\label{seiyakutheta}
Under \textrm{\textbf{[A1]}}, \textrm{\textbf{[B1]}},
\textrm{\textbf{[C1]}},
\textrm{\textbf{[D1]}} and
\textrm{\textbf{[E1]}},
as $h_n\longrightarrow0$, 
\begin{align*}
    \sqrt{n}(\underline{\theta}_{n}-\theta_{0})_{\mathcal{F}_{1}}
    \stackrel{d}{\longrightarrow}N_{|\mathcal{F}_{1}|}\Bigl(0, {\bf{A}}_{\mathcal{F}}^{11}(\theta_0)^{-1}\Bigr).
\end{align*}
\end{lemma}
\begin{proof}
    See Appendix \ref{seiyakuthetaproof}.
\end{proof}
\begin{proof}[\textbf{Proof of Lemma \ref{POasym}}]
Lemma \ref{seiyakutheta} yields
\begin{align*}
    \sqrt{n}(\underline{\theta}_{n}-\theta_{0})_{\mathcal{F}_{1}}
    \stackrel{d}{\longrightarrow}{\bf{A}}_{\mathcal{F}}^{11}(\theta_0)^{-\frac{1}{2}}\eta,
\end{align*}
where $\eta$ is a $|\mathcal{F}_1|$-dimensional standard normal random vector. Thus, Assumption 5 (ii) in Suzuki and Yoshida \cite{Suzuki(2020)} is satisfied, so that Theorem $5$ (a) in 
 Suzuki and Yoshida \cite{Suzuki(2020)} implies
\begin{align*}
    \sqrt{n}(\check{\theta}_{n}-\theta_{0})_{\mathcal{F}_{1}}\stackrel{d}{\longrightarrow}N_{|\mathcal{F}_{1}|}\Bigl(0,{\bf{A}}_{\mathcal{F}}^{11}(\theta_0)^{-1}\Bigr).
    &\qedhere
\end{align*}
\end{proof}
\begin{lemma}\label{seiyakuT}
 Under \textrm{\textbf{[A1]}}, \textrm{\textbf{[B1]}},
\textrm{\textbf{[C1]}},
\textrm{\textbf{[D1]}} and
\textrm{\textbf{[E1]}},
as $h_n\longrightarrow0$,  
\begin{align*}
    n\mathbb{F}_{n}(\underline{\theta}_{n})\stackrel{d}{\longrightarrow}\chi^2_{\bar{p}-|\mathcal{F}_{1}|}   
\end{align*}
under $H_0$.
\end{lemma}
\begin{proof}
    See Appendix \ref{seiyakuTproof}.
\end{proof}
\begin{proof}[\textbf{Proof of Theorem \ref{tildeT}}]
Set $D_{n}=\bigl\{\tilde{\mathcal{F}}_{\mathbb{I},n,1}=\mathcal{F}_{1}\bigr\}$.
Note that $\tilde{\Theta}_{n}=\underline{\Theta}$ on $D_{n}$. The definitions of $\check{\theta}_{n}$ and $\underline{\theta}_{n}$ imply that on $D_{n}$,
\begin{align*}
    \mathbb{F}_{n}(\check{\theta}_{n})=\mathbb{F}_{n}(\underline{\theta}_{n}).
\end{align*}
Thus, for all $\varepsilon>0$, we see that under $H_0$, 
\begin{align*}
    \mathbb{P}(D_{n})&\leq\mathbb{P}\Bigl(n\mathbb{F}_{n}(\check{\theta}_{n})=n\mathbb{F}_{n}(\underline{\theta}_{n})\Bigr)\leq\mathbb{P}\Bigl(\bigl|n\mathbb{F}_{n}(\check{\theta}_{n})-n\mathbb{F}_{n}(\underline{\theta}_{n})\bigr|<\varepsilon\Bigr),
\end{align*}
so that it follows from Lemma \ref{Aeoracle} that under $H_0$,
\begin{align*}
     n\mathbb{F}_{n}(\check{\theta}_{n})-n\mathbb{F}_{n}(\underline{\theta}_{n})\stackrel{P}{\longrightarrow}0.
\end{align*}
Therefore, Lemma \ref{seiyakuT} and Slutsky's theorem  yield
\begin{align*}
     \check{\mathbb{T}}_{n}=n\mathbb{F}_{n}(\check{\theta}_{n})\stackrel{d}{\longrightarrow}\chi^2_{\bar{p}-|\mathcal{F}_{1}|}   
\end{align*}
under $H_0$.
\end{proof}
\begin{proof}[\textbf{Proof of Proposition \ref{Tprop}}]
It holds
\begin{align*}
    &\quad\ \mathbb{P}\left(\check{\mathbb{T}}_{n}>\chi^2_{\bar{p}-|\mathcal{\tilde{F}}_{\mathbb{I},n,1}|}(\alpha)\right)-\mathbb{P}\left(\check{\mathbb{T}}_{n}>\chi^2_{\bar{p}-|\mathcal{F}_{1}|}(\alpha)\right)\\
    &\leq\mathbb{P}\left(\bigl\{\check{\mathbb{T}}_{n}>\chi^2_{\bar{p}-|\mathcal{\tilde{F}}_{\mathbb{I},n,1}|}(\alpha)\bigr\}\cap D_{n}\right)\\
    &\quad+\mathbb{P}\left(\bigl\{\check{\mathbb{T}}_{n}>\chi^2_{\bar{p}-|\mathcal{\tilde{F}}_{\mathbb{I},n,1}|}(\alpha)\bigr\}\cap D_{n}^{c}\right)-\mathbb{P}\left(\check{\mathbb{T}}_{n}>\chi^2_{\bar{p}-|\mathcal{F}_{1}|}(\alpha)\right)\\
    &\leq\mathbb{P}\left(\bigl\{\check{\mathbb{T}}>\chi^2_{\bar{p}-|\mathcal{F}_{1}|}(\alpha)\bigr\}\cap D_{n}\right)+\mathbb{P}\bigl(D_{n}^{c}\bigr)-\mathbb{P}\left(\check{\mathbb{T}}_{n}>\chi^2_{\bar{p}-|\mathcal{F}_{1}|}(\alpha)\right)\qquad\ \ \\
    &\leq \mathbb{P}\left(\check{\mathbb{T}}_{n}>\chi^2_{\bar{p}-|\mathcal{F}_{1}|}(\alpha)\right)+\mathbb{P}\bigl(D_{n}^{c}\bigr)-\mathbb{P}\left(\check{\mathbb{T}}_{n}>\chi^2_{\bar{p}-|\mathcal{F}_{1}|}(\alpha)\right)\\
    &=\mathbb{P}\bigl(D_{n}^{c}\bigr).
\end{align*}
In a similar way, one has
\begin{align*}
    &\quad\ \mathbb{P}\left(\check{\mathbb{T}}_{n}>\chi^2_{\bar{p}-|\mathcal{F}_{1}|}(\alpha)\right)-\mathbb{P}\left(\check{\mathbb{T}}_{n}>\chi^2_{\bar{p}-|\mathcal{\tilde{F}}_{\mathbb{I},n,1}|}(\alpha)\right)\\
    &\leq\mathbb{P}\left(\bigl\{\check{\mathbb{T}}_{n}>\chi^2_{\bar{p}-|\mathcal{F}_{1}|}(\alpha)\bigr\}\cap D_{n}\right)\\
    &\quad+\mathbb{P}\left(\bigl\{\check{\mathbb{T}}_{n}>\chi^2_{\bar{p}-|\mathcal{F}_{1}|}(\alpha)\bigr\}\cap D_{n}^c\right)-\mathbb{P}\left(\check{\mathbb{T}}_{n}>\chi^2_{\bar{p}-|\mathcal{\tilde{F}}_{\mathbb{I},n,1}|}(\alpha)\right)\\
    &\leq\mathbb{P}\left(\bigl\{\check{\mathbb{T}}_{n}>\chi^2_{\bar{p}-|\mathcal{\tilde{F}}_{\mathbb{I},n,1}|}(\alpha)\bigr\}\cap D_{n}\right)+\mathbb{P}\bigl(D_{n}^{c}\bigr)-\mathbb{P}\left(\check{\mathbb{T}}_{n}>\chi^2_{\bar{p}-|\mathcal{\tilde{F}}_{\mathbb{I},n,1}|}(\alpha)\right)\\
    &\leq\mathbb{P}\left(\check{\mathbb{T}}_{n}>\chi^2_{\bar{p}-|\mathcal{\tilde{F}}_{\mathbb{I},n,1}|}(\alpha)\right)+\mathbb{P}\bigl(D_{n}^{c}\bigr)-\mathbb{P}\left(\check{\mathbb{T}}_{n}>\chi^2_{\bar{p}-|\mathcal{\tilde{F}}_{\mathbb{I},n,1}|}(\alpha)\right)\\
    &=\mathbb{P}\bigl(D_{n}^{c}\bigr).
\end{align*}
Therefore, we obtain
\begin{align}
    \left|\mathbb{P}\left(\check{\mathbb{T}}_{n}>
    \chi^2_{\bar{p}-|\mathcal{\tilde{F}}_{\mathbb{I},n,1}|}(\alpha)\right)-\mathbb{P}\left(\check{\mathbb{T}}_{n}>\chi^2_{\bar{p}-|\mathcal{F}_{1}|}(\alpha)\right)\right|\leq\mathbb{P}\bigl(D_{n}^{c}\bigr),\label{propine}
\end{align}
so that Lemma \ref{Aeoracle} shows that under $H_0$, 
\begin{align*}
    \mathbb{P}\left(\check{\mathbb{T}}_{n}>\chi^2_{\bar{p}-|\mathcal{\tilde{F}}_{\mathbb{I},n,1}|}(\alpha)\right)-\mathbb{P}\left(\check{\mathbb{T}}_{n}>\chi^2_{\bar{p}-|\mathcal{F}_{1}|}(\alpha)\right)\longrightarrow 0
\end{align*}
 as $n\longrightarrow\infty$.
\end{proof}
Let
\begin{align*}
    \bar{\mathcal{F}}_{0}=\Bigl\{j\in\{1,\cdots,q\}\ |\ \bar{\theta}^{(j)}=0\Bigr\}.
\end{align*}
Set $\bar{a}_{n}=\max_{j\in\bar{\mathcal{F}}_{1}}\kappa_{n}^{(j)}$ and $\bar{b}_{n}=\min_{j\in\bar{\mathcal{F}}_{0}}\kappa_{n}^{(j)}$.
\begin{lemma}\label{abbarprob}
Under\textrm{\textbf{[A1]}}, \textrm{\textbf{[B1]}},
\textrm{\textbf{[C1]}},
\textrm{\textbf{[D1]}} and
\textrm{\textbf{[E2]}}, as $h_n\longrightarrow0$,
$\sqrt{n}\lambda_{1,n}\longrightarrow0$ and
$\sqrt{n}\lambda_{2,n}\longrightarrow\infty$,
\begin{align*}
    \sqrt{n}\bar{a}_{n}\stackrel{P}{\longrightarrow} 0,\quad \frac{1}{\sqrt{n}\bar{b}_{n}}\stackrel{P}{\longrightarrow}0
\end{align*}
under $H_1$.
\end{lemma}
\begin{proof}
Note that $0<\min_{j\in\bar{\mathcal{F}}_{1}}|\bar{\theta}^{(j)}|-\delta$ and it holds from Lemma \ref{starconslemma} that $\hat{\theta}_n\stackrel{P}{\longrightarrow}\bar{\theta}$ under $H_1$. In an analogous manner to Lemma \ref{abprob}, we can obtain the result.
\end{proof}
\begin{lemma}\label{starop1}
Under \textrm{\textbf{[A1]}}, \textrm{\textbf{[B1]}},
\textrm{\textbf{[C1]}},
\textrm{\textbf{[D1]}},
\textrm{\textbf{[E2]}} and \textrm{\textbf{[F2]}},
as $h_n\longrightarrow0$,
\begin{align*}
    \sqrt{n}(\hat{\theta}_{n}-\bar{\theta})=O_p(1)
\end{align*}
under $H_1$.
\end{lemma}
\begin{proof}
    See Appendix \ref{testconsproof}.
\end{proof}
\begin{lemma}\label{thetaop1}
Under \textrm{\textbf{[A1]}}, \textrm{\textbf{[B1]}},
\textrm{\textbf{[C1]}},
\textrm{\textbf{[D1]}},
\textrm{\textbf{[E2]}} and \textrm{\textbf{[F2]}}, as $h_n\longrightarrow0$, $\sqrt{n}\lambda_{1,n}\longrightarrow0$ and $\sqrt{n}\lambda_{2,n}\longrightarrow\infty$,
\begin{align*}
    \sqrt{n}(\tilde{\theta}_{\mathbb{I},n}-\bar{\theta})=O_p(1)
\end{align*}
under $H_1$.
\end{lemma}
\begin{proof}
Note that
\begin{align*}
    \sum_{j=1}^{q}\kappa_{n}^{(j)}
    |\tilde{\theta}_{\mathbb{I},n}^{(j)}|\geq \sum_{j\in\bar{\mathcal{F}}_{1}}\kappa_{n}^{(j)}|\tilde{\theta}_{\mathbb{I},n}^{(j)}|,\quad \sum_{j=1}^{q}\kappa_{n}^{(j)}|\bar{\theta}^{(j)}|=\sum_{j\in\bar{\mathcal{F}}_{1}}\kappa_{n}^{(j)}|\bar{\theta}^{(j)}|.
\end{align*}
Since it holds from the definition of $\tilde{\theta}_{\mathbb{I},n}$ that
\begin{align*}
    0&\geq \mathbb{Q}_{\mathbb{I},n}(\tilde{\theta}_{\mathbb{I},n})-\mathbb{Q}_{\mathbb{I},n}(\bar{\theta})\\
    &=\sum_{j=1}^{q}(\tilde{\theta}^{(j)}_{\mathbb{I},n}-\bar{\theta}^{(j)}+\bar{\theta}^{(j)}-\hat{\theta}^{(j)}_{n})^2+\sum_{j=1}^{q}\kappa_{n}^{(j)}|\tilde{\theta}_{\mathbb{I},n}^{(j)}|-\sum_{j=1}^{q}(\bar{\theta}^{(j)}-\hat{\theta}_{n}^{(j)})^2-\sum_{j=1}^{q}\kappa_{n}^{(j)}|\bar{\theta}^{(j)}|\\
    &=\sum_{j=1}^{q}(\tilde{\theta}_{\mathbb{I},n}^{(j)}-\bar{\theta}^{(j)})^2+2\sum_{i=1}^n(\tilde{\theta}_{\mathbb{I},n}^{(j)}-\bar{\theta}^{(j)})(\bar{\theta}^{(j)}-\hat{\theta}_{n}^{(j)})+\sum_{j=1}^{q}\kappa_{n}^{(j)}|\tilde{\theta}_{\mathbb{I},n}^{(j)}|-\sum_{j=1}^{q}\kappa_{n}^{(j)}|\bar{\theta}^{(j)}|\\
    &\geq\sum_{j=1}^{q}(\tilde{\theta}_{\mathbb{I},n}^{(j)}-\bar{\theta}^{(j)})^2-2\sum_{i=1}^n(\tilde{\theta}_{\mathbb{I},n}^{(j)}-\bar{\theta}^{(j)})(\hat{\theta}_{n}^{(j)}-\bar{\theta}^{(j)})+\sum_{j\in\bar{\mathcal{F}}_{1}}\kappa_{n}^{(j)}(|\tilde{\theta}_{\mathbb{I},n}^{(j)}|-|\bar{\theta}^{(j)}|)\\
    &\geq\sum_{j=1}^{q}(\tilde{\theta}_{\mathbb{I},n}^{(j)}-\bar{\theta}^{(j)})^2-2\sum_{i=1}^n(\tilde{\theta}_{\mathbb{I},n}^{(j)}-\bar{\theta}^{(j)})(\hat{\theta}_{n}^{(j)}-\bar{\theta}^{(j)})-\sum_{j\in\bar{\mathcal{F}}_{1}}\kappa_{n}^{(j)}|\tilde{\theta}_{\mathbb{I},n}^{(j)}-\bar{\theta}^{(j)}|\\
    &\geq\sum_{j=1}^{q}(\tilde{\theta}_{\mathbb{I},n}^{(j)}-\bar{\theta}^{(j)})^2-2\sum_{i=1}^n(\tilde{\theta}_{\mathbb{I},n}^{(j)}-\bar{\theta}^{(j)})(\hat{\theta}_{n}^{(j)}-\bar{\theta}^{(j)})-\bar{a}_n\sum_{j\in\bar{\mathcal{F}}_{1}}|\tilde{\theta}_{\mathbb{I},n}^{(j)}-\bar{\theta}^{(j)}|\\
    &\geq|\tilde{\theta}_{\mathbb{I},n}-\bar{\theta}|^2-2|\tilde{\theta}_{\mathbb{I},n}-\bar{\theta}||\hat{\theta}_{n}-\bar{\theta}|-\bar{a}_{n}|{\bar{\mathcal{F}}}_{1}||\tilde{\theta}_{\mathbb{I},n}-\bar{\theta}|,
\end{align*}
one has
\begin{align*}
    0&\geq|\tilde{\theta}_{\mathbb{I},n}-\bar{\theta}|\Bigl\{|\sqrt{n}(\tilde{\theta}_{\mathbb{I},n}-\bar{\theta})|-2|\sqrt{n}(\hat{\theta}_{n}-\bar{\theta})|-\sqrt{n}\bar{a}_{n}|{\bar{\mathcal{F}}}_{1}|\Bigr\},
\end{align*}
which yields 
\begin{align*}
    |\sqrt{n}(\tilde{\theta}_{\mathbb{I},n}-\bar{\theta})|\leq 2|\sqrt{n}(\hat{\theta}_{n}-\bar{\theta})|+\sqrt{n}\bar{a}_{n}|{\bar{\mathcal{F}}}_{1}|.
\end{align*}
Thus, it follows from Lemma \ref{abbarprob} and \ref{starop1} that under $H_1$,
\begin{align*}
    2|\sqrt{n}(\hat{\theta}_{n}-\bar{\theta})|+\sqrt{n}\bar{a}_{n}|{\bar{\mathcal{F}}}_{1}|=O_p(1),
\end{align*}
so that we have $\sqrt{n}(\tilde{\theta}_{\mathbb{I},n}-\bar{\theta})=O_p(1)$.
\end{proof}
\begin{proof}[\textbf{Proof of Lemma \ref{AeoracleH1}}]
Set $E_{n}^{(j)}=\bigl\{\tilde{\theta}_{\mathbb{I},n}^{(j)}\neq0\bigr\}$ for $j\in\bar{\mathcal{F}}_{0}$. On $E_{n}^{(j)}$, 
\begin{align*}
    \sqrt{n}\left.\frac{\partial\mathbb{Q}_{\mathbb{I},n}(\theta)}{\partial\theta^{(j)}}\right|_{\theta=\tilde{\theta}_{\mathbb{I},n}}=2\sqrt{n}(\tilde{\theta}_{\mathbb{I},n}^{(j)}-\hat{\theta}_{n}^{(j)})+\sqrt{n}\kappa_{n}^{(j)}sign(\tilde{\theta}_{\mathbb{I},n}^{(j)})
\end{align*}
for $j\in\bar{\mathcal{F}}_{0}$, where 
\begin{align*}
    sign(x)=\begin{cases}
    1, & (x>0),\\
    0, & (x=0),\\
    -1, & (x<0)
    \end{cases}
\end{align*}
for $x\in\mathbb{R}$. The definition of $\tilde{\theta}_{\mathbb{I},n}$ implies that for $j\in\bar{\mathcal{F}}_{0}$,
\begin{align*}
    2\sqrt{n}(\tilde{\theta}_{\mathbb{I},n}^{(j)}-\hat{\theta}_{n}^{(j)})
    =-\sqrt{n}\kappa_{n}^{(j)}sign(\tilde{\theta}_{\mathbb{I},n}^{(j)})
\end{align*}
on $E_{n}^{(j)}$, so that for $j\in\bar{\mathcal{F}}_{0}$,
\begin{align*}
\begin{split}
    \bigl|2\sqrt{n}(\tilde{\theta}_{\mathbb{I},n}^{(j)}-\hat{\theta}_{n}^{(j)})\bigr|&=\sqrt{n}\kappa_{n}^{(j)}\bigl|sign(\tilde{\theta}_{\mathbb{I},n}^{(j)})\bigr|=\sqrt{n}\kappa_{n}^{(j)}\geq \sqrt{n}\bar{b}_{n}
\end{split}
\end{align*}
on $E_{n}^{(j)}$. Note that Lemma \ref{starop1} and Lemma \ref{thetaop1} yield
\begin{align*}
    \sqrt{n}(\tilde{\theta}_{\mathbb{I},n}-\hat{\theta}_{n})=\sqrt{n}(\tilde{\theta}_{\mathbb{I},n}-\bar{\theta})-\sqrt{n}(\hat{\theta}_{n}-\bar{\theta})=O_p(1)
\end{align*}
under $H_1$. Since it holds from Lemma \ref{abbarprob} that for $j\in\bar{\mathcal{F}}_{0}$,
\begin{align*}
    2\sqrt{n}(\tilde{\theta}_{\mathbb{I},n}^{(j)}-\hat{\theta}_{n}^{(j)})\frac{1}{\sqrt{n}\bar{b}_{n}}=o_p(1)
\end{align*}
under $H_1$, one gets
\begin{align}       
    \mathbb{P}\bigl(E_{n}^{(j)}\bigr)
    \leq\mathbb{P}\left(\Bigl| 2\sqrt{n}(\tilde{\theta}_{\mathbb{I},n}^{(j)}-\hat{\theta}_{n}^{(j)})\frac{1}{\sqrt{n}\bar{b}_{n}}\Bigr|\geq 1\right)\longrightarrow 0\label{tilde0}
\end{align}
for $j\in\bar{\mathcal{F}}_{0}$ under $H_1$ as $n\longrightarrow\infty$. On
\begin{align*}
    \Bigl\{|(\tilde{\theta}_{\mathbb{I},n}-\bar{\theta})_{\bar{\mathcal{F}}_{1}}|<\min_{j\in\bar{\mathcal{F}}_{1}}|\bar{\theta}^{(j)}|\Bigr\},
\end{align*}
we see
\begin{align*}
    |\bar{\theta}^{(j)}|-|\tilde{\theta}_{\mathbb{I},n}^{(j)}|\leq
    |\tilde{\theta}_{\mathbb{I},n}^{(j)}-\bar{\theta}^{(j)}|\leq|(\tilde{\theta}_{\mathbb{I},n}-\bar{\theta})_{\bar{\mathcal{F}}_{1}}|<\min_{j\in\bar{\mathcal{F}}_{1}}|\bar{\theta}^{(j)}|\leq |\bar{\theta}^{(j)}|
\end{align*}
for $j\in\bar{\mathcal{F}}_{1}$, which implies that $|\tilde{\theta}_{\mathbb{I},n}^{(j)}|>0$ for $j\in\bar{\mathcal{F}}_{1}$. Thus, one has
\begin{align*}
    \Bigl(\bigcap_{j\in\bar{\mathcal{F}}_{0}}(E_n^{(j)})^c\Bigr)\cap\Bigl\{|(\tilde{\theta}_{\mathbb{I},n}-\bar{\theta})_{\bar{\mathcal{F}}_{1}}|<\min_{j\in\bar{\mathcal{F}}_{1}}|\bar{\theta}^{(j)}|\Bigr\}
    \subset\Bigl\{\tilde{\mathcal{F}}_{\mathbb{I},n,1}=\bar{\mathcal{F}}_{1}\Bigr\}.
\end{align*}
Therefore, it holds from (\ref{tilde0}) and Lemma \ref{thetaop1} that under $H_1$, 
\begin{align*}
    \mathbb{P}\left(\tilde{\mathcal{F}}_{\mathbb{I},n,1}\neq\bar{\mathcal{F}}_{1}\right)
    &\leq\sum_{j\in\bar{\mathcal{F}}_{0}}\mathbb{P}\bigl(E_{n}^{(j)}\bigr)
    +\mathbb{P}\left(|(\tilde{\theta}_{\mathbb{I},n}-
    \bar{\theta})_{\bar{\mathcal{F}}_{1}}|
    \geq\min_{j\in\bar{\mathcal{F}}_{1}}|\bar{\theta}^{(j)}|\right)\\
    &\leq\sum_{j\in\bar{\mathcal{F}}_{0}}\mathbb{P}\bigl(E_{n}^{(j)}\bigr)
    +\mathbb{P}\left(\frac{1}{\sqrt{n}}|\sqrt{n}(\tilde{\theta}_{\mathbb{I},n}-\bar{\theta})_{\bar{\mathcal{F}}_{1}}|\geq\min_{j\in\bar{\mathcal{F}}_{1}}|\bar{\theta}^{(j)}|\right)\longrightarrow 0,
\end{align*}
which implies
\begin{align*}
    \mathbb{P}\left(\tilde{\mathcal{F}}_{\mathbb{I},n,1}    =\bar{\mathcal{F}}_{1}\right)\longrightarrow 1
\end{align*}
under $H_1$ as $n\longrightarrow\infty$.
\end{proof}
\begin{lemma}\label{testcons}
Under \textrm{\textbf{[A1]}}, \textrm{\textbf{[B1]}},
\textrm{\textbf{[C1]}},
\textrm{\textbf{[D1]}}, \textrm{\textbf{[E2]}} and \textrm{\textbf{[F2]}},
as $h_n\longrightarrow0$,
$\sqrt{n}\lambda_{1,n}\longrightarrow0$ and
$\sqrt{n}\lambda_{2,n}\longrightarrow\infty$,
\begin{align*}
    \check{\theta}_{n}\stackrel{P}{\longrightarrow}\bar{\theta}
\end{align*}
under $H_1$.
\end{lemma}
\begin{proof}
From \textrm{\textbf{[E2]}}, for any $\varepsilon>0$, there exists $\delta>0$ such that
\begin{align}
    \bigl|\check{\theta}_{n}-\bar{\theta}\bigr|>\varepsilon\Longrightarrow \mathbb{U}(\check{\theta}_{n})-\mathbb{U}(\bar{\theta})>\delta.\label{assumptionIstar}
\end{align}
Let $\bar{D}_{n}=\{\tilde{\mathcal{F}}_{\mathbb{I},n,1}=\bar{\mathcal{F}}_{1}\}$ and
\begin{align*}
   \bar{\Theta}=\Bigl\{\theta\in\Theta\ \big|\ \theta^{(j)}=0\ \ (j\in\bar{\mathcal{F}}_{0})\Bigr\}.
\end{align*}
Since $\tilde{\Theta}=\bar{\Theta}$ on $\bar{D}_{n}$ and $\bar{\theta}\in\bar{\Theta}$, we have
\begin{align*}
    \mathbb{F}_{n}(\check{\theta}_{n})=\inf_{\theta\in\bar{\Theta}}\mathbb{F}_{n}(\theta)\leq\mathbb{F}_{n}(\bar{\theta})
\end{align*}
on $\bar{D}_{n}$. Thus, one gets
\begin{align*}
    \mathbb{P}\bigl(\bar{D}_{n}\bigr)\leq\mathbb{P}\left(\tilde{\rm{F}}(\mathbb{Q}_{\mathbb{XX}},{\bf{\Sigma}}(\check{\theta}_{n}))-\tilde{\rm{F}}(\mathbb{Q}_{\mathbb{XX}},{\bf{\Sigma}}(\bar{\theta}))\leq\frac{\delta}{3}\right),
\end{align*}
so that it follows from Lemma \ref{AeoracleH1} that under $H_1$,
\begin{align}
    \mathbb{P}\left(\tilde{\rm{F}}(\mathbb{Q}_{\mathbb{XX}},{\bf{\Sigma}}(\check{\theta}_{n}))-\tilde{\rm{F}}(\mathbb{Q}_{\mathbb{XX}},{\bf{\Sigma}}(\bar{\theta}))>\frac{\delta}{3}\right)\longrightarrow 0\label{asymin}
\end{align}
as $n\longrightarrow\infty$. It holds from Lemma \ref{Fproblemma}, (\ref{assumptionIstar}) and (\ref{asymin}) that
\begin{align*}
    0&\leq \PP\Bigl(|\check{\theta}_{n}-\bar{\theta}|>\varepsilon\Bigr)\\
    &\leq\PP\Bigl(\mathbb{U}(\check{\theta}_{n})-\mathbb{U}(\bar{\theta})>\delta\Bigr)\\
    &\leq\PP\left(\rm{F}({\bf{\Sigma}}_0,{\bf{\Sigma}}(\check{\theta}_{n}))-\tilde{\rm{F}}(\mathbb{Q}_{\mathbb{XX}},{\bf{\Sigma}}(\check{\theta}_{n}))>\frac{\delta}{3}\right)\\
    &\qquad+\PP\left(\tilde{\rm{F}}(\mathbb{Q}_{\mathbb{XX}},{\bf{\Sigma}}(\check{\theta}_{n}))-\tilde{\rm{F}}(\mathbb{Q}_{\mathbb{XX}},{\bf{\Sigma}}(\bar{\theta}))>\frac{\delta}{3}\right)\\
    &\qquad\qquad+\PP\left(\tilde{\rm{F}}(\mathbb{Q}_{\mathbb{XX}},{\bf{\Sigma}}(\bar{\theta}))-\rm{F}({\bf{\Sigma}}_0,{\bf{\Sigma}}(\bar{\theta}))>\frac{\delta}{3}\right)\\
    &\leq 2\PP\left(\sup_{\theta\in\Theta}\left|\tilde{\rm{F}}(\mathbb{Q}_{\mathbb{XX}},{\bf{\Sigma}}(\theta))-\rm{F}({\bf{\Sigma}}_0,{\bf{\Sigma}}(\theta))\right|>\frac{\delta}{3}\right)\\
    &\qquad\qquad\quad+\mathbb{P}\left(\tilde{\rm{F}}(\mathbb{Q}_{\mathbb{XX}},{\bf{\Sigma}}(\check{\theta}_{n}))-\tilde{\rm{F}}(\mathbb{Q}_{\mathbb{XX}},{\bf{\Sigma}}(\bar{\theta}))>\frac{\delta}{3}\right)\stackrel{}{\longrightarrow}0
\end{align*}
under $H_1$ as $n\longrightarrow\infty$, which yields $\check{\theta}_{n}\stackrel{P}{\longrightarrow}\bar{\theta}$ under $H_1$.
\end{proof}
\begin{proposition}\label{TpropH1}
Under \textrm{\textbf{[A1]}}, \textrm{\textbf{[B1]}},
\textrm{\textbf{[C1]}},
\textrm{\textbf{[D1]}},
\textrm{\textbf{[E2]}} and \textrm{\textbf{[F2]}},
as $h_n\longrightarrow 0$,
$\sqrt{n}\lambda_{1,n}\longrightarrow0$ and
$\sqrt{n}\lambda_{2,n}\longrightarrow\infty$,
\begin{align*}
    \mathbb{P}\left(\check{\mathbb{T}}_{n}>\chi^2_{\bar{p}-|\mathcal{\tilde{F}}_{\mathbb{I},n,1}|}(\alpha)\right)-\mathbb{P}\left(\check{\mathbb{T}}_{n}>\chi^2_{\bar{p}-|\bar{\mathcal{F}}_{1}|}(\alpha)\right)\longrightarrow 0
\end{align*}
under $H_1$.
\end{proposition}
\begin{proof}
In a similar way to Proposition \ref{Tprop}, we obtain
\begin{align*}
    \left|\mathbb{P}\left(\check{\mathbb{T}}_{n}>\chi^2_{\bar{p}-|\mathcal{\tilde{F}}_{\mathbb{I},n,1}|}(\alpha)\right)-\mathbb{P}\left(\check{\mathbb{T}}_{n}>\chi^2_{\bar{p}-|\bar{\mathcal{F}}_{1}|}(\alpha)\right)\right|\leq\mathbb{P}\bigl(\bar{D}_{n}^{c}\bigr).
\end{align*}
Since it follows from Lemma \ref{AeoracleH1} that $\mathbb{P}(\bar{D}_{n}^c)\stackrel{}{\longrightarrow}0$ under $H_1$ as $n\longrightarrow\infty$, 
\begin{align*}
    \mathbb{P}\left(\check{\mathbb{T}}_{n}>\chi^2_{\bar{p}-|\mathcal{\tilde{F}}_{\mathbb{I},n,1}|}(\alpha)\right)-\mathbb{P}\left(\check{\mathbb{T}}_{n}>\chi^2_{\bar{p}-|\bar{\mathcal{F}}_{1}|}(\alpha)\right)\longrightarrow 0
\end{align*}
under $H_1$ as $n\longrightarrow\infty$.
\end{proof}
\begin{proof}[\textbf{Proof of Theorem \ref{checktcons}}]
In an analogous manner to Theorem $4$, it holds from Lemma \ref{testcons} that
\begin{align*}
    \frac{1}{n}\check{\mathbb{T}}_{n}\stackrel{P}{\longrightarrow} \mathbb{U}(\bar{\theta})
\end{align*}
under $H_1$. Recall that $\mathbb{U}(\bar{\theta})>0$ under $H_1$. Under $H_1$, 
\begin{align*}
    \mathbb{P}\left(\check{\mathbb{T}}_{n}>\chi^2_{\bar{p}-|\bar{\mathcal{F}}_{1}|}(\alpha)\right)
    &=1-\mathbb{P}\left(\frac{1}{n}\check{\mathbb{T}}_{n}\leq\frac{1}{n}\chi^2_{\bar{p}-|\bar{\mathcal{F}}_{1}|}(\alpha)\right)\stackrel{}{\longrightarrow}1
\end{align*}
as $n\longrightarrow\infty$. Therefore, Proposition \ref{TpropH1} and  Slutsky’s theorem imply that under $H_1$,
\begin{align*}
    \mathbb{P}\left(\check{\mathbb{T}}_{n}>\chi^2_{\bar{p}-|\mathcal{\tilde{F}}_{\mathbb{I},n,1}|}(\alpha)\right)\stackrel{}{\longrightarrow}1
\end{align*}
as $n\longrightarrow\infty$.
\end{proof}
\begin{proof}[\textbf{Proofs of Theorems 11-12}]
Using Theorem 5, we can prove Theorems 11-12 in a similar way to Theorems 9-10.
\end{proof}

\clearpage
\section{Appendix}
First, we introduce notation. Decompose $\Delta \mathbb{X}_{1,i}$ into
\begin{align*}
    \Delta \mathbb{X}_{1,i}=A_{i,n}+B_{i,n},
\end{align*}
where $A_{i,n}={\bf{\Lambda}}_{x_1,0}\Delta\xi_{0,i}$ and $B_{i,n}=\Delta\delta_{0,i}$. Recall that 
\begin{align*}
    \mathbb{X}_{2,t}={\bf{\Lambda}}_{x_2,0}{\bf{\Psi}}_{0}^{-1}{\bf{\Gamma}}_{0}\xi_{0,t}
    +{\bf{\Lambda}}_{x_2,0}{\bf{\Psi}}_{0}^{-1}\zeta_{0,t}+\varepsilon_{0,t}.
\end{align*}
$\Delta \mathbb{X}_{2,i}$ is decomposed into
\begin{align*}
    \Delta \mathbb{X}_{2,i}=C_{i,n}+D_{i,n}+E_{i,n},
\end{align*}
where
\begin{align*}
    C_{i,n}={\bf{\Lambda}}_{x_2,0}{\bf{\Psi}}_0^{-1}{\bf{\Gamma}}_0\Delta\xi_{0,i},\ 
    D_{i,n}={\bf{\Lambda}}_{x_2,0}{\bf{\Psi}}_0^{-1}\Delta\zeta_{0,i},\
    E_{i,n}=\Delta\varepsilon_{0,i}.
\end{align*}
In addition, we decompose ${\bf{\Sigma}}_0$ into
\begin{align*}
    {\bf{\Sigma}}_0=\begin{pmatrix}
    {\bf{\Sigma}}_{0}^{11} & {\bf{\Sigma}}_{0}^{12}\\
    {\bf{\Sigma}}_{0}^{12\top} & {\bf{\Sigma}}_{0}^{22}
    \end{pmatrix}=
    \begin{pmatrix}
    A+B & F\\
    F^{\top} & C+D+E
    \end{pmatrix},
\end{align*}
where
\begin{align*}
    &\quad A={\bf{\Lambda}}_{x_1,0}{\bf{\Sigma}}_{\xi\xi,0}{\bf{\Lambda}}_{x_1,0}^{\top},\ B={\bf{\Sigma}}_{\delta\delta,0},\ C={\bf{\Lambda}}_{x_2,0}{\bf{\Psi}}_0^{-1}{\bf{\Gamma}}_0{\bf{\Sigma}}_{\xi\xi,0}{\bf{\Gamma}}_0^{\top}{\bf{\Psi}}_0^{-1\top}{\bf{\Lambda}}_{x_2,0}^{\top},\\
    &D={\bf{\Lambda}}_{x_2,0}{\bf{\Psi}}_0^{-1}{\bf{\Sigma}}_{\zeta\zeta,0}
    {\bf{\Psi}}_0^{-1\top}{\bf{\Lambda}}_{x_2,0}^{\top},\
    E={\bf{\Sigma}}_{\varepsilon\varepsilon,0},\ F={\bf{\Lambda}}_{x_1,0}{\bf{\Sigma}}_{\xi\xi,0}{\bf{\Gamma}}_{0}^{\top}{\bf{\Psi}}_0^{-1\top}{\bf{\Lambda}}_{x_2,0}^{\top}.
\end{align*}
Let $R_i(h_n^{\ell},\xi)=R(h_n^{\ell},\xi_{0,t_{i-1}^n})$, $R_i(h_n^{\ell},\delta)=R(h_n^{\ell},\delta_{0,t_{i-1}^n})$, $R_i(h_n^{\ell},\varepsilon)=R(h_n^{\ell},\varepsilon_{0,t_{i-1}^n})$ and $R_i(h_n^{\ell},\zeta)=R(h_n^{\ell},\zeta_{0,t_{i-1}^n})$ for $\ell\geq 0$.  For any twice-differentiable function $f$, $\partial_{x}f(a)=\left.\partial_{x}f(x)\right|_{x=a}$ and $\partial^2_{x}f(a)=\left.\partial^2_{x}f(x)\right|_{x=a}$.
\subsection{Details of (\ref{quasi}) and (\ref{quasilog})}\label{quasi-likelihood}
Let $\Xi_{t}$ be the Euler-Maruyama approximation of $\xi_{t}$. One has 
\begin{align*}
    \Delta\Xi_i&=B_{1}(\Xi_{t_{i-1}^n})h_n+{\bf{S}}_1\Delta W_{1,i}.
\end{align*}
In the same way, $D_{t}$, $E_{t}$ and $Z_{t}$ are set to the Euler-Maruyama approximation of $\delta_{t}$, $\varepsilon_{t}$ and $\zeta_{t}$ respectively. We get 
\begin{align*}
    \Delta D_{i}&=B_{2}(D_{t_{i-1}^n})h_n+{\bf{S}}_2\Delta W_{2,i},\\
    \Delta E_{i}&=B_{3}(E_{t_{i-1}^n})h_n+{\bf{S}}_3\Delta W_{3,i},\\
    \Delta Z_{i}&=B_{4}(Z_{t_{i-1}^n})h_n+{\bf{S}}_4\Delta W_{4,i}.
\end{align*}
Note that it holds from \textbf{[A1]}(iii) and \textbf{[B1]}(iii) that
\begin{align*}
    \begin{split}
    {\bf{\Lambda}}_{x_1}\Delta \Xi_i +\Delta D_i
    &={\bf{\Lambda}}_{x_1}{\bf{S}}_1\Delta W_{1,i}+{\bf{S}}_2\Delta W_{2,i}+R(h_n,\Xi_{t_{i-1}^n})+R(h_n,\Delta_{t_{i-1}^n}).
    \end{split}
\end{align*}
If we set $\bar{\mathbb{X}}_{1,t}$ as an approximation of $\mathbb{X}_{1,t}$, (\ref{X}) yields
\begin{align}
    \Delta\bar{\mathbb{X}}_{1,i}={\bf{\Lambda}}_{x_1}{\bf{S}}_1\Delta W_{1,i}+{\bf{S}}_2\Delta W_{2,i}. \label{X1app}
\end{align}
In a similar way, since it follows from \textbf{[A1]} (iii), \textbf{[C1]} (iii) and \textbf{[D1]} (iii) that
\begin{align*}
    \begin{split}
    &{\bf{\Lambda}}_{x_2}{\bf{\Psi}}^{-1}{\bf{\Gamma}}\Delta\Xi_{i}+{\bf{\Lambda}}_{x_2}{\bf{\Psi}}^{-1}\Delta Z_i+\Delta E_i\\
    &\qquad\qquad\qquad={\bf{\Lambda}}_{x_2}{\bf{\Psi}}^{-1}{\bf{\Gamma}} {\bf{S}}_1\Delta W_{1,i}+{\bf{\Lambda}}_{x_2}{\bf{\Psi}}^{-1}{\bf{S}}_4\Delta W_{4,i}+{\bf{S}}_3\Delta W_{3,i}\\
    &\qquad\qquad\qquad\qquad\qquad\qquad+R(h_n,\Xi_{t_{i-1}^n})+R(h_n,Z_{t_{i-1}^n})+R(h_n,E_{t_{i-1}^n}), \\
    \end{split}
\end{align*}
(\ref{Y}) and (\ref{eta}) imply
\begin{align}
    \begin{split}
    \Delta\bar{\mathbb{X}}_{2,i}&={\bf{\Lambda}}_{x_2}{\bf{\Psi}}^{-1}{\bf{\Gamma}} {\bf{S}}_1\Delta W_{1,i}+{\bf{\Lambda}}_{x_{2}}{\bf{\Psi}}^{-1}{\bf{S}}_4\Delta W_{4,i}+{\bf{S}}_3\Delta W_{3,i},\label{X2app}
    \end{split}
\end{align}
where $\bar{\mathbb{X}}_{2,t}$ denotes an approximation of $\mathbb{X}_{2,t}$.
Set $\bar{\mathbb{X}}_{t}=(\bar{\mathbb{X}}_{1,t}^{\top},\bar{\mathbb{X}}_{2,t}^{\top})^{\top}$. It holds from (\ref{X1app}) and (\ref{X2app}) that
\begin{align*}
    \Delta\bar{\mathbb{X}}_{i}
    &=\begin{pmatrix}
    {\bf{\Lambda}}_{x_1}{\bf{S}}_1 & O_{p_1\times r_1}\\
    O_{p_2\times r_1} &{\bf{\Lambda}}_{x_2}{\bf{\Psi}}^{-1}{\bf{\Gamma}} {\bf{S}}_1
    \end{pmatrix}
    \begin{pmatrix}
    \Delta W_{1,i}\\
    \Delta W_{1,i}
    \end{pmatrix}\\
    &\qquad\quad+\begin{pmatrix}
    {\bf{S}}_2 \\
    O_{p_2\times r_2} 
    \end{pmatrix}\Delta W_{2,i}+\begin{pmatrix}
    O_{p_1\times r_3} \\
    {\bf{S}}_3
    \end{pmatrix}\Delta W_{3,i}+\begin{pmatrix}
    O_{p_1\times r_4}\\
    {\bf{\Lambda}}_{x_2}{\bf{\Psi}}^{-1}{\bf{S}}_4
    \end{pmatrix}\Delta W_{4,i}.
\end{align*}
The property of the Brownian motion implies
\begin{align}
    \begin{pmatrix}
    \Delta W_{1,i}\\
    \Delta W_{1,i}
    \end{pmatrix}\sim N_{2r_1}\left(0,h_n\begin{pmatrix}
    \mathbb{I}_{r_{1}} & \mathbb{I}_{r_{1}}\\
    \mathbb{I}_{r_{1}} & \mathbb{I}_{r_{1}}
    \end{pmatrix}\right),\label{W1-1}
\end{align}
so that we see from (\ref{W1-1}) that
\begin{align*}
    \begin{split}
    &\begin{pmatrix}
    {\bf{\Lambda}}_{x_1}{\bf{S}}_1 & O_{p_1\times r_1}\\
    O_{p_2\times r_1} &{\bf{\Lambda}}_{x_2}{\bf{\Psi}}^{-1}{\bf{\Gamma}} {\bf{S}}_1
    \end{pmatrix}
    \begin{pmatrix}
    \Delta W_{1,i}\\
    \Delta W_{1,i}
    \end{pmatrix}\\
    &\qquad\qquad\qquad\sim N_{p}\left(0,h_n\begin{pmatrix}
    {\bf{\Lambda}}_{x_1}{\bf{\Sigma}}_{\xi\xi}{\bf{\Lambda}}_{x_1}^{\top}
    & {\bf{\Lambda}}_{x_1}{\bf{\Sigma}}_{\xi\xi}{\bf{\Gamma}}^{\top}{\bf{\Psi}}^{-1\top}{\bf{\Lambda}}_{x_2}^{\top}\\
    {\bf{\Lambda}}_{x_2}{\bf{\Psi}}^{-1}{\bf{\Gamma}}{\bf{\Sigma}}_{\xi\xi}{\bf{\Lambda}}_{x_1}^{\top}	&{\bf{\Lambda}}_{x_2}{\bf{\Psi}}^{-1}{\bf{\Gamma}}{\bf{\Sigma}}_{\xi\xi}{\bf{\Gamma}}^{\top}{\bf{\Psi}}^{-1\top}{\bf{\Lambda}}_{x_2}^{\top}
    \end{pmatrix}\right).
    \end{split}
\end{align*}
In an analogous manner, we have
\begin{align*}
    \begin{pmatrix}
    {\bf{S}}_2 \\
    O_{p_2\times r_2} 
    \end{pmatrix}
    \Delta W_{2,i}
    &\sim N_{p}\left(0,h_n\begin{pmatrix}
    {\bf{\Sigma}}_{\delta\delta} & O_{p_1\times p_2}\\
    O_{p_2\times p_1} & O_{p_2\times p_2}
    \end{pmatrix}\right),\\
    \begin{pmatrix}
    O_{p_1\times r_3} \\
    {\bf{S}}_3
    \end{pmatrix}\Delta W_{3,i}
    &\sim N_{p}\left(0,h_n\begin{pmatrix}
    O_{p_1\times p_1} & O_{p_1\times p_2}\\
    O_{p_2\times p_1} & {\bf{\Sigma}}_{\varepsilon\varepsilon}
    \end{pmatrix}\right),\\
    \begin{split}
    \quad\begin{pmatrix}
    O_{p_1\times r_4} \\
    {\bf{\Lambda}}_{x_2}{\bf{\Psi}}^{-1}{\bf{S}}_4
    \end{pmatrix}
    \Delta W_{4,i}&\sim N_{p}\left(0,h_n\begin{pmatrix}
    O_{p_1\times p_1} & O_{p_1\times p_2}\\
    O_{p_2\times p_1}&{\bf{\Lambda}}_{x_2}{\bf{\Psi}}^{-1}{\bf{\Sigma}}_{\zeta\zeta}{\bf{\Psi}}^{-1\top}{\bf{\Lambda}}_{x_2}^{\top} 
\end{pmatrix}\right).
\end{split}
\end{align*}
Therefore, since $W_{1,t}$, $W_{2,t}$, $W_{3,t}$ and $W_{4,t}$ are independent, it follows that 
\begin{align*}
    \Delta\bar{\mathbb{X}}_{i}\sim N_{p}\bigl(0,h_n{\bf{\Sigma}}(\theta)\bigr).
\end{align*}
Hence, one has the following joint probability density function of $(\bar{\mathbb{X}}_{t_{i}^n})_{0\leq i\leq n}$: 
\begin{align*}
    \prod_{i=1}^{n}\frac{1}{(2\pi)^{\frac{p}{2}}\det{(h_n{\bf{\Sigma}}(\theta)})^{\frac{1}{2}}}\exp{\left\{-\frac{1}{2h_n}(\bar{x}_{t_{i}^n}-\bar{x}_{t_{i-1}^n})^{\top}{\bf{\Sigma}}(\theta)^{-1}(\bar{x}_{t_{i}^n}-\bar{x}_{t_{i-1}^n})\right\}}.
\end{align*}
The quasi-likelihood is set to (\ref{quasi}). Since
\begin{align*}
    \log \mathbb{L}_{n}(\theta)
    &=\sum_{i=1}^n\left\{-\frac{p}{2}\log(2\pi)-\frac{p}{2}\log h_n
    -\frac{1}{2}\log\det{{\bf{\Sigma}}(\theta)}-\frac{1}{2h_n}(\Delta \mathbb{X}_i)^{\top}{\bf{\Sigma}}(\theta)^{-1}(\Delta \mathbb{X}_i)\right\}\\
    &=-\frac{pn}{2}\log(2\pi)-\frac{pn}{2}\log h_n-\frac{n}{2}\log\det{{\bf{\Sigma}}(\theta)}-\frac{1}{2h_n}\sum_{i=1}^n\tr{\Bigl\{{\bf{\Sigma}}(\theta)^{-1}(\Delta \mathbb{X}_i)(\Delta \mathbb{X}_i)^{\top}\Bigr\}}\\
    &=-\frac{pn}{2}\log(2\pi)-\frac{pn}{2}\log h_n-\frac{n}{2}\log\det{{\bf{\Sigma}}(\theta)}-\frac{n}{2}\tr{\Bigl\{{\bf{\Sigma}}(\theta)^{-1}\mathbb{Q}_{\mathbb{XX}}\Bigr\}},
\end{align*}
we obtain the quasi-log likelihood function (\ref{quasilog}).
\subsection{Proof of (\ref{model1iden})}\label{idenap}
Assume that
\begin{align}
    {\bf{\Sigma}}(\theta_1)={\bf{\Sigma}}(\theta_2).\label{sigmam1}
\end{align}
From the (1,3)-th element of (\ref{sigmam1}), we obtain
\begin{align}
    ({\bf{\Sigma}}_{\xi\xi,1})_{12}=({\bf{\Sigma}}_{\xi\xi,2})_{12}.\label{sigmaxi12}
\end{align}
Since it holds from the (2,3)-th and (1,4)-th elements of (\ref{sigmam1}) that
\begin{align*}
    ({\bf{\Lambda}}_{x_1,1})_{21}({\bf{\Sigma}}_{\xi\xi,1})_{12}&=({\bf{\Lambda}}_{x_1,2})_{21}({\bf{\Sigma}}_{\xi\xi,2})_{12},\\  ({\bf{\Lambda}}_{x_1,1})_{42}({\bf{\Sigma}}_{\xi\xi,1})_{12}&=({\bf{\Lambda}}_{x_1,2})_{42}({\bf{\Sigma}}_{\xi\xi,2})_{12}
\end{align*}
and $({\bf{\Sigma}}_{\xi\xi})_{12}$ is not zero, we see from  (\ref{sigmaxi12}) that
\begin{align}
    ({\bf{\Lambda}}_{x_1,1})_{21}=({\bf{\Lambda}}_{x_1,2})_{21},\  ({\bf{\Lambda}}_{x_1,1})_{42}=({\bf{\Lambda}}_{x_1,2})_{42}.\label{lambdax}
\end{align}
As it follows from the (1,2)-th and (3,4)-th elements of (\ref{sigmam1}) that
\begin{align*}
    ({\bf{\Lambda}}_{x_1,1})_{21}({\bf{\Sigma}}_{\xi\xi,1})_{11}&=({\bf{\Lambda}}_{x_1,2})_{21}({\bf{\Sigma}}_{\xi\xi,2})_{11},\\  ({\bf{\Lambda}}_{x_1,1})_{42}({\bf{\Sigma}}_{\xi\xi,1})_{22}&=({\bf{\Lambda}}_{x_1,2})_{42}({\bf{\Sigma}}_{\xi\xi,2})_{22}
\end{align*}
and $({\bf{\Lambda}}_{x_1})_{21}$ and $({\bf{\Lambda}}_{x_1})_{42}$ are not zero, it holds from (\ref{lambdax}) that
\begin{align}
    ({\bf{\Sigma}}_{\xi\xi,1})_{11}=({\bf{\Sigma}}_{\xi\xi,2})_{11}, 
    \ ({\bf{\Sigma}}_{\xi\xi,1})_{22}=({\bf{\Sigma}}_{\xi\xi,2})_{22}.
    \label{sigmaxi}
\end{align}
Hence, (\ref{sigmaxi12})-(\ref{sigmaxi}) imply that
\begin{align*} 
    {\bf{\Lambda}}_{x_1,1}{\bf{\Sigma}}_{\xi\xi,1}{\bf{\Lambda}}_{x_1,1}^{\top}
    ={\bf{\Lambda}}_{x_1,2}{\bf{\Sigma}}_{\xi\xi,2}{\bf{\Lambda}}_{x_1,2}^{\top},
\end{align*}
so that we see from (\ref{sigmam1}) that
\begin{align}
    {\bf{\Sigma}}_{\delta\delta,1}
    ={\bf{\Sigma}}_{\delta\delta,2}.\label{delta2}
\end{align}
Moreover, it holds from the (1,5) and (3,5)-th elements of (\ref{sigmam1}) that
\begin{align*}
    {\bf{\Sigma}}_{\xi\xi,1}{\bf{\Gamma}}_1^{\top}
    ={\bf{\Sigma}}_{\xi\xi,2}{\bf{\Gamma}}_2^{\top}
\end{align*}
and ${\bf{\Sigma}}_{\xi\xi}$ is a positive definite matrix, which yields
\begin{align}
    {\bf{\Gamma}}_1={\bf{\Gamma}}_2 \label{gamma}
\end{align}
from (\ref{sigmaxi12}) and (\ref{sigmaxi}). Note that ${\bf{\Sigma}}_{\xi\xi}{\bf{\Gamma}}^{\top}\neq 0$ since ${\bf{\Gamma}}$ is not a zero vector and ${\bf{\Sigma}}_{\xi\xi}$ is a positive definite matrix. It follows from the (1,6)-th element of (\ref{sigmam1}) that
\begin{align*}
    ({\bf{\Lambda}}_{x_2,1})_{21}({\bf{\Sigma}}_{\xi\xi,1}{\bf{\Gamma}}_1^{\top})_{11}=({\bf{\Lambda}}_{x_2,2})_{21}
    ({\bf{\Sigma}}_{\xi\xi,2}{\bf{\Gamma}}_2^{\top})_{11},
\end{align*}
which implies
\begin{align}
    {\bf{\Lambda}}_{x_2,1}={\bf{\Lambda}}_{x_2,2}\label{lambday}
\end{align}
from (\ref{sigmaxi12}), (\ref{sigmaxi}) and (\ref{gamma}). 
Since it holds from the (5,6)-th element of (\ref{sigmam1}) that
\begin{align*}
    &({\bf{\Lambda}}_{x_2,1})_{21}{\bf{\Gamma}}_1{\bf{\Sigma}}_{\xi\xi,1}{\bf{\Gamma}}_1^{\top}+({\bf{\Lambda}}_{x_2,1})_{21}{\bf{\Sigma}}_{\zeta\zeta,1}\\
    &\qquad\qquad\qquad\qquad\quad=({\bf{\Lambda}}_{x_2,2})_{21}{\bf{\Gamma}}_2{\bf{\Sigma}}_{\xi\xi,2}{\bf{\Gamma}}_2^{\top}+({\bf{\Lambda}}_{x_2,2})_{21}{\bf{\Sigma}}_{\zeta\zeta,2}
\end{align*}
and $({\bf{\Lambda}}_{x_2})_{21}$ is not zero, (\ref{sigmaxi12}), (\ref{sigmaxi}), (\ref{gamma}) and (\ref{lambday}) show
\begin{align}
    {\bf{\Sigma}}_{\zeta\zeta,1}={\bf{\Sigma}}_{\zeta\zeta,2}. \label{sigmazeta}
\end{align}
Furthermore, we see from (\ref{sigmam1}) that
\begin{align*}
    {\bf{\Lambda}}_{x_2,1}({\bf{\Gamma}}_1{\bf{\Sigma}}_{\xi\xi,1}{\bf{\Gamma}}_1^{\top}+{\bf{\Sigma}}_{\zeta\zeta,1}){\bf{\Lambda}}_{x_2,1}^{\top}+{\bf{\Sigma}}_{\varepsilon\varepsilon,1}={\bf{\Lambda}}_{x_2,2}({\bf{\Gamma}}_2{\bf{\Sigma}}_{\xi\xi,2}{\bf{\Gamma}}_2^{\top}+{\bf{\Sigma}}_{\zeta\zeta,2}){\bf{\Lambda}}_{x_2,2}^{\top}+{\bf{\Sigma}}_{\varepsilon\varepsilon,2}
\end{align*}
and it follows from (\ref{sigmaxi12}), (\ref{sigmaxi}), (\ref{gamma}) and (\ref{sigmazeta}) that
\begin{align}
    {\bf{\Sigma}}_{\varepsilon\varepsilon,1}
    ={\bf{\Sigma}}_{\varepsilon\varepsilon,2}\label{sigmavarepsilon}.
\end{align}
Therefore, from (\ref{sigmaxi12})-(\ref{sigmavarepsilon}), 
we obtain $\theta_1=\theta_2$, which implies (\ref{model1iden}).
\subsection{Proof of (\ref{model2iden})}\label{idenap2}
Suppose that
\begin{align}
    {\bf{\Sigma}}(\theta_1)={\bf{\Sigma}}(\theta_2). \label{sigmam2}
\end{align}
It follows from (\ref{sigmam2}) and Theorem 5.1 in Anderson and Rubin \cite{Anderson(1956)} that
\begin{align}
    {\bf{\Lambda}}_{x_1,1}{\bf{\Sigma}}_{\xi\xi,1}{\bf{\Lambda}}_{x_1,1}^{\top}
    ={\bf{\Lambda}}_{x_1,2}{\bf{\Sigma}}_{\xi\xi,2}{\bf{\Lambda}}_{x_1,2}^{\top}\label{Lxi}
\end{align}
and
\begin{align}
    {\bf{\Sigma}}_{\delta\delta,1}={\bf{\Sigma}}_{\delta\delta,2}. \label{delta}
\end{align}
Recall that
\begin{align*}
    {\bf{\Lambda}}_{x_1}{\bf{\Sigma}}_{\xi\xi}{\bf{\Lambda}}_{x_1}^{\top}
    =\begin{pmatrix}
    {\bf{\Sigma}}_{\xi\xi} & {\bf{\Sigma}}_{\xi\xi}{\bf{A}}_{x_1}^{\top}\\
    {\bf{A}}_{x_1}{\bf{\Sigma}}_{\xi\xi} & {\bf{A}}_{x_1}{\bf{\Sigma}}_{\xi\xi}{\bf{A}}_{x_1}^{\top}
    \end{pmatrix}
\end{align*}
and ${\bf{\Sigma}}_{\xi\xi}$ is a positive definite matrix. From (\ref{Lxi}), we have
\begin{align}
    {\bf{\Lambda}}_{x_1,1}={\bf{\Lambda}}_{x_1,2},\ {\bf{\Sigma}}_{\xi\xi,1}={\bf{\Sigma}}_{\xi\xi,2}.\label{L1}
\end{align}
Note that
\begin{align*}
    {\bf{\Sigma}}^{12}(\theta)
    =\begin{pmatrix}
    {\bf{\Sigma}}_{\xi\xi}{\bf{\Gamma}}^{\top} & {\bf{\Sigma}}_{\xi\xi}{\bf{\Gamma}}^{\top}{\bf{A}}_{x_2}^{\top}\\
    {\bf{A}}_{x_1}{\bf{\Sigma}}_{\xi\xi}{\bf{\Gamma}}^{\top} & {\bf{A}}_{x_1}{\bf{\Sigma}}_{\xi\xi}{\bf{\Gamma}}^{\top}{\bf{A}}_{x_2}^{\top}
    \end{pmatrix}
\end{align*}
and ${\bf{\Gamma}}$ is a full row rank matrix. As it holds from (\ref{sigmam2}) and (\ref{L1}) that
\begin{align*}
    {\bf{\Sigma}}_{\xi\xi,1}({\bf{\Gamma}}_1^{\top}-{\bf{\Gamma}}_2^{\top})=O_{3\times 2},
\end{align*}
one has
\begin{align}
    {\bf{\Gamma}}_1={\bf{\Gamma}}_2.\label{G}
\end{align}
In a similar way, (\ref{sigmam2}), (\ref{L1}) and (\ref{G}) yield
\begin{align*}
    {\bf{\Sigma}}_{\xi\xi,1}{\bf{\Gamma}}_1^{\top}({\bf{A}}_{x_2,1}^{\top}-{\bf{A}}_{x_2,2}^{\top})=O_{3\times 4}.
\end{align*}
Since ${\bf{\Sigma}}_{\xi\xi}{\bf{\Gamma}}^{\top}$ is a full column rank matrix, it follows Lemma 11.3.1 in Harville \cite{Harville(1998)} that
\begin{align}
    {\bf{\Lambda}}_{x_2,1}={\bf{\Lambda}}_{x_2,2}.\label{L2}
\end{align}
Recall that
\begin{align*}
    {\bf{\Sigma}}^{22}(\theta)={\bf{\Lambda}}_{x_2}({\bf{\Gamma}}{\bf{\Sigma}}_{\xi\xi}{\bf{\Gamma}}^{\top}
    +{\bf{\Sigma}}_{\zeta\zeta}){\bf{\Lambda}}_{x_2}^{\top}+{\bf{\Sigma}}_{\varepsilon\varepsilon}
\end{align*}
and ${\bf{\Sigma}}_{\varepsilon\varepsilon}$ is a diagonal matrix. As it holds from (\ref{L1})-(\ref{L2}) that
\begin{align*}
    {\bf{\Lambda}}_{x_2,1}{\bf{\Gamma}}_1{\bf{\Sigma}}_{\xi\xi,1}{\bf{\Gamma}}_1^{\top}{\bf{\Lambda}}_{x_2,1}^{\top}
    ={\bf{\Lambda}}_{x_2,2}{\bf{\Gamma}}_2{\bf{\Sigma}}_{\xi\xi,2}{\bf{\Gamma}}_2^{\top}{\bf{\Lambda}}_{x_2,2}^{\top},
\end{align*}
we see from (\ref{sigmam2}) that
\begin{align}
    {\bf{\Lambda}}_{x_2,1}{\bf{\Sigma}}_{\zeta\zeta,1}{\bf{\Lambda}}_{x_2,1}^{\top}
    +{\bf{\Sigma}}_{\varepsilon\varepsilon,1}
    ={\bf{\Lambda}}_{x_2,2}{\bf{\Sigma}}_{\zeta\zeta,2}{\bf{\Lambda}}_{x_2,2}^{\top}
    +{\bf{\Sigma}}_{\varepsilon\varepsilon,2}.\label{sigma2}
\end{align}
Note that
\begin{align*}
    {\bf{\Lambda}}_{x_2}{\bf{\Sigma}}_{\zeta\zeta}{\bf{\Lambda}}_{x_2}^{\top}+{\bf{\Sigma}}_{\varepsilon\varepsilon}
    =\begin{pmatrix}
    {\bf{\Sigma}}_{\zeta\zeta} & {\bf{\Sigma}}_{\zeta\zeta}{\bf{A}}_{x_2}^{\top}\\
    {\bf{A}}_{x_2}{\bf{\Sigma}}_{\zeta\zeta} & {\bf{A}}_{x_2}{\bf{\Sigma}}_{\zeta\zeta}{\bf{A}}_{x_2}^{\top}
    \end{pmatrix}+
    \begin{pmatrix}
    {\bf{\Sigma}}^{11}_{\varepsilon\varepsilon} & O_{2\times 4}\\
    O_{4\times 2} & {\bf{\Sigma}}^{22}_{\varepsilon\varepsilon}
    \end{pmatrix},
\end{align*}
where ${\bf{\Sigma}}^{11}_{\varepsilon\varepsilon}=\Diag(({\bf{\Sigma}}_{\varepsilon\varepsilon})_{11},({\bf{\Sigma}}_{\varepsilon\varepsilon})_{22})^{\top}$ and ${\bf{\Sigma}}^{22}_{\varepsilon\varepsilon}=\Diag(({\bf{\Sigma}}_{\varepsilon\varepsilon})_{33},\cdots,({\bf{\Sigma}}_{\varepsilon\varepsilon})_{66})^{\top}$. (\ref{L2}) and (\ref{sigma2}) imply  
\begin{align*}
    {\bf{A}}_{x_2,1}({\bf{\Sigma}}_{\zeta\zeta,1}-{\bf{\Sigma}}_{\zeta\zeta,2})=O_{4\times 2}.
\end{align*}
${\bf{A}}_{x_2}$ is a full column rank matrix, so that one gets
\begin{align}
    {\bf{\Sigma}}_{\zeta\zeta,1}={\bf{\Sigma}}_{\zeta\zeta,2}.\label{zeta}
\end{align}
From (\ref{L2})-(\ref{zeta}), we obtain
\begin{align}
    {\bf{\Sigma}}_{\varepsilon\varepsilon,1}={\bf{\Sigma}}_{\varepsilon\varepsilon,2}. \label{e}
\end{align}
Therefore, it holds from (\ref{delta})-(\ref{L2}), (\ref{zeta}) and (\ref{e}) that $\theta_1=\theta_2$.
\subsection{Proof of Lemma \ref{EXlemmanon}}\label{EXlemmanonproof}
\begin{lemma}\label{Alemma}
Under \textrm{\textbf{[A1]}},
\begin{align*}
    &\quad\E\left[A_{i,n}^{(j_1)}\big|\mathscr{F}^{n}_{i-1}\right]=R_i(h_n,\xi),\\
    &\quad\E\left[A_{i,n}^{(j_1)}A_{i,n}^{(j_2)}
    \big|\mathscr{F}^{n}_{i-1}\right]=h_nA_{j_1j_2}+R_i(h_n^2,\xi),\\
    &\quad\E\left[A_{i,n}^{(j_1)}A_{i,n}^{(j_2)}A_{i,n}^{(j_3)}
    \big|\mathscr{F}^{n}_{i-1}\right]=R_i(h_n^2,\xi),\\	&\quad\E\left[A_{i,n}^{(j_1)}A_{i,n}^{(j_2)}A_{i,n}^{(j_3)}A_{i,n}^{(j_4)}\big|\mathscr{F}^{n}_{i-1}\right]=h_n^2(A_{j_1j_2}A_{j_3j_4}+A_{j_1j_3}A_{j_2j_4}+A_{j_1j_4}A_{j_2j_3})
    +R_i(h_n^3,\xi)
\end{align*}
for $j_1,j_2,j_3,j_4=1,\cdots,p_1$.
\end{lemma}
\begin{lemma}\label{Blemma}
Under \textrm{\textbf{[B1]}},
\begin{align*}
    &\quad\E\left[B_{i,n}^{(j_1)}\big|\mathscr{F}^{n}_{i-1}\right]=R_i(h_n,\delta),\\
    &\quad\E\left[B_{i,n}^{(j_1)}B_{i,n}^{(j_2)}
    \big|\mathscr{F}^{n}_{i-1}\right]=h_nB_{j_1j_2}+R_i(h_n^2,\delta),\\
    &\quad\E\left[B_{i,n}^{(j_1)}B_{i,n}^{(j_2)}B_{i,n}^{(j_3)}
    \big|\mathscr{F}^{n}_{i-1}\right]=R_i(h_n^2,\delta),\\
    &\quad\E\left[B_{i,n}^{(j_1)}B_{i,n}^{(j_2)}
    B_{i,n}^{(j_3)}B_{i,n}^{(j_4)}
    \big|\mathscr{F}^{n}_{i-1}\right]=h_n^2(B_{j_1j_2}B_{j_3j_4}+B_{j_1j_3}B_{j_2j_4}
    +B_{j_1j_4}B_{j_2j_3})+R_i(h_n^3,\delta)
\end{align*}
for $j_1,j_2,j_3,j_4=1,\cdots,p_1$.
\end{lemma}
\begin{lemma}\label{Clemma}
Under \textrm{\textbf{[A1]}},
\begin{align*}
    &\quad\E\left[C_{i,n}^{(j_1)}\big|\mathscr{F}^{n}_{i-1}\right]=R_i(h_n,\xi),\\
    &\quad\E\left[C_{i,n}^{(j_1)}C_{i,n}^{(j_2)}
    \big|\mathscr{F}^{n}_{i-1}\right]=h_nC_{j_1j_2}+R_i(h_n^2,\xi),\\
    &\quad\E\left[C_{i,n}^{(j_1)}C_{i,n}^{(j_2)}C_{i,n}^{(j_3)}
    \big|\mathscr{F}^{n}_{i-1}\right]=R_i(h_n^2,\xi),\\
    &\quad\E\left[C_{i,n}^{(j_1)}C_{i,n}^{(j_2)}
    C_{i,n}^{(j_3)}C_{i,n}^{(j_4)}\big|\mathscr{F}^{n}_{i-1}\right]=h_n^2(C_{j_1j_2}C_{j_3j_4}+C_{j_1j_3}C_{j_2j_4}+C_{j_1j_4}C_{j_2j_3})+R_i(h_n^3,\xi)
\end{align*}
for $j_1,j_2,j_3,j_4=1,\cdots,p_2$.
\end{lemma}
\begin{lemma}\label{Dlemma}
Under \textrm{\textbf{[D1]}},
\begin{align*}
    &\quad\E\left[D_{i,n}^{(j_1)}\big|\mathscr{F}^{n}_{i-1}\right]=R_i(h_n,\zeta),\\
    &\quad\E\left[D_{i,n}^{(j_1)}D_{i,n}^{(j_2)}
    \big|\mathscr{F}^{n}_{i-1}\right]=h_nD_{j_1j_2}+R_i(h_n^2,\zeta),\\
    &\quad\E\left[D_{i,n}^{(j_1)}D_{i,n}^{(j_2)}
    D_{i,n}^{(j_3)}\big|\mathscr{F}^{n}_{i-1}\right]=R_i(h_n^2,\zeta),\\
    &\quad\E\left[D_{i,n}^{(j_1)}D_{i,n}^{(j_2)}
    D_{i,n}^{(j_3)}D_{i,n}^{(j_4)}\big|\mathscr{F}^{n}_{i-1}\right]
    =h_n^2(D_{j_1j_2}D_{j_3j_4}+D_{j_1j_3}
    D_{j_2j_4}+D_{j_1j_4}D_{j_2j_3})
     +R_i(h_n^3,\zeta)
\end{align*}
for $j_1,j_2,j_3,j_4=1,\cdots,p_2$.
\end{lemma}
\begin{lemma}\label{Elemma}
Under \textrm{\textbf{[C1]}},
\begin{align*}
    &\quad\E\left[E_{i,n}^{(j_1)}\big|\mathscr{F}^{n}_{i-1}\right]=R_i(h_n,\varepsilon),\\
    &\quad\E\left[E_{i,n}^{(j_1)}E_{i,n}^{(j_2)}
    \big|\mathscr{F}^{n}_{i-1}\right]=h_nE_{j_1j_2}+R_i(h_n^2,\varepsilon),\\
    &\quad\E\left[E_{i,n}^{(j_1)}E_{i,n}^{(j_2)}
    E_{i,n}^{(j_3)}\big|\mathscr{F}^{n}_{i-1}\right]=R_i(h_n^2,\varepsilon),\\
    &\quad\E\left[E_{i,n}^{(j_1)}E_{i,n}^{(j_2)}
    E_{i,n}^{(j_3)}E_{i,n}^{(j_4)}\big|\mathscr{F}^{n}_{i-1}\right]=h_n^2(E_{j_1j_2}E_{j_3j_4}+E_{j_1j_3}E_{j_2j_4}+E_{j_1j_4}E_{j_2j_3})+R_i(h_n^3,\varepsilon)
\end{align*}
for $j_1,j_2,j_3,j_4=1,\cdots,p_2$.
\end{lemma}
\begin{lemma}\label{AClemma}
Under \textrm{\textbf{[A1]}},
\begin{align*}
    \E\left[A_{i,n}^{(j_1)}C_{i,n}^{(j_2)}\big|\mathscr{F}^{n}_{i-1}\right]
    =h_n F_{j_1j_2}+R_i(h_n^2,\xi)\qquad\qquad\qquad\qquad\qquad\qquad\qquad\qquad\quad\ \  
\end{align*}
for $j_1=1,\cdots,p_1, j_2=1,\cdots,p_2$,
\begin{align*}
    \E\left[A_{i,n}^{(j_1)}A_{i,n}^{(j_2)}C_{i,n}^{(j_3)}\big|\mathscr{F}^{n}_{i-1}\right]=R_i(h_n^2,\xi)\qquad\qquad\qquad\qquad\qquad\qquad\qquad\qquad\qquad\qquad
\end{align*}
for $j_1,j_2=1,\cdots,p_1, j_3=1,\cdots,p_2$,
\begin{align*}
    \E\left[A_{i,n}^{(j_1)}C_{i,n}^{(j_2)}C_{i,n}^{(j_3)}\big|\mathscr{F}^{n}_{i-1}\right]=R_i(h_n^2,\xi)\qquad\qquad\qquad\qquad\qquad\qquad\qquad\qquad\qquad\qquad
\end{align*}
for $j_1=1,\cdots,p_1, j_2,j_3=1,\cdots,p_2$,
\begin{align*}
    \quad\E\left[A_{i,n}^{(j_1)}A_{i,n}^{(j_2)}
    A_{i,n}^{(j_3)}C_{i,n}^{(j_4)}\big|\mathscr{F}^{n}_{i-1}\right]=h_n^2(A_{j_1j_2}F_{j_3j_4}+A_{j_1j_3}F_{j_2j_4}+F_{j_1j_4}A_{j_2j_3})+R_i(h_n^3,\xi)
\end{align*}
for $j_1,j_2,j_3=1,\cdots,p_1, j_4=1,\cdots,p_2$,
\begin{align*}
    \quad\E\left[A_{i,n}^{(j_1)}A_{i,n}^{(j_2)}
    C_{i,n}^{(j_3)}C_{i,n}^{(j_4)}\big|\mathscr{F}^{n}_{i-1}\right]=h_n^2(A_{j_1j_2}C_{j_3j_4}+F_{j_1j_3}F_{j_2j_4}+F_{j_1j_4}F_{j_2j_3})+R_i(h_n^3,\xi)
\end{align*}
for $j_1,j_2=1,\cdots,p_1, j_3,j_4=1,\cdots,p_2$, and
\begin{align*}
    \quad\E\left[A_{i,n}^{(j_1)}C_{i,n}^{(j_2)}
    C_{i,n}^{(j_3)}C_{i,n}^{(j_4)}\big|\mathscr{F}^{n}_{i-1}\right]=h_n^2(F_{j_1j_2}C_{j_3j_4}+F_{j_1j_3}C_{j_2j_4}+F_{j_1j_4}C_{j_2j_3})+R_i(h_n^3,\xi)
\end{align*}
for $j_1,j_2,j_3,j_4=1,\cdots,p_2$.
\end{lemma}
\begin{proof}[Proofs of Lemmas \ref{Alemma}-\ref{AClemma}.]
The results can be shown in a similar way to Lemmas 2-3 in Kusano and Uchida \cite{Kusano(2022)}.
\end{proof}
\begin{lemma}\label{klemma}
Under \textrm{\textbf{[A1]}}, \textrm{\textbf{[B1]}}, \textrm{\textbf{[C1]}} and \textrm{\textbf{[D1]}},
\begin{align*}
    \E\left[\bigl|A_{i,n}^{(j)}\bigr|^{\ell}\big|\mathscr{F}^{n}_{i-1}\right]&=R_i(h_n^\frac{\ell}{2},\xi)\quad(j=1,\cdots, p_1),\\
    \E\left[\bigl|B_{i,n}^{(j)}\bigr|^{\ell}\big|\mathscr{F}^{n}_{i-1}\right]&=R_i(h_n^\frac{\ell}{2},\delta)\quad(j=1,\cdots, p_1),\\
    \E\left[\bigl|C_{i,n}^{(j)}\bigr|^{\ell}\big|\mathscr{F}^{n}_{i-1}\right]&=R_i(h_n^\frac{\ell}{2},\xi)\quad(j=1,\cdots, p_2),\\
    \E\left[\bigl|D_{i,n}^{(j)}\bigr|^{\ell}\big|\mathscr{F}^{n}_{i-1}\right]&=R_i(h_n^\frac{\ell}{2},\zeta)\quad(j=1,\cdots, p_2),\\  
    \E\left[\bigl|E_{i,n}^{(j)}\bigr|^{\ell}\big|\mathscr{F}^{n}_{i-1}\right]&=R_i(h_n^\frac{\ell}{2},\varepsilon)\quad(j=1,\cdots, p_2)
\end{align*}
for $\ell\geq 2$.
\end{lemma}
\begin{proof}
In an analogous manner to Lemma 6 in Kessler \cite{kessler(1997)}, 
the results can be shown.
\end{proof}
\begin{lemma}\label{EX1X1lemma}
Under \textrm{\textbf{[A1]}}, and \textrm{\textbf{[B1]}},
\begin{align}
    \begin{split}
    &\quad\ \E\left[\Delta \mathbb{X}^{(j_1)}_{1,i}\Delta \mathbb{X}^{(j_2)}_{1,i}|\mathscr{F}^{n}_{i-1}\right]\\
    &=h_n({\bf{\Sigma}}_0^{11})_{j_1j_2}+h_n^2\bigl\{R_{i}(1,\xi)+R_{i}(1,\delta)+R_{i}(1,\xi)R_{i}(1,\delta)\bigr\},\label{EX1X1}
    \end{split}\\
    \begin{split}
    &\quad\ \E\left[\Delta \mathbb{X}^{(j_1)}_{1,i}\Delta \mathbb{X}^{(j_2)}_{1,i}\Delta \mathbb{X}^{(j_3)}_{1,i}\Delta \mathbb{X}^{(j_4)}_{1,i}\big|\mathscr{F}^{n}_{i-1}\right]\\
    &=h_n^2\{({\bf{\Sigma}}_0^{11})_{j_1j_2}({\bf{\Sigma}}_0^{11})_{j_3j_4}+({\bf{\Sigma}}_0^{11})_{j_1j_3}({\bf{\Sigma}}_0^{11})_{j_2j_4}+({\bf{\Sigma}}_0^{11})_{j_1j_4}({\bf{\Sigma}}_0^{11})_{j_2j_3}\}\\
    &\quad\qquad\qquad\qquad\qquad\qquad\quad+h_n^3\{R_{i}(1,\xi)+R_{i}(1,\delta)+R_{i}(1,\xi)R_{i}(1,\delta)\}\label{EX1X1X1X1}
    \end{split}
\end{align}
for $j_1,j_2,j_3,j_4=1,\cdots,p_1$.
\end{lemma}
\begin{proof}
From Lemma \ref{Alemma} and Lemma \ref{Blemma}, 
the results can be shown in a similar way to Lemma 4 in Kusano and Uchida \cite{Kusano(2022)}.
\end{proof}
\begin{lemma}\label{EX2X2lemma}
Under \textrm{\textbf{[A1]}}, \textrm{\textbf{[C1]}} and \textrm{\textbf{[D1]}},
\begin{align}
    \begin{split}
    &\quad\ \E\left[\Delta \mathbb{X}^{(j_1)}_{2,i}\Delta \mathbb{X}^{(j_2)}_{2,i}\big|\mathscr{F}^{n}_{i-1}\right]\\
    &=h_n({\bf{\Sigma}}_0^{22})_{j_1j_2}+h_n^2\bigl\{R_{i}(1,\xi)+R_{i}(1,\varepsilon)+R_{i}(1,\zeta)\\
    &\qquad\qquad+R_i(1,\xi)R_i(1,\varepsilon)+R_i(1,\xi)R_i(1,\zeta)
    +R_i(1,\varepsilon)R_i(1,\zeta)\bigr\},\label{EX2X2}
    \end{split}\\
    \begin{split}
    &\quad\ \E\left[\Delta \mathbb{X}^{(j_1)}_{2,i}\Delta \mathbb{X}^{(j_2)}_{2,i}\Delta \mathbb{X}^{(j_3)}_{2,i}\Delta \mathbb{X}^{(j_4)}_{2,i}\big|\mathscr{F}^{n}_{i-1}\right]\\
    &=h_n^2\{({\bf{\Sigma}}_0^{22})_{j_1j_2}({\bf{\Sigma}}_0^{22})_{j_3j_4}+({\bf{\Sigma}}_0^{22})_{j_1j_3}({\bf{\Sigma}}_0^{22})_{j_2j_4}+({\bf{\Sigma}}_0^{22})_{j_1j_4}({\bf{\Sigma}}_0^{22})_{j_2j_3}\}\\
    &\quad+h_n^3\bigl\{R_i(1,\xi)+R_i(1,\varepsilon)+R_i(1,\zeta)+R_i(1,\xi)R(1,\varepsilon)
    \\
    &\qquad\qquad\qquad+R_i(1,\xi)R_i(1,\zeta)+R_i(1,\varepsilon)R_i(1,\zeta)
    +R_i(1,\xi)R_i(1,\varepsilon)R_i(1,\zeta)\bigr\}\label{EX2X2X2X2}
    \end{split}
\end{align}
for $j_1,j_2,j_3,j_4=1,\cdots,p_2$.
\end{lemma}
\begin{proof}
First, we consider (\ref{EX2X2}). It holds from Lemmas \ref{Clemma}-\ref{Elemma} that for $j_1,j_2=1,\cdots,p_2$, 
\begin{align*}
    &\quad\ \E\left[C_{i,n}^{(j_1)}C_{i,n}^{(j_2)}\big|\mathscr{F}^{n}_{i-1}\right]+\E\left[D_{i,n}^{(j_1)}D_{i,n}^{(j_2)}\big|\mathscr{F}^{n}_{i-1}\right]+\E\left[E_{i,n}^{(j_1)}E_{i,n}^{(j_2)}\big|\mathscr{F}^{n}_{i-1}\right]\\
    &=h_n({\bf{\Sigma}}^{22}_0)_{j_1j_2}+h_n^2\bigl\{R_i(1,\xi)+R_i(1,\varepsilon)+R_i(1,\zeta)\bigr\}.
\end{align*}
It follows from Lemmas \ref{Clemma}-\ref{Elemma}
and the independence of $\xi_{0,t}$, $\varepsilon_{0,t}$, and $\zeta_{0,t}$ that
\begin{align*}
    &\E\left[C_{i,n}^{(j_1)}D_{i,n}^{(j_2)}\big|\mathscr{F}^{n}_{i-1}\right]=R_i(h_n,\xi)R_i(h_n,\zeta),\\
    &\E\left[C_{i,n}^{(j_1)}E_{i,n}^{(j_2)}\big|\mathscr{F}^{n}_{i-1}\right]=R_i(h_n,\xi)R_i(h_n,\varepsilon),\\
    &\E\left[D_{i,n}^{(j_1)}E_{i,n}^{(j_2)}\big|\mathscr{F}^{n}_{i-1}\right]=R_i(h_n,\zeta)R_i(h_n,\varepsilon)
\end{align*}
for $j_1,j_2=1,\cdots,p_2$. Therefore, we obtain
\begin{align*}
    \E\left[\Delta \mathbb{X}^{(j_1)}_{2,i}\Delta \mathbb{X}^{(j_2)}_{2,i}\big|\mathscr{F}^{n}_{i-1}\right]
    &=\E\left[(C_{i,n}^{(j_1)}+D_{i,n}^{(j_1)}+E_{i,n}^{(j_1)})(C_{i,n}^{(j_2)}+D_{i,n}^{(j_2)}+E_{i,n}^{(j_2)})\big|\mathscr{F}^{n}_{i-1}\right]\\
    &=h_n({\bf{\Sigma}}^{22}_0)_{j_1j_2}+h_n^2\bigl\{R_{i}(1,\xi)+R_{i}(1,\varepsilon)+R_{i}(1,\zeta)\\
    &\quad+R_i(1,\xi)R_i(1,\varepsilon)+R_i(1,\xi)R_i(1,\zeta)+R_i(1,\varepsilon)R_i(1,\zeta)\bigr\}
\end{align*}
for $j_1,j_2=1,\cdots,p_2$. Furthermore, from 
Lemmas \ref{Clemma}-\ref{Elemma}, (\ref{EX2X2X2X2}) 
can be shown in the same way.
\end{proof}
\begin{lemma}\label{EX1X2lemma}
Under \textrm{\textbf{[A1]}}, \textrm{\textbf{[B1]}}, \textrm{\textbf{[C1]}} and \textrm{\textbf{[D1]}},
\begin{align*}
    &\quad\ \E\left[\Delta \mathbb{X}^{(j_1)}_{1,i}\Delta \mathbb{X}^{(j_2)}_{2,i}\big|\mathscr{F}^{n}_{i-1}\right]\\
    &=h_n({\bf{\Sigma}}_0^{12})_{j_1j_2}+h_n^2\bigl\{R_i(1,\xi)+R_i(1,\xi)R_i(1,\delta)+R_i(1,\xi)R_i(1,\varepsilon)
    \qquad\qquad\qquad\quad\ \\
    &\quad+R_i(1,\xi)R_i(1,\zeta)+R_i(1,\delta)R_i(1,\zeta)
    +R_i(1,\delta)R_i(1,\varepsilon)\bigr\}
\end{align*}
for $j_1=1,\cdots,p_1,\ j_2=1,\cdots,p_2$,
\begin{align*}
    &\quad\ \E\left[\Delta \mathbb{X}^{(j_1)}_{1,i}\Delta \mathbb{X}^{(j_2)}_{1,i}\Delta \mathbb{X}^{(j_3)}_{1,i}\Delta \mathbb{X}^{(j_4)}_{2,i}\big|\mathscr{F}^{n}_{i-1}\right]\\
    &=h_n^2\{({\bf{\Sigma}}_0^{11})_{j_1j_2}({\bf{\Sigma}}_0^{12})_{j_3j_4}+({\bf{\Sigma}}_0^{11})_{j_1j_3}({\bf{\Sigma}}_0^{12})_{j_2j_4}+({\bf{\Sigma}}_0^{12})_{j_1j_4}({\bf{\Sigma}}_0^{11})_{j_2j_3}\}\qquad\qquad\qquad\quad\ \\
    &\quad+h_n^3\bigl\{R_i(1,\xi)+R_i(1,\delta)+R_i(1,\xi)R_i(1,\delta)+R_i(1,\xi)R_i(1,\varepsilon)
    \\  &\quad+R_i(1,\xi)R_i(1,\zeta)+R_i(1,\delta)R_i(1,\varepsilon)+R_i(1,\delta)R_i(1,\zeta)\\
    &\quad+R_i(1,\xi)R_i(1,\delta)R_i(1,\varepsilon)
    +R_i(1,\xi)R_i(1,\delta)R_i(1,\zeta)\bigr\}\quad\ \
\end{align*}
for $j_1,j_2,j_3=1,\cdots,p_1,\ j_4=1,\cdots,p_2$,
\begin{align*}
    &\quad\ \E\left[\Delta \mathbb{X}^{(j_1)}_{1,i}\Delta \mathbb{X}^{(j_2)}_{1,i}\Delta \mathbb{X}^{(j_3)}_{2,i}\Delta \mathbb{X}^{(j_4)}_{2,i}\big|\mathscr{F}^{n}_{i-1}\right]\\
    &=h_n^2\{({\bf{\Sigma}}^{11}_0)_{j_1j_2}({\bf{\Sigma}}^{22}_0)_{j_3j_4}+({\bf{\Sigma}}^{12}_0)_{j_1j_3}({\bf{\Sigma}}^{12}_0)_{j_2j_4}+({\bf{\Sigma}}^{12}_0)_{j_1j_4}({\bf{\Sigma}}^{12}_0)_{j_2j_3}\}\\
    &\quad+h_n^3\bigl\{R_i(1,\xi)+R_i(1,\delta)+R_i(1,\varepsilon)+R_i(1,\zeta)
    +R_i(1,\xi)R_i(1,\delta)\\
    &\quad+R_i(1,\xi)R_i(1,\varepsilon)
    +R_i(1,\xi)R_i(1,\zeta)
    +R_i(1,\delta)R_i(1,\varepsilon)+R_i(1,\delta)R_i(1,\zeta)\\
    &\quad+R_i(1,\xi)R_i(1,\delta)R_i(1,\varepsilon)
    +R_i(1,\xi)R_i(1,\delta)R_i(1,\zeta)
    +R_i(1,\xi)R_i(1,\varepsilon)R_i(1,\zeta)\quad\ \\
    &\quad+R_i(1,\delta)R_i(1,\varepsilon)
    R_i(1,\zeta)\bigr\}
    +h_n^4R_i(1,\xi)R_i(1,\delta)R_i(1,\zeta)R_i(1,\varepsilon)
\end{align*}
for $j_1,j_2=1,\cdots,p_1, \ j_3,j_4=1,\cdots,p_2$, and
\begin{align*}
    &\quad\ \E\left[\Delta \mathbb{X}^{(j_1)}_{1,i}\Delta \mathbb{X}^{(j_2)}_{2,i}\Delta \mathbb{X}^{(j_3)}_{2,i}\Delta \mathbb{X}^{(j_4)}_{2,i}\big|\mathscr{F}^{n}_{i-1}\right]\\
    &=h_n^2\{({\bf{\Sigma}}^{12}_0)_{j_1j_2}({\bf{\Sigma}}^{22}_0)_{j_3j_4}+({\bf{\Sigma}}^{12}_0)_{j_1j_3}({\bf{\Sigma}}^{22}_0)_{j_2j_4}+({\bf{\Sigma}}^{12}_0)_{j_1j_4}({\bf{\Sigma}}^{22}_0)_{j_2j_3}\}\\
    &\quad+h_n^3\bigl\{R_i(1,\xi)+R_i(1,\varepsilon)+R_i(1,\zeta)+R_i(1,\xi)R_i(1,\delta)+R_i(1,\xi)R_i(1,\varepsilon)\\
    &\quad+R_i(1,\xi)R_i(1,\zeta)
    +R_i(1,\delta)R_i(1,\varepsilon)
    +R_i(1,\delta)R_i(1,\zeta)
    +R_i(1,\xi)R_i(1,\delta)R_i(1,\varepsilon)
    \\
    &\quad+R_i(1,\xi)R_i(1,\delta)R_i(1,\zeta)
    +R_i(1,\xi)R_i(1,\varepsilon)R_i(1,\zeta)
    +R_i(1,\delta)R_i(1,\varepsilon)
    R_i(1,\zeta)\bigr\}\\
    &\quad+h_n^4 R_i(1,\xi)R_i(1,\delta)R_i(1,\varepsilon)R_i(1,\zeta)\ 
\end{align*}
for $j_1=1,\cdots,p_1, \ j_2,j_3,j_4=1,\cdots,p_2$.
\end{lemma}
\begin{proof}
From Lemmas \ref{Alemma}-\ref{AClemma},
the results can be shown in an analogous manner to Lemma \ref{EX2X2lemma}.
\end{proof}
\begin{lemma}\label{Ilemma}
Under \textbf{[A1]},
\begin{align*}
    \mathbb{E}\left[\Bigl|\int_{t_{i-1}^n}^{t_{i}^n}(B_1(\xi_{0,s})-B_1(\xi_{0,t_{i-1}^n}))ds\Bigr|^k\Big|\mathscr{F}^{n}_{i-1}\right]=R_i(h_n^{k+\frac{k}{2}},\xi)
\end{align*}
for $k\geq 2$.
\end{lemma}
\begin{proof}
Note that we see from Lemma 6 in Kessler \cite{kessler(1997)} that
\begin{align*}
    t_{i-1}^n\leq s\leq t_{i}^n \Longrightarrow \E\Bigl[|\xi_{0,s}-\xi_{0,t_{i-1}^n}|^{k}\big|\mathscr{F}^{n}_{i-1}\Bigr]=R_i(h_n^{\frac{k}{2}},\xi)
\end{align*}
for $k\geq 2$. It holds from H{\"o}lder's inequality that 
\begin{align*}
   &\quad\  \mathbb{E}\left[\Bigl|\int_{t_{i-1}^n}^{t_{i}^n}(B_1(\xi_{0,s})-B_1(\xi_{0,t_{i-1}^n}))ds\Bigr|^k\Big|\mathscr{F}^{n}_{i-1}\right]\\
   &\leq\E\left[\left(\int_{t_{i-1}^n}^{t_{i}^n}|B_1(\xi_{0,s})-B_1(\xi_{0,t_{i-1}^n})|ds\right)^{k}\Big|\mathscr{F}^{n}_{i-1}\right]\\
    &\leq\E\left[\left\{\left(\int_{t_{i-1}^n}^{t_{i}^n}|B_1(\xi_{0,s})-B_1(\xi_{0,t_{i-1}^n})|^{k}ds\right)^{\frac{1}{k}}\left(\int_{t_{i-1}^n}^{t_{i}^n}1^{\frac{k}{k-1}}ds\right)^{\frac{k-1}{k}}\right\}^{k}
    \Big|\mathscr{F}^{n}_{i-1}\right]\\
    &\leq h_n^{k-1}\E\left[\int_{t_{i-1}^n}^{t_{i}^n}|B_1(\xi_{0,s})-B_1(\xi_{0,t_{i-1}^n})|^{k}ds\Big|\mathscr{F}^{n}_{i-1}\right]\\
    &\leq h_n^{k-1}\int_{t_{i-1}^n}^{t_{i}^n}\E\Bigl[|B_1(\xi_{0,s})-B_1(\xi_{0,t_{i-1}^n})|^{k}\big|\mathscr{F}^{n}_{i-1}\Bigr]ds\\
    &\leq C_1h_n^{k-1}\int_{t_{i-1}^n}^{t_{i}^n}\E\Bigl[|\xi_{0,s}-\xi_{0,t_{i-1}^n}|^{k}\big|\mathscr{F}^{n}_{i-1}\Bigr]ds\\
    &\leq C_2 h_n^{k+\frac{k}{2}}\bigl(1+|\xi_{0,t_{i-1}^n}|\bigr)^{C_2}
\end{align*}
for $k\geq 2$, which yields
\begin{align*}
    \mathbb{E}\left[\Bigl|\int_{t_{i-1}^n}^{t_{i}^n}(B_1(\xi_{0,s})-B_1(\xi_{0,t_{i-1}^n}))ds\Bigr|^k\Big|\mathscr{F}^{n}_{i-1}\right]=R_i(h_n^{k+\frac{k}{2}},\xi) 
\end{align*}
for $k\geq 2$. 
\end{proof}
\begin{lemma}\label{AWlemma}
Under \textbf{[A1]},
\begin{align}
    \E\Bigl[A_{i,n}^{(j_1)}\Delta W_{1,i}^{(j_2)}\big|\mathscr{F}^{n}_{i-1}\Bigr]&=R_i(h_n,\xi)\label{AW}
\end{align}
for $j_1=1,\cdots,p_1, j_2=1,\cdots,r_1$,
\begin{align}
    \E\Bigl[A_{i,n}^{(j_1)}A_{i,n}^{(j_2)}\Delta W_{1,i}^{(j_3)}\big|\mathscr{F}^{n}_{i-1}\Bigr]&=R_i(h_n^2,\xi) \label{AAW}
\end{align}
for $j_1,j_2=1,\cdots,p_1, j_3=1,\cdots,r_1$, and
\begin{align}
    \E\Bigl[A_{i,n}^{(j_1)}C_{i,n}^{(j_2)}\Delta W_{1,i}^{(j_3)}\big|\mathscr{F}^{n}_{i-1}\Bigr]&=R_i(h_n^2,\xi) \label{ACW}
\end{align}
for $j_1=1,\cdots, p_1, j_2=1,\cdots, p_2, j_3=1,\cdots,r_1$.
\end{lemma}
\begin{proof}
First, we will prove (\ref{AW}). One has
\begin{align*}
    \Delta\xi_{0,i}=I_{1,i}+h_n B_1(\xi_{0,t_{i-1}^n})+{\bf{S}}_{1,0}\Delta W_{1,i},
\end{align*}
where
\begin{align*}
    I_{1,i}=\int_{t_{i-1}^n}^{t_{i}^n}(B_1(\xi_{0,s})-B_1(\xi_{0,t_{i-1}^n}))ds.
\end{align*}
For $j_1=1,\cdots,p_1,j_2=1,\cdots,r_1$, we see
\begin{align*}
    \E\Bigl[A_{i,n}^{(j_1)}\Delta W_{1,i}^{(j_2)}\big|\mathscr{F}^{n}_{i-1}\Bigr]&=\E\Bigl[\bigl({\bf{\Lambda}}_{x_1,0}[j_1,]I_{1,i}\bigr)\Delta W_{1,i}^{(j_2)}\big|\mathscr{F}^{n}_{i-1}\Bigr]\\
    &\quad+h_n{\bf{\Lambda}}_{x_1,0}[j_1,]B_1(\xi_{0,t_{i-1}^n})\E\Bigl[\Delta W_{1,i}^{(j_2)}\Bigr]\\
    &\quad+\E\Bigl[\bigl({\bf{\Lambda}}_{x_1,0}[j_1,]{\bf{S}}_{1,0}\Delta W_{1,i}\bigr)
    \Delta W_{1,i}^{(j_2)}\big|\mathscr{F}^{n}_{i-1}\Bigr].
\end{align*}
The Cauchy-Schwartz inequality and Lemma \ref{Ilemma} yield
\begin{align}
    \begin{split}
    \left|\E\Bigl[\bigl({\bf{\Lambda}}_{x_1,0}[j_1,]I_{1,i}\bigr)\Delta W_{1,i}^{(j_2)}\big|\mathscr{F}^{n}_{i-1}\Bigr]\right|
    &\leq C_1\E\Bigl[|I_{1,i}||\Delta W_{1,i}^{(j_2)}|\big|\mathscr{F}^{n}_{i-1}\Bigr]\\
    &\leq C_1\E\Bigl[|I_{1,i}|^2\big|\mathscr{F}^{n}_{i-1}\Bigr]^{\frac{1}{2}}\E\Bigl[\bigr|\Delta W_{1,i}^{(j_2)}\big|^2\Bigr]^{\frac{1}{2}}\\
    &\leq R_i(h_n^{2},\xi)\label{EIW}
    \end{split}
\end{align}
for $j_1=1,\cdots,p_1,j_2=1,\cdots,r_1$. In a similar way, it is shown that
\begin{align}
    \begin{split}
    \left|\E\Bigl[\bigl({\bf{\Lambda}}_{x_1,0}[j_1,]{\bf{S}}_{1,0}\Delta W_{1,i}\bigr)
    \Delta W_{1,i}^{(j_2)}\big|\mathscr{F}^{n}_{i-1}\Bigr]\right|\leq C_2 h_n
    \end{split}\label{EWW}
\end{align}
for $j_1=1,\cdots,p_1,j_2=1,\cdots,r_1$. Thus, one gets (\ref{AW}). 

Next, we consider (\ref{AAW}). It holds from the Cauchy-Schwartz inequality and Lemma \ref{Ilemma} that 
\begin{align*}
    &\left|\E\Bigl[\bigl({\bf{\Lambda}}_{x_1,0}[j_1,]I_{1,i}\bigr)\bigl({\bf{\Lambda}}_{x_1,0}[j_2,]I_{1,i}\bigr)\Delta W_{1,i}^{(j_3)}\big|\mathscr{F}^{n}_{i-1}\Bigr]\right|\\
    &\qquad\qquad\qquad\qquad\leq C_3\E\Bigl[|I_{1,i}|^4\big|\mathscr{F}^{n}_{i-1}\Bigr]^{\frac{1}{2}}\E\Bigl[\bigl|\Delta W_{1,i}^{(j_3)}\bigr|^2\Bigr]^{\frac{1}{2}}\leq R_i(h_n^\frac{7}{2},\xi)
\end{align*}
for $j_1,j_2=1,\cdots,p_1, j_3=1,\cdots,r_1$. In an analogous way, 
\begin{align*}
    &\left|\E\Bigl[\bigl({\bf{\Lambda}}_{x_1,0}[j_1,]I_{1,i}\bigr)\bigl({\bf{\Lambda}}_{x_1,0}[j_2,]{\bf{S}}_{1,0}\Delta W_{1,i}\bigr)\Delta W_{1,i}^{(j_3)}\big|\mathscr{F}^{n}_{i-1}\Bigr]\right|\\
    &\qquad\quad\leq C_4\E\Bigl[|I_{1,i}|^2\big|\mathscr{F}^{n}_{i-1}\Bigr]^{\frac{1}{2}}\E\Bigl[\bigl|\Delta W_{1,i}\bigr|^4\Bigr]^{\frac{1}{4}}\E\Bigl[\bigl|\Delta W_{1,i}^{(j_3)}\bigr|^4\Bigr]^{\frac{1}{4}}\leq R_i(h_n^\frac{5}{2},\xi)
\end{align*}
for $j_1,j_2=1,\cdots,p_1, j_3=1,\cdots,r_1$. (\ref{EIW}) and (\ref{EWW}) imply
\begin{align*}
    h_n{\bf{\Lambda}}_{x_1,0}[j_1,]B_1(\xi_{0,t_{i-1}^n})\E\Bigl[\bigl({\bf{\Lambda}}_{x_1,0}[j_2,]I_{1,i}\bigr)\Delta W_{1,i}^{(j_3)}\big|\mathscr{F}^{n}_{i-1}\Bigr]&=R_i(h_n^3,\xi)
\end{align*}
and
\begin{align*}
    h_n{\bf{\Lambda}}_{x_1,0}[j_1,]B_1(\xi_{0,t_{i-1}^n})\E\Bigl[\bigl({\bf{\Lambda}}_{x_1,0}[j_2,]{\bf{S}}_{1,0}\Delta W_{1,i}\bigr)\Delta W_{1,i}^{(j_3)}\Bigr]&=R_i(h_n^2,\xi)
\end{align*}
for $j_1,j_2=1,\cdots,p_1,j_3=1,\cdots,r_1$. Furthermore, 
\begin{align*}
    &\quad\ \E\Bigl[\bigl({\bf{\Lambda}}_{x_1,0}[j_1,]{\bf{S}}_{1,0}\Delta W_{1,i}\bigr)\bigl({\bf{\Lambda}}_{x_1,0}[j_2,]{\bf{S}}_{1,0}\Delta W_{1,i}\bigr)\Delta W_{1,i}^{(j_3)}\Bigr]\\
    &=\sum_{\ell=1}^{r}\sum_{m=1}^{r}({\bf{\Lambda}}_{x_1,0} {\bf{S}}_{1,0})_{j_1\ell}({\bf{\Lambda}}_{x_1,0} {\bf{S}}_{1,0})_{j_2m}\E\Bigl[\Delta W_{1,i}^{(\ell)}\Delta W_{1,i}^{(m)}\Delta W_{1,i}^{(j_3)}\Bigr]=0
\end{align*}
for $j_1,j_2=1,\cdots,p_1, j_3=1,\cdots,r_1$.
Therefore, we obtain (\ref{AAW}). In an analogous manner, (\ref{ACW}) holds.
\end{proof}
\begin{lemma}\label{BWlemma}
Under \textbf{[B1]},
\begin{align*}
    \E\Bigl[B_{i,n}^{(j_1)}\Delta W_{2,i}^{(j_2)}\big|\mathscr{F}^{n}_{i-1}\Bigr]=R_i(h_n,\delta)
\end{align*}
for $j_1=1,\cdots,p_1,j_2=1,\cdots,r_2$, and
\begin{align*}
    \E\Bigl[B_{i,n}^{(j_1)}B_{i,n}^{(j_2)}\Delta W_{2,i}^{(j_3)}\big|\mathscr{F}^{n}_{i-1}\Bigr]=R_i(h_n^2,\delta)
\end{align*}
for $j_1,j_2=1,\cdots,p_1,j_3=1,\cdots,r_2$.
\end{lemma}
\begin{lemma}
Under \textbf{[A1]},
\begin{align*}
    \E\Bigl[C_{i,n}^{(j_1)}\Delta W_{1,i}^{(j_2)}\big|\mathscr{F}^{n}_{i-1}\Bigr]=R_i(h_n,\xi)
\end{align*}
for $j_1=1,\cdots,p_2,j_2=1,\cdots,r_1$, and
\begin{align*}
    \E\Bigl[C_{i,n}^{(j_1)}C_{i,n}^{(j_2)}\Delta W_{1,i}^{(j_3)}\big|\mathscr{F}^{n}_{i-1}\Bigr]=R_i(h_n^2,\xi)
\end{align*}
for $j_1,j_2=1,\cdots,p_2,j_3=1,\cdots,r_1$.
\end{lemma}
\begin{lemma}
Under \textbf{[D1]},
\begin{align*}
    \E\Bigl[D_{i,n}^{(j_1)}\Delta W_{4,i}^{(j_2)}\big|\mathscr{F}^{n}_{i-1}\Bigr]=R_i(h_n,\zeta)
\end{align*}
for $j_1=1,\cdots,p_2,j_2=1,\cdots,r_4$, and
\begin{align*}
    \E\Bigl[D_{i,n}^{(j_1)}D_{i,n}^{(j_2)}\Delta W_{4,i}^{(j_3)}\big|\mathscr{F}^{n}_{i-1}\Bigr]=R_i(h_n^2,\zeta)
\end{align*}
for $j_1,j_2=1,\cdots,p_2,j_3=1,\cdots,r_4$.
\end{lemma}
\begin{lemma}\label{EWlemma}
Under \textbf{[C1]},
\begin{align*}
    \E\Bigl[E_{i,n}^{(j_1)}\Delta W_{3,i}^{(j_2)}\big|\mathscr{F}^{n}_{i-1}\Bigr]=R_i(h_n,\varepsilon)
\end{align*}
for $j_1=1,\cdots,p_2,j_2=1,\cdots,r_3$, and
\begin{align*}
    \E\Bigl[E_{i,n}^{(j_1)}E_{i,n}^{(j_2)}\Delta W_{3,i}^{(j_3)}\big|\mathscr{F}^{n}_{i-1}\Bigr]=R_i(h_n^2,\varepsilon)
\end{align*}
for $j_1,j_2=1,\cdots,p_2,j_3=1,\cdots,r_3$.
\end{lemma}
\begin{proof}[Proofs of Lemmas \ref{BWlemma}-\ref{EWlemma}]
The results can be shown in an analogous manner to Lemma \ref{AWlemma}. 
\end{proof}
\begin{proof}[\textbf{Proof of Lemma \ref{EXlemmanon}}]
First, we prove (\ref{EXXnonp}). It is sufficient to show that
\begin{align}
    \sum_{i=1}^{n}\left|\E\left[\sqrt{n}\Delta \mathbb{X}_{1,i}^{(j_1)}\Delta \mathbb{X}_{1,i}^{(j_2)}-\frac{1}{\sqrt{n}}({\bf{\Sigma}}_0^{11})_{j_1j_2}\Big|\mathscr{F}^{n}_{i-1}\right]\right|\stackrel{P}{\longrightarrow}0 \label{EX11nonp}
\end{align}
for $j_1,j_2=1,\cdots,p_1$,                          
\begin{align}
    \sum_{i=1}^{n}\left|\E\left[\sqrt{n}\Delta \mathbb{X}_{1,i}^{(j_1)}\Delta \mathbb{X}_{2,i}^{(j_2)}-\frac{1}{\sqrt{n}}({\bf{\Sigma}}_0^{12})_{j_1j_2}\Big|\mathscr{F}^{n}_{i-1}\right]\right|&\stackrel{P}{\longrightarrow}0 \label{EX12nonp}
\end{align}
for $j_1=1,\cdots,p_1$, $j_2=1,\cdots,p_2$, and
\begin{align}
    \sum_{i=1}^{n}\left|\E\left[\sqrt{n}\Delta \mathbb{X}_{2,i}^{(j_1)}\Delta \mathbb{X}_{2,i}^{(j_2)}-\frac{1}{\sqrt{n}}({\bf{\Sigma}}_0^{22})_{j_1j_2}\Big|\mathscr{F}^{n}_{i-1}\right]\right|&\stackrel{P}{\longrightarrow}0 \label{EX22nonp}
\end{align}
for $j_1,j_2=1,\cdots,p_2$. From (\ref{EX1X1}), we have
\begin{align*}
    &\quad\ \sum_{i=1}^{n}\left|\E\left[\sqrt{n}\Delta \mathbb{X}_{1,i}^{(j_1)}\Delta \mathbb{X}_{1,i}^{(j_2)}-\frac{1}{\sqrt{n}}({\bf{\Sigma}}_0^{11})_{j_1j_2}\Big|\mathscr{F}^{n}_{i-1}\right]\right|\\
    &=\frac{1}{\sqrt{n}}\sum_{i=1}^{n}\left|n\E\Bigl[\Delta \mathbb{X}_{1,i}^{(j_1)}\Delta \mathbb{X}_{1,i}^{(j_2)}\big|\mathscr{F}^{n}_{i-1}\Bigr]-({\bf{\Sigma}}_0^{11})_{j_1j_2}\right|\\
    &\leq\frac{1}{\sqrt{n}}\sum_{i=1}^{n}\Bigl\{n^{-1}R_{i}(1,\xi)+n^{-1}R_{i}(1,\delta)+n^{-1}R_{i}(1,\xi)R_{i}(1,\delta)\Bigr\}\\
    &\leq\frac{1}{\sqrt{n}}\left\{\frac{1}{n}\sum_{i=1}^n R_{i}(1,\xi)+\frac{1}{n}\sum_{i=1}^n 
    R_{i}(1,\delta)+\frac{1}{n}\sum_{i=1}^n R_{i}(1,\xi)R_{i}(1,\delta)\right\}\stackrel{P}{\longrightarrow}0
\end{align*}
for $j_1,j_2=1,\cdots,p_1$, which implies (\ref{EX11nonp}). In the same way, from Lemmas \ref{EX2X2lemma}-\ref{EX1X2lemma}, we obtain (\ref{EX1X2prob}) and (\ref{EX2X2prob}). Next, we show (\ref{EXXXXnonp}). Note that $nt-1<[nt]\leq nt$. Since it holds from Lemma \ref{EX1X1lemma} that
\begin{align*}
    &\quad\ \sum_{i=1}^{[nt]}\E\left[\Bigl\{\sqrt{n}\Delta \mathbb{X}_{1,i}^{(j_1)}\Delta \mathbb{X}_{1,i}^{(j_2)}-\frac{1}{\sqrt{n}}({\bf{\Sigma}}^{11}_0)_{j_1j_2}\Bigr\}\right.\\
    &\left.\qquad\qquad\qquad\qquad\qquad\qquad\quad\times \Bigl\{\sqrt{n}\Delta \mathbb{X}_{1,i}^{(j_3)}\Delta \mathbb{X}_{1,i}^{(j_4)}-\frac{1}{\sqrt{n}}({\bf{\Sigma}}^{11}_0)_{j_3j_4}\Bigr\}\Big|\mathscr{F}^{n}_{i-1}\right]\\
    &=n\sum_{i=1}^{[nt]}\E\Bigl[\Delta \mathbb{X}_{1,i}^{(j_1)}\Delta \mathbb{X}_{1,i}^{(j_2)}\Delta \mathbb{X}_{1,i}^{(j_3)}\Delta \mathbb{X}_{1,i}^{(j_4)}\big|\mathscr{F}^{n}_{i-1}\Bigr]\\
    &\quad-({\bf{\Sigma}}^{11}_0)_{j_3j_4}\times\sum_{i=1}^{[nt]}\E\Bigl[\Delta \mathbb{X}_{1,i}^{(j_1)}\Delta \mathbb{X}_{1,i}^{(j_2)}\big|\mathscr{F}^{n}_{i-1}\Bigr]\\
    &\quad-({\bf{\Sigma}}^{11}_0)_{j_1j_2}\times\sum_{i=1}^{[nt]}\E\Bigl[\Delta \mathbb{X}_{1,i}^{(j_3)}\Delta \mathbb{X}_{1,i}^{(j_4)}\big|\mathscr{F}^{n}_{i-1}\Bigr]+\frac{[nt]}{n}({\bf{\Sigma}}^{11}_0)_{j_1j_2}({\bf{\Sigma}}^{11}_0)_{j_3j_4}\\
    &=\frac{[nt]}{n}\bigl\{({\bf{\Sigma}}_0^{11})_{j_{1}j_{3}}({\bf{\Sigma}}^{11}_0)_{j_{2}j_{4}}+({\bf{\Sigma}}^{11}_0)_{j_{1}j_{4}}({\bf{\Sigma}}^{11}_0)_{j_{2}j_{3}}\bigr\}\\
    &\qquad\qquad\qquad+\frac{1}{n^2}\sum_{i=1}^{[nt]}\bigl\{R_{i}(1,\xi)+R_{i}(1,\delta)+R_{i}(1,\xi)R_{i}(1,\delta)\bigr\},
\end{align*}
one has
\begin{align*}
    &\left|\sum_{i=1}^{[nt]}\mathbb{E}\left[\Bigl\{\sqrt{n}\Delta \mathbb{X}_{1,i}^{(j_1)}\Delta \mathbb{X}_{1,i}^{(j_2)}-\frac{1}{\sqrt{n}}({\bf{\Sigma}}^{11}_0)_{j_1j_2}\Bigr\}\right.\right.\\
    &\left.\qquad\qquad\quad\times\Bigl\{\sqrt{n}\Delta \mathbb{X}_{1,i}^{(j_3)}\Delta \mathbb{X}_{1,i}^{(j_4)}-\frac{1}{\sqrt{n}}({\bf{\Sigma}}^{11}_0)_{j_3j_4}\Bigr\}\Big|\mathscr{F}^{n}_{i-1}\right]\\
    &\qquad\qquad\qquad\qquad\qquad\qquad\qquad-t\bigl\{({\bf{\Sigma}}_0^{11})_{j_{1}j_{3}}({\bf{\Sigma}}^{11}_0)_{j_{2}j_{4}}+({\bf{\Sigma}}^{11}_0)_{j_{1}j_{4}}({\bf{\Sigma}}^{11}_0)_{j_{2}j_{3}}\bigr\}\Bigr|\\
    &\leq\Bigl|\frac{[nt]-nt}{n}\Bigr|\bigl|({\bf{\Sigma}}_0^{11})_{j_{1}j_{3}}({\bf{\Sigma}}^{11}_0)_{j_{2}j_{4}}+({\bf{\Sigma}}^{11}_0)_{j_{1}j_{4}}({\bf{\Sigma}}^{11}_0)_{j_{2}j_{3}}\bigr|\\
    &\qquad\qquad\qquad\qquad+\frac{1}{n^2}\sum_{i=1}^{[nt]}\bigl|R_{i}(1,\xi)+R_{i}(1,\delta)+R_{i}(1,\xi)R_{i}(1,\delta)\bigr|\\
    &\leq\frac{1}{n}\bigl|({\bf{\Sigma}}_0^{11})_{j_{1}j_{3}}({\bf{\Sigma}}^{11}_0)_{j_{2}j_{4}}+({\bf{\Sigma}}^{11}_0)_{j_{1}j_{4}}({\bf{\Sigma}}^{11}_0)_{j_{2}j_{3}}\bigr|\\
    &\qquad+\frac{1}{n}\left\{\frac{1}{n}\sum_{i=1}^n R_{i}(1,\xi)+\frac{1}{n}\sum_{i=1}^n 
    R_{i}(1,\delta)+\frac{1}{n}\sum_{i=1}^n R_{i}(1,\xi)R_{i}(1,\delta)\right\}\stackrel{P}{\longrightarrow}0,
\end{align*}
which implies
\begin{align}
\begin{split}
    &\sum_{i=1}^{[nt]}\E\left[\Bigl\{\sqrt{n}\Delta \mathbb{X}_{1,i}^{(j_1)}\Delta \mathbb{X}_{1,i}^{(j_2)}-\frac{1}{\sqrt{n}}({\bf{\Sigma}}^{11}_0)_{j_1j_2}\Bigr\}\right.\\
    &\left.\qquad\quad\times\Bigl\{\sqrt{n}\Delta \mathbb{X}_{1,i}^{(j_3)}\Delta \mathbb{X}_{1,i}^{(j_4)}-\frac{1}{\sqrt{n}}({\bf{\Sigma}}^{11}_0)_{j_3j_4}\Bigr\}\Big|\mathscr{F}^{n}_{i-1}\right]\\
    &\qquad\qquad\qquad\qquad\qquad\stackrel{P}{\longrightarrow}t\bigl\{({\bf{\Sigma}}_0^{11})_{j_{1}j_{3}}({\bf{\Sigma}}^{11}_0)_{j_{2}j_{4}}+({\bf{\Sigma}}^{11}_0)_{j_{1}j_{4}}({\bf{\Sigma}}^{11}_0)_{j_{2}j_{3}}\bigr\}\label{EXXXXnon}
\end{split}
\end{align}
for $j_1,j_2,j_3,j_4=1,\cdots,p_1$. Furthermore, it holds from Lemma \ref{EX1X1lemma} that
\begin{align*}
    &\quad\left|\sum_{i=1}^{[nt]}\E\left[\sqrt{n}\Delta \mathbb{X}_{1,i}^{(j_1)}\Delta \mathbb{X}_{1,i}^{(j_2)}-\frac{1}{\sqrt{n}}({\bf{\Sigma}}^{11}_0)_{j_1j_2}\Big|\mathscr{F}^{n}_{i-1}\right]\right.\\
    &\left.\qquad\qquad\qquad\qquad\qquad\quad\times\E\left[\sqrt{n}\Delta \mathbb{X}_{1,i}^{(j_3)}\Delta \mathbb{X}_{1,i}^{(j_3)}-\frac{1}{\sqrt{n}}({\bf{\Sigma}}^{11}_0)_{j_3j_4}\Big|\mathscr{F}^{n}_{i-1}\right]\right|\\
    &\leq \frac{1}{n^2}\times\left\{\frac{1}{n}\sum_{i=1}^n R_{i}(1,\xi)+\frac{1}{n}\sum_{i=1}^n 
    R_{i}(1,\delta)+\frac{1}{n}\sum_{i=1}^n R_{i}(1,\xi)R_{i}(1,\delta)\right\}\stackrel{P}{\longrightarrow}0,
\end{align*}
so that
\begin{align}
\begin{split}
    &\left|\sum_{i=1}^{[nt]}\E\left[\sqrt{n}\Delta \mathbb{X}_{1,i}^{(j_1)}\Delta \mathbb{X}_{1,i}^{(j_2)}-\frac{1}{\sqrt{n}}({\bf{\Sigma}}^{11}_0)_{j_1j_2}\Big|\mathscr{F}^{n}_{i-1}\right]\right.\\
    &\left.\qquad\qquad\qquad\qquad\times\E\left[\sqrt{n}\Delta \mathbb{X}_{1,i}^{(j_3)}\Delta \mathbb{X}_{1,i}^{(j_4)}-\frac{1}{\sqrt{n}}({\bf{\Sigma}}^{11}_0)_{j_3j_4}\Big|\mathscr{F}^{n}_{i-1}\right]\right|\stackrel{P}{\longrightarrow}0\label{EXXEXXnon}
\end{split}
\end{align}
for $j_1,j_2,j_3,j_4=1,\cdots,p_1$. Hence, (\ref{EXXXXnon}) and (\ref{EXXEXXnon}) yield
\begin{align*}
    &\quad\ \ \sum_{i=1}^{[nt]}\E\left[\left\{\sqrt{n}\Delta \mathbb{X}_{1,i}^{(j_1)}\Delta \mathbb{X}_{1,i}^{(j_2)}-\frac{1}{\sqrt{n}}({\bf{\Sigma}}_0^{11})_{j_1j_2}\right\}\right.\\
    &\left.\qquad\qquad\qquad\qquad\qquad\qquad\times\left\{\sqrt{n}\Delta \mathbb{X}_{1,i}^{(j_3)}\Delta \mathbb{X}_{1,i}^{(j_4)}-\frac{1}{\sqrt{n}}({\bf{\Sigma}}_0^{11})_{j_3j_4}\right\}\Big|\mathscr{F}^{n}_{i-1}\right]\\
    &\qquad\qquad\qquad-\sum_{i=1}^{[nt]}\E\left[\sqrt{n}\Delta \mathbb{X}_{1,i}^{(j_1)}\Delta \mathbb{X}_{1,i}^{(j_2)}-\frac{1}{\sqrt{n}}({\bf{\Sigma}}^{11}_0)_{j_1j_2}\Big|\mathscr{F}^{n}_{i-1}\right]\\
    &\qquad\qquad\qquad\qquad\qquad\qquad\qquad\qquad\times\E\left[\sqrt{n}\bigl(\Delta \mathbb{X}_{1,i}^{(j_3)}\Delta \mathbb{X}_{1,i}^{(j_4)}-\frac{1}{\sqrt{n}}({\bf{\Sigma}}^{11}_0)_{j_3j_4}\Big|\mathscr{F}^{n}_{i-1}\right]\\
    &\stackrel{P}{\longrightarrow} t\bigl\{({\bf{\Sigma}}_0^{11})_{j_{1}j_{3}}({\bf{\Sigma}}^{11}_0)_{j_{2}j_{4}}+({\bf{\Sigma}}^{11}_0)_{j_{1}j_{4}}({\bf{\Sigma}}^{11}_0)_{j_{2}j_{3}}\bigr\}
\end{align*}
for $j_1,j_2,j_3,j_4=1,\cdots,p_1$. In an analogous manner, we have
\begin{align*}
    &\quad\ \ \sum_{i=1}^{[nt]}\E\left[\left\{\sqrt{n}\Delta \mathbb{X}_{1,i}^{(j_1)}\Delta \mathbb{X}_{1,i}^{(j_2)}-\frac{1}{\sqrt{n}}({\bf{\Sigma}}_0^{11})_{j_1j_2}\right\}\right.\\
    &\left.\qquad\qquad\qquad\qquad\qquad\qquad\times\left\{\sqrt{n}\Delta \mathbb{X}_{1,i}^{(j_3)}\Delta \mathbb{X}_{2,i}^{(j_4)}-\frac{1}{\sqrt{n}}({\bf{\Sigma}}_0^{12})_{j_3j_4}\right\}\Big|\mathscr{F}^{n}_{i-1}\right]\\
    &\qquad\qquad\qquad-\sum_{i=1}^{[nt]}\E\left[\sqrt{n}\Delta \mathbb{X}_{1,i}^{(j_1)}\Delta \mathbb{X}_{1,i}^{(j_2)}-\frac{1}{\sqrt{n}}({\bf{\Sigma}}^{11}_0)_{j_1j_2}\Big|\mathscr{F}^{n}_{i-1}\right]\\
    &\qquad\qquad\qquad\qquad\qquad\qquad\qquad\qquad\times\E\left[\sqrt{n}\bigl(\Delta \mathbb{X}_{1,i}^{(j_3)}\Delta \mathbb{X}_{2,i}^{(j_4)}-\frac{1}{\sqrt{n}}({\bf{\Sigma}}^{12}_0)_{j_3j_4}\Big|\mathscr{F}^{n}_{i-1}\right]\\
    &\stackrel{P}{\longrightarrow} t\bigl\{({\bf{\Sigma}}_0^{11})_{j_{1}j_{3}}({\bf{\Sigma}}^{12}_0)_{j_{2}j_{4}}+({\bf{\Sigma}}^{12}_0)_{j_{1}j_{4}}({\bf{\Sigma}}^{11}_0)_{j_{2}j_{3}}\bigr\}
\end{align*}
for $j_1,j_2,j_3=1,\cdots,p_1,\ j_4=1,\cdots,p_2$,
\begin{align*}
    &\quad\ \ \sum_{i=1}^{[nt]}\E\left[\left\{\sqrt{n}\Delta \mathbb{X}_{1,i}^{(j_1)}\Delta \mathbb{X}_{1,i}^{(j_2)}-\frac{1}{\sqrt{n}}({\bf{\Sigma}}_0^{11})_{j_1j_2}\right\}\right.\\
    &\left.\qquad\qquad\qquad\qquad\qquad\qquad\times\left\{\sqrt{n}\Delta \mathbb{X}_{2,i}^{(j_3)}\Delta \mathbb{X}_{2,i}^{(j_4)}-\frac{1}{\sqrt{n}}({\bf{\Sigma}}_0^{22})_{j_3j_4}\right\}\Big|\mathscr{F}^{n}_{i-1}\right]\\
    &\qquad\qquad\qquad-\sum_{i=1}^{[nt]}\E\left[\sqrt{n}\Delta \mathbb{X}_{1,i}^{(j_1)}\Delta \mathbb{X}_{1,i}^{(j_2)}-\frac{1}{\sqrt{n}}({\bf{\Sigma}}^{11}_0)_{j_1j_2}\Big|\mathscr{F}^{n}_{i-1}\right]\\
    &\qquad\qquad\qquad\qquad\qquad\qquad\qquad\qquad\times\E\left[\sqrt{n}\bigl(\Delta \mathbb{X}_{1,i}^{(j_3)}\Delta \mathbb{X}_{2,i}^{(j_4)}-\frac{1}{\sqrt{n}}({\bf{\Sigma}}^{12}_0)_{j_3j_4}\Big|\mathscr{F}^{n}_{i-1}\right]\\
    &\stackrel{P}{\longrightarrow} t\bigl\{({\bf{\Sigma}}_0^{12})_{j_{1}j_{3}}({\bf{\Sigma}}^{22}_0)_{j_{2}j_{4}}+({\bf{\Sigma}}^{12}_0)_{j_{1}j_{4}}({\bf{\Sigma}}^{12}_0)_{j_{2}j_{3}}\bigr\}
\end{align*}
for $j_1,j_3=1,\cdots,p_1,\ j_2,j_4=1,\cdots,p_2$,
\begin{align*}
    &\quad\ \ \sum_{i=1}^{[nt]}\E\left[\left\{\sqrt{n}\Delta \mathbb{X}_{1,i}^{(j_1)}\Delta \mathbb{X}_{2,i}^{(j_2)}-\frac{1}{\sqrt{n}}({\bf{\Sigma}}_0^{12})_{j_1j_2}\right\}\right.\\
    &\left.\qquad\qquad\qquad\qquad\qquad\qquad\times\left\{\sqrt{n}\Delta \mathbb{X}_{1,i}^{(j_3)}\Delta \mathbb{X}_{2,i}^{(j_4)}-\frac{1}{\sqrt{n}}({\bf{\Sigma}}_0^{12})_{j_3j_4}\right\}\Big|\mathscr{F}^{n}_{i-1}\right]\\
    &\qquad\qquad\qquad-\sum_{i=1}^{[nt]}\E\left[\sqrt{n}\Delta \mathbb{X}_{1,i}^{(j_1)}\Delta \mathbb{X}_{1,i}^{(j_2)}-\frac{1}{\sqrt{n}}({\bf{\Sigma}}^{11}_0)_{j_1j_2}\Big|\mathscr{F}^{n}_{i-1}\right]\\
    &\qquad\qquad\qquad\qquad\qquad\qquad\qquad\qquad\times\E\left[\sqrt{n}\bigl(\Delta \mathbb{X}_{1,i}^{(j_3)}\Delta \mathbb{X}_{2,i}^{(j_4)}-\frac{1}{\sqrt{n}}({\bf{\Sigma}}^{12}_0)_{j_3j_4}\Big|\mathscr{F}^{n}_{i-1}\right]\\
    &\stackrel{P}{\longrightarrow} t\bigl\{({\bf{\Sigma}}_0^{12})_{j_{1}j_{3}}({\bf{\Sigma}}^{22}_0)_{j_{2}j_{4}}+({\bf{\Sigma}}^{12}_0)_{j_{1}j_{4}}({\bf{\Sigma}}^{12\top}_{0})_{j_{2}j_{3}}\bigr\}
\end{align*}
for $j_1,j_3=1,\cdots,p_1,\ j_2,j_4=1,\cdots,p_2$,
\begin{align*}
    &\quad\ \ \sum_{i=1}^{[nt]}\E\left[\left\{\sqrt{n}\Delta \mathbb{X}_{1,i}^{(j_1)}\Delta \mathbb{X}_{2,i}^{(j_2)}-\frac{1}{\sqrt{n}}({\bf{\Sigma}}_0^{12})_{j_1j_2}\right\}\right.\\
    &\left.\qquad\qquad\qquad\qquad\qquad\qquad\times\left\{\sqrt{n}\Delta \mathbb{X}_{2,i}^{(j_3)}\Delta \mathbb{X}_{2,i}^{(j_4)}-\frac{1}{\sqrt{n}}({\bf{\Sigma}}_0^{22})_{j_3j_4}\right\}\Big|\mathscr{F}^{n}_{i-1}\right]\\
    &\qquad\qquad\qquad-\sum_{i=1}^{[nt]}\E\left[\sqrt{n}\Delta \mathbb{X}_{1,i}^{(j_1)}\Delta \mathbb{X}_{1,i}^{(j_2)}-\frac{1}{\sqrt{n}}({\bf{\Sigma}}^{11}_0)_{j_1j_2}\Big|\mathscr{F}^{n}_{i-1}\right]\\
    &\qquad\qquad\qquad\qquad\qquad\qquad\qquad\qquad\times\E\left[\sqrt{n}\bigl(\Delta \mathbb{X}_{1,i}^{(j_3)}\Delta \mathbb{X}_{2,i}^{(j_4)}-\frac{1}{\sqrt{n}}({\bf{\Sigma}}^{12}_0)_{j_3j_4}\Big|\mathscr{F}^{n}_{i-1}\right]\\
    &\stackrel{P}{\longrightarrow} t\bigl\{({\bf{\Sigma}}_0^{12})_{j_{1}j_{3}}({\bf{\Sigma}}^{22}_0)_{j_{2}j_{4}}+({\bf{\Sigma}}^{12}_0)_{j_{1}j_{4}}({\bf{\Sigma}}^{22}_{0})_{j_{2}j_{3}}\bigr\}
\end{align*}
for $j_1=1,\cdots,p_1,\ j_2,j_3,j_4=1,\cdots,p_2$, and
\begin{align*}
    &\quad\ \ \sum_{i=1}^{[nt]}\E\left[\left\{\sqrt{n}\Delta \mathbb{X}_{2,i}^{(j_1)}\Delta \mathbb{X}_{2,i}^{(j_2)}-\frac{1}{\sqrt{n}}({\bf{\Sigma}}_0^{22})_{j_1j_2}\right\}\right.\\
    &\left.\qquad\qquad\qquad\qquad\qquad\qquad\times\left\{\sqrt{n}\Delta \mathbb{X}_{2,i}^{(j_3)}\Delta \mathbb{X}_{2,i}^{(j_4)}-\frac{1}{\sqrt{n}}({\bf{\Sigma}}_0^{22})_{j_3j_4}\right\}\Big|\mathscr{F}^{n}_{i-1}\right]\\
    &\qquad\qquad\qquad-\sum_{i=1}^{[nt]}\E\left[\sqrt{n}\Delta \mathbb{X}_{1,i}^{(j_1)}\Delta \mathbb{X}_{1,i}^{(j_2)}-\frac{1}{\sqrt{n}}({\bf{\Sigma}}^{22}_0)_{j_1j_2}\Big|\mathscr{F}^{n}_{i-1}\right]\\
    &\qquad\qquad\qquad\qquad\qquad\qquad\qquad\qquad\times\E\left[\sqrt{n}\bigl(\Delta \mathbb{X}_{2,i}^{(j_3)}\Delta \mathbb{X}_{2,i}^{(j_4)}-\frac{1}{\sqrt{n}}({\bf{\Sigma}}^{22}_0)_{j_3j_4}\Big|\mathscr{F}^{n}_{i-1}\right]\\
    &\stackrel{P}{\longrightarrow} t\bigl\{({\bf{\Sigma}}_0^{22})_{j_{1}j_{3}}({\bf{\Sigma}}^{22}_0)_{j_{2}j_{4}}+({\bf{\Sigma}}^{22}_0)_{j_{1}j_{4}}({\bf{\Sigma}}^{22}_{0})_{j_{2}j_{3}}\bigr\}
\end{align*}
for $j_1,j_2,j_3,j_4=1,\cdots,p_2$, which yields (\ref{EXXXXnonp}). Next, we will prove (\ref{EXWnonp}). 
For the proof of (\ref{EXWnonp}), it is sufficient to show
\begin{align}\label{X11W}
    \sum_{i=1}^{[nt]}\E\left[\sqrt{n}\Delta \mathbb{X}_{1,i}^{(j_1)}\Delta \mathbb{X}_{1,i}^{(j_2)}\Delta \bar{W}_i^{(j_3)}\big|\mathscr{F}^{n}_{i-1}\right]\stackrel{P}{\longrightarrow}0
\end{align}
for $j_1,j_2=1,\cdots, p_1, j_3=1,\cdots,\bar{r}$, 
\begin{align}\label{X12W}
    \sum_{i=1}^{[nt]}\E\left[\sqrt{n}\Delta \mathbb{X}_{1,i}^{(j_1)}\Delta \mathbb{X}_{2,i}^{(j_2)}\Delta \bar{W}_i^{(j_3)}\big|\mathscr{F}^{n}_{i-1}\right]\stackrel{P}{\longrightarrow}0
\end{align}
for $j_1=1,\cdots, p_1, j_2=1,\cdots, p_2, j_3=1,\cdots,\bar{r}$, and
\begin{align}\label{X22W}
    \sum_{i=1}^{[nt]}\E\left[\sqrt{n}\Delta \mathbb{X}_{2,i}^{(j_1)}\Delta \mathbb{X}_{2,i}^{(j_2)}\Delta \bar{W}_i^{(j_3)}\big|\mathscr{F}^{n}_{i-1}\right]\stackrel{P}{\longrightarrow}0
\end{align}
for $j_1,j_2=1,\cdots, p_2, j_3=1,\cdots,\bar{r}$. Recalling that
\begin{align*}
    \E\Bigl[B_{i,n}^{(j_1)}B_{i,n}^{(j_2)}\Delta W_{1,i}^{(j_3)}\big|\mathscr{F}^{n}_{i-1}\Bigr]=\E\Bigl[B_{i,n}^{(j_1)}B_{i,n}^{(j_2)}\big|\mathscr{F}^{n}_{i-1}\Bigr]\E\Bigl[\Delta W_{1,i}^{(j_3)}\Bigr],
\end{align*}
Lemma \ref{Blemma} and Lemma \ref{AWlemma} yield
\begin{align*}
    \E\Bigl[\Delta \mathbb{X}_{1,i}^{(j_1)}\Delta \mathbb{X}_{1,i}^{(j_2)}\Delta W_{1,i}^{(j_3)}\big|\mathscr{F}^{n}_{i-1}\Bigr]=h_n^2\bigl\{R_{i}(1,\xi)+R_{i}(1,\xi)R_{i}(1,\delta)\bigr\}
\end{align*}
for $j_1,j_2=1,\cdots, p_1, j_3=1,\cdots,r_1$. Since
\begin{align*}
    \left|\sum_{i=1}^{[nt]}\E\left[\sqrt{n}\Delta \mathbb{X}_{1,i}^{(j_1)}\Delta \mathbb{X}_{1,i}^{(j_2)}\Delta W_{1,i}^{(j_3)}\big|\mathscr{F}^{n}_{i-1}\right]\right|
    &\leq\sqrt{n}\sum_{i=1}^{[nt]}\left|\E\Bigl[\Delta \mathbb{X}_{1,i}^{(j_1)}\Delta \mathbb{X}_{1,i}^{(j_2)}\Delta W_{1,i}^{(j_3)}\big|\mathscr{F}^{n}_{i-1}\Bigr]\right|\\
    &\leq\sqrt{n}\sum_{i=1}^{n}\left|\E\Bigl[\Delta \mathbb{X}_{1,i}^{(j_1)}\Delta \mathbb{X}_{1,i}^{(j_2)}\Delta W_{1,i}^{(j_3)}\big|\mathscr{F}^{n}_{i-1}\Bigr]\right|\\
    &\leq\frac{1}{\sqrt{n}}\left\{\frac{1}{n}\sum_{i=1}^{n}R_{i}(1,\xi)+\frac{1}{n}\sum_{i=1}^{n}R_{i}(1,\xi)R_{i}(1,\delta)\right\}\\
    &\stackrel{P}{\longrightarrow}0,
\end{align*}
we see
\begin{align}
    \sum_{i=1}^{[nt]}\E\left[\sqrt{n}\Delta \mathbb{X}_{1,i}^{(j_1)}\Delta \mathbb{X}_{1,i}^{(j_2)}\Delta W_{1,i}^{(j_3)}\big|\mathscr{F}^{n}_{i-1}\right]\stackrel{P}{\longrightarrow}0
    \label{X11W1}
\end{align}
for $j_1,j_2=1,\cdots,p_1, j_3=1,\cdots,r_1$. In an analogous manner,
\begin{align*}
    \E\Bigl[\Delta \mathbb{X}_{1,i}^{(j_1)}\Delta \mathbb{X}_{1,i}^{(j_2)}\Delta W_{2,i}^{(j_3)}\big|\mathscr{F}^{n}_{i-1}\Bigr]=h_n^2\bigl\{R_{i}(1,\delta)+R_{i}(1,\xi)R_{i}(1,\delta)\bigr\}
\end{align*}
for $j_1,j_2=1,\cdots,p_1, j_3=1,\cdots,r_1$, which yields
\begin{align}
    \sum_{i=1}^{[nt]}\E\left[\sqrt{n}\Delta \mathbb{X}_{1,i}^{(j_1)}\Delta \mathbb{X}_{1,i}^{(j_2)}\Delta W_{2,i}^{(j_3)}\big|\mathscr{F}^{n}_{i-1}\right]\stackrel{P}{\longrightarrow}0
    \label{X11W2}
\end{align}
for $j_1,j_2=1,\cdots,p_1, j_3=1,\cdots,r_2$. It holds from the independence of $W_{1,t}$, $W_{2,t}$, $W_{3,t}$ and $W_{4,t}$ that
\begin{align}
    \E\Bigl[\Delta \mathbb{X}_{1,i}^{(j_1)}\Delta \mathbb{X}_{1,i}^{(j_2)}\Delta W_{3,i}^{(j_3)}\big|\mathscr{F}^{n}_{i-1}\Bigr]=0 \label{X11W3}
\end{align}
for $j_1,j_2=1,\cdots, p_1, j_3=1,\cdots,r_3$, and
\begin{align}
    \E\Bigl[\Delta \mathbb{X}_{1,i}^{(j_1)}\Delta \mathbb{X}_{1,i}^{(j_2)}\Delta W_{4,i}^{(j_4)}\big|\mathscr{F}^{n}_{i-1}\Bigr]=0 \label{X11W4}
\end{align}
for $j_1,j_2=1,\cdots, p_1, j_3=1,\cdots,r_4$. Hence, (\ref{X11W1})-(\ref{X11W4}) imply (\ref{X11W}). In a similar way, we obtain (\ref{X12W}) and (\ref{X22W}). Finally, we prove (\ref{EXX4nonp}). It is sufficient to show that
\begin{align}
    \begin{split}
    &\sum_{i=1}^n\E\left[\left|\sqrt{n}\Delta \mathbb{X}_{1,i}^{(j_1)}\Delta \mathbb{X}_{1,i}^{(j_2)}-\frac{1}{\sqrt{n}}({\bf{\Sigma}}_0^{11})_{j_1j_2}\right|^4\Big|\mathscr{F}^{n}_{i-1}\right]\stackrel{P}{\longrightarrow}0\label{EX1X14}
    \end{split}
\end{align}
for $j_1,j_2=1,\cdots, p_1$,
\begin{align}
    \begin{split}
    &\sum_{i=1}^n\E\left[\left|\sqrt{n}\Delta \mathbb{X}_{1,i}^{(j_1)}\Delta \mathbb{X}_{2,i}^{(j_2)}-\frac{1}{\sqrt{n}}({\bf{\Sigma}}_0^{12})_{j_1j_2}\right|^4\Big|\mathscr{F}^{n}_{i-1}\right]\stackrel{P}{\longrightarrow}0\label{EX1X24}
    \end{split}
\end{align}
for $j_1=1,\cdots, p_1,\ j_2=1,\cdots, p_2$, and
\begin{align}
    \begin{split}
    &\sum_{i=1}^n\E\left[\left|\sqrt{n}\Delta \mathbb{X}_{2,i}^{(j_1)}\Delta \mathbb{X}_{2,i}^{(j_2)}-\frac{1}{\sqrt{n}}({\bf{\Sigma}}_0^{22})_{j_1j_2}\right|^4\Big|\mathscr{F}^{n}_{i-1}\right]\stackrel{P}{\longrightarrow}0\label{EX2X24}
    \end{split}
\end{align}
for $j_1,j_2=1,\cdots, p_2$. Note that for $j_1,j_2=1,\cdots,p_1$,
\begin{align}
    \begin{split}
    0&\leq\sum_{i=1}^n\E\left[\left|\sqrt{n}\Delta \mathbb{X}_{1,i}^{(j_1)}\Delta \mathbb{X}_{1,i}^{(j_2)}-\frac{1}{\sqrt{n}}({\bf{\Sigma}}_0^{11})_{j_1j_2}\right|^4\Big|\mathscr{F}^{n}_{i-1}\right]\\
    &\leq C_1 n^2\sum_{i=1}^n\E\left[\big|\Delta \mathbb{X}_{1,i}^{(j_1)}\Delta \mathbb{X}_{1,i}^{(j_2)}\bigr|^4\big|\mathscr{F}^{n}_{i-1}\right]+\frac{C_{1}}{n}({\bf{\Sigma}}_0^{11})_{j_1j_2}^4.\label{X1X14ine}
    \end{split}
\end{align}
Using the Cauchy-Schwartz's inequality and Lemma \ref{klemma}, we have
\begin{align*}
    \E\left[\big|A^{(j_1)}_{i,n}A^{(j_2)}_{i,n}\big|^4
    \big|\mathscr{F}^{n}_{i-1}\right]
    &\leq\E\left[\big|A^{(j_1)}_{i,n}\big|^8\big|\mathscr{F}^{n}_{i-1}\right]^{\frac{1}{2}}\E\left[\big|A^{(j_2)}_{i,n}\big|^8\big|\mathscr{F}^{n}_{i-1}\right]^{\frac{1}{2}}\leq R_i(h_n^4,\xi)
\end{align*}
and
\begin{align*}
    \E\left[\big|B^{(j_1)}_{i,n}B^{(j_2)}_{i,n}\big|^4
    \big|\mathscr{F}^{n}_{i-1}\right]
    &\leq\E\left[\big|B^{(j_1)}_{i,n}\big|^8\big|\mathscr{F}^{n}_{i-1}\right]^{\frac{1}{2}}\E\left[\big|B^{(j_2)}_{i,n}\big|^8\big|\mathscr{F}^{n}_{i-1}\right]^{\frac{1}{2}}\leq
    R_i(h_n^4,\delta)
\end{align*}
for $j_1,j_2=1,\cdots,p_1$. Recall that $\xi_{0,t}$ and $\delta_{0,t}$ is independent. It follows from Lemma \ref{klemma} that
\begin{align*}
    \E\left[\bigl|A^{(j_1)}_{i,n}B^{(j_2)}_{i,n}\bigr|^4\big|\mathscr{F}^{n}_{i-1}\right]&=R_i(h_n^2,\xi)R_i(h_n^2,\delta)
\end{align*}
for $j_1,j_2=1,\cdots,p_1$. Thus, for $j_1,j_2=1,\cdots,p_1$, one has
\begin{align*}
    0&\leq C_1 n^2\sum_{i=1}^n\E\left[\left|\Delta \mathbb{X}_{1,i}^{(j_1)}\Delta \mathbb{X}_{1,i}^{(j_2)}\right|^4\Big|\mathscr{F}^{n}_{i-1}\right]\\
    &\leq C_{1}n^2\sum_{i=1}^n\E\left[\bigl|(A^{(j_1)}_{i,n}+B^{(j_1)}_{i,n})(A^{(j_2)}_{i,n}+B^{(j_2)}_{i,n})\bigr|^4\big|\mathscr{F}^{n}_{i-1}\right]\\
    &\leq C_{2}n^2\sum_{i=1}^n
    \E\left[\bigl|A^{(j_1)}_{i,n}A^{(j_2)}_{i,n}\bigr|^4\big|\mathscr{F}^{n}_{i-1}\right]+C_{2}n^2\sum_{i=1}^n
    \E\left[\bigl|A^{(j_1)}_{i,n}B^{(j_2)}_{i,n}\bigr|^4\big|\mathscr{F}^{n}_{i-1}\right]\\
    &\quad +C_{2}n^2\sum_{i=1}^n\E\left[\bigl|B^{(j_1)}_{i,n}A^{(j_2)}_{i,n}\bigr|^4\big|\mathscr{F}^{n}_{i-1}\right]+C_{2}n^2\sum_{i=1}^n\E\left[\bigl|B^{(j_1)}_{i,n}B^{(j_2)}_{i,n}\bigr|^4\big|\mathscr{F}^{n}_{i-1}\right]\\
    &\leq\frac{C_{2}}{n}\left\{\frac{1}{n}\sum_{i=1}^n R_i(1,\xi)+\frac{1}{n}\sum_{i=1}^n R_i(1,\xi)R_i(1,\delta)+\frac{1}{n}\sum_{i=1}^n R_i(1,\delta)\right\}\stackrel{P}{\longrightarrow}0,
\end{align*}
which yields
\begin{align}
    C_{1}n^2\sum_{i=1}^n\E\left[\left|\Delta \mathbb{X}_{1,i}^{(j_1)}\Delta \mathbb{X}_{1,i}^{(j_2)}\right|^4\Big|\mathscr{F}^{n}_{i-1}\right]\stackrel{P}{\longrightarrow}0\label{CX1X14}
\end{align}
for $j_1,j_2=1,\cdots,p_1$. Hence, we obtain (\ref{EX1X14}) from (\ref{X1X14ine}) and (\ref{CX1X14}). In the same way, we can show (\ref{EX1X24}) and (\ref{EX2X24}).     
\end{proof}
\subsection{Proofs of Theorems 2-4}\label{Proof2-4}
For any matrix $M$, $\mathbb{C}(M)$ denotes the column space of $M$.
\begin{lemma}\label{positive}
$M\in\mathbb{R}^{(p_1+p_2)\times(p_1+p_2)}$ is set to
\begin{align*}
    M=\begin{pmatrix}
    M^{11} & M^{12}\\
    M^{12\top} & M^{22}
\end{pmatrix},
\end{align*}
where $M^{11}\in\mathbb{R}^{p_1\times p_1}$, $M^{12}\in\mathbb{R}^{p_1\times p_2}$ and $M^{22}\in\mathbb{R}^{p_2\times p_2}$. 
\begin{enumerate}[$(i)$]
    \item If $M^{11}$ is a positive definite matrix, and $M^{22}-M^{12\top}(M^{11})^{-1}M^{12}$ is a positive definite matrix, then $M$ is a positive definite matrix.
    \item If $M^{11}$ is a semi-positive definite matrix, $M^{22}-M^{12\top}(M^{11})^{-}M^{12}$ is a semi-positive definite matrix, and $\mathbb{C}(M^{12})\subset \mathbb{C}(M^{11})$, then $M$ is a semi-positive definite matrix.
\end{enumerate}
\end{lemma}
\begin{proof}
See Theorem 14.8.5 in Harville \cite{Harville(1998)}.
\end{proof}
\begin{lemma}\label{Sigmaposlemma}
${\bf{\Sigma}}_0$, ${\bf{\Sigma}}(\theta)$ and $\rm{V}({\bf{\Sigma}}_0,{\bf{\Sigma}}(\theta))$ are  positive definite matrices.
\end{lemma}
\begin{proof}
First, we decompose ${\bf{\Sigma}}(\theta)$ into
\begin{align*}
    {\bf{\Sigma}}(\theta)&=\begin{pmatrix}
    T & U\\
    U^{\top} & V
    \end{pmatrix}
    +\begin{pmatrix}
    {\bf{\Sigma}}_{\delta\delta} & O_{p_1\times p_2}\\
    O_{p_2\times p_1} & {\bf{\Sigma}}_{\varepsilon\varepsilon}
    \end{pmatrix},
\end{align*}
where 
\begin{align*}
    T={\bf{\Lambda}}_{x_1}{\bf{\Sigma}}_{\xi\xi}{\bf{\Lambda}}_{x_1}^{\top},\ 
    U={\bf{\Lambda}}_{x_1}{\bf{\Sigma}}_{\xi\xi}{\bf{\Gamma}}^{\top}{\bf{\Psi}}^{-1\top}{\bf{\Lambda}}_{x_2}^{\top},\ 
    V={\bf{\Lambda}}_{x_2}{\bf{\Psi}}^{-1}({\bf{\Gamma}}{\bf{\Sigma}}_{\xi\xi}{\bf{\Gamma}}^{\top}+{\bf{\Sigma}}_{\zeta\zeta}){\bf{\Psi}}^{-1\top}{\bf{\Lambda}}_{x_2}^{\top}.
\end{align*}
Recalling that ${\bf{\Sigma}}_{\xi\xi}$ is a semi-positive definite matrix, one has
\begin{align}
    T=\Bigl({\bf{\Lambda}}_{x_1}{\bf{\Sigma}}_{\xi\xi}^{\frac{1}{2}}\Bigr)\Bigl({\bf{\Lambda}}_{x_1}{\bf{\Sigma}}_{\xi\xi}^{\frac{1}{2}}\Bigr)^{\top}\geq 0.\label{l1-1}
\end{align}
Note that ${\bf{\Lambda}}_{x_1}^{-}{\bf{\Lambda}}_{x_1}=\mathbb{I}_{k_{1}}$ and ${\bf{\Lambda}}_{x_1}^{\top}({\bf{\Lambda}}_{x_1}^{\top})^{-}=\mathbb{I}_{k_{1}}$ since $\rank {\bf{\Lambda}}_{x_{1}}$ is a full row rank matrix. As it holds 
\begin{align*}
    {\bf{\Lambda}}_{x_1}^{-}T({\bf{\Lambda}}_{x_1}^{\top})^{-}={\bf{\Sigma}}_{\xi\xi},
\end{align*}
we obtain
\begin{align*}
    U^{\top} T^{-} U
    &={\bf{\Lambda}}_{x_2}{\bf{\Psi}}^{-1}{\bf{\Gamma}}{\bf{\Sigma}}_{\xi\xi}{\bf{\Lambda}}_{x_1}^{\top}T^{-}{\bf{\Lambda}}_{x_1}
    {\bf{\Sigma}}_{\xi\xi}{\bf{\Gamma}}^{\top}{\bf{\Psi}}^{-1\top}{\bf{\Lambda}}_{x_2}^{\top}\\
    &={\bf{\Lambda}}_{x_2}{\bf{\Psi}}^{-1}{\bf{\Gamma}}{\bf{\Lambda}}_{x_1}^{-}
    TT^{-}T({\bf{\Lambda}}_{x_1}^{\top})^{-}{\bf{\Gamma}}^{\top}{\bf{\Psi}}^{-1\top}{\bf{\Lambda}}_{x_2}^{\top}\\
    &={\bf{\Lambda}}_{x_2}{\bf{\Psi}}^{-1}{\bf{\Gamma}}{\bf{\Sigma}}_{\xi\xi}{\bf{\Gamma}}^{\top}
    {\bf{\Psi}}^{-1\top}{\bf{\Lambda}}_{x_2}^{\top}.
\end{align*}
Noting that ${\bf{\Sigma}}_{\zeta\zeta}$ is a semi-positive definite matrix, one gets
\begin{align}
    \begin{split}
    V-U^{\top}T^{-}U
    &={\bf{\Lambda}}_{x_2}{\bf{\Psi}}^{-1}{\bf{\Sigma}}_{\zeta\zeta}{\bf{\Psi}}^{-1\top}{\bf{\Lambda}}_{x_2}^{\top}\\
    &=\Bigl({\bf{\Lambda}}_{x_2}{\bf{\Psi}}^{-1}{\bf{\Sigma}}_{\zeta\zeta}^{\frac{1}{2}}\Bigr)\Bigl({\bf{\Lambda}}_{x_2}{\bf{\Psi}}^{-1}{\bf{\Sigma}}_{\zeta\zeta}^{\frac{1}{2}}\Bigr)^{\top}\geq 0.
    \end{split}\label{l1-2}
\end{align}
Furthermore, we set 
\begin{align*}
    F=({\bf{\Lambda}}_{x_1}^{\top})^{-}{\bf{\Gamma}}^{\top}{\bf{\Psi}}^{-1\top}{\bf{\Lambda}}_{x_2}^{\top},
\end{align*}
which yields $U=TF$. Thus, it follows from Lemma 4.2.2 in Harville \cite{Harville(1998)} that
\begin{align}
    \mathbb{C}(U)\subset \mathbb{C}(T).\label{l1-3}
\end{align}
Hence, Lemma \ref{positive} (ii), (\ref{l1-1}), (\ref{l1-2}) and (\ref{l1-3}) imply
\begin{align}
    \begin{pmatrix}
    T & U\\
    U^{\top} & V
    \end{pmatrix}\geq 0 .\label{l1-4}
\end{align}
Since
\begin{align*}
    {\bf{\Sigma}}_{\varepsilon\varepsilon}-O_{p_1\times p_2}^{\top}{\bf{\Sigma}}_{\delta\delta}^{-1}O_{p_1\times p_2}={\bf{\Sigma}}_{\varepsilon\varepsilon}>0,
\end{align*}
we see  from Lemma \ref{positive} (i) that
\begin{align}
    \begin{pmatrix}
    {\bf{\Sigma}}_{\delta\delta} & O_{p_1\times p_2}\\
    O_{p_2\times p_1} & {\bf{\Sigma}}_{\varepsilon\varepsilon}
    \end{pmatrix}>0. \label{l1-5}
\end{align}
Therefore, it holds from (\ref{l1-4}) and (\ref{l1-5}) that ${\bf{\Sigma}}(\theta)$ is a positive definite matrix. In the same way, we obtain ${\bf{\Sigma}}_0>0$. Since ${\bf{\Sigma}}_0$ and ${\bf{\Sigma}}(\theta)$ are positive definite matrices, it holds
\begin{align*}
    {\bf{\Sigma}}(\theta)+\lambda_1\lambda_2({\bf{\Sigma}}_0-{\bf{\Sigma}}(\theta))>0
\end{align*}
for $\lambda_1,\lambda_2\in[0,1]$. Noting that
\begin{align*}
    \mathbb{D}_{p}^{+}x=0\Longleftrightarrow x=0
\end{align*}
for $x\in\mathbb{R}^{\bar{p}}(\neq 0)$, if $\lambda_2$ is not zero, one has
\begin{align*}
    &\lambda_2x^{\top}\mathbb{D}_{p}^{+\top}({\bf{\Sigma}}(\theta)+\lambda_1\lambda_2({\bf{\Sigma}}_0-{\bf{\Sigma}}(\theta)))^{-1}\\
    &\qquad\qquad\qquad\qquad\otimes({\bf{\Sigma}}(\theta)+\lambda_1\lambda_2({\bf{\Sigma}}_0-{\bf{\Sigma}}(\theta)))^{-1}\mathbb{D}_{p}^{+}x>0
\end{align*}
for $\lambda_1,\lambda_2\in[0,1]$ and $x\in\mathbb{R}^{\bar{p}}(\neq 0)$. 
Therefore, we obtain $x^{\top}\rm{V}({\bf{\Sigma}}_0,{\bf{\Sigma}}(\theta))x>0$ for $x\in\mathbb{R}^{\bar{p}}(\neq 0)$. 
\end{proof}
\begin{lemma}\label{suplemma}
Let $f:\mathbb{R}^{p\times q}\times \mathbb{R}^{r}\rightarrow \mathbb{R}$ denote a continuous function
and $A$ be a compact subset set of $\mathbb{R}^{r}$. Then,
\begin{align*}
    \sup_{\alpha\in A}\left|f(Y,\alpha)-f(Y_0,\alpha)\right|\longrightarrow 0
\end{align*}
as $Y\longrightarrow Y_0$.
\end{lemma}
\begin{proof}
For all $\varepsilon>0$, there exists $\alpha_0\in\mathbb{R}^{r}$ such that
\begin{align}
\begin{split}
    &\sup_{\alpha\in A}\left|f(Y,\alpha)-f(Y_0,\alpha)\right|-\varepsilon \\
    &\qquad\qquad<\left|f(Y,\alpha_0)-f(Y_0,\alpha_0)\right|\leq\sup_{\alpha\in A}\left|f(Y,\alpha)-f(Y_0,\alpha)\right|.\label{supproof1}
\end{split}
\end{align}
By the continuity of $f$,
there exists $\delta>0$ such that
\begin{align}
    \|Y-Y_0\|<\delta\Longrightarrow \left|f(Y,\alpha_0)-f(Y_0,\alpha_0)\right|<\varepsilon. \label{supproof2}
\end{align}
Therefore, we see 
from (\ref{supproof1}) and (\ref{supproof2}) that
\begin{align*}
    \|Y-Y_0\|<\delta\Longrightarrow \sup_{\alpha\in A}\left|f(Y,\alpha)-f(Y_0,\alpha)\right|<2\varepsilon,
\end{align*}
which implies
\begin{align*}
    \sup_{\alpha\in A}\left|f(Y,\alpha)-f(Y_0,\alpha)\right|\longrightarrow 0
\end{align*}
as $Y\longrightarrow Y_0$.
\end{proof}
\begin{lemma}\label{Fproblemma}
Under \textrm{\textbf{[A1]}}, \textrm{\textbf{[B1]}},
\textrm{\textbf{[C1]}} and
\textrm{\textbf{[D1]}}, as $h_n\longrightarrow0$, 
\begin{align*} 
    \tilde{\rm{F}}(\mathbb{Q}_{\mathbb{XX}},{\bf{\Sigma}}(\theta))&\stackrel{P}{\longrightarrow}\rm{F}({\bf{\Sigma}}_0,{\bf{\Sigma}}(\theta))\quad\mbox{uniformly in }\theta,\\
    \partial_{\theta}^2\tilde{\rm{F}}(\mathbb{Q}_{\mathbb{XX}},{\bf{\Sigma}}(\theta))&\stackrel{P}{\longrightarrow}\partial_{\theta}^2 \rm{F}({\bf{\Sigma}}_0,{\bf{\Sigma}}(\theta))\quad\mbox{uniformly in }\theta.
\end{align*}
\end{lemma}
\begin{proof}
Set
\begin{align*}
    J_n=\bigl\{\mathbb{Q}_{\mathbb{XX}}\ \mbox{is non-singular}\bigr\}.
\end{align*}
Since ${\rm{F}}$ is continuous in $\theta$, from Lemma \ref{suplemma}, 
for any $\varepsilon>0$, there exists $\delta>0$ such that
\begin{align*} 
    \|\mathbb{Q}_{\mathbb{XX}}-{\bf{\Sigma}}_0\|<\delta
    \Longrightarrow \sup_{\theta\in\Theta}\bigl|\rm{F}(\mathbb{Q}_{\mathbb{XX}},{\bf{\Sigma}}(\theta))-\rm{F}({\bf{\Sigma}}_0,{\bf{\Sigma}}(\theta))\bigr|<\varepsilon
\end{align*}
on $J_{n}$. Therefore, one has
\begin{align}
    \begin{split}
    0&\leq\PP\Bigl(\bigl\{\|\mathbb{Q}_{\mathbb{XX}}-{\bf{\Sigma}}_0\|<\delta\bigr\}\cap J_n\Bigr)\\
    &\leq\PP\left(\left\{\sup_{\theta\in\Theta}\bigl|\rm{F}(\mathbb{Q}_{\mathbb{XX}},{\bf{\Sigma}}(\theta))-\rm{F}({\bf{\Sigma}}_0,{\bf{\Sigma}}(\theta))\bigr|<\varepsilon\right\}\cap J_n\right)\\
    &\leq\PP\left(\sup_{\theta\in\Theta}\bigl|\tilde{\rm{F}}(\mathbb{Q}_{\mathbb{XX}},{\bf{\Sigma}}(\theta))-\rm{F}({\bf{\Sigma}}_0,{\bf{\Sigma}}(\theta))\bigr|<\varepsilon\right).\label{Qine}
    \end{split}
\end{align}
Since it holds from Theorem \ref{Qtheoremnon} and Lemma \ref{Sigmaposlemma} that $\PP\bigl(J_{n}\bigr)\stackrel{}{\longrightarrow}1$,
\begin{align*}
    \PP\Bigl(\bigl\{\|\mathbb{Q}_{\mathbb{XX}}-{\bf{\Sigma}}_0\|<\delta\bigr\}\cap J_n\Bigr)\stackrel{}{\longrightarrow}1
\end{align*}
as $n\longrightarrow\infty$. Hence, it follows from (\ref{Qine}) that for all $\varepsilon>0$,
\begin{align*}
    \PP\left(\sup_{\theta\in\Theta}\bigl|\tilde{\rm{F}}(\mathbb{Q}_{\mathbb{XX}},{\bf{\Sigma}}(\theta))-\rm{F}({\bf{\Sigma}}_0,{\bf{\Sigma}}(\theta))\bigr|<\varepsilon\right)\longrightarrow 1,
\end{align*}
which implies 
\begin{align*}
    \sup_{\theta\in\Theta}\bigl|\tilde{\rm{F}}(\mathbb{Q}_{\mathbb{XX}},{\bf{\Sigma}}(\theta))-\rm{F}({\bf{\Sigma}}_0,{\bf{\Sigma}}(\theta))\bigr|\stackrel{P}{\longrightarrow}0.
\end{align*}
In the same way, we can show
\begin{align*}
    \sup_{\theta\in\Theta}\bigl|\partial^2_{\theta}\tilde{\rm{F}}(\mathbb{Q}_{\mathbb{XX}},{\bf{\Sigma}}(\theta))-\partial^2_{\theta}\rm{F}({\bf{\Sigma}}_0,{\bf{\Sigma}}(\theta))\bigr|\stackrel{P}{\longrightarrow}0. &\qedhere
\end{align*}
\end{proof}
\begin{lemma}\label{Vproblemma}
Under \textrm{\textbf{[A1]}}, \textrm{\textbf{[B1]}},
\textrm{\textbf{[C1]}} and
\textrm{\textbf{[D1]}}, 
as $h_n\longrightarrow0$, 
\begin{align*}
    \rm{V}(\mathbb{Q}_{\mathbb{XX}},{\bf{\Sigma}}(\theta))&\stackrel{P}{\longrightarrow}\rm{V}({\bf{\Sigma}}_0,{\bf{\Sigma}}(\theta)),\\
    \partial_{\theta^{(i)}}\rm{V}(\mathbb{Q}_{\mathbb{XX}},{\bf{\Sigma}}(\theta))&\stackrel{P}{\longrightarrow}\partial_{\theta^{(i)}}\rm{V}({\bf{\Sigma}}_0,{\bf{\Sigma}}(\theta))
\end{align*}
for $i=1,\cdots,q$.
\end{lemma}
\begin{proof}
The results can be shown in an analogous manner to Lemma \ref{Fproblemma}.
\end{proof}
\begin{lemma}\label{Aposlemma}
Under \textrm{\textbf{[E1]}}(ii), $\Delta^{\top}{\bf{W}}(\theta_0)^{-1}\Delta$ is a positive definite matrix.
\end{lemma}
\begin{proof}
It holds from Lemma \ref{Sigmaposlemma} that ${\bf{W}}(\theta_0)$ is a positive definite matrix, which implies
\begin{align*}
    \Delta^{\top}{\bf{W}}(\theta_0)^{-1}\Delta=\bigl(\Delta^{\top}{\bf{W}}(\theta_0)^{-\frac{1}{2}}\bigr)\bigl(\Delta^{\top}{\bf{W}}(\theta_0)^{-\frac{1}{2}}\bigr)^{\top}\geq 0.
\end{align*}
In a similar way to Lemma 6 in Kusano and Uchida \cite{Kusano(2022)}, we see 
\begin{align*}
    \det \Delta^{\top}{\bf{W}}(\theta_0)^{-1}\Delta\neq 0.
\end{align*}
Therefore, $\Delta^{\top}{\bf{W}}(\theta_0)^{-1}\Delta$ is a positive definite matrix.
\end{proof}
\begin{proof}[\textbf{Proof of Theorem 2}.]
We first prove 
\begin{align}
    \hat{\theta}_{n}\stackrel{P}{\longrightarrow}\theta_{0}.\label{thetacons}
\end{align}
$\bf{[E1]}$ (i) and Lemma \ref{Sigmaposlemma} yield
\begin{align*}
    \begin{split}
    \rm{F}({\bf{\Sigma}}(\theta_{0}),{\bf{\Sigma}}(\theta))=0
    &\Longleftrightarrow\vech{\bf{\Sigma}}(\theta_{0})-\vech{\bf{\Sigma}}(\theta)=0\Longleftrightarrow \theta_{0}=\theta.
    \end{split}
\end{align*}
For any $\varepsilon>0$, there exists $\delta>0$ such that
\begin{align*}
    |\hat{\theta}_{n}-\theta_{0}|>\varepsilon\Longrightarrow \rm{F}({\bf{\Sigma}}(\theta_{0}),{\bf{\Sigma}}(\hat{\theta}_{n}))-\rm{F}({\bf{\Sigma}}(\theta_{0}),{\bf{\Sigma}}(\theta_{0}))>\delta.
\end{align*} 
From the definition of $\hat{\theta}_{n}$, 
\begin{align*}
    \tilde{\rm{F}}(\mathbb{Q}_{\mathbb{XX}},{\bf{\Sigma}}(\hat{\theta}_{n}))=\mathbb{F}_{n}(\hat{\theta}_{n})\leq\mathbb{F}_{n}(\theta_0)=\tilde{\rm{F}}(\mathbb{Q}_{\mathbb{XX}},{\bf{\Sigma}}(\theta_0)).
\end{align*}
Note that ${\bf{\Sigma}}_0={\bf{\Sigma}}(\theta_{0})$. It follows from Lemma \ref{Fproblemma} that  
\begin{align*}
    0&\leq \PP\left(|\hat{\theta}_{n}-\theta_{0}|>\varepsilon\right)\\
    &\leq\PP\Bigl(\rm{F}({\bf{\Sigma}}(\theta_{0}),{\bf{\Sigma}}(\hat{\theta}_{n}))-\rm{F}({\bf{\Sigma}}(\theta_{0}),{\bf{\Sigma}}(\theta_{0}))>\delta\Bigr)\\
    &\leq\PP\left(\rm{F}({\bf{\Sigma}}(\theta_{0}),{\bf{\Sigma}}(\hat{\theta}_{n}))-\tilde{\rm{F}}(\mathbb{Q}_{\mathbb{XX}},{\bf{\Sigma}}(\hat{\theta}_{n}))>\frac{\delta}{3}\right)\\
    &\quad+\PP\left(\tilde{\rm{F}}(\mathbb{Q}_{\mathbb{XX}},{\bf{\Sigma}}(\hat{\theta}_{n}))-\tilde{\rm{F}}(\mathbb{Q}_{\mathbb{XX}},{\bf{\Sigma}}(\theta_{0}))>\frac{\delta}{3}\right)\\
    &\quad+\PP\left(\tilde{\rm{F}}(\mathbb{Q}_{\mathbb{XX}},{\bf{\Sigma}}(\theta_{0}))-\rm{F}({\bf{\Sigma}}(\theta_{0}),{\bf{\Sigma}}(\theta_{0}))>\frac{\delta}{3}\right)\\
    &\leq 2\PP\left(\sup_{\theta\in\Theta}\bigl|\tilde{\rm{F}}(\mathbb{Q}_{\mathbb{XX}},{\bf{\Sigma}}(\theta))-\rm{F}({\bf{\Sigma}}(\theta_{0}),{\bf{\Sigma}}(\theta))\bigr|>\frac{\delta}{3}\right)+0\stackrel{}{\longrightarrow}0
\end{align*}
as $n\longrightarrow\infty$, which yields (\ref{thetacons}). 

Next, we prove
\begin{align*}
    \sqrt{n}(\hat{\theta}_{n}-\theta_{0})\stackrel{d}{\longrightarrow}N_{q}\bigl(0,\bigl(\Delta^{\top}{\bf{W}}(\theta_0)^{-1}\Delta\bigr)^{-1}\bigr).
\end{align*}
The Taylor expansion of $\partial_{\theta}\mathbb{F}_{n}(\hat{\theta}_{n})$ around $\hat{\theta}_{n}=\theta_{0}$ is given by
\begin{align*}
    \begin{split}
    \partial_{\theta}\mathbb{F}_{n}(\hat{\theta}_{n})&=\partial_{\theta}\mathbb{F}_{n}(\theta_{0})+\int_{0}^{1}\partial^2_{\theta}\mathbb{F}_{n}(\ddot{\theta}_{n})d\lambda(\hat{\theta}_{n}-\theta_{0}),
    \end{split}
\end{align*}
where $\ddot{\theta}_{n}=\theta_{0}+\lambda(\hat{\theta}_{n}-\theta_{0})$. Since $\partial_{\theta}\mathbb{F}_{n}(\hat{\theta}_{n})=0$ from the definition of $\hat{\theta}_{n}$, one gets
\begin{align}
    -\sqrt{n}\partial_{\theta}\mathbb{F}_{n}(\theta_{0})=\int_{0}^{1}\partial^2_{\theta}\mathbb{F}_{n}(\ddot{\theta}_{n})d\lambda\sqrt{n}(\hat{\theta}_{n}-\theta_{0}).\label{thetataylor}
\end{align}
Let $\tilde{V}_{n}(\theta)=\rm{V}(\mathbb{Q}_{\mathbb{XX}},{\bf{\Sigma}}(\theta))$. Theorem 1 and Lemma \ref{Vproblemma} imply that the left-hand side of (\ref{thetataylor}) is given by
\begin{align*}
    \begin{split}
    -\sqrt{n}\partial_{\theta^{(i)}}\mathbb{F}_{n}(\theta_{0})
    &=2\bigl\{\partial_{\theta^{(i)}}\vech{\bf{\Sigma}}(\theta_{0})\bigr\}^{\top}\tilde{V}_{n}(\theta_0)\sqrt{n}(\vech \mathbb{Q}_{\mathbb{XX}}-\vech{\bf{\Sigma}}(\theta_{0}))\\
    &\quad-(\vech \mathbb{Q}_{\mathbb{XX}}-\vech{\bf{\Sigma}}(\theta_{0}))^{\top}\partial_{\theta^{(i)}}\tilde{V}_{n}(\theta_0)\sqrt{n}(\vech \mathbb{Q}_{\mathbb{XX}}-\vech{\bf{\Sigma}}(\theta_{0}))\\
    &=2\bigl\{\partial_{\theta^{(i)}}\vech{\bf{\Sigma}}(\theta_{0})\bigr\}^{\top}\tilde{V}_{n}(\theta_0)\sqrt{n}(\vech \mathbb{Q}_{\mathbb{XX}}-\vech{\bf{\Sigma}}(\theta_{0}))+o_{p}(1)
    \end{split}
\end{align*}
for $i=1,\cdots,q$. Thus, it follows from Theorem 1 and Lemma \ref{Vproblemma} that
\begin{align}
    \begin{split}
    -\sqrt{n}\partial_{\theta}\mathbb{F}_{n}(\theta_{0})&=2\Delta^{\top}\tilde{V}_{n}(\theta_0)\sqrt{n}(\vech \mathbb{Q}_{\mathbb{XX}}-\vech{\bf{\Sigma}}(\theta_{0}))+o_{p}(1)\\
    &\stackrel{d}{\longrightarrow}2\Delta^{\top}{\bf{W}}(\theta_0)^{-1}N_{\bar{p}}\bigl(0,{\bf{W}}(\theta_0)\bigr)\sim N_{q}\bigl(0,4\Delta^{\top}{\bf{W}}(\theta_0)^{-1}\Delta\bigr).\label{partialFprob}
    \end{split}
\end{align}
Set $A_{n}=\bigl\{|\hat{\theta}_{n}-\theta_{0}|\leq\rho_n\bigr\}$, where $\{\rho_n\}_{n\in\mathbb{N}}$ is a positive sequence such that  $\rho_n\longrightarrow0$ as $n\longrightarrow\infty$. Note that $\partial_{\theta}^2{\rm{F}}$ is uniform continuous in $\theta$ on $\Theta$ since $\partial_{\theta}^2{\rm{F}}$ is continuous in $\theta$ and $\Theta$ is a compact set. As it holds that
\begin{align*} 
    \partial_{\theta}^2
    \rm{F}({\bf{\Sigma}}(\theta_{0}),{\bf{\Sigma}}(\theta_{0}))=2\Delta^{\top}{\bf{W}}(\theta_0)^{-1}\Delta,
\end{align*}
we see
\begin{align}
    \sup_{|\theta-\theta_{0}|\leq\rho_n}
    \Bigl\|\partial_{\theta}^2\rm{F}({\bf{\Sigma}}(\theta_{0}),{\bf{\Sigma}}(\theta))
    -2\Delta^{\top}{\bf{W}}(\theta_0)^{-1}\Delta\Bigr\|\longrightarrow 0\label{uniprob}
\end{align}
as $n\longrightarrow\infty$. Hence, we see from Lemma \ref{Fproblemma}, (\ref{thetacons}) and (\ref{uniprob}) that for any $\varepsilon>0$,
\begin{align*}
    0&\leq \PP\left(\Bigl\|\int_{0}^{1}\partial_{\theta}^2\mathbb{F}_{n}(\ddot{\theta}_{n})d\lambda-2\Delta^{\top}{\bf{W}}(\theta_0)^{-1}\Delta\Bigr\|>\varepsilon\right)\\
    &\leq\mathbb{P}\left(\left\{\Bigl\|\int_{0}^{1}\partial_{\theta}^2\mathbb{F}_{n}(\ddot{\theta}_{n})d\lambda-2\Delta^{\top}{\bf{W}}(\theta_0)^{-1}\Delta\Bigr\|>\varepsilon\right\}\cap A_{n}\right)\\
    &\quad+\PP\left(\left\{\Bigl\|\int_{0}^{1}\partial_{\theta}^2\mathbb{F}_{n}(\ddot{\theta}_{n})d\lambda-2\Delta^{\top}{\bf{W}}(\theta_0)^{-1}\Delta\Bigr\|>\varepsilon\right\}\cap A_{n}^c\right)\\
    &\leq \PP\left(\sup_{|\theta-\theta_{0}|\leq\rho_n}\Bigl\|\partial_{\theta}^2\tilde{\rm{F}}(\mathbb{Q}_{\mathbb{XX}},{\bf{\Sigma}}(\theta))-2\Delta^{\top}{\bf{W}}(\theta_0)^{-1}\Delta\Bigr\|>\varepsilon\right)+\PP\bigl(A_{n}^{c}\bigr)\\
    &\leq\PP\left(\sup_{|\theta-\theta_{0}|\leq\rho_n}\Bigl\||\partial_{\theta}^2\tilde{\rm{F}}(\mathbb{Q}_{\mathbb{XX}},{\bf{\Sigma}}(\theta))
    -\partial_{\theta}^2\rm{F}({\bf{\Sigma}}(\theta_{0}),{\bf{\Sigma}}(\theta))\Bigr\|>\frac{\varepsilon}{2}\right)\\
    &\quad +\PP\left(\sup_{|\theta-\theta_{0}|\leq\rho_n}\Bigl\|\partial_{\theta}^2\rm{F}({\bf{\Sigma}}(\theta_{0}),{\bf{\Sigma}}(\theta))-2\Delta^{\top}{\bf{W}}(\theta_0)^{-1}\Delta\Bigr\|>\frac{\varepsilon}{2}\right)+\PP\bigl(A_{n}^{c}\bigr)\\
    &\leq\PP\left(\sup_{\theta\in\Theta}\Bigl\||\partial_{\theta}^2\tilde{\rm{F}}(\mathbb{Q}_{\mathbb{XX}},{\bf{\Sigma}}(\theta))
    -\partial_{\theta}^2\rm{F}({\bf{\Sigma}}(\theta_{0}),{\bf{\Sigma}}(\theta))\Bigr\|>\frac{\varepsilon}{2}\right)\\
    &\quad +\PP\left(\sup_{|\theta-\theta_{0}|\leq\rho_n}\Bigl\|\partial_{\theta}^2\rm{F}({\bf{\Sigma}}(\theta_{0}),{\bf{\Sigma}}(\theta))-2\Delta^{\top}{\bf{W}}(\theta_0)^{-1}\Delta\Bigr\|>\frac{\varepsilon}{2}\right)+\PP\bigl(A_{n}^{c}\bigr)\longrightarrow 0
\end{align*}
as $n\longrightarrow\infty$, which yields
\begin{align}
    \int_{0}^{1}\partial_{\theta}^2\mathbb{F}_{n}(\ddot{\theta}_{n})d\lambda\stackrel{P}{\longrightarrow}2\Delta^{\top}{\bf{W}}(\theta_0)^{-1}\Delta.\label{intprob}
\end{align}
Therefore, from (\ref{thetataylor}), (\ref{partialFprob}), (\ref{intprob}) and Lemma \ref{Aposlemma}, we obtain
\begin{align*}
    \qquad\qquad\qquad\sqrt{n}(\hat{\theta}_{n}-\theta_{0})&\stackrel{d}{\longrightarrow}(2\Delta^{\top}{\bf{W}}(\theta_0)^{-1}\Delta)^{-1}N_{q}\bigl(0,4\Delta^{\top}{\bf{W}}(\theta_0)^{-1}\Delta\bigr)\\
    &\ \sim N_{q}\bigl(0,(\Delta^{\top}{\bf{W}}(\theta_0)^{-1}\Delta)^{-1}\bigr).&&\qedhere
\end{align*}
\end{proof}
\begin{proof}[\textbf{Proof of Theorem 3}]
The Taylor expansion of $\mathbb{T}_{n}=n\mathbb{F}_{n}(\hat{\theta}_{n})$ around $\hat{\theta}_{n}=\theta_{0}$ is given by
\begin{align}
    \begin{split}
    \mathbb{T}_{n}&=n\mathbb{F}_{n}(\theta_{0})+n\partial_{\theta}\mathbb{F}_{n}(\theta_{0})^{\top}(\hat{\theta}_{n}-\theta_{0})\\
    &\qquad+n(\hat{\theta}_{n}-\theta_{0})^{\top}\left\{\int_{0}^{1}(1-\lambda)\partial^2_{\theta}\mathbb{F}_{n}(\ddot{\theta}_{n})d\lambda\right\}(\hat{\theta}_{n}-\theta_{0}).\label{testtaylor}
    \end{split}
\end{align}
In a similar way to Theorem 2, we obtain
\begin{align}
    \begin{split}
    \sqrt{n}\partial_{\theta}\mathbb{F}_{n}(\theta_{0})&=-2\Delta^{\top}\tilde{V}_{n}(\theta_0)\sqrt{n}(\vech \mathbb{Q}_{\mathbb{XX}}-\vech{\bf{\Sigma}}(\theta_{0}))+o_{p}(1) \label{nFstar}
    \end{split}
\end{align}
under $H_0$ and
\begin{align}
    \begin{split}
    \sqrt{n}(\hat{\theta}_{n}-\theta_{0})=(\Delta^{\top}{\bf{W}}(\theta_0)^{-1}\Delta)^{-1}\Delta^{\top}\tilde{V}_{n}(\theta_0)\sqrt{n}(\vech \mathbb{Q}_{\mathbb{XX}}-\vech{\bf{\Sigma}}(\theta_{0}))+o_{p}(1)\label{thetamstar} 
    \end{split}
\end{align}
under $H_0$. Let
\begin{align*}
    H_{n}(\theta_0)=\tilde{V}_{n}(\theta_0)\Delta(\Delta^{\top}{\bf{W}}(\theta_0)^{-1}\Delta)^{-1}\Delta^{\top}\tilde{V}_{n}(\theta_0).
\end{align*}
Theorem 1, Lemma \ref{Vproblemma}, (\ref{nFstar}) and (\ref{thetamstar}) imply that the second term on the right-hand side of (\ref{testtaylor}) is expressed as
\begin{align}
    \begin{split}
    &\quad\ \ n\partial_{\theta}\mathbb{F}_{n}(\theta_{0})^{\top}(\hat{\theta}_{n}-\theta_{0})\\
    &=-2\sqrt{n}(\vech \mathbb{Q}_{\mathbb{XX}}-\vech{\bf{\Sigma}}(\theta_{0}))^{\top}H_{n}(\theta_0)\sqrt{n}(\vech \mathbb{Q}_{\mathbb{XX}}-\vech{\bf{\Sigma}}(\theta_{0}))+o_p(1)\label{nFpar}
    \end{split}
\end{align}
under $H_0$. Recall that $A_{n}=\bigl\{|\hat{\theta}_{n}-\theta_{0}|\leq\rho_n\bigr\}$,
where a positive sequence $\{\rho_n\}_{n\in\mathbb{N}}$ satisfies $\rho_n\longrightarrow0$ as $n\longrightarrow\infty$. In an analogous manner to Theorem 2, it follows that for all $\varepsilon>0$,
\begin{align*}
    0&\leq\PP\left(\Bigl\|\int_{0}^{1}(1-\lambda)\partial^2_{\theta}\mathbb{F}_{n}(\ddot{\theta}_{n})d\lambda-\Delta^{\top}{\bf{W}}(\theta_0)^{-1}\Delta\Bigr\|>\varepsilon\right)\\
    &\leq\PP\left(\left\{\Bigl\|\int_{0}^{1}(1-\lambda)\bigl\{\partial^2_{\theta}\mathbb{F}_{n}(\ddot{\theta}_{n})-2\Delta^{\top}{\bf{W}}(\theta_0)^{-1}\Delta\bigr\}d\lambda\Bigr\|>\varepsilon\right\}\cap A_{n}\right)\\
    &\quad+\PP\left(\left\{\Bigl\|\int_{0}^{1}(1-\lambda)\bigl\{\partial^2_{\theta}\mathbb{F}_{n}(\ddot{\theta}_{n})-2\Delta^{\top}{\bf{W}}(\theta_0)^{-1}\Delta\bigr\}d\lambda\Bigr\|>\varepsilon\right\}\cap A_{n}^{c}\right)\\
    &\leq\PP\left(\sup_{|\theta-\theta_{0}|\leq\rho_{n}}\left\|\partial^2_{\theta}\tilde{\rm{F}}(\mathbb{Q}_{\mathbb{XX}},{\bf{\Sigma}}(\theta))-2\Delta^{\top}{\bf{W}}(\theta_0)^{-1}\Delta\right\|>\varepsilon\right)+\PP\bigl(A_{n}^{c}\bigr)\\
    &\leq\PP\left(\sup_{\theta\in\Theta}\left\|\partial^2_{\theta}\tilde{\rm{F}}(\mathbb{Q}_{\mathbb{XX}},{\bf{\Sigma}}(\theta))-\partial^2_{\theta}\tilde{\rm{F}}({\bf{\Sigma}}(\theta_{0}),{\bf{\Sigma}}(\theta))\right\|>\varepsilon\right)\\
    &\quad+\PP\left(\sup_{|\theta-\theta_{0}|\leq\rho_{n}}\left\|\partial^2_{\theta}\tilde{\rm{F}}({\bf{\Sigma}}(\theta_{0}),{\bf{\Sigma}}(\theta))-2\Delta^{\top}{\bf{W}}(\theta_0)^{-1}\Delta\right\|>\varepsilon\right)+\PP\bigl(A_{n}^{c}\bigr)\longrightarrow 0
\end{align*} 
as $n\longrightarrow\infty$ under $H_0$, so that
\begin{align}
    \int_{0}^{1}(1-\lambda)\partial_{\theta}^2\mathbb{F}_{n}(\ddot{\theta}_{n})d\lambda\stackrel{P}{\longrightarrow}\Delta^{\top}{\bf{W}}(\theta_0)^{-1}\Delta\label{intprob2}
\end{align}
under $H_{0}$. Thus, Theorem 1, Lemma \ref{Vproblemma} and (\ref{thetamstar}) imply that the third term on the right-hand side of (\ref{testtaylor}) is
\begin{align}
    \begin{split}
    &\quad\ n(\hat{\theta}_{n}-\theta_{0})^\top\left\{\int_{0}^{1}(1-\lambda)\partial_{\theta}^2\mathbb{F}_{n}(\ddot{\theta}_{n})d\lambda\right\}(\hat{\theta}_{n}-\theta_{0})\\
    &=\sqrt{n}(\vech \mathbb{Q}_{\mathbb{XX}}-\vech{\bf{\Sigma}}(\theta_{0}))^{\top}H_{n}(\theta_0)\sqrt{n}(\vech \mathbb{Q}_{\mathbb{XX}}-\vech{\bf{\Sigma}}(\theta_{0}))+o_p(1)\label{nFpar2}
    \end{split}
\end{align}
under $H_0$. Therefore, it follows from (\ref{nFpar}) and (\ref{nFpar2}) that (\ref{testtaylor}) is given by
\begin{align}
    \begin{split}
    \mathbb{T}_{n}=\sqrt{n}(\vech \mathbb{Q}_{\mathbb{XX}}-\vech{\bf{\Sigma}}(\theta_{0}))^{\top}(\tilde{V}_{n}(\theta_0)-H_{n}(\theta_0))\sqrt{n}(\vech \mathbb{Q}_{\mathbb{XX}}-\vech{\bf{\Sigma}}(\theta_{0}))+o_p(1)\label{testre1}
    \end{split}
\end{align}
under $H_0$. Set
\begin{align*}
    \gamma_{n}=\tilde{V}_{n}(\theta_0)^{\frac{1}{2}}\sqrt{n}(\vech \mathbb{Q}_{\mathbb{XX}}-\vech{\bf{\Sigma}}(\theta_{0}))
\end{align*}
and
\begin{align*}
    P_{n}(\theta_0)=\tilde{V}_{n}(\theta_0)^{-\frac{1}{2}}(\tilde{V}_{n}(\theta_0)-H_{n}(\theta_0))\tilde{V}_{n}(\theta_0)^{-\frac{1}{2}}.
\end{align*}
We can rewrite (\ref{testre1}) as
\begin{align}
    \mathbb{T}_{n}=\gamma_{n}^{\top}P_{n}(\theta_0)\gamma_{n}+o_p(1)\label{Tgamma}
\end{align}
under $H_0$. It follows from Lemma \ref{Vproblemma} and the continuous mapping theorem  that
under $H_0$,
\begin{align*}
    \begin{split}
    \tilde{V}_{n}(\theta_0)^{\frac{1}{2}}=f_1(\tilde{V}_{n}(\theta_0))\stackrel{P}{\longrightarrow}f_1({\bf{W}}(\theta_0)^{-1})={\bf{W}}(\theta_0)^{-\frac{1}{2}},
    \end{split}
\end{align*}
where $f_1(X)=X^{\frac{1}{2}}$ for $X\in\mathbb{R}^{\bar{p}\times\bar{p}}$. Theorem 1 and Slutsky's theorem show that
under $H_0$, 
\begin{align}
    \gamma_{n}\stackrel{d}{\longrightarrow}\gamma, \label{gammad}
\end{align}
where $\gamma\sim N_{\bar{p}}(0,\mathbb{I}_{\bar{p}})$. Set
\begin{align*}
    H(\theta_0)={\bf{W}}(\theta_0)^{-1}\Delta(\Delta^{\top}{\bf{W}}(\theta_0)^{-1}\Delta)^{-1}
    \Delta^{\top}{\bf{W}}(\theta_0)^{-1}.
\end{align*}
It follows from Lemma \ref{Vproblemma} and the continuous mapping theorem that under $H_0$,
\begin{align}
    H_{n}(\theta_0)=f_2(\tilde{V}_{n}(\theta_0))\stackrel{P}{\longrightarrow}f_2({\bf{W}}(\theta_0)^{-1})=H(\theta_0), \label{Hprob}
\end{align}
where for $X\in\mathbb{R}^{\bar{p}\times\bar{p}}$,
\begin{align*}
    f_2(X)=X\Delta(\Delta^{\top}{\bf{W}}(\theta_0)^{-1}\Delta)^{-1}
    \Delta^{\top}X.
\end{align*}
Since Lemma \ref{Vproblemma} and the continuous mapping theorem imply that under $H_0$,
\begin{align*}
    \begin{split}
    \tilde{V}_{n}(\theta_0)^{-\frac{1}{2}}=f_3(\tilde{V}_{n}(\theta_0))\stackrel{P}{\longrightarrow}f_3({\bf{W}}(\theta_0)^{-1})={\bf{W}}(\theta_0)^{\frac{1}{2}},
    \end{split}
\end{align*}
where $f_3(X)=X^{-\frac{1}{2}}$ for $X\in\mathbb{R}^{\bar{p}\times\bar{p}}$, 
we see from Lemma \ref{Vproblemma}, (\ref{Hprob}) and Slutsky's theorem that under $H_0$,
\begin{align}
    P_{n}(\theta_0)\stackrel{P}{\longrightarrow}P(\theta_0), \label{Pprob}
\end{align}
where 
\begin{align*}
    P(\theta_0)={\bf{W}}(\theta_0)^{\frac{1}{2}}({\bf{W}}(\theta_0)^{-1}-H(\theta_0)){\bf{W}}(\theta_0)^{\frac{1}{2}}.
\end{align*}
Furthermore, it follows  from the continuous mapping theorem and (\ref{gammad}) that under $H_0$,
\begin{align}
    \gamma_{n}^{\top}P(\theta_0)\gamma_{n}=f_4(\gamma_{n})\stackrel{d}{\longrightarrow}f_4(\gamma)=\gamma^{\top}P(\theta_0)\gamma,   \label{gammaPd}
\end{align}
where $f_4(x)=x^{\top}P(\theta_0)x$ for $x\in\mathbb{R}^{\bar{p}}$. We see from (\ref{gammad}) that $\gamma_{n}=O_{p}(1)$ under $H_0$, and it holds from (\ref{Pprob}) that
\begin{align}
    \gamma_{n}^{\top}P_{n}(\theta_0)\gamma_{n}-\gamma_{n}^{\top}P(\theta_0)\gamma_{n}\stackrel{P}{\longrightarrow} 0 \label{Pslu}
\end{align}
under $H_0$. Therefore, (\ref{Tgamma}), (\ref{gammaPd}), (\ref{Pslu}) and Slutsky's theorem yield 
\begin{align*}
    \mathbb{T}_{n}\stackrel{d}{\longrightarrow}\gamma^{\top}P(\theta_0)\gamma
\end{align*}
under $H_0$. Since one gets $\gamma^{\top}P(\theta_0)\gamma\sim\chi^2_{\bar{p}-q}$
in the same manner as Theorem 3 in Kusano and Uchida \cite{Kusano(2022)},
\begin{align*}
    \mathbb{T}_{n}\stackrel{d}{\longrightarrow}\chi^2_{\bar{p}-q}
\end{align*}
under $H_0$. 
\end{proof}
\begin{lemma}\label{starconslemma}
Under \textrm{\textbf{[A1]}}, \textrm{\textbf{[B1]}},
\textrm{\textbf{[C1]}},
\textrm{\textbf{[D1]}} and \textrm{\textbf{[E2]}},
as $h_n\longrightarrow0$, 
\begin{align*}
    \hat{\theta}_{n}\stackrel{P}{\longrightarrow}\bar{\theta}
\end{align*}
under $H_1$.
\end{lemma}
\begin{proof}
From \textrm{\textbf{[E2]}}, for any $\varepsilon>0$, there exists $\delta>0$ such that
\begin{align}
    |\hat{\theta}_{n}-\bar{\theta}|>\varepsilon\Longrightarrow \mathbb{U}(\hat{\theta}_{n})-\mathbb{U}(\bar{\theta})>\delta.\label{assumptionI}
\end{align}
As it holds from the definition of $\hat{\theta}_{n}$ that
\begin{align*}
    \tilde{\rm{F}}(\mathbb{Q}_{\mathbb{XX}},{\bf{\Sigma}}(\hat{\theta}_{n}))\leq\tilde{\rm{F}}(\mathbb{Q}_{\mathbb{XX}},{\bf{\Sigma}}(\bar{\theta})),
\end{align*}
we see from Lemma \ref{Fproblemma} and (\ref{assumptionI}) that  
for all $\varepsilon>0$,
\begin{align*}
    0&\leq \PP\left(|\hat{\theta}_{n}-\bar{\theta}|>\varepsilon\right)\\
    &\leq\PP\Bigl(\mathbb{U}(\hat{\theta}_{n})-\mathbb{U}(\bar{\theta})>\delta\Bigr)\\
    &\leq\PP\left(\rm{F}({\bf{\Sigma}}_0,{\bf{\Sigma}}(\hat{\theta}_{n}))-\tilde{\rm{F}}(\mathbb{Q}_{\mathbb{XX}},{\bf{\Sigma}}(\hat{\theta}_{n}))>\frac{\delta}{3}\right)\\
    &\quad+\PP\left(\tilde{\rm{F}}(\mathbb{Q}_{\mathbb{XX}},{\bf{\Sigma}}(\hat{\theta}_{n}))-\tilde{\rm{F}}(\mathbb{Q}_{\mathbb{XX}},{\bf{\Sigma}}(\bar{\theta}))>\frac{\delta}{3}\right)\\
    &\quad+\PP\left(\tilde{\rm{F}}(\mathbb{Q}_{\mathbb{XX}},{\bf{\Sigma}}(\bar{\theta}))-\rm{F}({\bf{\Sigma}}_0,{\bf{\Sigma}}(\bar{\theta}))>\frac{\delta}{3}\right)\\
    &\leq 2\PP\left(\sup_{\theta\in\Theta}\bigl|\tilde{\rm{F}}(\mathbb{Q}_{\mathbb{XX}},{\bf{\Sigma}}(\theta))-\rm{F}({\bf{\Sigma}}_0,{\bf{\Sigma}}(\theta))\bigr|>\frac{\delta}{3}\right)+0\stackrel{}{\longrightarrow}0
\end{align*}
under $H_1$ as $n\longrightarrow\infty$,
which implies $\hat{\theta}_{n}\stackrel{P}{\longrightarrow}\bar{\theta}$ under $H_1$.
\end{proof}
\begin{proof}[\textbf{Proof of Theorem 4}]
Since $\mathbb{U}(\theta)$ is continuous in $\theta$, it holds 
from the continuous mapping theorem and Lemma \ref{starconslemma} that
\begin{align}
    \mathbb{U}(\hat{\theta}_{n})\stackrel{P}{\longrightarrow}\mathbb{U}(\bar{\theta})\label{Fu1prob}
\end{align}
under $H_1$. It follows from Lemma \ref{Fproblemma} and (\ref{Fu1prob}) that for all $\varepsilon>0$,
\begin{align*}
    0&\leq \PP\left(\Bigl|\frac{1}{n}\mathbb{T}_{n}-\mathbb{U}(\bar{\theta})\Bigr|>\epsilon\right)\\
    &\leq\PP\left(\bigl|\tilde{\rm{F}}(\mathbb{Q}_{\mathbb{XX}},{\bf{\Sigma}}(\hat{\theta}_{n}))-\mathbb{U}(\hat{\theta}_{n})\bigr|>\frac{\epsilon}{2}\right)+\PP\left(\bigl|\mathbb{U}(\hat{\theta}_{n})-\mathbb{U}(\bar{\theta})\bigr|>\frac{\epsilon}{2}\right)\\
    &\leq\PP\left(\sup_{\theta\in\Theta}\bigl|\tilde{\rm{F}}(\mathbb{Q}_{\mathbb{XX}},{\bf{\Sigma}}(\theta))-\rm{F}({\bf{\Sigma}}_0,{\bf{\Sigma}}(\theta))\bigr|>\frac{\epsilon}{2}\right)\\
    &\qquad+\PP\left(\bigl|\mathbb{U}(\hat{\theta}_{n})-\mathbb{U}(\bar{\theta})\bigr|>\frac{\epsilon}{2}\right)\stackrel{}{\longrightarrow}0
\end{align*}
under $H_1$ as $n\longrightarrow\infty$, which implies that
\begin{align}
    \frac{1}{n}\mathbb{T}_{n}\stackrel{P}{\longrightarrow} \mathbb{U}(\bar{\theta})\label{Tcons}
\end{align}
under $H_1$. Note that $\vech{{\bf{\Sigma}}_0}-\vech{{\bf{\Sigma}}(\bar{\theta})}\neq 0$ under $H_1$. It follows from Lemma \ref{Sigmaposlemma} that $\mathbb{U}(\bar{\theta})>0$ under $H_1$.
Therefore, Lemma 3 in Kitagawa and Uchida \cite{kitagawa(2014)} and (\ref{Tcons}) imply that under $H_1$,
\begin{align*}
    \PP\Bigl(\mathbb{T}_{n}>\chi^2_{\bar{p}-q}(\alpha)\Bigr)
    &=1-\PP\left(\frac{1}{n}\mathbb{T}_{n}\leq \frac{1}{n}\chi^2_{\bar{p}-q}(\alpha)\right)\stackrel{}{\longrightarrow}1
\end{align*}
as $n\longrightarrow\infty$.
\end{proof}
\subsection{Proof of Theorem \ref{Qtheorem}}\label{Qproof}
\begin{lemma}\label{Qlemma}
Under \textrm{\textbf{[A1]}}-\textrm{\textbf{[A2]}}, \textrm{\textbf{[B1]}}-\textrm{\textbf{[B2]}},
\textrm{\textbf{[C1]}}-\textrm{\textbf{[C2]}}
and \textrm{\textbf{[D1]}}-\textrm{\textbf{[D2]}},
as $h_n\longrightarrow 0$ and $nh_n\longrightarrow\infty$,
\begin{align*}
    \mathbb{Q}_{\xi\xi,0}\stackrel{P}{\longrightarrow}{\bf{\Sigma}}_{\xi\xi,0}, \  
    \mathbb{Q}_{\delta\delta,0}\stackrel{P}{\longrightarrow}{\bf{\Sigma}}_{\delta\delta,0}, \ 
    \mathbb{Q}_{\varepsilon\varepsilon,0}\stackrel{P}{\longrightarrow}
    {\bf{\Sigma}}_{\varepsilon\varepsilon,0},\qquad\qquad\quad\\
    \mathbb{Q}_{\zeta\zeta,0}\stackrel{P}{\longrightarrow} {\bf{\Sigma}}_{\zeta\zeta,0},\ \mathbb{Q}_{\xi\delta,0}\stackrel{P}{\longrightarrow} O_{k_1\times p_1},\
    \mathbb{Q}_{\xi\varepsilon,0}\stackrel{P}{\longrightarrow} O_{k_1\times p_2},\qquad\qquad\\
    \mathbb{Q}_{\xi\zeta,0}\stackrel{P}{\longrightarrow} O_{k_1\times k_2},\  \mathbb{Q}_{\delta\varepsilon,0}\stackrel{P}{\longrightarrow} O_{p_1\times p_2},\ \mathbb{Q}_{\delta\zeta} \stackrel{P}{\longrightarrow}O_{p_1\times k_2}, \  \mathbb{Q}_{\varepsilon\zeta,0}\stackrel{P}{\longrightarrow} O_{p_2\times k_2}.
\end{align*}
\end{lemma}
\begin{proof}
The results can be shown in a similar way to Lemma 1 in Kusano and Uchida \cite{Kusano(2022)}.
\end{proof}
\begin{proof}[\textbf{Proof of Theorem \ref{Qtheorem}}]
In a similar way to Theorem \ref{Qtheoremnon}, Lemma \ref{Qlemma} and Slutsky’s theorem yield
\begin{align*}
    \mathbb{Q}_{\mathbb{XX}}\stackrel{P}{\longrightarrow}{\bf{\Sigma}}_0.
\end{align*}
Next, we consider
\begin{align}
    \sqrt{n}(\vech\mathbb{Q}_{\mathbb{XX}}-\vech {\bf{\Sigma}}_0)\stackrel{d}{\longrightarrow} N_{\bar{p}}(0,{\bf{W}}_0).\label{vechasym}
\end{align}
Recall that 
\begin{align*}
    \sqrt{n}(\vec\mathbb{Q}_{\mathbb{XX}}-\vec{\bf{\Sigma}}_0 )=\sum_{i=1}^{n}L_{i,n},
\end{align*}
where 
\begin{align*}
    L_{i,n}=\frac{1}{\sqrt{n}h_n}\vec \Delta \mathbb{X}_i\Delta \mathbb{X}_i^\top-\frac{1}{\sqrt{n}}\vec{\bf{\Sigma}}_0.
\end{align*}
If it is shown that
\begin{align}
    \sum_{i=1}^{n}L_{i,n}\stackrel{d}{\longrightarrow} N_{p^2}(0,\bar{{\bf{W}}}_0),\label{vecasym}
\end{align}
then, we have (\ref{vechasym}) in an analogous way to Theorem \ref{Qtheoremnon}. In a similar way to Lemma 5 in Kessler \cite{kessler(1997)}, if it holds that
\begin{align}
    \label{L}
    &\qquad\qquad\qquad\qquad\sum_{i=1}^n\E\left[L_{i,n}| \mathscr{F}^{n}_{i-1}\right]\stackrel{P}{\longrightarrow}0,\\
    \label{LL}
    \begin{split}
    &\sum_{i=1}^n\E\left[L_{i,n}L_{i,n}^\top| \mathscr{F}^{n}_{i-1}\right]-\sum_{i=1}^n\E\left[L_{i,n}| \mathscr{F}^{n}_{i-1}\right]
    \E\left[L_{i,n}| \mathscr{F}^{n}_{i-1}\right]^\top
    \stackrel{P}{\longrightarrow}\bar{{\bf{W}}}_0,
    \end{split}\\
    \label{L4}
    &\qquad\qquad\qquad\qquad
    \sum_{i=1}^n\E\left[|L_{i,n}|^4|\mathscr{F}^{n}_{i-1}\right]\stackrel{P}{\longrightarrow}0,
\end{align}
then we can obtain (\ref{vecasym}) from Theorems 3.2 and 3.4 in Hall and Heyde \cite{Hall(1981)}. It holds from Lemma \ref{EX1X1lemma} that
\begin{align*}
    &\quad\ \sum_{i=1}^n\E\left[\frac{1}{\sqrt{n}h_n}\Delta \mathbb{X}_{1,i}^{(j_1)}\Delta \mathbb{X}_{1,i}^{(j_2)}-\frac{1}{\sqrt{n}}({\bf{\Sigma}}_0)^{11}_{j_1j_2}\Big|\mathscr{F}^{n}_{i-1}\right]\\
    &=\frac{1}{\sqrt{n}}\sum_{i=1}^n\left\{\frac{1}{h_n}\E\left[\Delta \mathbb{X}_{1,i}^{(j_1)}\Delta \mathbb{X}_{1,i}^{(j_2)}\big|\mathscr{F}^{n}_{i-1}\right]-({\bf{\Sigma}}_0^{11})_{j_1j_2}\right\}\\
    &=\frac{h_n}{\sqrt{n}}\sum_{i=1}^n\bigl\{R_i(1,\xi)+R_i(1,\delta)+R_i(1,\xi)R_i(1,\delta)\bigr\}\\
    &=\sqrt{nh_n^2}\ \frac{1}{n}\sum_{i=1}^n\bigl\{R_i(1,\xi)+R_i(1,\delta)+R_i(1,\xi)R_i(1,\delta)\bigr\}\stackrel{P}{\longrightarrow} 0
\end{align*}
for $j_1,j_2=1,\cdots,p_1$, which yields
\begin{align}
    \begin{split}
    \sum_{i=1}^n\E\left[\frac{1}{\sqrt{n}h_n}\Delta \mathbb{X}_{1,i}^{(j_1)}\Delta \mathbb{X}_{1,i}^{(j_2)}-\frac{1}{\sqrt{n}}({\bf{\Sigma}}_0^{11})_{j_1j_2}\Big|\mathscr{F}^{n}_{i-1}\right]\stackrel{P}{\longrightarrow}0\label{EX1X1prob}
    \end{split}
\end{align}
for $j_1,j_2=1,\cdots,p_1$. In a similar way,                          
\begin{align}
    \begin{split}
    \sum_{i=1}^n\E\left[\frac{1}{\sqrt{n}h_n}\Delta \mathbb{X}_{1,i}^{(j_1)}\Delta \mathbb{X}_{2,i}^{(j_2)}-\frac{1}{\sqrt{n}}({\bf{\Sigma}}_0^{12})_{j_1j_2}\Big|\mathscr{F}^{n}_{i-1}\right]\stackrel{P}{\longrightarrow}0\label{EX1X2prob}
    \end{split}
\end{align}
for $j_1=1,\cdots,p_1$, $j_2=1,\cdots,p_2$, and
\begin{align}
    \begin{split}
    \sum_{i=1}^n\E\left[\frac{1}{\sqrt{n}h_n}\Delta \mathbb{X}_{2,i}^{(j_1)}\Delta \mathbb{X}_{2,i}^{(j_2)}-\frac{1}{\sqrt{n}}({\bf{\Sigma}}_0^{22})_{j_1j_2}\Big|\mathscr{F}^{n}_{i-1}\right]\stackrel{P}{\longrightarrow}0\label{EX2X2prob}
    \end{split}
\end{align}
for $j_1,j_2=1,\cdots,p_2$. Consequently, from (\ref{EX1X1prob})-(\ref{EX2X2prob}), one gets (\ref{L}). Next, we consider (\ref{LL}). It follows from Lemma \ref{EX1X1lemma} that
\begin{align*}
    &\quad\ \sum_{i=1}^n\E\left[\left\{\frac{1}{\sqrt{n}h_n}\Delta \mathbb{X}_{1,i}^{(j_1)}\Delta \mathbb{X}_{1,i}^{(j_2)}-\frac{1}{\sqrt{n}}({\bf{\Sigma}}_0^{11})_{j_1j_2}\right\}\right.\\
    &\qquad\qquad\qquad\qquad\qquad\times\left.\left\{\frac{1}{\sqrt{n}h_n}\Delta \mathbb{X}_{1,i}^{(j_3)}\Delta \mathbb{X}_{1,i}^{(j_4)}-\frac{1}{\sqrt{n}}({\bf{\Sigma}}_0^{11})_{j_3j_4}\right\}\Big|\mathscr{F}^{n}_{i-1}\right]\\
    &=\frac{1}{nh_n^2}\sum_{i=1}^n\E\left[\Delta \mathbb{X}^{(j_1)}_{1,i}\Delta \mathbb{X}^{(j_2)}_{1,i}\Delta \mathbb{X}^{(j_3)}_{1,i}\Delta \mathbb{X}^{(j_4)}_{1,i}\big|\mathscr{F}^{n}_{i-1}\right]\\
    &\quad-\frac{1}{nh_n}\sum_{i=1}^n\E\left[\Delta \mathbb{X}^{(j_1)}_{1,i}\Delta \mathbb{X}^{(j_2)}_{1,i}\big|\mathscr{F}^{n}_{i-1}\right]({\bf{\Sigma}}_0^{11})_{j_3j_4}\\
    &\quad-\frac{1}{nh_n}\sum_{i=1}^n\E\left[\Delta \mathbb{X}^{(j_3)}_{1,i}\Delta \mathbb{X}^{(j_4)}_{1,i}\big|\mathscr{F}^{n}_{i-1}\right]({\bf{\Sigma}}_0^{11})_{j_1j_2}
    +\frac{1}{n}\sum_{i=1}^n({\bf{\Sigma}}_0^{11})_{j_1j_2}({\bf{\Sigma}}_0^{11})_{j_3j_4}\\
    &=({\bf{\Sigma}}_0^{11})_{j_{1}j_{3}}({\bf{\Sigma}}_0^{11})_{j_{2}j_{4}}+({\bf{\Sigma}}_0^{11})_{j_{1}j_{4}}({\bf{\Sigma}}^{11}_0)_{j_{2}j_{3}}\\
    &\quad+\frac{h_n}{n}\sum_{i=1}^n\bigl\{R_i(1,\xi)+R_i(1,\delta)+R_i(1,\xi)R_i(1,\delta)\bigr\},
\end{align*}
which implies 
\begin{align}
\begin{split}
    &\quad\ \sum_{i=1}^n\E\left[\left\{\frac{1}{\sqrt{n}h_n}\Delta \mathbb{X}_{1,i}^{(j_1)}\Delta \mathbb{X}_{1,i}^{(j_2)}-\frac{1}{\sqrt{n}}({\bf{\Sigma}}_0^{11})_{j_1j_2}\right\}\right.\\
    &\qquad\qquad\qquad\qquad\qquad\times\left.\left\{\frac{1}{\sqrt{n}h_n}\Delta \mathbb{X}_{1,i}^{(j_3)}\Delta \mathbb{X}_{1,i}^{(j_4)}-\frac{1}{\sqrt{n}}({\bf{\Sigma}}_0^{11})_{j_3j_4}\right\}\Big|\mathscr{F}^{n}_{i-1}\right]\\
    &\stackrel{P}{\longrightarrow}({\bf{\Sigma}}_0^{11})_{j_{1}j_{3}}({\bf{\Sigma}}_0^{11})_{j_{2}j_{4}}+({\bf{\Sigma}}_0^{11})_{j_{1}j_{4}}({\bf{\Sigma}}_0^{11})_{j_{2}j_{3}}\label{XXXX}
\end{split}
\end{align}
for $j_1,j_2,j_3,j_4=1,\cdots,p_1$. Furthermore, it holds from Lemma \ref{EX1X1lemma} that
\begin{align*}
    &\quad\left|\sum_{i=1}^{n}\E\left[\frac{1}{\sqrt{n}h_n}\Delta \mathbb{X}_{1,i}^{(j_1)}\Delta \mathbb{X}_{1,i}^{(j_2)}-\frac{1}{\sqrt{n}}({\bf{\Sigma}}^{11}_0)_{j_1j_2}\Big|\mathscr{F}^{n}_{i-1}\right]\right.\\
    &\left.\qquad\qquad\qquad\qquad\qquad\quad\times\E\left[\frac{1}{\sqrt{n}h_n}\Delta \mathbb{X}_{1,i}^{(j_3)}\Delta \mathbb{X}_{1,i}^{(j_3)}-\frac{1}{\sqrt{n}}({\bf{\Sigma}}^{11}_0)_{j_3j_4}\Big|\mathscr{F}^{n}_{i-1}\right]\right|\\
    &\leq \sqrt{nh_n^2}\times\left\{\frac{1}{n}\sum_{i=1}^n R_{i}(1,\xi)+\frac{1}{n}\sum_{i=1}^n 
    R_{i}(1,\delta)+\frac{1}{n}\sum_{i=1}^n R_{i}(1,\xi)R_{i}(1,\delta)\right\}\stackrel{P}{\longrightarrow}0,
\end{align*}
so that
\begin{align}
\begin{split}
    &\left|\sum_{i=1}^{n}\E\left[\frac{1}{\sqrt{n}h_n}\Delta \mathbb{X}_{1,i}^{(j_1)}\Delta \mathbb{X}_{1,i}^{(j_2)}-\frac{1}{\sqrt{n}}({\bf{\Sigma}}^{11}_0)_{j_1j_2}\Big|\mathscr{F}^{n}_{i-1}\right]\right.\\
    &\left.\qquad\qquad\qquad\qquad\times\E\left[\frac{1}{\sqrt{n}h_n}\Delta \mathbb{X}_{1,i}^{(j_3)}\Delta \mathbb{X}_{1,i}^{(j_4)}-\frac{1}{\sqrt{n}}({\bf{\Sigma}}^{11}_0)_{j_3j_4}\Big|\mathscr{F}^{n}_{i-1}\right]\right|\stackrel{P}{\longrightarrow}0\label{EXXEXX}
\end{split}
\end{align}
for $j_1,j_2,j_3,j_4=1,\cdots,p_1$. Hence, (\ref{XXXX}) and (\ref{EXXEXX}) yield
\begin{align*}
    &\quad\ \ \sum_{i=1}^{n}\E\left[\left\{\frac{1}{\sqrt{n}h_n}\Delta \mathbb{X}_{1,i}^{(j_1)}\Delta \mathbb{X}_{1,i}^{(j_2)}-\frac{1}{\sqrt{n}}({\bf{\Sigma}}_0^{11})_{j_1j_2}\right\}\right.\\
    &\left.\qquad\qquad\qquad\qquad\times\left\{\frac{1}{\sqrt{n}h_n}\Delta \mathbb{X}_{1,i}^{(j_3)}\Delta \mathbb{X}_{1,i}^{(j_4)}-\frac{1}{\sqrt{n}}({\bf{\Sigma}}_0^{11})_{j_3j_4}\right\}\Big|\mathscr{F}^{n}_{i-1}\right]\\
    &\qquad\qquad-\sum_{i=1}^{n}\E\left[\frac{1}{\sqrt{n}h_n}\Delta \mathbb{X}_{1,i}^{(j_1)}\Delta \mathbb{X}_{1,i}^{(j_2)}-\frac{1}{\sqrt{n}}({\bf{\Sigma}}^{11}_0)_{j_1j_2}\Big|\mathscr{F}^{n}_{i-1}\right]\\
    &\qquad\qquad\qquad\qquad\qquad\qquad\times\E\left[\frac{1}{\sqrt{n}h_n}\Delta \mathbb{X}_{1,i}^{(j_3)}\Delta \mathbb{X}_{1,i}^{(j_4)}-\frac{1}{\sqrt{n}}({\bf{\Sigma}}^{11}_0)_{j_3j_4}\Big|\mathscr{F}^{n}_{i-1}\right]\\
    &\stackrel{P}{\longrightarrow} ({\bf{\Sigma}}_0^{11})_{j_{1}j_{3}}({\bf{\Sigma}}^{11}_0)_{j_{2}j_{4}}+({\bf{\Sigma}}^{11}_0)_{j_{1}j_{4}}({\bf{\Sigma}}^{11}_0)_{j_{2}j_{3}}
\end{align*}
for $j_1,j_2,j_3,j_4=1,\cdots,p_1$. In an analogous manner, we have
\begin{align*}
    &\quad\ \ \sum_{i=1}^{n}\E\left[\left\{\frac{1}{\sqrt{n}h_n}\Delta \mathbb{X}_{1,i}^{(j_1)}\Delta \mathbb{X}_{1,i}^{(j_2)}-\frac{1}{\sqrt{n}}({\bf{\Sigma}}_0^{11})_{j_1j_2}\right\}\right.\\
    &\left.\qquad\qquad\qquad\qquad\times\left\{\frac{1}{\sqrt{n}h_n}\Delta \mathbb{X}_{1,i}^{(j_3)}\Delta \mathbb{X}_{2,i}^{(j_4)}-\frac{1}{\sqrt{n}}({\bf{\Sigma}}_0^{12})_{j_3j_4}\right\}\Big|\mathscr{F}^{n}_{i-1}\right]\\
    &\qquad\qquad-\sum_{i=1}^{n}\E\left[\frac{1}{\sqrt{n}h_n}\Delta \mathbb{X}_{1,i}^{(j_1)}\Delta \mathbb{X}_{1,i}^{(j_2)}-\frac{1}{\sqrt{n}}({\bf{\Sigma}}^{11}_0)_{j_1j_2}\Big|\mathscr{F}^{n}_{i-1}\right]\\
    &\qquad\qquad\qquad\qquad\qquad\qquad\times\E\left[\frac{1}{\sqrt{n}h_n}\Delta \mathbb{X}_{1,i}^{(j_3)}\Delta \mathbb{X}_{2,i}^{(j_4)}-\frac{1}{\sqrt{n}}({\bf{\Sigma}}^{12}_0)_{j_3j_4}\Big|\mathscr{F}^{n}_{i-1}\right]\\
    &\stackrel{P}{\longrightarrow} ({\bf{\Sigma}}_0^{11})_{j_{1}j_{3}}({\bf{\Sigma}}^{12}_0)_{j_{2}j_{4}}+({\bf{\Sigma}}^{12}_0)_{j_{1}j_{4}}({\bf{\Sigma}}^{11}_0)_{j_{2}j_{3}}
\end{align*}
for $j_1,j_2,j_3=1,\cdots,p_1,\ j_4=1,\cdots,p_2$,
\begin{align*}
    &\quad\ \ \sum_{i=1}^{n}\E\left[\left\{\frac{1}{\sqrt{n}h_n}\Delta \mathbb{X}_{1,i}^{(j_1)}\Delta \mathbb{X}_{1,i}^{(j_2)}-\frac{1}{\sqrt{n}}({\bf{\Sigma}}_0^{11})_{j_1j_2}\right\}\right.\\
    &\left.\qquad\qquad\qquad\qquad\times\left\{\frac{1}{\sqrt{n}h_n}\Delta \mathbb{X}_{2,i}^{(j_3)}\Delta \mathbb{X}_{2,i}^{(j_4)}-\frac{1}{\sqrt{n}}({\bf{\Sigma}}_0^{22})_{j_3j_4}\right\}\Big|\mathscr{F}^{n}_{i-1}\right]\\
    &\qquad\qquad-\sum_{i=1}^{n}\E\left[\frac{1}{\sqrt{n}h_n}\Delta \mathbb{X}_{1,i}^{(j_1)}\Delta \mathbb{X}_{1,i}^{(j_2)}-\frac{1}{\sqrt{n}}({\bf{\Sigma}}^{11}_0)_{j_1j_2}\Big|\mathscr{F}^{n}_{i-1}\right]\\
    &\qquad\qquad\qquad\qquad\qquad\qquad\times\E\left[\frac{1}{\sqrt{n}h_n}\Delta \mathbb{X}_{1,i}^{(j_3)}\Delta \mathbb{X}_{2,i}^{(j_4)}-\frac{1}{\sqrt{n}}({\bf{\Sigma}}^{12}_0)_{j_3j_4}\Big|\mathscr{F}^{n}_{i-1}\right]\\
    &\stackrel{P}{\longrightarrow} ({\bf{\Sigma}}_0^{12})_{j_{1}j_{3}}({\bf{\Sigma}}^{22}_0)_{j_{2}j_{4}}+({\bf{\Sigma}}^{12}_0)_{j_{1}j_{4}}({\bf{\Sigma}}^{12}_0)_{j_{2}j_{3}}
\end{align*}
for $j_1,j_3=1,\cdots,p_1,\ j_2,j_4=1,\cdots,p_2$,
\begin{align*}
    &\quad\ \ \sum_{i=1}^{n}\E\left[\left\{\frac{1}{\sqrt{n}h_n}\Delta \mathbb{X}_{1,i}^{(j_1)}\Delta \mathbb{X}_{2,i}^{(j_2)}-\frac{1}{\sqrt{n}}({\bf{\Sigma}}_0^{12})_{j_1j_2}\right\}\right.\\
    &\left.\qquad\qquad\qquad\qquad\times\left\{\frac{1}{\sqrt{n}h_n}\Delta \mathbb{X}_{1,i}^{(j_3)}\Delta \mathbb{X}_{2,i}^{(j_4)}-\frac{1}{\sqrt{n}}({\bf{\Sigma}}_0^{12})_{j_3j_4}\right\}\Big|\mathscr{F}^{n}_{i-1}\right]\\
    &\qquad\qquad-\sum_{i=1}^{n}\E\left[\frac{1}{\sqrt{n}h_n}\Delta \mathbb{X}_{1,i}^{(j_1)}\Delta \mathbb{X}_{1,i}^{(j_2)}-\frac{1}{\sqrt{n}}({\bf{\Sigma}}^{11}_0)_{j_1j_2}\Big|\mathscr{F}^{n}_{i-1}\right]\\
    &\qquad\qquad\qquad\qquad\qquad\qquad\times\E\left[\frac{1}{\sqrt{n}h_n}\Delta \mathbb{X}_{1,i}^{(j_3)}\Delta \mathbb{X}_{2,i}^{(j_4)}-\frac{1}{\sqrt{n}}({\bf{\Sigma}}^{12}_0)_{j_3j_4}\Big|\mathscr{F}^{n}_{i-1}\right]\\
    &\stackrel{P}{\longrightarrow} ({\bf{\Sigma}}_0^{12})_{j_{1}j_{3}}({\bf{\Sigma}}^{22}_0)_{j_{2}j_{4}}+({\bf{\Sigma}}^{12}_0)_{j_{1}j_{4}}({\bf{\Sigma}}^{12\top}_{0})_{j_{2}j_{3}}
\end{align*}
for $j_1,j_3=1,\cdots,p_1,\ j_2,j_4=1,\cdots,p_2$,
\begin{align*}
    &\quad\ \ \sum_{i=1}^{n}\E\left[\left\{\frac{1}{\sqrt{n}h_n}\Delta \mathbb{X}_{1,i}^{(j_1)}\Delta \mathbb{X}_{2,i}^{(j_2)}-\frac{1}{\sqrt{n}}({\bf{\Sigma}}_0^{12})_{j_1j_2}\right\}\right.\\
    &\left.\qquad\qquad\qquad\qquad\times\left\{\frac{1}{\sqrt{n}h_n}\Delta \mathbb{X}_{2,i}^{(j_3)}\Delta \mathbb{X}_{2,i}^{(j_4)}-\frac{1}{\sqrt{n}}({\bf{\Sigma}}_0^{22})_{j_3j_4}\right\}\Big|\mathscr{F}^{n}_{i-1}\right]\\
    &\qquad\qquad-\sum_{i=1}^{n}\E\left[\frac{1}{\sqrt{n}h_n}\Delta \mathbb{X}_{1,i}^{(j_1)}\Delta \mathbb{X}_{1,i}^{(j_2)}-\frac{1}{\sqrt{n}}({\bf{\Sigma}}^{11}_0)_{j_1j_2}\Big|\mathscr{F}^{n}_{i-1}\right]\\
    &\qquad\qquad\qquad\qquad\qquad\qquad\times\E\left[\frac{1}{\sqrt{n}h_n}\Delta \mathbb{X}_{1,i}^{(j_3)}\Delta \mathbb{X}_{2,i}^{(j_4)}-\frac{1}{\sqrt{n}}({\bf{\Sigma}}^{12}_0)_{j_3j_4}\Big|\mathscr{F}^{n}_{i-1}\right]\\
    &\stackrel{P}{\longrightarrow} ({\bf{\Sigma}}_0^{12})_{j_{1}j_{3}}({\bf{\Sigma}}^{22}_0)_{j_{2}j_{4}}+({\bf{\Sigma}}^{12}_0)_{j_{1}j_{4}}({\bf{\Sigma}}^{22}_{0})_{j_{2}j_{3}}
\end{align*}
for $j_1=1,\cdots,p_1,\ j_2,j_3,j_4=1,\cdots,p_2$, and
\begin{align*}
    &\quad\ \ \sum_{i=1}^{n}\E\left[\left\{\frac{1}{\sqrt{n}h_n}\Delta \mathbb{X}_{2,i}^{(j_1)}\Delta \mathbb{X}_{2,i}^{(j_2)}-\frac{1}{\sqrt{n}}({\bf{\Sigma}}_0^{22})_{j_1j_2}\right\}\right.\\
    &\left.\qquad\qquad\qquad\qquad\times\left\{\frac{1}{\sqrt{n}h_n}\Delta \mathbb{X}_{2,i}^{(j_3)}\Delta \mathbb{X}_{2,i}^{(j_4)}-\frac{1}{\sqrt{n}}({\bf{\Sigma}}_0^{22})_{j_3j_4}\right\}\Big|\mathscr{F}^{n}_{i-1}\right]\\
    &\qquad\qquad-\sum_{i=1}^{n}\E\left[\frac{1}{\sqrt{n}h_n}\Delta \mathbb{X}_{1,i}^{(j_1)}\Delta \mathbb{X}_{1,i}^{(j_2)}-\frac{1}{\sqrt{n}}({\bf{\Sigma}}^{22}_0)_{j_1j_2}\Big|\mathscr{F}^{n}_{i-1}\right]\\
    &\qquad\qquad\qquad\qquad\qquad\qquad\times\E\left[\frac{1}{\sqrt{n}h_n}\Delta \mathbb{X}_{2,i}^{(j_3)}\Delta \mathbb{X}_{2,i}^{(j_4)}-\frac{1}{\sqrt{n}}({\bf{\Sigma}}^{22}_0)_{j_3j_4}\Big|\mathscr{F}^{n}_{i-1}\right]\\
    &\stackrel{P}{\longrightarrow} ({\bf{\Sigma}}_0^{22})_{j_{1}j_{3}}({\bf{\Sigma}}^{22}_0)_{j_{2}j_{4}}+({\bf{\Sigma}}^{22}_0)_{j_{1}j_{4}}({\bf{\Sigma}}^{22}_{0})_{j_{2}j_{3}}
\end{align*}
for $j_1,j_2,j_3,j_4=1,\cdots,p_2$, which yields (\ref{LL}). In an analogous manner to Theorem \ref{Qtheoremnon}, (\ref{L4}) holds.
\end{proof}
\subsection{Proof of Lemma \ref{seiyakutheta}}\label{seiyakuthetaproof}
\begin{proof}[\textbf{Proof of Lemma \ref{seiyakutheta}}]
In an analogous manner to Theorem \ref{thetatheoremnon}, for any $\varepsilon>0$, there exists $\delta>0$ such that
\begin{align*}
    |\underline{\theta}_{n}-\theta_{0}|>\varepsilon\Longrightarrow \rm{F}({\bf{\Sigma}}(\theta_{0}),{\bf{\Sigma}}(\underline{\theta}_{n}))-\rm{F}({\bf{\Sigma}}(\theta_{0}),{\bf{\Sigma}}(\theta_{0}))>\delta.
\end{align*} 
Note that $\theta_{0}\in\underline{\Theta}$. Since it holds from the definition of $\underline{\theta}_n$ that
\begin{align*}
    \mathbb{F}_{n}(\underline{\theta}_{n})
    =\inf_{\theta\in\underline{\Theta}}\mathbb{F}_{n}(\theta)\leq\mathbb{F}_{n}(\theta_{0}),
\end{align*}
we see from Lemma \ref{Fproblemma} that
\begin{align*}
    0&\leq \PP\Bigl(|\underline{\theta}_{n}-\theta_{0}|>\varepsilon\Bigr)\\
    &\leq\PP\left(\rm{F}({\bf{\Sigma}}(\theta_{0}),{\bf{\Sigma}}(\underline{\theta}_{n}))-\tilde{\rm{F}}(\mathbb{Q}_{\mathbb{XX}},{\bf{\Sigma}}(\underline{\theta}_{n}))>\frac{\delta}{3}\right)\\
    &\quad+\PP\left(\tilde{\rm{F}}(\mathbb{Q}_{\mathbb{XX}},{\bf{\Sigma}}(\underline{\theta}_{n}))-\tilde{\rm{F}}(\mathbb{Q}_{\mathbb{XX}},{\bf{\Sigma}}(\theta_{0}))>\frac{\delta}{3}\right)\\
    &\quad+\PP\left(\tilde{\rm{F}}(\mathbb{Q}_{\mathbb{XX}},{\bf{\Sigma}}(\theta_{0}))-\rm{F}({\bf{\Sigma}}(\theta_{0}),{\bf{\Sigma}}(\theta_{0}))>\frac{\delta}{3}\right)\\
    &\leq 2\PP\left(\sup_{\theta\in\Theta}\bigl|\tilde{\rm{F}}(\mathbb{Q}_{\mathbb{XX}},{\bf{\Sigma}}(\theta))-\rm{F}({\bf{\Sigma}}(\theta_{0}),{\bf{\Sigma}}(\theta))\bigr|>\frac{\delta}{3}\right)+0\stackrel{}{\longrightarrow}0
\end{align*}
as $n\longrightarrow\infty$, which yields
\begin{align}
    \underline{\theta}_{n}\stackrel{P}{\longrightarrow}\theta_{0}.\label{barcons}
\end{align}
The Taylor expansion of $\partial_{\theta}\mathbb{F}_{n}(\underline{\theta}_{n})$ around $\underline{\theta}_{n}=\theta_{0}$ is given by
\begin{align*}
    \partial_{\theta}\mathbb{F}_{n}(\underline{\theta}_{n})
    &=\partial_{\theta}\mathbb{F}_{n}(\theta_{0})+\int_{0}^{1}\partial^2_{\theta}\mathbb{F}_{n}(\underline{\ddot{\theta}}_{n})d\lambda(\underline{\theta}_{n}-\theta_{0}), 
\end{align*}
where $\underline{\ddot{\theta}}_{n}=\theta_{0}+\lambda(\underline{\theta}_{n}-\theta_{0})$. Noting that
\begin{align*}
\underline{\theta}_{n}-\theta_{0}=\Bigl(
    (\underline{\theta}_{n}-\theta_{0})_{\mathcal{F}_{1}}^{\top},\ 
     0^{\top}\Bigr)^{\top},
\end{align*}
we have
\begin{align*}
    \partial_{\theta_{\mathcal{F}_{1}}}\mathbb{F}_{n}(\underline{\theta}_{n})
    &=\partial_{\theta_{\mathcal{F}_{1}}}\mathbb{F}_{n}(\theta_{0})+\int_{0}^{1}\partial^2_{\theta_{\mathcal{F}_{1}}}\mathbb{F}_{n}(\underline{\ddot{\theta}}_{n})d\lambda(\underline{\theta}_{n}-\theta_{0})_{\mathcal{F}_{1}}.
\end{align*}
As it holds from the definition of $\underline{\theta}_{n}$ that $\partial_{\theta_{\mathcal{F}_{1}}}\mathbb{F}_{n}(\underline{\theta}_{n})=0$, 
\begin{align}
    -\sqrt{n}\partial_{\theta_{\mathcal{F}_{1}}}\mathbb{F}_{n}(\theta_{0})
    =\int_{0}^{1}\partial^2_{\theta_{\mathcal{F}_{1}}}\mathbb{F}_{n}(\underline{\ddot{\theta}}_{n})d\lambda\sqrt{n}(\underline{\theta}_{n}-\theta_{0})_{\mathcal{F}_{1}}.
    \label{eqseiyaku}
\end{align}
From (\ref{partialFprob}), one gets 
\begin{align}
    \begin{split}
    -\sqrt{n}\partial_{\theta_{\mathcal{F}_{1}}}\mathbb{F}_{n}(\theta_{0})&=2\Delta_{\mathcal{F}_1}^{\top}\tilde{V}_{n}(\theta_0)\sqrt{n}(\vech \mathbb{Q}_{\mathbb{XX}}-\vech{\bf{\Sigma}}(\theta_{0}))+o_{p}(1)\\
    &\stackrel{d}{\longrightarrow} N_{|\mathcal{F}_{1}|}\Bigl(0,4\Delta_{\mathcal{F}_1}^{\top}{\bf{W}}(\theta_0)^{-1}\Delta_{\mathcal{F}_1}\Bigr).
    \end{split}\label{partialnFseiyaku}
\end{align}
Furthermore, in a similar way to Theorem \ref{thetatheoremnon}, it follows from (\ref{barcons}) that
\begin{align}
    \int_{0}^{1}\partial^2_{\theta_{\mathcal{F}_{1}}}\mathbb{F}_{n}(\underline{\ddot{\theta}}_{n})d\lambda\stackrel{P}{\longrightarrow}2\Delta_{\mathcal{F}_1}^{\top}{\bf{W}}(\theta_0)^{-1}\Delta_{\mathcal{F}_1}.\label{intseiyaku}
\end{align}
Recall that $\Delta_{\mathcal{F}_1}^{\top}{\bf{W}}(\theta_0)^{-1}\Delta_{\mathcal{F}_1}$ is non-singular. Therefore, we see from (\ref{eqseiyaku})-(\ref{intseiyaku}) that
\begin{align*}
    \sqrt{n}(\underline{\theta}_{n}-\theta_{0})_{\mathcal{F}_{1}}\stackrel{d}{\longrightarrow}N_{|\mathcal{F}_{1}|}\Bigl(0,(\Delta_{\mathcal{F}_1}^{\top}{\bf{W}}(\theta_0)^{-1}\Delta_{\mathcal{F}_1})^{-1}\Bigr).
    &\qedhere
\end{align*}
\end{proof}
\subsection{Proof of Lemma \ref{seiyakuT}}\label{seiyakuTproof}
\begin{proof}[\textbf{Proof of Lemma \ref{seiyakuT}}]
The Taylor expansion of $n\mathbb{F}_{n}(\underline{\theta}_{n})$ around $\underline{\theta}_n=\theta_{0}$ is given by
\begin{align}
    \begin{split}
    n\mathbb{F}_{n}(\underline{\theta}_{n})
    &=n\mathbb{F}_{n}(\theta_{0})+n\partial_{\theta}\mathbb{F}_{n}(\theta_{0})^{\top}(\underline{\theta}_{n}-\theta_{0})\\
    &\quad+n(\underline{\theta}_{n}-\theta_{0})^{\top}\int_{0}^{1}(1-\lambda)\partial^2_{\theta}\mathbb{F}_{n}(\underline{\ddot{\theta}}_{n})d\lambda(\underline{\theta}_{n}-\theta_{0})\\
    &=n\mathbb{F}_{n}(\theta_{0})+n\partial_{\theta_{\mathcal{F}_{1}}}\mathbb{F}_{n}(\theta_{0})^{\top}(\underline{\theta}_{n}-\theta_{0})_{\mathcal{F}_{1}}\\
    &\quad+n(\underline{\theta}_{n}-\theta_{0})_{\mathcal{F}_{1}}^{\top}
    \int_{0}^{1}(1-\lambda)\partial^2_{\theta_{\mathcal{F}_1}}\mathbb{F}_{n}(\underline{\ddot{\theta}}_{n})d\lambda(\underline{\theta}_{n}-\theta_{0})_{\mathcal{F}_{1}}.
    \end{split}\label{nFseiyaku}
\end{align}
In a similar way to Lemma \ref{seiyakutheta}, Theorem \ref{Qtheoremnon} implies
\begin{align*}
    -\sqrt{n}\partial_{\theta_{\mathcal{F}_{1}}}\mathbb{F}_{n}(\theta_{0})&=2\Delta_{\mathcal{F}_1}^{\top}\tilde{V}_{n}(\theta_0)\sqrt{n}(\vech \mathbb{Q}_{\mathbb{XX}}-\vech{\bf{\Sigma}}(\theta_{0}))+o_{p}(1)
\end{align*}
and
\begin{align}
    \begin{split}
    \sqrt{n}(\underline{\theta}_{n}-\theta_{0})_{\mathcal{F}_{1}}&={\bf{A}}_{\mathcal{F}}^{11}(\theta_0)^{-1}\Delta_{\mathcal{F}_{1}}^{\top}\tilde{V}_{n}(\theta_0)\sqrt{n}(\vech \mathbb{Q}_{\mathbb{XX}}-\vech{\bf{\Sigma}}(\theta_{0}))+o_{p}(1)\label{nthetaseiyaku}
    \end{split}
\end{align}
under $H_0$. Let
\begin{align*}
     \underline{H}_{n}(\theta_0)=\tilde{V}_{n}(\theta_0)\Delta_{\mathcal{F}_{1}}(\Delta_{\mathcal{F}_1}^{\top}{\bf{W}}(\theta_0)^{-1}\Delta_{\mathcal{F}_1})^{-1}\Delta_{\mathcal{F}_{1}}^{\top}\tilde{V}_{n}(\theta_0).
\end{align*}
It holds from Theorem \ref{Qtheoremnon}, (\ref{partialnFseiyaku}) and (\ref{nthetaseiyaku}) that under $H_0$,
\begin{align}
\begin{split}
    &\quad\ n\partial_{\theta_{\mathcal{F}_{1}}}\mathbb{F}_{n}(\theta_{0})^{\top}(\underline{\theta}_{n}-\theta_{0})_{\mathcal{F}_{1}}\\
    &=-2\sqrt{n}(\vech \mathbb{Q}_{\mathbb{XX}}-\vech{\bf{\Sigma}}(\theta_{0}))^{\top}\underline{H}_{n}(\theta_0)\sqrt{n}(\vech \mathbb{Q}_{\mathbb{XX}}-\vech{\bf{\Sigma}}(\theta_{0}))+o_p(1).
\end{split}\label{nF}
\end{align}
In an analogous manner to Theorem $3$, (\ref{barcons}) yields
\begin{align*}
    \int_{0}^{1}(1-\lambda)\partial^2_{\theta_{\mathcal{F}_{1}}}\mathbb{F}_{n}(\underline{\ddot{\theta}}_{n})d\lambda\stackrel{P}{\longrightarrow}\Delta_{\mathcal{F}_1}^{\top}{\bf{W}}(\theta_0)^{-1}\Delta_{\mathcal{F}_1}
\end{align*}
under $H_0$, so that we see from Theorem \ref{Qtheoremnon} and (\ref{nthetaseiyaku}) that 
\begin{align}
\begin{split}
    &\quad\ n(\underline{\theta}_{n}-\theta_{0})_{\mathcal{F}_{1}}^{\top}
    \int_{0}^{1}(1-\lambda)\partial^2_{\theta_{\mathcal{F}_1}}\mathbb{F}_{n}(\underline{\ddot{\theta}}_{n})d\lambda(\underline{\theta}_{n}-\theta_{0})_{\mathcal{F}_{1}}\\
    &=\sqrt{n}(\vech \mathbb{Q}_{\mathbb{XX}}-\vech{\bf{\Sigma}}(\theta_{0}))^{\top}\underline{H}_{n}(\theta_0)\sqrt{n}(\vech \mathbb{Q}_{\mathbb{XX}}-\vech{\bf{\Sigma}}(\theta_{0}))+o_p(1)\label{nFn}
\end{split}
\end{align}
under $H_0$. Let 
\begin{align*}
    \underline{P}_{n}(\theta_0)&=\tilde{V}_{n}(\theta_0)^{-\frac{1}{2}}(\tilde{V}_{n}(\theta_0)-\underline{H}_{n}(\theta_0))\tilde{V}_{n}(\theta_0)^{-\frac{1}{2}}.
\end{align*}
Lemma \ref{Vproblemma} and Slutsky's theorem imply that
under $H_0$, 
\begin{align*}
    \underline{H}_n(\theta_0)\stackrel{P}{\longrightarrow}\underline{H}(\theta_0),\quad \underline{P}_n(\theta_0)\stackrel{P}{\longrightarrow}\underline{P}(\theta_0),
\end{align*}
where
\begin{align*}
    \underline{H}(\theta_0)&={\bf{W}}(\theta_0)^{-1}\Delta_{\mathcal{F}_{1}}(\Delta_{\mathcal{F}_1}^{\top}{\bf{W}}(\theta_0)^{-1}\Delta_{\mathcal{F}_1})^{-1}\Delta_{\mathcal{F}_{1}}^{\top}{\bf{W}}(\theta_0)^{-1},\\
    \underline{P}(\theta_0)&={\bf{W}}(\theta_0)^{\frac{1}{2}}({\bf{W}}(\theta_0)^{-1}-\underline{H}(\theta_0)){\bf{W}}(\theta_0)^{\frac{1}{2}}.
\end{align*}
Recall that
\begin{align*}
    \gamma_{n}=\tilde{V}_n(\theta_0)^{\frac{1}{2}}\sqrt{n}(\vech \mathbb{Q}_{\mathbb{XX}}-\vech{\bf{\Sigma}}(\theta_{0})).
\end{align*}
In a similar way to Theorem $3$, it follows from (\ref{nFseiyaku}), (\ref{nF}) and (\ref{nFn}) that under $H_0$,
\begin{align*}
    &\quad\ n\mathbb{F}_{n}(\bar{\theta}_{n})\\
    &=\sqrt{n}(\vech \mathbb{Q}_{\mathbb{XX}}-\vech{\bf{\Sigma}}(\theta_{0}))^{\top}
    \bigl(\tilde{V}_{n}(\theta_0)-\underline{H}_{n}(\theta_0)\bigr)\sqrt{n}(\vech \mathbb{Q}_{\mathbb{XX}}-\vech{\bf{\Sigma}}(\theta_{0}))+o_p(1)\\
    &=\gamma_{n}^{\top}\underline{P}_{n}(\theta_0)\gamma_{n}+o_p(1)\\
    &\stackrel{d}{\longrightarrow}\gamma^{\top}\underline{P}(\theta_0)\gamma.
\end{align*}
Since ${\bf{W}}(\theta_0)$ and $\underline{H}(\theta_0)$ are symmetric matrices, one gets $\underline{P}(\theta_0)=\underline{P}(\theta_0)^{\top}$. Furthermore, we have
\begin{align*}
    \underline{H}(\theta_0){\bf{W}}(\theta_0)\underline{H}(\theta_0)&=\underline{H}(\theta_0),
\end{align*}
which implies
\begin{align*}
    \underline{P}(\theta_0)^2&={\bf{W}}(\theta_0)^{\frac{1}{2}}\bigl({\bf{W}}(\theta_0)^{-1}-\underline{H}(\theta_0)\bigr){\bf{W}}(\theta_0)\bigl({\bf{W}}(\theta_0)^{-1}-\underline{H}(\theta_0)\bigr){\bf{W}}(\theta_0)^{\frac{1}{2}}\\
    &={\bf{W}}(\theta_0)^{\frac{1}{2}}\bigl({\bf{W}}(\theta_0)^{-1}-2\underline{H}(\theta_0)+\underline{H}(\theta_0){\bf{W}}(\theta_0)\underline{H}(\theta_0)\bigr){\bf{W}}(\theta_0)^{\frac{1}{2}}\\
    &={\bf{W}}(\theta_0)^{\frac{1}{2}}\bigl({\bf{W}}(\theta_0)^{-1}-\underline{H}(\theta_0)\bigr){\bf{W}}(\theta_0)^{\frac{1}{2}}=\underline{P}(\theta_0).
\end{align*}
Hence, $\underline{P}(\theta_0)$ is a projection matrix and
\begin{align*}
    \rank{\underline{P}(\theta_0)}&=\tr{\underline{P}(\theta_0)}\\
    &=\tr{\bigl(\mathbb{I}_{\bar{p}}-{\bf{W}}(\theta_0)\underline{H}(\theta_0)\bigr)}\\
    &=\bar{p}-\tr{\bigl(\Delta_{\mathcal{F}_{1}}(\Delta_{\mathcal{F}_1}^{\top}{\bf{W}}(\theta_0)^{-1}\Delta_{\mathcal{F}_1})^{-1}\Delta_{\mathcal{F}_{1}}^{\top}{\bf{W}}(\theta_0)^{-1}\bigr)}\\
    &=\bar{p}-\tr{\mathbb{I}_{|\mathcal{F}_{1}|}}=\bar{p}-|\mathcal{F}_{1}|.
\end{align*}
Therefore, it holds from Lemma 7 in Kusano and Uchida \cite{Kusano(2022)} that
\begin{align*}
    \gamma^{\top}\underline{P}(\theta_0)\gamma\sim\chi^2_{\bar{p}-|\mathcal{F}_{1}|},
\end{align*}
which implies
\begin{align*}
     n\mathbb{F}_{n}(\underline{\theta}_{n})\stackrel{d}{\longrightarrow}\chi^2_{\bar{p}-|\mathcal{F}_{1}|}   
  \end{align*}
under $H_0$.
\end{proof}
\subsection{Proof of Lemma \ref{starop1}}\label{testconsproof}
Let $\bar{{\bf{W}}}(\bar{\theta})=\rm{V}({\bf{\Sigma}}_0,{\bf{\Sigma}}(\bar{\theta}))$.
\begin{lemma}\label{Vbar}
Under \textrm{\textbf{[A1]}}, \textrm{\textbf{[B1]}},
\textrm{\textbf{[C1]}} and
\textrm{\textbf{[D1]}},
as $h_n\longrightarrow0$,
\begin{align}
    \sqrt{n}(\tilde{V}_n(\bar{\theta})-\bar{{\bf{W}}}(\bar{\theta}))=O_p(1) \label{Vasym}
\end{align}
and
\begin{align}
    \sqrt{n}(\partial_{\theta^{(i)}}\tilde{V}_n(\bar{\theta})-\partial_{\theta^{(i)}}\bar{{\bf{W}}}(\bar{\theta}))=O_p(1) \label{Vdasym}
\end{align}
for $i=1,\cdots,q$.
\end{lemma}
\begin{proof}
Set $Y_n=\vech{\mathbb{Q}_{\mathbb{XX}}}$ and $Y_0=\vech{{\bf{\Sigma}}_0}$. It holds from Theorem \ref{Qtheoremnon} that
\begin{align}
    \sqrt{n}(Y_n-Y_0)\stackrel{d}{\longrightarrow}Z,\label{YOp1}
\end{align}
where $Z\sim N_{\bar{p}}(0,{\bf{W}}_0)$. 
Let $\tilde{V}_{\theta}(X)=\rm{V}(X,{\bf{\Sigma}}(\theta))$ for any matrix $X\in\mathbb{R}^{p\times p}$, and
\begin{align*}
   g=\tilde{V}_{\bar{\theta}}\circ\vech^{-1},
\end{align*}
where $\vech^{-1}$ is the inverse of vech operator. It follows from the definition of $g$ and Lemma \ref{Sigmaposlemma} that
\begin{align*}
    g(Y_n)&=\tilde{V}_{\bar{\theta}}(\mathbb{Q}_{\mathbb{XX}})=\tilde{V}_n(\bar{\theta})
\end{align*}
and
\begin{align*}
    g(Y_0)&=\tilde{V}_{\bar{\theta}}({\bf{\Sigma}}_0)=\bar{{\bf{W}}}(\bar{\theta}).
\end{align*}
Note that $g(y)_{ij}$ is differentiable at $y=Y_0$ for $i,j=1,\cdots,p$. Therefore, we see from the delta method and (\ref{YOp1}) that
\begin{align*}
    \sqrt{n}(g(Y_n)_{ij}-g(Y_0)_{ij})\stackrel{d}{\longrightarrow}\partial_{y}g(Y_0)_{ij}^{\top}Z
\end{align*}
for $i,j=1,\cdots,p$, which implies (\ref{Vasym}). In an analogous manner, one has (\ref{Vdasym}).
\end{proof}
\begin{proof}[\textbf{Proof of Lemma \ref{starop1}}]
The Taylor expansion of $\mathbb{F}_{n}(\hat{\theta}_{n})$ around $\hat{\theta}_n=\bar{\theta}$ is given by
\begin{align*}
    \begin{split}
    \partial_{\theta}\mathbb{F}_{n}(\hat{\theta}_{n})&=\partial_{\theta}\mathbb{F}_{n}(\bar{\theta})+\int_{0}^{1}\partial^2_{\theta}\mathbb{F}_{n}(\ddot{\bar{\theta}}_{n})d\lambda(\hat{\theta}_{n}-\bar{\theta}),
    \end{split}
\end{align*}
where $\ddot{\bar{\theta}}_{n}=\bar{\theta}+\lambda(\hat{\theta}_{n}-\bar{\theta})$. The definition of $\hat{\theta}_{n}$ yields
\begin{align}
    \begin{split}
    -\sqrt{n}\partial_{\theta}\mathbb{F}_{n}(\bar{\theta})=\int_{0}^{1}\partial^2_{\theta}\mathbb{F}_{n}(\ddot{\bar{\theta}}_{n})d\lambda\sqrt{n}(\hat{\theta}_{n}-\bar{\theta}).
    \end{split}\label{bareq}
\end{align}
Let $u_{n}=\vech \mathbb{Q}_{\mathbb{XX}}-\vech{\bf{\Sigma}}_0$ and $r=\vech{{\bf{\Sigma}}_0}-\vech{{\bf{\Sigma}}(\bar{\theta})}$.
Note that for $i=1,\cdots, q$, 
\begin{align*}
    -\partial_{\theta^{(i)}}\mathbb{F}_{n}(\bar{\theta})&=2\bar{\Delta}_{i}^{\top}V_n(\bar{\theta})(u_n+r)-(u_n+r)^{\top}\partial_{\theta^{(i)}}V_n(\bar{\theta})(u_n+r),
\end{align*}
where $\bar{\Delta}_{i}=\partial_{\theta^{(i)}}\vech{\bf{\Sigma}}(\bar{\theta})$. 
One has 
\begin{align*}
    -\partial_{\theta^{(i)}}\mathbb{F}_{n}(\bar{\theta})&=2\bar{\Delta}_{i}^{\top}\tilde{V}_n(\bar{\theta})u_n+2\bar{\Delta}_{i}^{\top}(
    \tilde{V}_n(\bar{\theta})-\bar{{\bf{W}}}(\bar{\theta}))r+2\bar{\Delta}_{i}^{\top}\bar{{\bf{W}}}(\bar{\theta})r\\
    &\quad-u_n^{\top}\partial_{\theta^{(i)}}\tilde{V}_n(\bar{\theta})u_n-2u_n^{\top}\partial_{\theta^{(i)}}\tilde{V}_n(\bar{\theta})r\\
    &\quad-r^{\top}(\partial_{\theta^{(i)}}\tilde{V}_n(\bar{\theta})-\partial_{\theta^{(i)}}\bar{{\bf{W}}}(\bar{\theta}))r-r^{\top}\partial_{\theta^{(i)}}\bar{{\bf{W}}}(\bar{\theta})r.
\end{align*}
Since it holds from the definition of $\bar{\theta}$ that
\begin{align*}
    \partial_{\theta^{(i)}}\mathbb{U}(\bar{\theta})
    &=2\bar{\Delta}_{i}^{\top}\bar{{\bf{W}}}(\bar{\theta})r-r^{\top}\partial_{\theta^{(i)}}\bar{{\bf{W}}}(\bar{\theta})r=0,
\end{align*}
Theorem \ref{Qtheoremnon}, Lemma \ref{Vproblemma} and Lemma \ref{Vbar} imply
\begin{align}
    \begin{split}
    -\sqrt{n}\partial_{\theta^{(i)}}\mathbb{F}_{n}(\bar{\theta})&=2\bar{\Delta}_{i}^{\top}\tilde{V}_n(\bar{\theta})\sqrt{n}u_n+2\bar{\Delta}_{i}^{\top}\sqrt{n}(\tilde{V}_n(\bar{\theta})-\bar{{\bf{W}}}(\bar{\theta}))r\\
    &\quad-u_n^{\top}\partial_{\theta^{(i)}}\tilde{V}_n(\bar{\theta})\sqrt{n}u_n-2\sqrt{n}u_n^{\top}\partial_{\theta^{(i)}}\tilde{V}_n(\bar{\theta})r\\
    &\quad-r^{\top}\sqrt{n}(\partial_{\theta^{(i)}}\tilde{V}_n(\bar{\theta})-\partial_{\theta^{(i)}}\bar{{\bf{W}}}(\bar{\theta}))r\\
    &=O_p(1)\label{nFOp1}
    \end{split}
\end{align}
under $H_1$ for $i=1,\cdots, q$. In a similar way to Theorem \ref{thetatheoremnon}, from Lemma \ref{starconslemma}, we obtain
\begin{align}
    \int_{0}^{1}\partial^2_{\theta}\mathbb{F}_{n}(\ddot{\bar{\theta}}_{n})d\lambda\stackrel{P}{\longrightarrow}\partial^2_{\theta}\mathbb{U}(\bar{\theta})\label{starprob}
\end{align}
under $H_1$. Therefore, it follows from (\ref{bareq})-(\ref{starprob}) and \textbf{[F2]} that under $H_1$,
\begin{align*}
    \sqrt{n}(\hat{\theta}_{n}-\bar{\theta})=O_p(1). &\qedhere
\end{align*}
\end{proof}
\newpage
\subsection{Simulation results of the ergodic case}\label{ergodic}
Set $(n,h_n,T)=(10^6,10^{-4},10^2)$. First, we consider the correctly specified parametric model in Section 5.2. Table \ref{Qtableer}, Table \ref{thetatableer} and Table\ref{testtable1er} show a sample mean and a sample standard deviation of $\mathbb{Q}_{\mathbb{XX}}$, $\hat{\theta}_n$ and $\mathbb{T}_n$ respectively. Figure \ref{Qfigureer} shows histograms, Q-Q plots and empirical distributions of $\sqrt{n}((\mathbb{Q}_{\mathbb{XX}})_{ij}-({\bf{\Sigma}}_0)_{ij})$ for $i\leq j$ and $i,j=1,\cdots,6$. Figure \ref{thetafigureer} shows histograms, Q-Q plots and empirical distributions of $\sqrt{n}(\hat{\theta}_{n}^{(i)}-\theta_{0}^{(i)})$ for $i=1,\cdots,15$. Figure \ref{testfigure1er} shows Histogram, Q-Q plot  and empirical distribution of the test statistic $\mathbb{T}_n$. It seems from Tables \ref{Qtableer}-\ref{testtable1er} and Figures \ref{Qfigureer}-\ref{testfigure1er} that Theorems 5-7 hold true for this example. Next, we consider the missspecified parametric model in Section 5.3. Table 
\ref{testtable2er} shows the number of rejections of the quasi-likelihood ratio test in Model A and Model B. Table \ref{testtable3er} shows Quartiles of the test statistic $\mathbb{T}_{n}$ in Model A and Model B, which indicates that Theorem 8 holds true for this example.

\begin{table}[h]
    \ \\ \ \\ \ \\ \ \\ \ \\ \ \\ \ \\ 
    \begin{tabular}{ccccc}
    & $(\mathbb{Q}_{\mathbb{XX}})_{11}$ & $(\mathbb{Q}_{\mathbb{XX}})_{12}$ & $(\mathbb{Q}_{\mathbb{XX}})_{13}$ & $(\mathbb{Q}_{\mathbb{XX}})_{14}$ \\\hline
    Mean (True value) & 3.0003 (3.0000) & 4.0001 (4.0000) & 2.0001 (2.0000)& 6.0002 (6.0000) \\
    SD (Theoretical value) & 0.0042 (0.0042) & 0.0073 (0.0072)& 0.0053 (0.0053) & 0.0121 (0.0121)\\
    & $(\mathbb{Q}_{\mathbb{XX}})_{15}$ & $(\mathbb{Q}_{\mathbb{XX}})_{16}$ & $(\mathbb{Q}_{\mathbb{XX}})_{22}$ & $(\mathbb{Q}_{\mathbb{XX}})_{23}$\\\hline
    Mean (True value) & 6.0002 (6.0000) & 18.0008 (18.0000) & 12.0007 (12.0000) & 4.0002 (4.0000)\\
    SD (Theoretical value) & 0.0114 (0.0114) & 0.0342 (0.0341)& 0.0171 (0.0170) & 0.0105 (0.0106)  \\
    & $(\mathbb{Q}_{\mathbb{XX}})_{24}$ & $(\mathbb{Q}_{\mathbb{XX}})_{25}$ & $(\mathbb{Q}_{\mathbb{XX}})_{26}$ & $(\mathbb{Q}_{\mathbb{XX}})_{33}$\\\hline
    Mean (True value) & 12.0006 (12.0000) & 12.0005 (12.0000) & 36.0017 (36.0000) & 8.0007 (8.0000)\\
    SD (Theoretical value) & 0.0243 (0.0242) & 0.0227 (0.0227) & 0.0680 (0.0681) & 0.0113 (0.0113) \\
    & $(\mathbb{Q}_{\mathbb{XX}})_{34}$ & $(\mathbb{Q}_{\mathbb{XX}})_{35}$ & $(\mathbb{Q}_{\mathbb{XX}})_{36}$ & $(\mathbb{Q}_{\mathbb{XX}})_{44}$\\\hline
    Mean (True value) & 12.0002 (12.0000) & 10.0002 (10.0000) & 30.0006 (30.0000) & 37.0008 (37.0000)\\
    SD (Theoretical value) & 0.0209 (0.0210) & 0.0186 (0.0187)& 0.0559 (0.0560) & 0.0525 (0.0523) \\
    & $(\mathbb{Q}_{\mathbb{XX}})_{45}$ & $(\mathbb{Q}_{\mathbb{XX}})_{46}$ & $(\mathbb{Q}_{\mathbb{XX}})_{55}$ & $(\mathbb{Q}_{\mathbb{XX}})_{56}$\\\hline
    Mean (True value) & 30.0008 (30.0000) & 90.0024 (90.0000) & 31.0011 (31.0000) & 90.0031 (90.0000)\\
    SD (Theoretical value) & 0.0453 (0.0452) & 0.1364 (0.1357) & 0.0439 (0.0438) & 0.1298 (0.1294) \\
    & $(\mathbb{Q}_{\mathbb{XX}})_{66}$ &&& \\\hline
    Mean (True value) & 279.0107 (279.0000) & & &\\
    SD (Theoretical value) & 0.3965 (0.3946) & &  \\ \\
\end{tabular}
\caption{Sample mean and sample standard deviation (SD) of $\mathbb{Q}_{\mathbb{XX}}$.}\label{Qtableer}
\end{table}
\clearpage
\begin{table}[h]
\centering
\ \\ \ \\ \ \\ \ \\ \ \\ \ \\ \ \\ \ \\ \ \\ \ \\ \ \\ \ \\
\begin{tabular}{ccccc}
    & $\hat{\theta}_{n}^{(1)}$ & $\hat{\theta}_{n}^{(2)}$ & $\hat{\theta}_{n}^{(3)}$ & $\hat{\theta}_{n}^{(4)}$ \\\hline
    Mean (True value) & 2.0000 (2.0000) & 3.0000 (3.0000) & 3.0000 (3.0000)& 1.0000 (1.0000) \\
    SD (Theoretical value) & 0.0026 (0.0026) & 0.0033 (0.0034) & 0.0009 (0.0008) & 0.0036 (0.0036) \\
    & $\hat{\theta}_{n}^{(5)}$ & $\hat{\theta}_{n}^{(6)}$ & $\hat{\theta}_{n}^{(7)}$ & $\hat{\theta}_{n}^{(8)}$\\\hline
    Mean (True value) & 2.0000 (2.0000) & 2.0001 (2.0000) & 2.0000 (2.0000) & 4.0001 (4.0000)\\
    SD (Theoretical value) & 0.0030 (0.0030) & 0.0044 (0.0044)& 0.0045 (0.0046) & 0.0100 (0.0100) \\
    & $\hat{\theta}_{n}^{(9)}$ & $\hat{\theta}_{n}^{(10)}$ & $\hat{\theta}_{n}^{(11)}$ & $\hat{\theta}_{n}^{(12)}$\\\hline
    Mean (True value) & 1.0002 (1.0000)  &  4.0004 (4.0000) & 4.0006 (4.0000) & 0.9999 (1.0000)\\
    SD (Theoretical value) & 0.0024 (0.0024) & 0.0095 (0.0096) & 0.0059 (0.0060) & 0.0181 (0.0182) \\
    & $\hat{\theta}_{n}^{(13)}$ & $\hat{\theta}_{n}^{(14)}$ & $\hat{\theta}_{n}^{(15)}$ & \\\hline
    Mean (True value) & 1.0001 (1.0000) & 9.0013 (9.0000) & 4.0002 (4.0000)&\\
    SD (Theoretical value) & 0.0038 (0.0038) & 0.0342 (0.0343) & 0.0110 (0.0109)\\ \\
\end{tabular}
\caption{Sample mean and sample standard deviation (SD) of $\hat{\theta}_{n}$.} \label{thetatableer}
\end{table}
\newpage
\begin{figure}[h]
    \ \\ \ \\ \ \\
    \includegraphics[width=0.3\columnwidth]{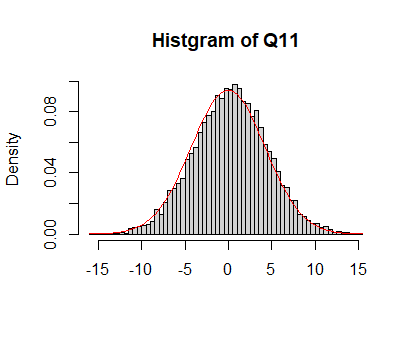}
    \includegraphics[width=0.3\columnwidth]{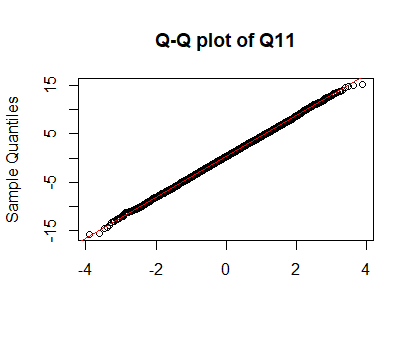}
    \includegraphics[width=0.3\columnwidth]{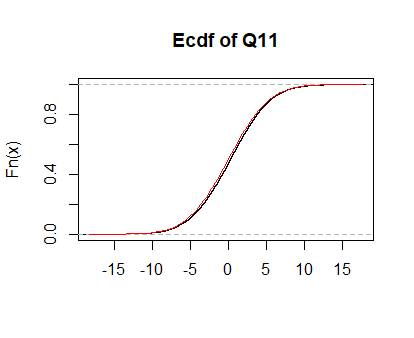}
    \\
    \includegraphics[width=0.3\columnwidth]{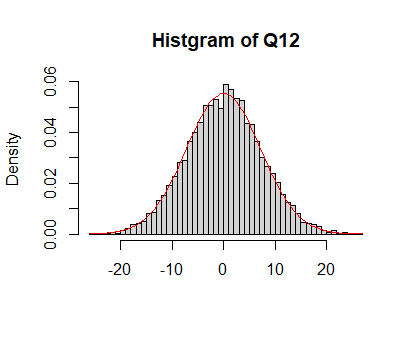}
    \includegraphics[width=0.3\columnwidth]{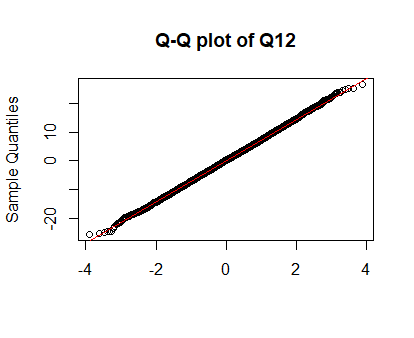}
    \includegraphics[width=0.3\columnwidth]{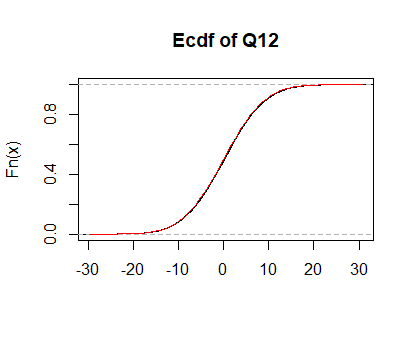}
    \\
    \includegraphics[width=0.3\columnwidth]{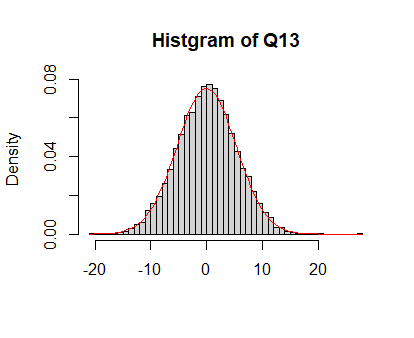}
    \includegraphics[width=0.3\columnwidth]{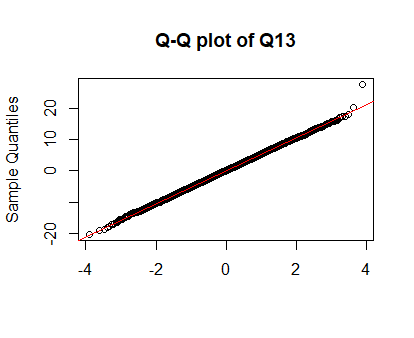}
    \includegraphics[width=0.3\columnwidth]{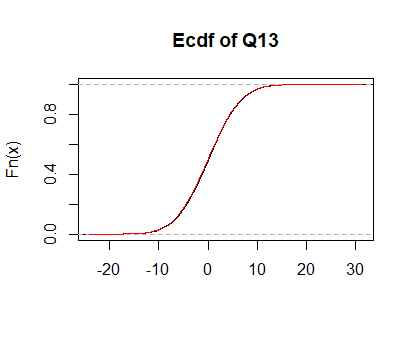}
    \\
    \includegraphics[width=0.3\columnwidth]{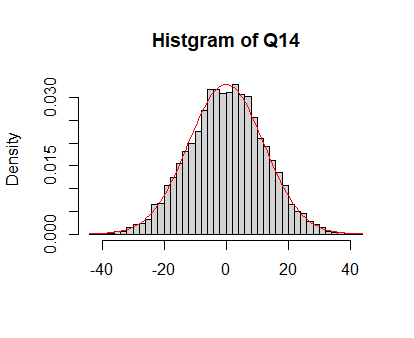}
    \includegraphics[width=0.3\columnwidth]{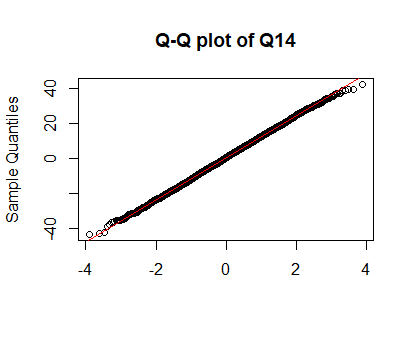}
    \includegraphics[width=0.3\columnwidth]{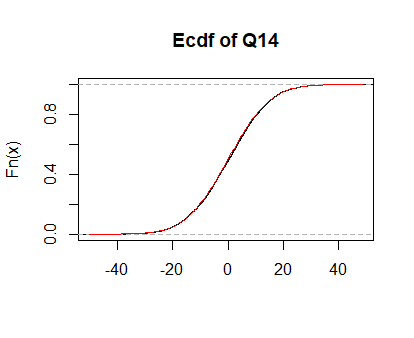}
    \\ 
    \includegraphics[width=0.3\columnwidth]{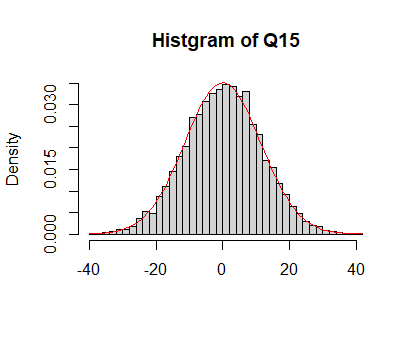}
    \includegraphics[width=0.3\columnwidth]{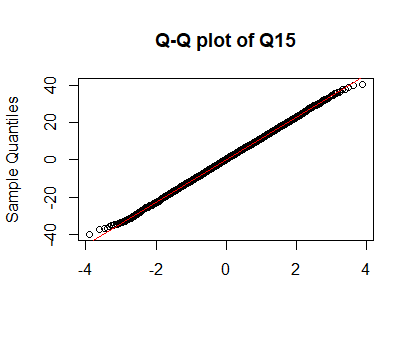}
    \includegraphics[width=0.3\columnwidth]{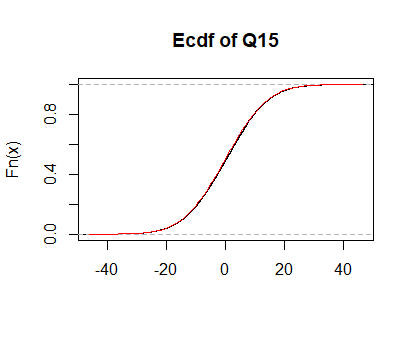}
\end{figure}
\newpage
\begin{figure}[h]
    \ \\ \ \\ \ \\
    \includegraphics[width=0.3\columnwidth]{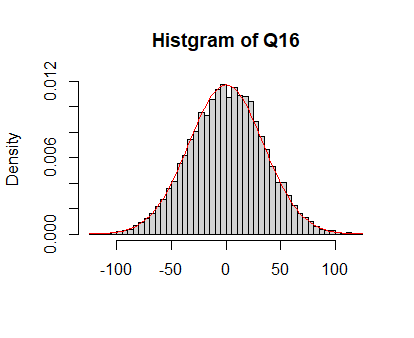}
    \includegraphics[width=0.3\columnwidth]{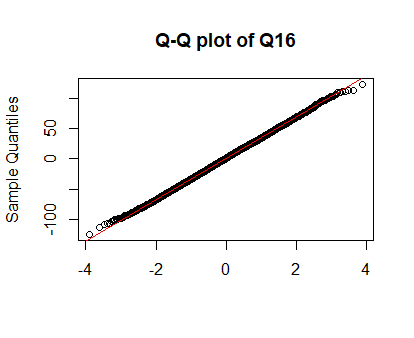}
    \includegraphics[width=0.3\columnwidth]{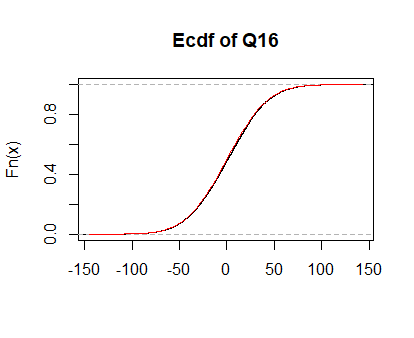}
    \\
    \includegraphics[width=0.3\columnwidth]{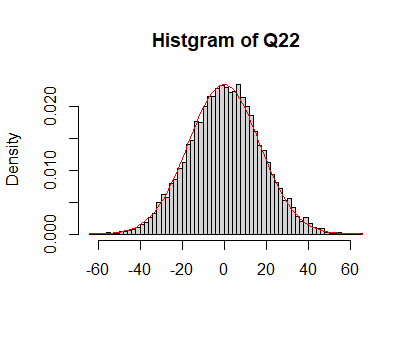}
    \includegraphics[width=0.3\columnwidth]{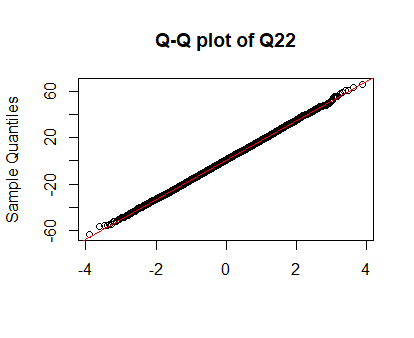}
    \includegraphics[width=0.3\columnwidth]{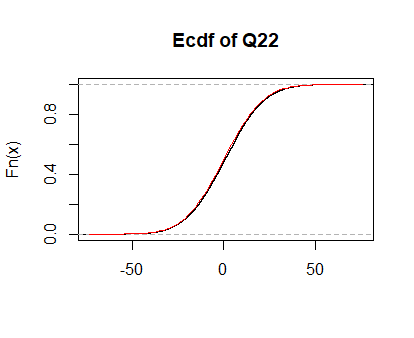}
    \\ 
    \includegraphics[width=0.3\columnwidth]{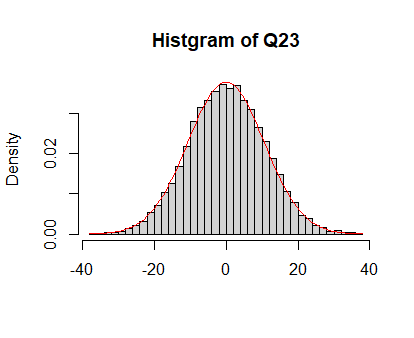}
    \includegraphics[width=0.3\columnwidth]{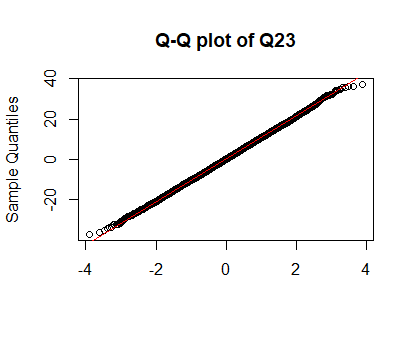}
    \includegraphics[width=0.3\columnwidth]{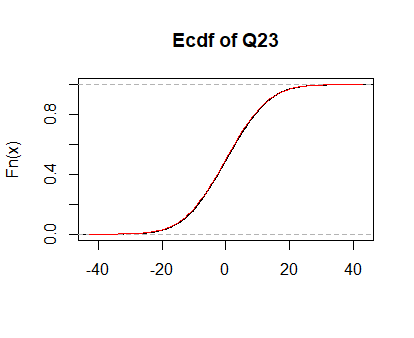}
    \\ 
    \includegraphics[width=0.3\columnwidth]{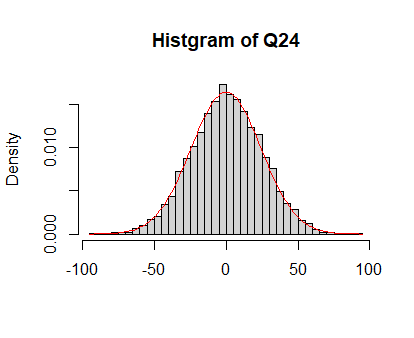}
    \includegraphics[width=0.3\columnwidth]{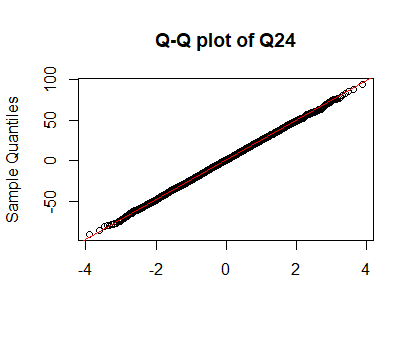}
    \includegraphics[width=0.3\columnwidth]{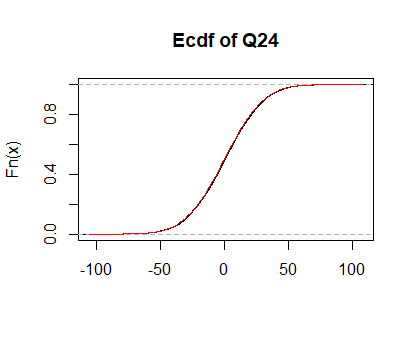}
    \\
    \includegraphics[width=0.3\columnwidth]{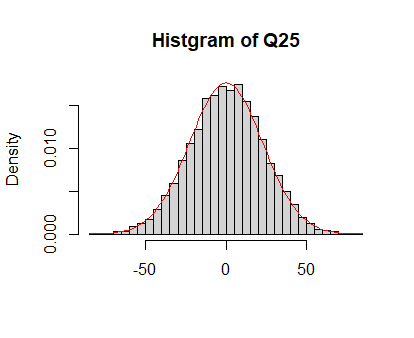}
    \includegraphics[width=0.3\columnwidth]{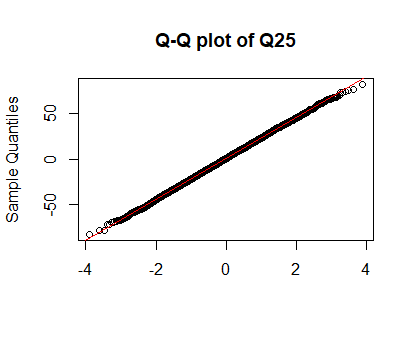}
    \includegraphics[width=0.3\columnwidth]{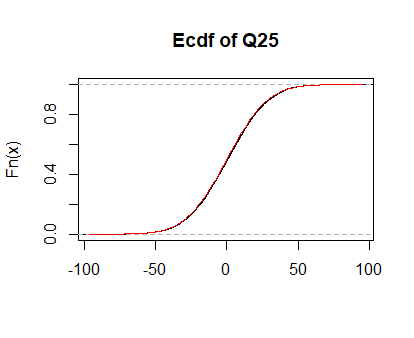}
    \\
\end{figure}
\newpage
\begin{figure}[h]
    \ \\ \ \\ \ \\
    \includegraphics[width=0.3\columnwidth]{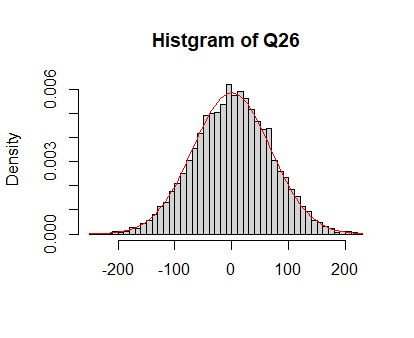}
    \includegraphics[width=0.3\columnwidth]{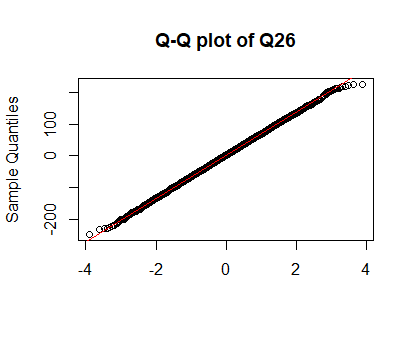}
    \includegraphics[width=0.3\columnwidth]{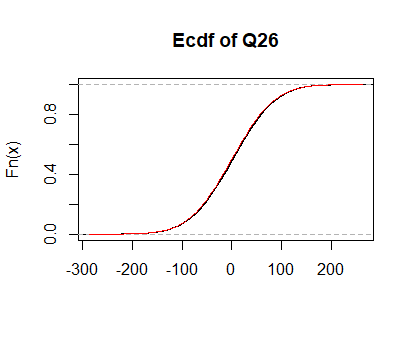}
    \\ 
    \includegraphics[width=0.3\columnwidth]{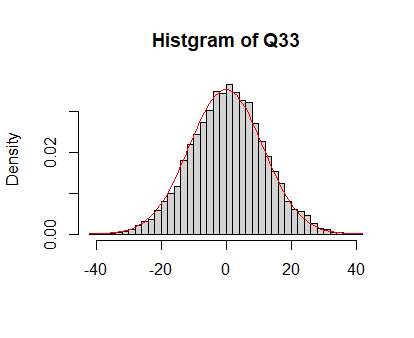}
    \includegraphics[width=0.3\columnwidth]{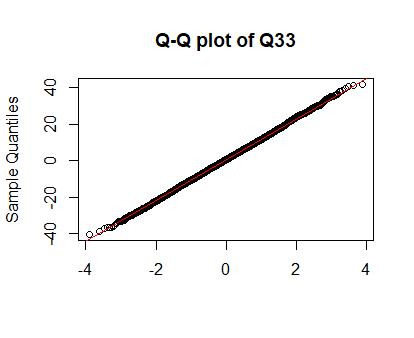}
    \includegraphics[width=0.3\columnwidth]{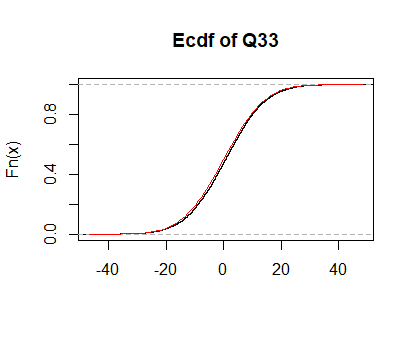}
    \\ 
    \includegraphics[width=0.3\columnwidth]{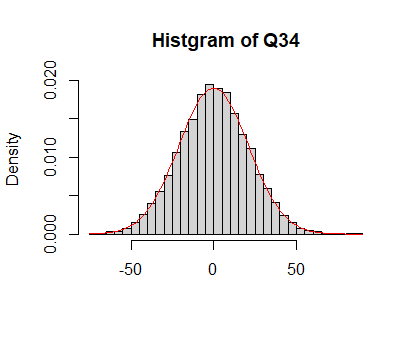}
    \includegraphics[width=0.3\columnwidth]{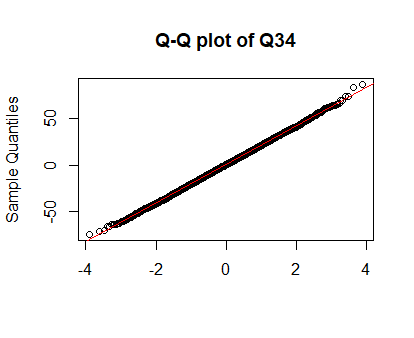}
    \includegraphics[width=0.3\columnwidth]{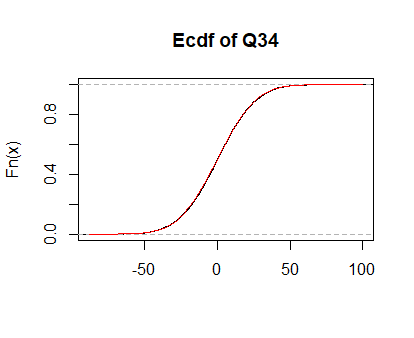}
    \\
    \includegraphics[width=0.3\columnwidth]{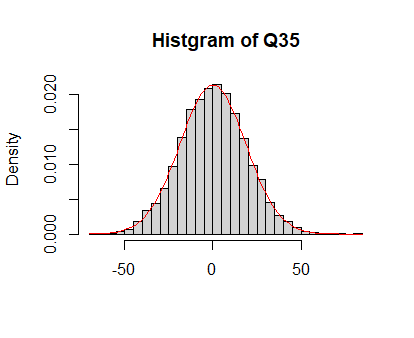}
    \includegraphics[width=0.3\columnwidth]{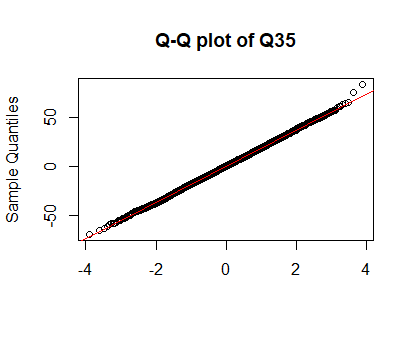}
    \includegraphics[width=0.3\columnwidth]{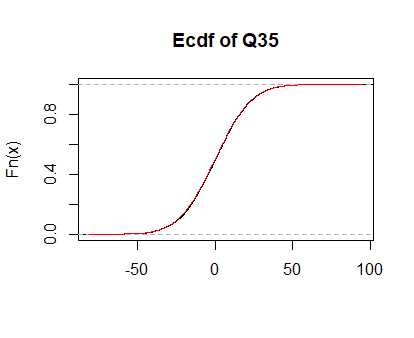}
    \\ 
    \includegraphics[width=0.3\columnwidth]{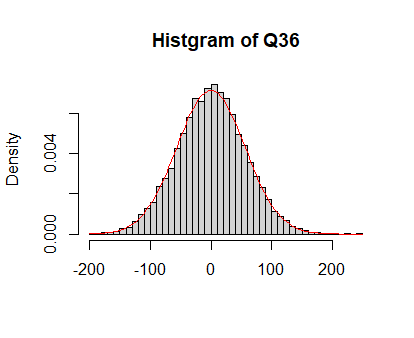}
    \includegraphics[width=0.3\columnwidth]{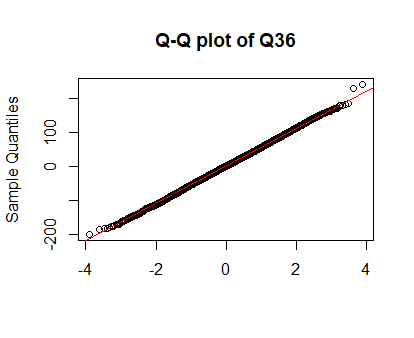}
    \includegraphics[width=0.3\columnwidth]{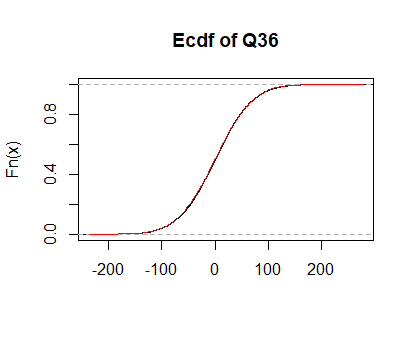}
    \\
\end{figure}
\newpage
\begin{figure}[h]
    \ \\ \ \\ \ \\
    \includegraphics[width=0.3\columnwidth]{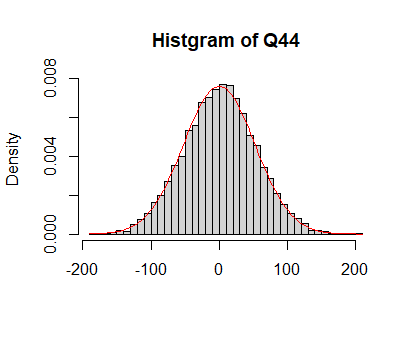}
    \includegraphics[width=0.3\columnwidth]{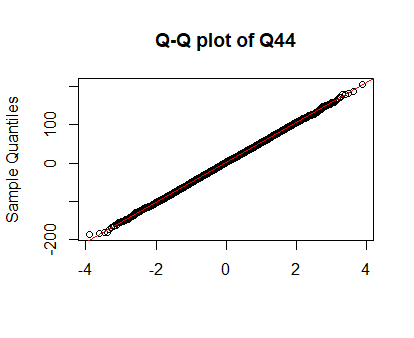}
    \includegraphics[width=0.3\columnwidth]{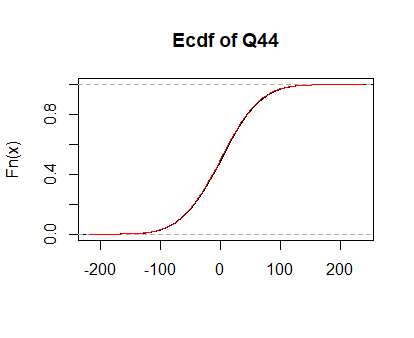}
    \\ 
    \includegraphics[width=0.3\columnwidth]{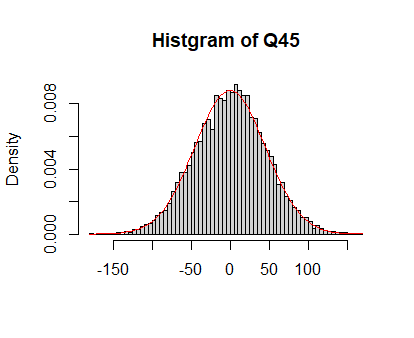}
    \includegraphics[width=0.3\columnwidth]{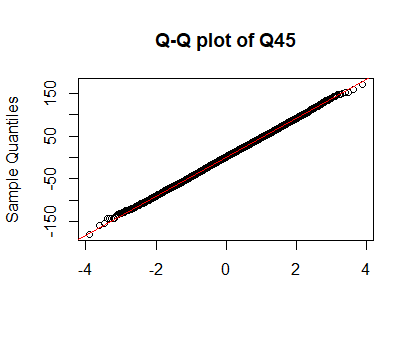}
    \includegraphics[width=0.3\columnwidth]{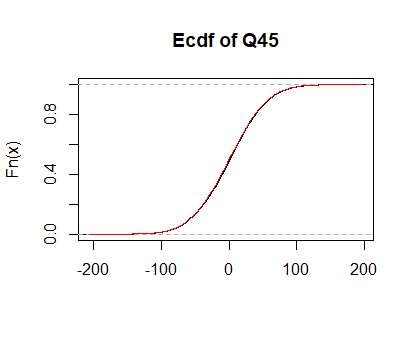}
    \\
    \includegraphics[width=0.3\columnwidth]{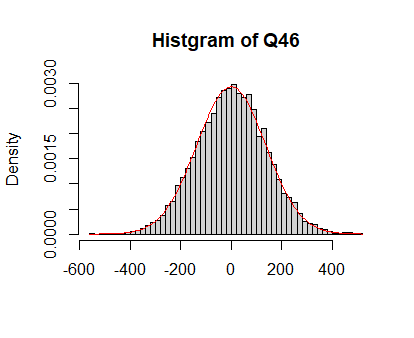}
    \includegraphics[width=0.3\columnwidth]{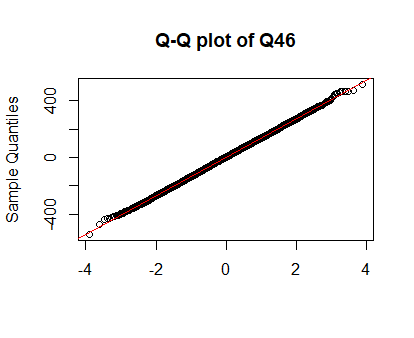}
    \includegraphics[width=0.3\columnwidth]{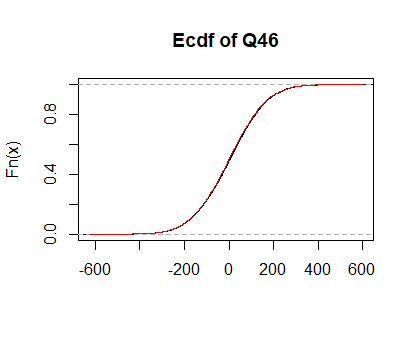}
    \\
    \includegraphics[width=0.3\columnwidth]{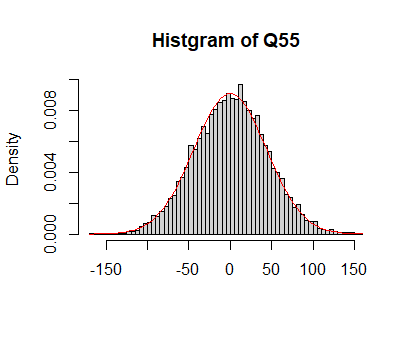}
    \includegraphics[width=0.3\columnwidth]{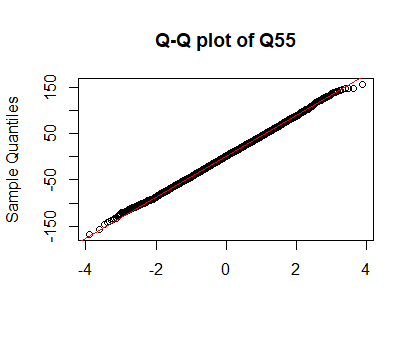}
    \includegraphics[width=0.3\columnwidth]{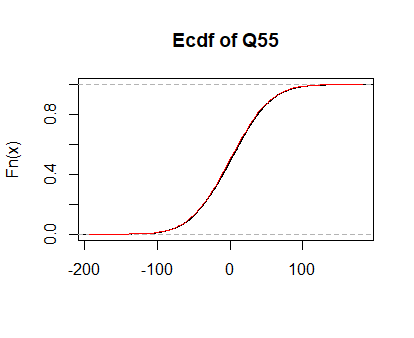}
    \\
    \includegraphics[width=0.3\columnwidth]{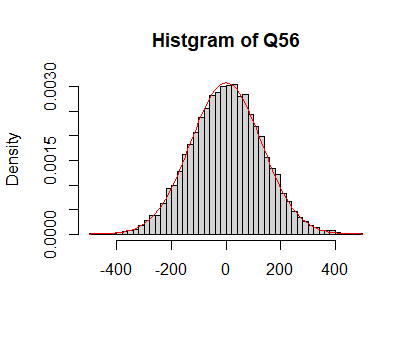}
    \includegraphics[width=0.3\columnwidth]{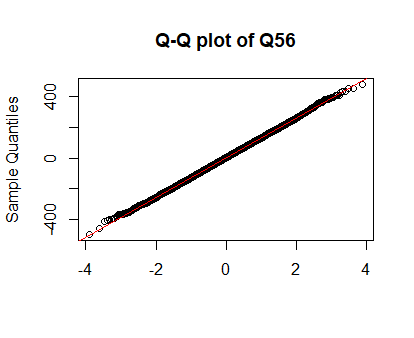}
    \includegraphics[width=0.3\columnwidth]{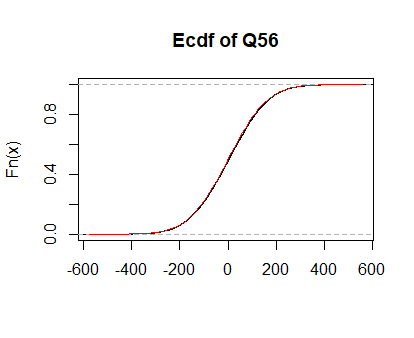}
    \\
\end{figure}
\newpage
\begin{figure}[h]
    \ \\ \ \\ \ \\
    \includegraphics[width=0.3\columnwidth]{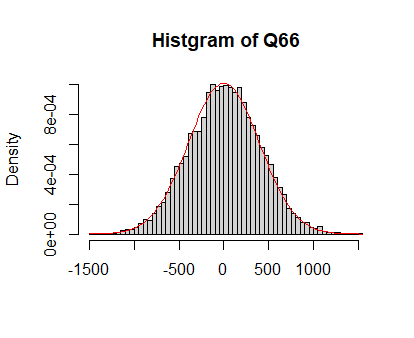}
    \includegraphics[width=0.3\columnwidth]{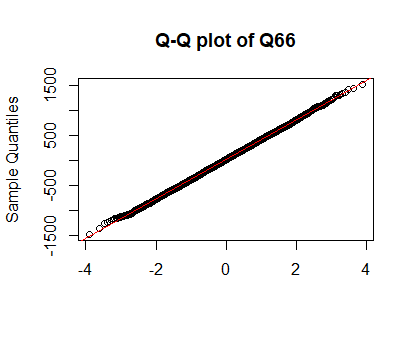}
    \includegraphics[width=0.3\columnwidth]{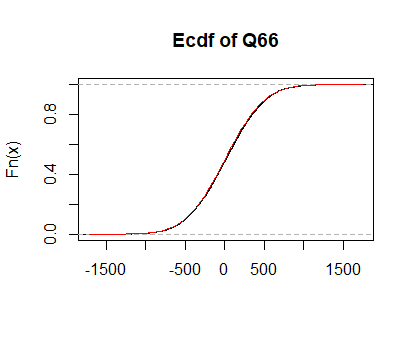}
    \caption{Histograms (left), Q-Q plots (middle) and empirical distributions (right) of $\sqrt{n}((\mathbb{Q}_{\mathbb{XX}})_{ij}-({\bf{\Sigma}}_0)_{ij})$ for $i\leq j$ and $i,j=1,\cdots,6$. The red lines are theoretical curves.}\label{Qfigureer}
\end{figure}
\newpage
\begin{figure}[h]
    \ \\ \ \\ \ \\
    \includegraphics[width=0.3\columnwidth]{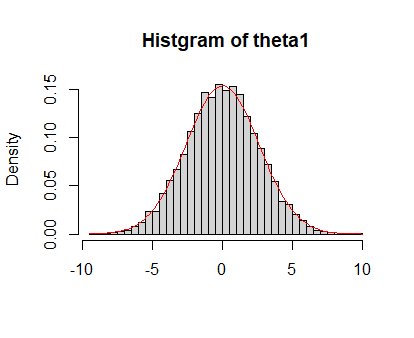}
    \includegraphics[width=0.3\columnwidth]{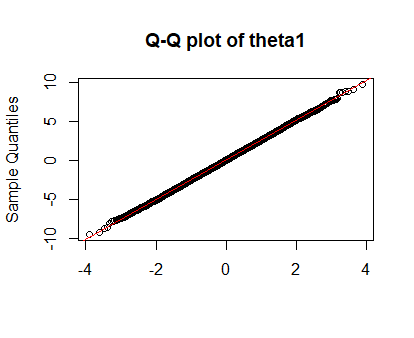}
    \includegraphics[width=0.3\columnwidth]{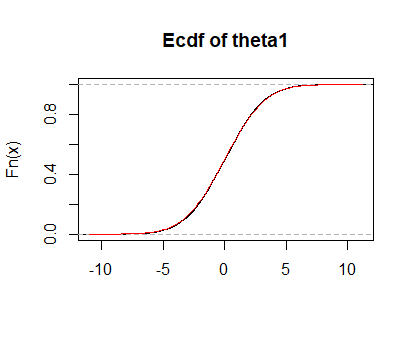}
    \\
    \includegraphics[width=0.3\columnwidth]{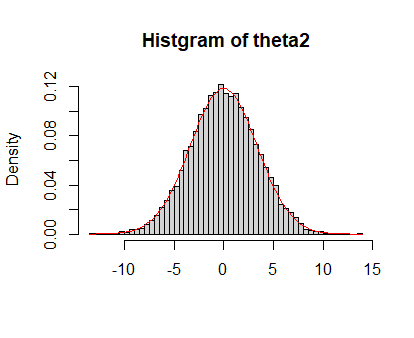}
    \includegraphics[width=0.3\columnwidth]{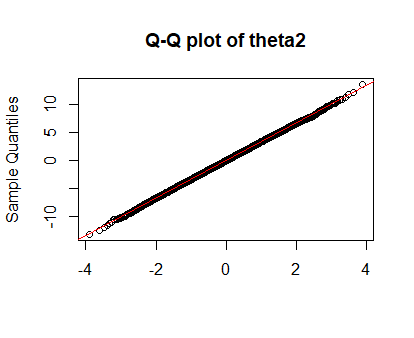}
    \includegraphics[width=0.3\columnwidth]{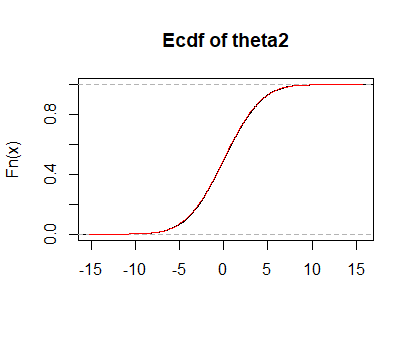}
    \\ 
    \includegraphics[width=0.3\columnwidth]{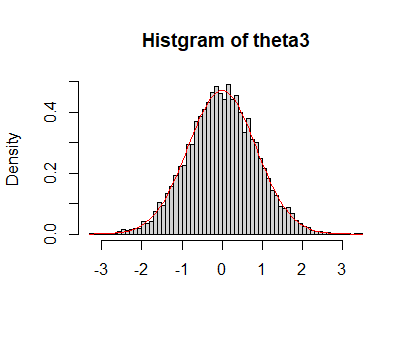}
    \includegraphics[width=0.3\columnwidth]{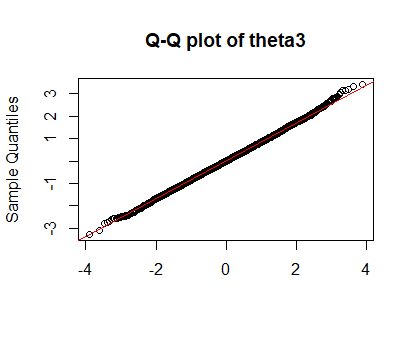}
    \includegraphics[width=0.3\columnwidth]{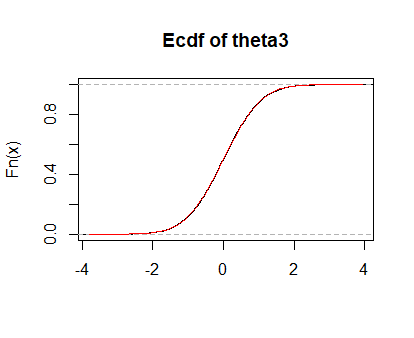}
    \\
    \includegraphics[width=0.3\columnwidth]{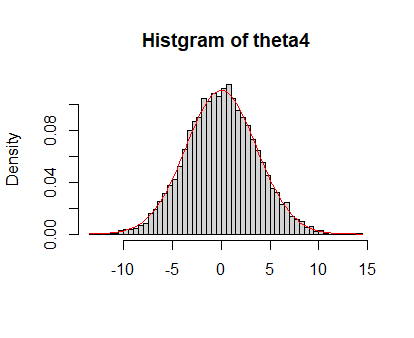}
    \includegraphics[width=0.3\columnwidth]{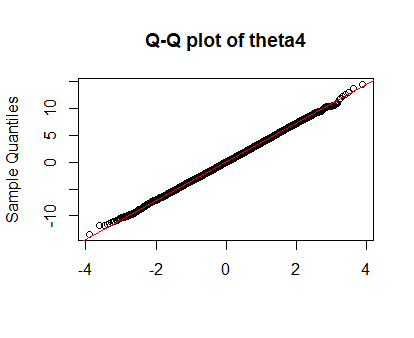}
    \includegraphics[width=0.3\columnwidth]{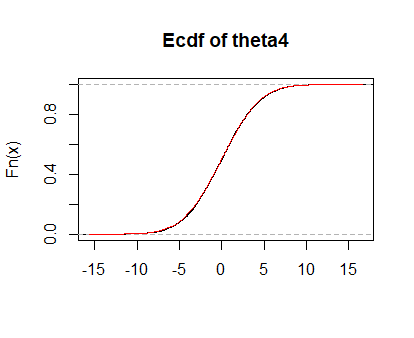}
    \\
    \includegraphics[width=0.3\columnwidth]{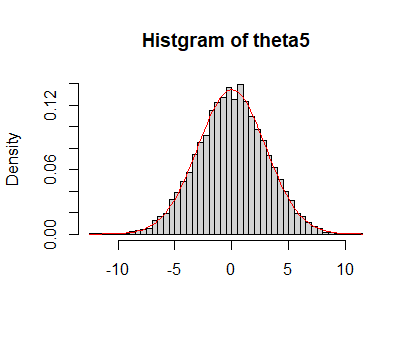}
    \includegraphics[width=0.3\columnwidth]{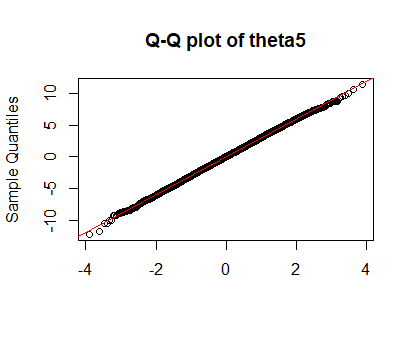}
    \includegraphics[width=0.3\columnwidth]{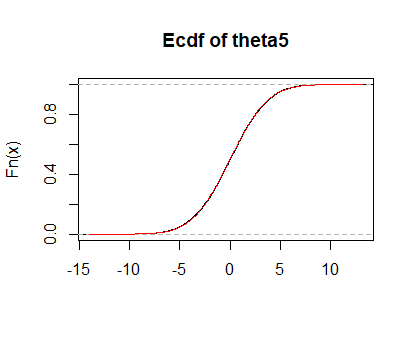}
\end{figure}
\newpage
\begin{figure}[h]
    \ \\ \ \\ \ \\
    \includegraphics[width=0.3\columnwidth]{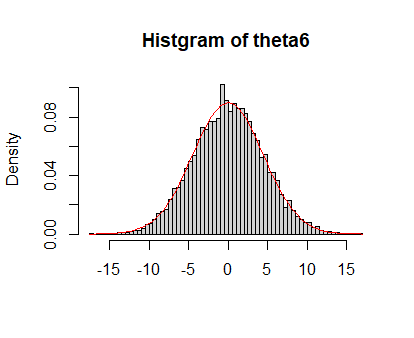}
    \includegraphics[width=0.3\columnwidth]{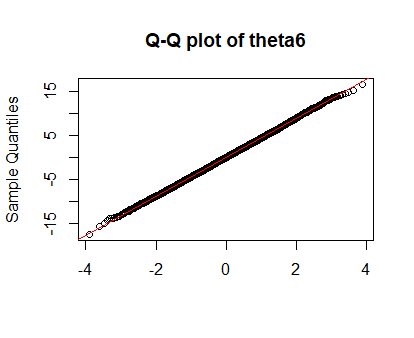}
    \includegraphics[width=0.3\columnwidth]{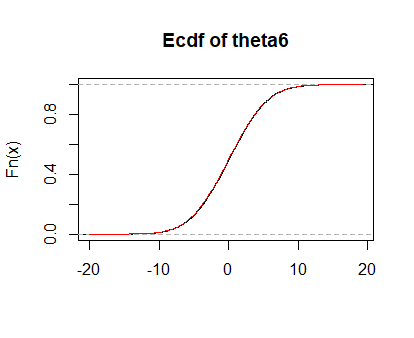}
    \\ 
    \includegraphics[width=0.3\columnwidth]{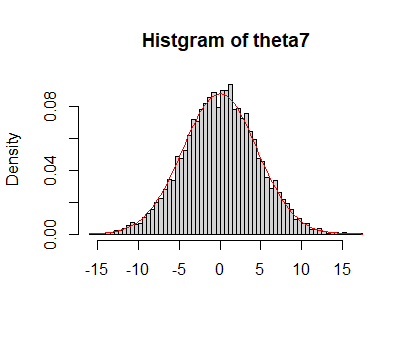}
    \includegraphics[width=0.3\columnwidth]{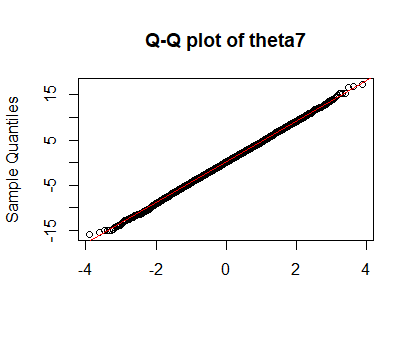}
    \includegraphics[width=0.3\columnwidth]{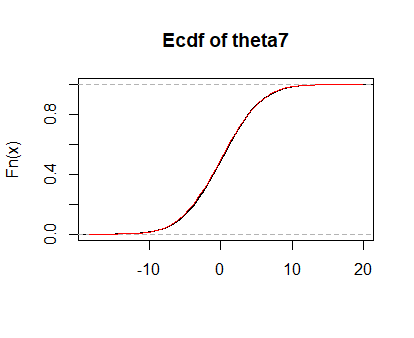}
    \\ 
    \includegraphics[width=0.3\columnwidth]{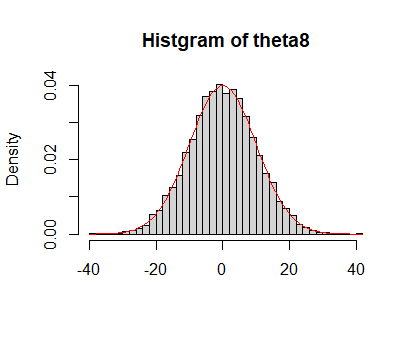}
    \includegraphics[width=0.3\columnwidth]{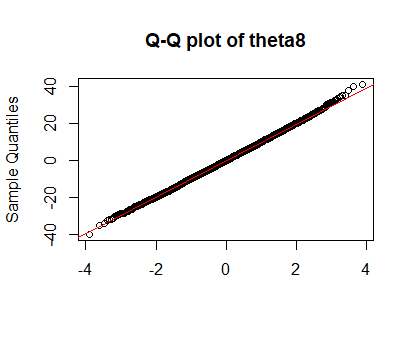}
    \includegraphics[width=0.3\columnwidth]{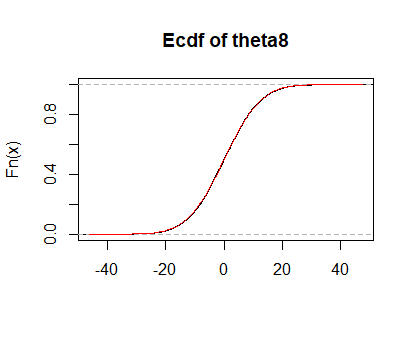}
    \\ 
    \includegraphics[width=0.3\columnwidth]{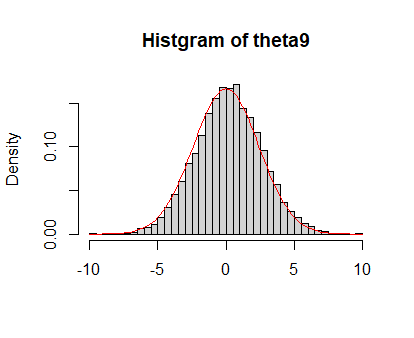}
    \includegraphics[width=0.3\columnwidth]{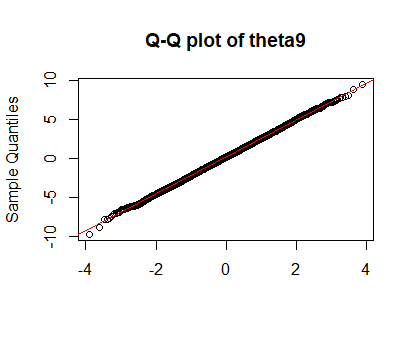}
    \includegraphics[width=0.3\columnwidth]{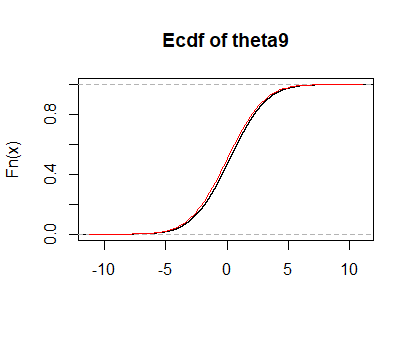}
    \\
    \includegraphics[width=0.3\columnwidth]{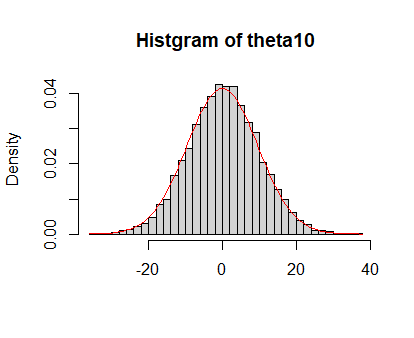}
    \includegraphics[width=0.3\columnwidth]{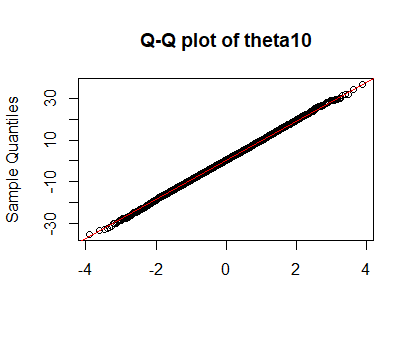}
    \includegraphics[width=0.3\columnwidth]{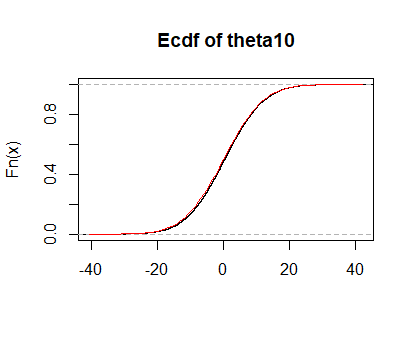}
\end{figure}
\newpage
\begin{figure}[h]
    \ \\ \ \\
    \includegraphics[width=0.3\columnwidth]{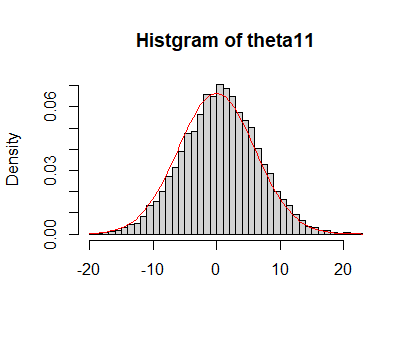}
    \includegraphics[width=0.3\columnwidth]{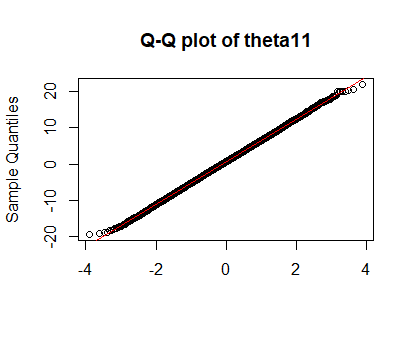}
    \includegraphics[width=0.3\columnwidth]{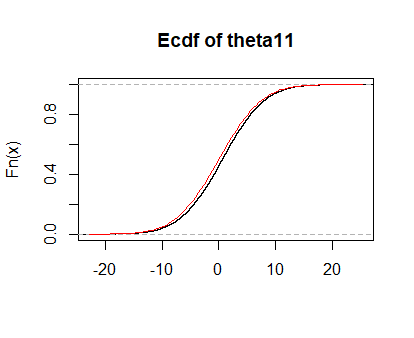}
    \\
    \includegraphics[width=0.3\columnwidth]{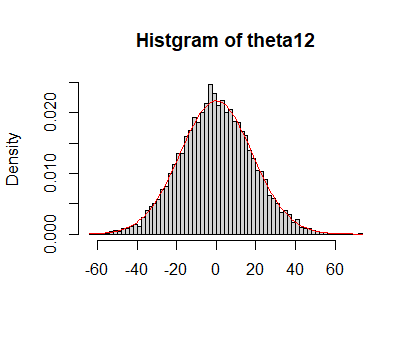}
    \includegraphics[width=0.3\columnwidth]{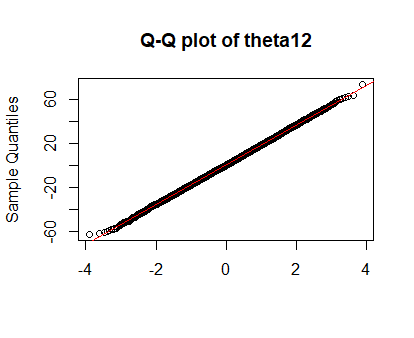}
    \includegraphics[width=0.3\columnwidth]{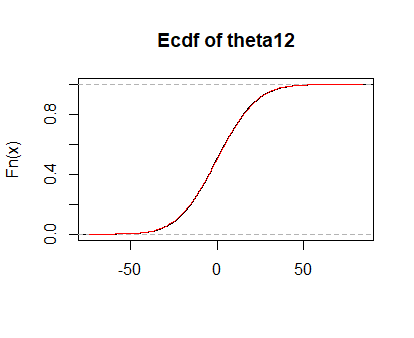}
    \\
    \includegraphics[width=0.3\columnwidth]{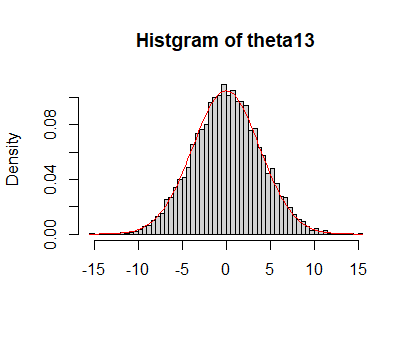}
    \includegraphics[width=0.3\columnwidth]{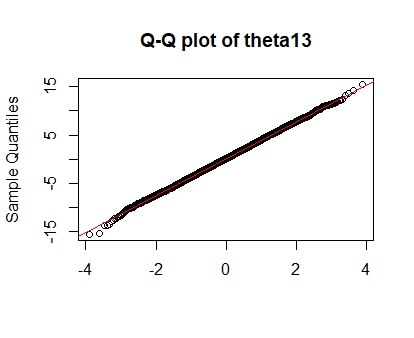}
    \includegraphics[width=0.3\columnwidth]{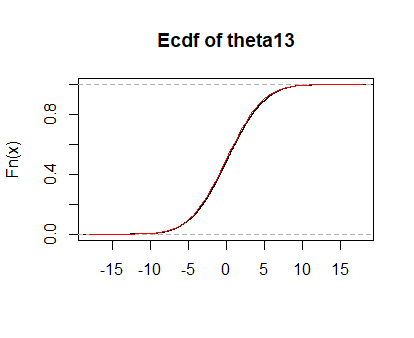}
    \\
    \includegraphics[width=0.3\columnwidth]{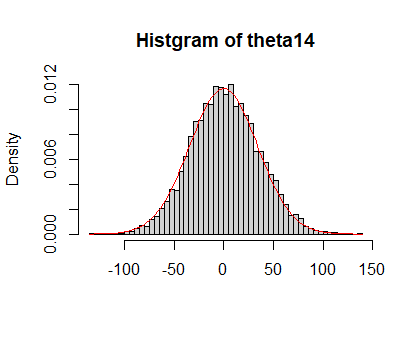}
    \includegraphics[width=0.3\columnwidth]{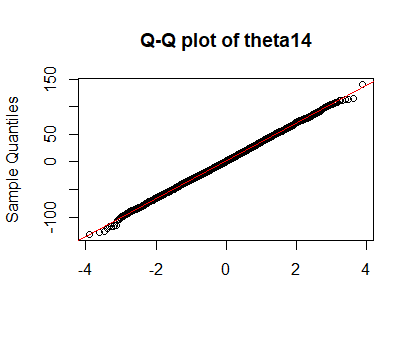}
    \includegraphics[width=0.3\columnwidth]{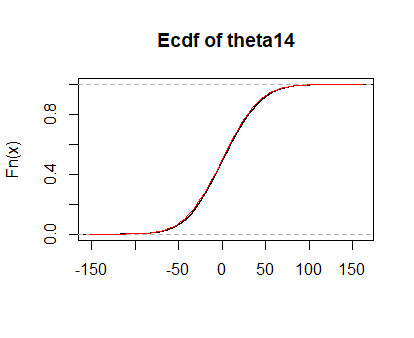}
    \\
    \includegraphics[width=0.3\columnwidth]{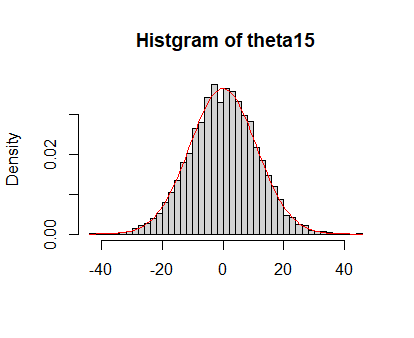}
    \includegraphics[width=0.3\columnwidth]{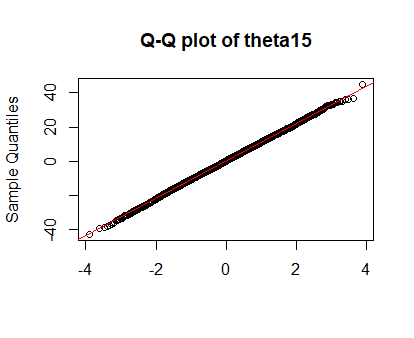}
    \includegraphics[width=0.3\columnwidth]{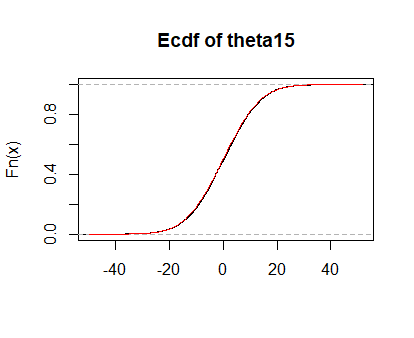}
    \caption{Histograms (left), Q-Q plots (middle) and empirical distributions (right) of $\sqrt{n}(\hat{\theta}_{n}^{(i)}-\theta_{0}^{(i)})$ for $i=1,\cdots,15$. The red lines are theoretical curves.}\label{thetafigureer}
\end{figure}
\clearpage
\ \\
\ \\
\ \\
\begin{table}[h]
    \centering
    \begin{tabular}{cc}
    \hline
    Mean\ \  (True value) &\quad  5.999\ \ (6.000)\\
    SD\ \ (Theoretical value) &\quad 3.449\ \ (3.464)\\ \\
    \end{tabular}
\caption{Sample mean and sample standard deviation (SD) of the test statistic $\mathbb{T}_{n}$.}
\label{testtable1er}
\end{table}
\ \\
\ \\
\begin{figure}[h]
    \centering
    \includegraphics[width=0.3\columnwidth]{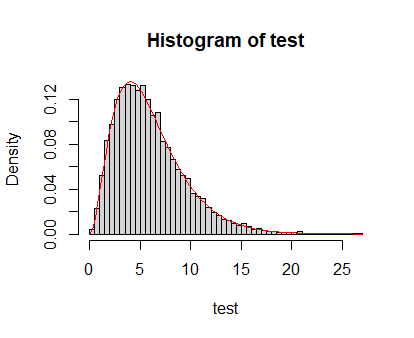}
    \includegraphics[width=0.3\columnwidth]{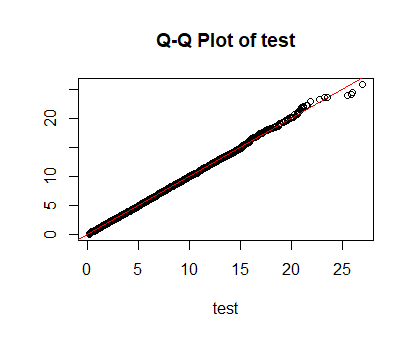}
    \includegraphics[width=0.3\columnwidth]{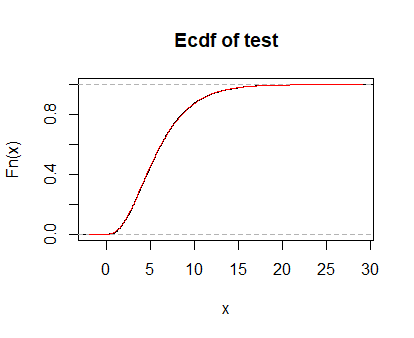}
\caption{Histogram (left), Q-Q plot (middle) and empirical distribution (right) of the test statistic $\mathbb{T}_{n}$. The red lines are theoretical curves.}
\label{testfigure1er}
\end{figure}
\ \\
\ \\
\begin{table}[h]
    \centering
    \begin{tabular}{cc}
    \hline
    Model A &\quad 10000\\
    Model B &\quad 10000\\ \ \\
    \end{tabular}
\caption{The number of rejections of the quasi-likelihood ratio test in Model A and Model B.}\label{testtable2er}
\end{table}
\begin{table}[h]
    \centering
    \begin{tabular}{cccccc}
    &\quad Min &\  $Q1$ &\  Median &\  $Q3$ &\  Max\\\hline
    Model A &\quad 274634 &\ 277845 &\ 278548 &\ 279250 &\ 282330  \\
    Model B &\quad 208785 &\ 211270 &\ 211870 &\ 212435 &\ 215120 \\ \\
    \end{tabular}
\caption{Quartile of the test statistic $\mathbb{T}_{n}$ in Model A and Model B.}
\label{testtable3er}
\end{table}
\clearpage
\subsection{Details of simulation results in Section 5}\label{simulation}
Figure \ref{Qfigurenon2} shows histograms, Q-Q plots and empirical distributions of $\sqrt{n}((\mathbb{Q}_{\mathbb{XX}})_{ij}-({\bf{\Sigma}}_0)_{ij})$ for $i\leq j$ and $i,j=1,\cdots,6$. Figure \ref{thetafigurenon2} shows histograms, Q-Q plots and empirical distributions of $\sqrt{n}(\hat{\theta}_{n}^{(i)}-\theta_{0}^{(i)})$ for $i=1,\cdots,15$.
\begin{figure}[h]
    \ \\ \ \\ \ \\ \ \\
    \includegraphics[width=0.3\columnwidth]{./files/histQ11.png}
    \includegraphics[width=0.3\columnwidth]{./files/QQQ11.png}
    \includegraphics[width=0.3\columnwidth]{./files/ecdfQ11.png}
    \\ 
    \includegraphics[width=0.3\columnwidth]{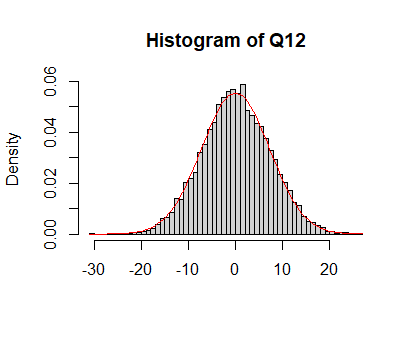}
    \includegraphics[width=0.3\columnwidth]{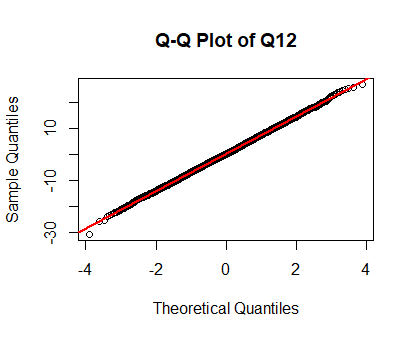}
    \includegraphics[width=0.3\columnwidth]{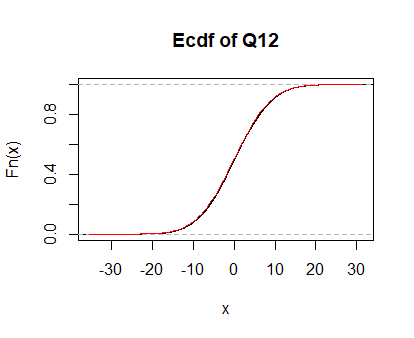}
    \\ 
    \includegraphics[width=0.3\columnwidth]{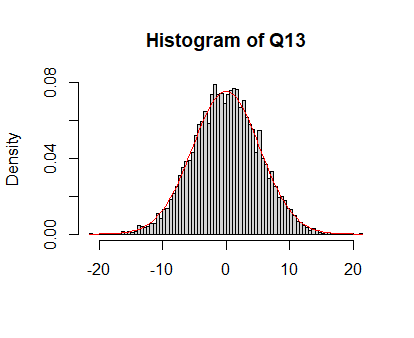}
    \includegraphics[width=0.3\columnwidth]{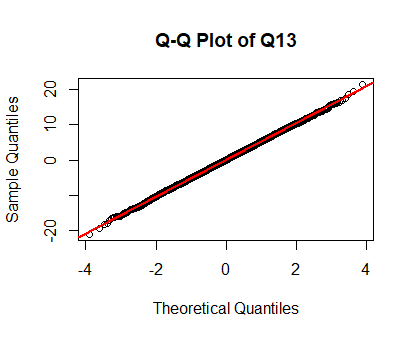}
    \includegraphics[width=0.3\columnwidth]{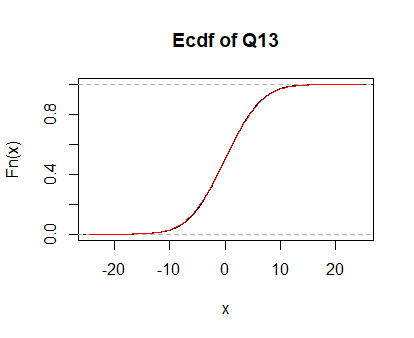}
    \\
    \includegraphics[width=0.3\columnwidth]{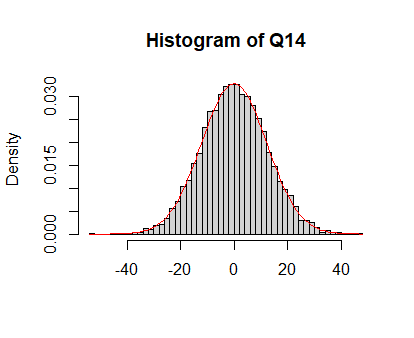}
    \includegraphics[width=0.3\columnwidth]{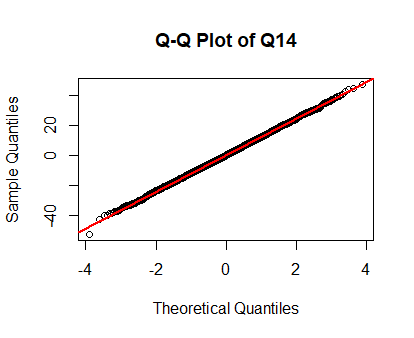}
    \includegraphics[width=0.3\columnwidth]{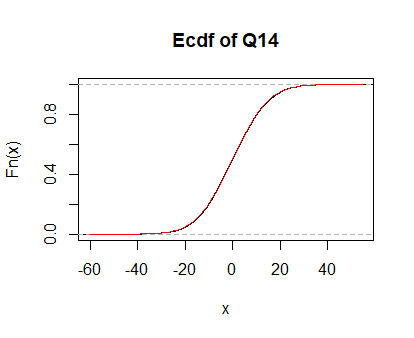}
\end{figure}
\newpage
\begin{figure}[h]
    \ \\ \ \\ \ \\
    \includegraphics[width=0.3\columnwidth]{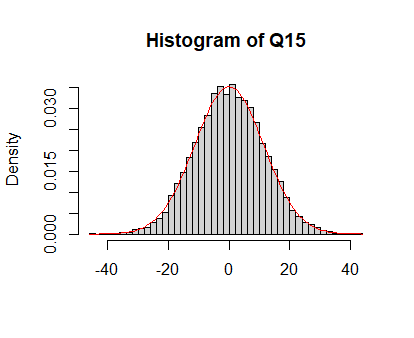}
    \includegraphics[width=0.3\columnwidth]{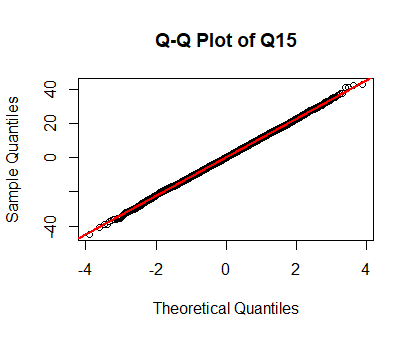}
    \includegraphics[width=0.3\columnwidth]{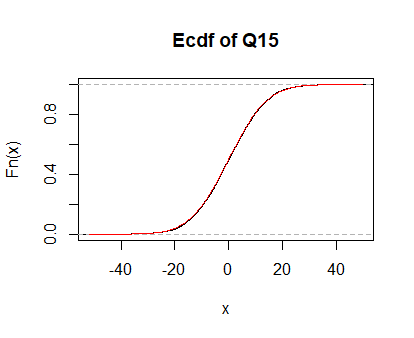}
    \\
    \includegraphics[width=0.3\columnwidth]{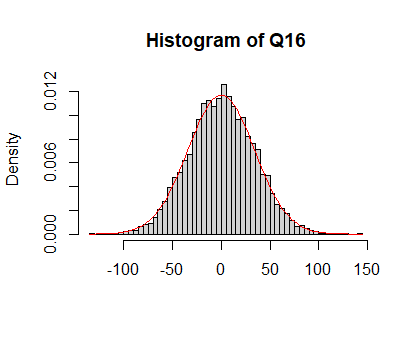}
    \includegraphics[width=0.3\columnwidth]{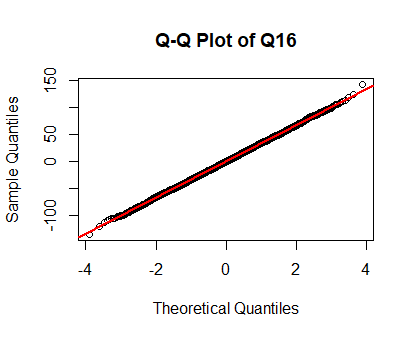}
    \includegraphics[width=0.3\columnwidth]{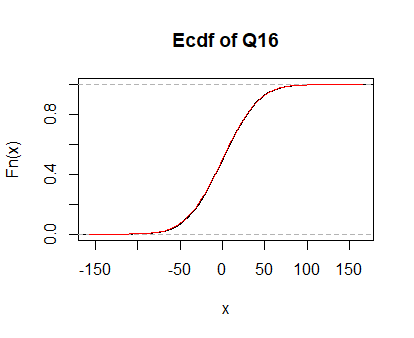}
    \\
    \includegraphics[width=0.3\columnwidth]{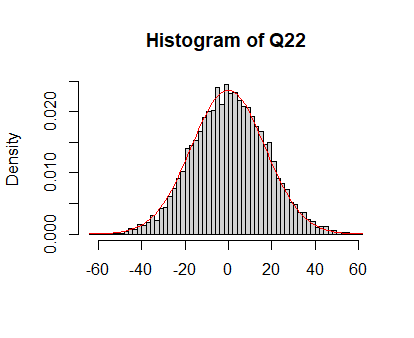}
    \includegraphics[width=0.3\columnwidth]{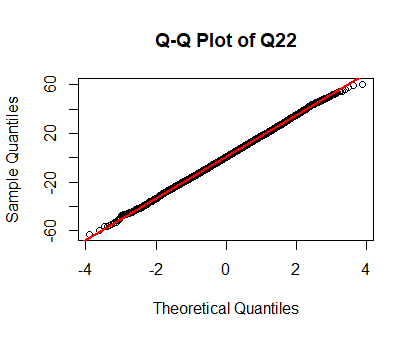}
    \includegraphics[width=0.3\columnwidth]{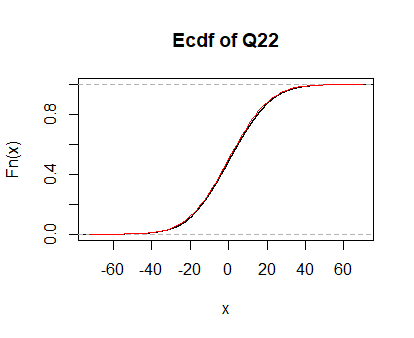}
    \\ 
    \includegraphics[width=0.3\columnwidth]{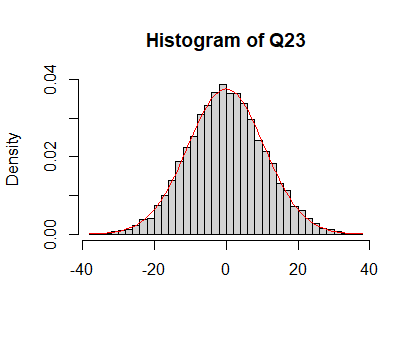}
    \includegraphics[width=0.3\columnwidth]{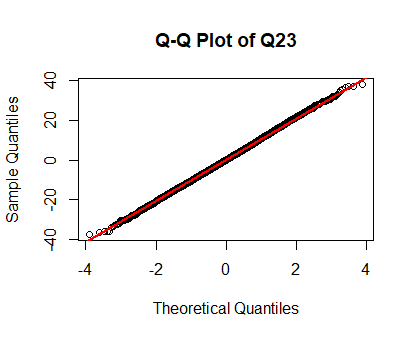}
    \includegraphics[width=0.3\columnwidth]{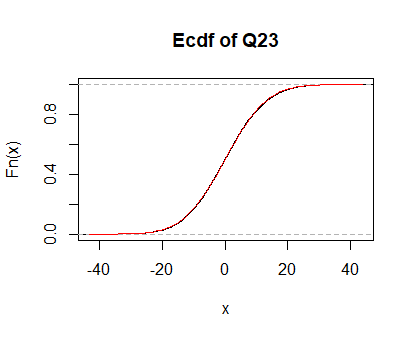}
    \\
    \includegraphics[width=0.3\columnwidth]{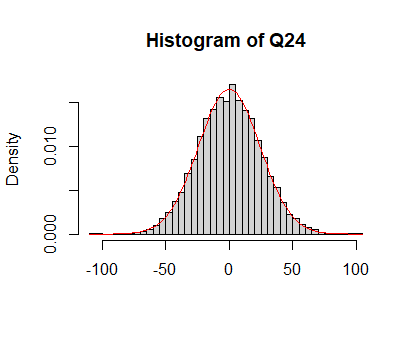}
    \includegraphics[width=0.3\columnwidth]{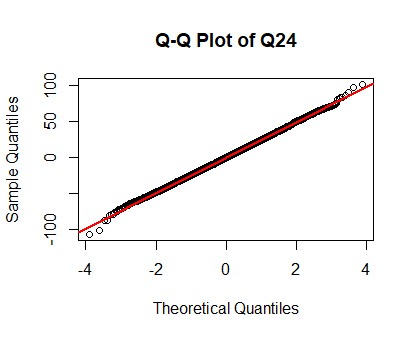}
    \includegraphics[width=0.3\columnwidth]{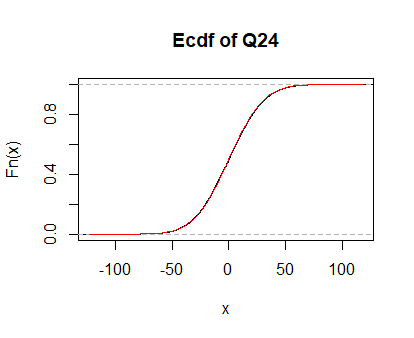}
\end{figure}
\newpage
\begin{figure}[h]
    \ \\ \ \\ \ \\
    \includegraphics[width=0.3\columnwidth]{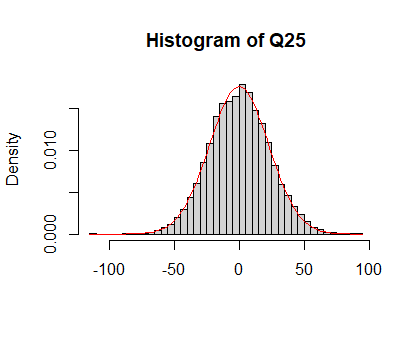}
    \includegraphics[width=0.3\columnwidth]{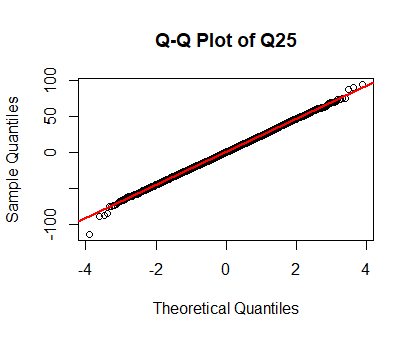}
    \includegraphics[width=0.3\columnwidth]{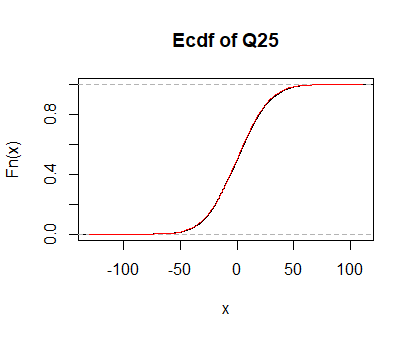}
    \\
    \includegraphics[width=0.3\columnwidth]{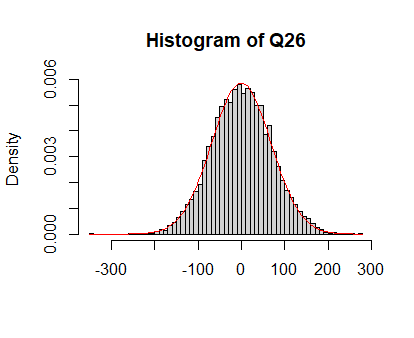}
    \includegraphics[width=0.3\columnwidth]{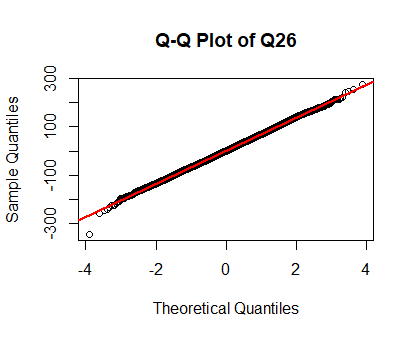}
    \includegraphics[width=0.3\columnwidth]{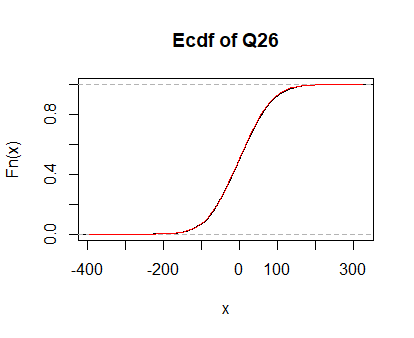}
    \\
    \includegraphics[width=0.3\columnwidth]{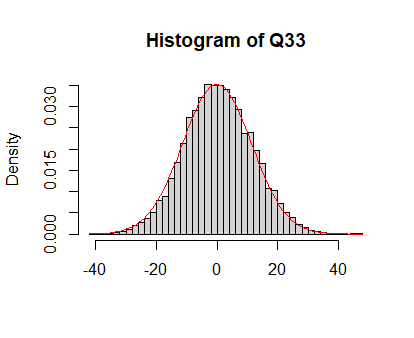}
    \includegraphics[width=0.3\columnwidth]{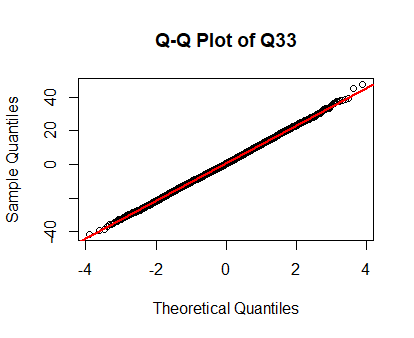}
    \includegraphics[width=0.3\columnwidth]{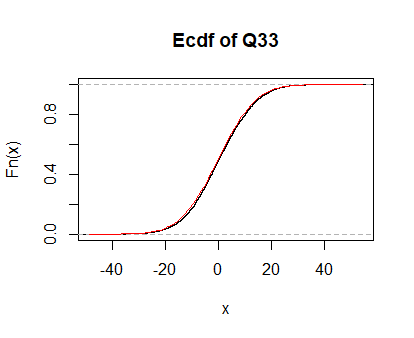}
    \\ 
    \includegraphics[width=0.3\columnwidth]{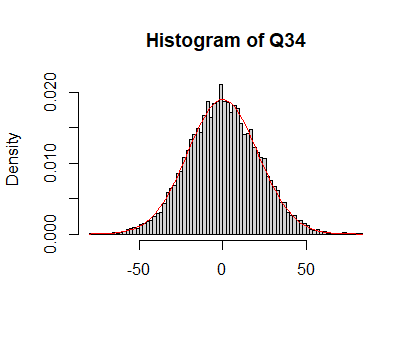}
    \includegraphics[width=0.3\columnwidth]{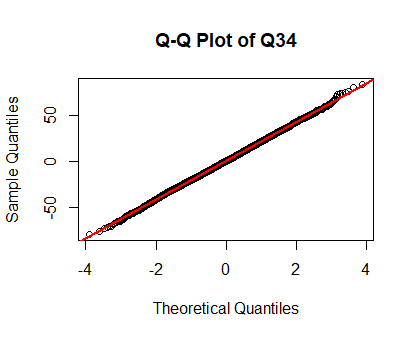}
    \includegraphics[width=0.3\columnwidth]{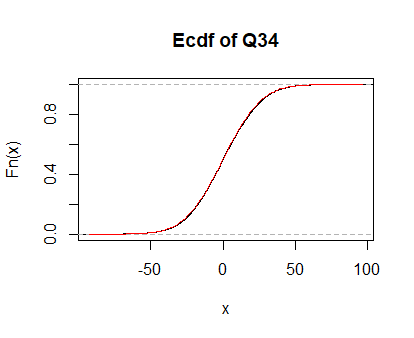}
    \\
    \includegraphics[width=0.3\columnwidth]{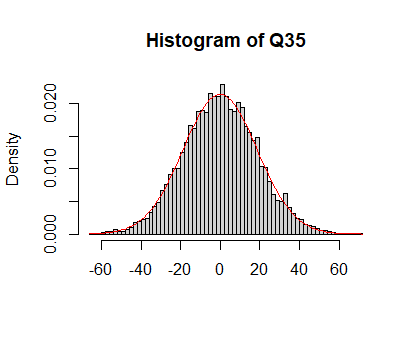}
    \includegraphics[width=0.3\columnwidth]{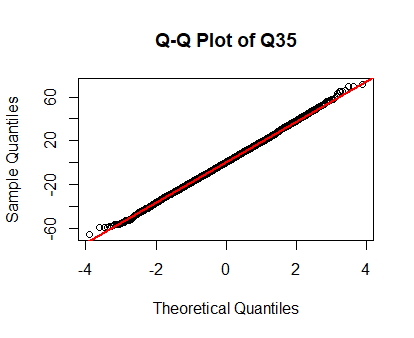}
    \includegraphics[width=0.3\columnwidth]{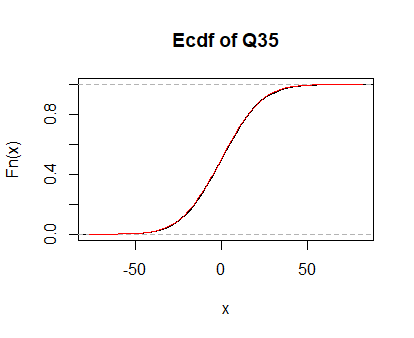}
\end{figure}
\newpage
\begin{figure}[h]
    \ \\ \ \\ \ \\
    \includegraphics[width=0.3\columnwidth]{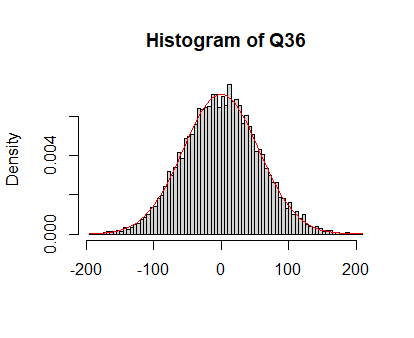}
    \includegraphics[width=0.3\columnwidth]{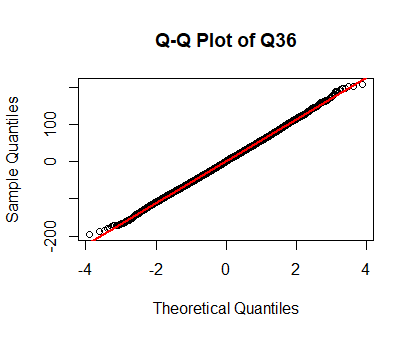}
    \includegraphics[width=0.3\columnwidth]{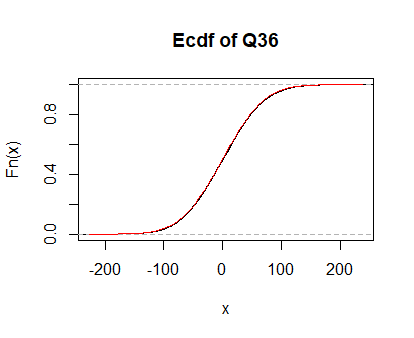}
    \\
    \includegraphics[width=0.3\columnwidth]{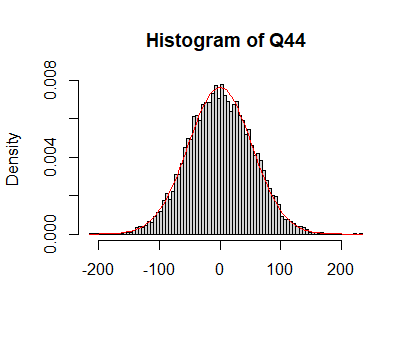}
    \includegraphics[width=0.3\columnwidth]{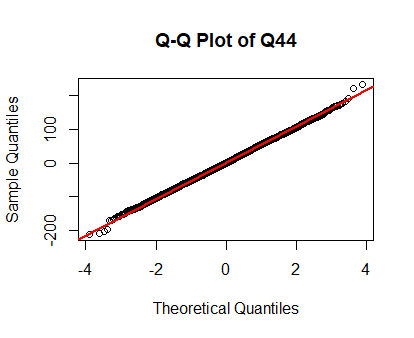}
    \includegraphics[width=0.3\columnwidth]{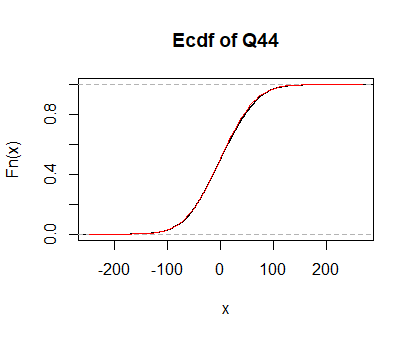}
    \\
    \includegraphics[width=0.3\columnwidth]{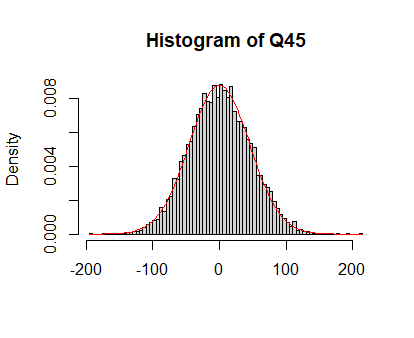}
    \includegraphics[width=0.3\columnwidth]{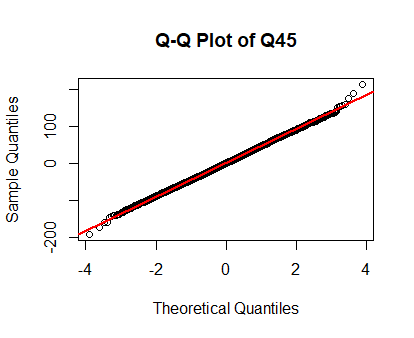}
    \includegraphics[width=0.3\columnwidth]{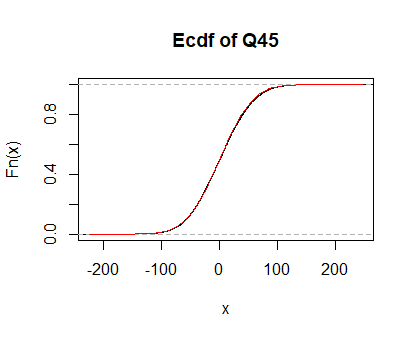}
    \\
    \includegraphics[width=0.3\columnwidth]{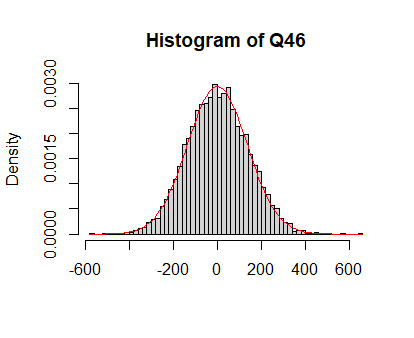}
    \includegraphics[width=0.3\columnwidth]{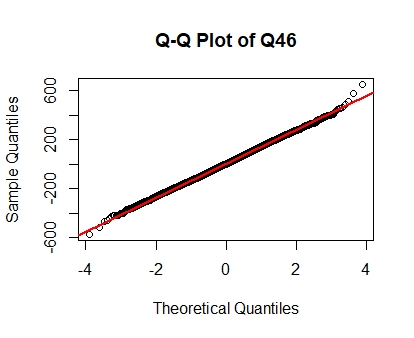}
    \includegraphics[width=0.3\columnwidth]{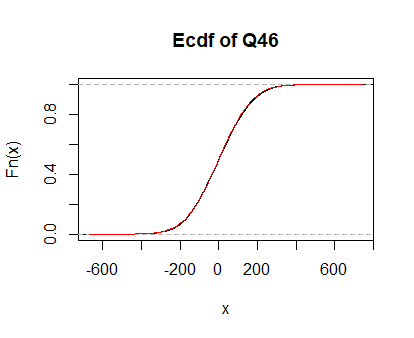}
    \\
    \includegraphics[width=0.3\columnwidth]{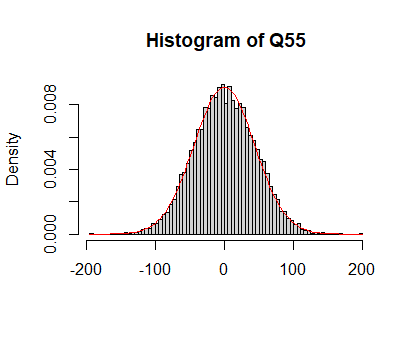}
    \includegraphics[width=0.3\columnwidth]{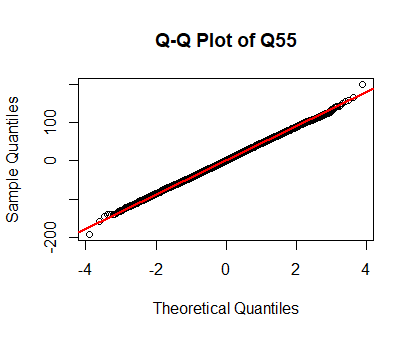}
    \includegraphics[width=0.3\columnwidth]{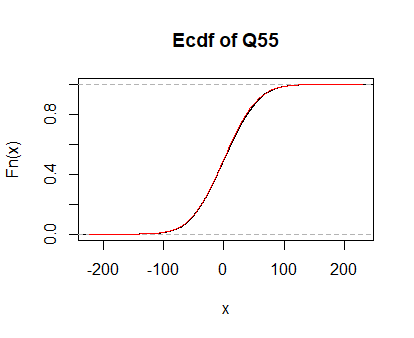}
\end{figure}
\newpage
\begin{figure}[h]
    \ \\ \ \\ \ \\
    \includegraphics[width=0.3\columnwidth]{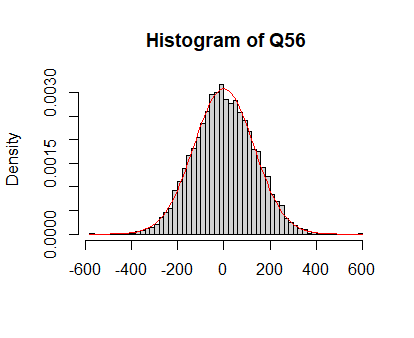}
    \includegraphics[width=0.3\columnwidth]{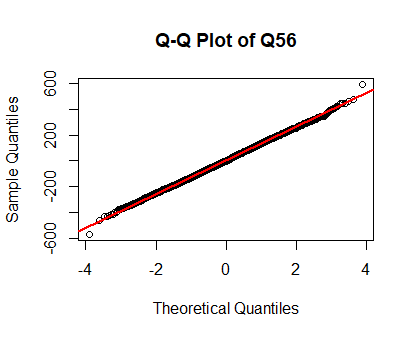}
    \includegraphics[width=0.3\columnwidth]{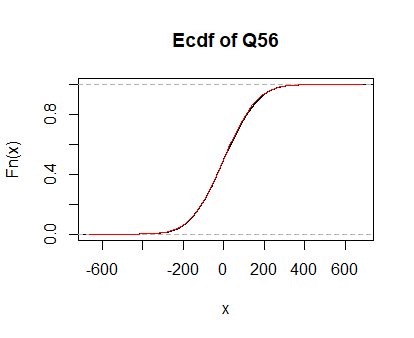}
    \\
    \includegraphics[width=0.3\columnwidth]{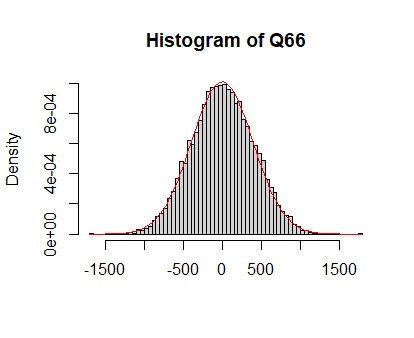}
    \includegraphics[width=0.3\columnwidth]{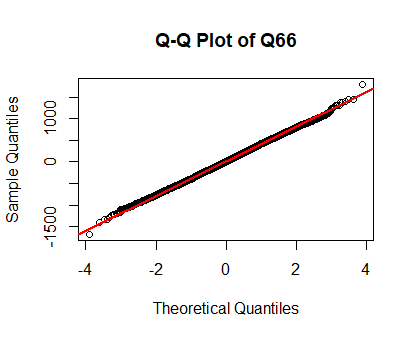}
    \includegraphics[width=0.3\columnwidth]{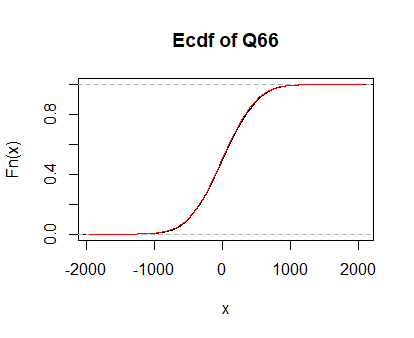}
    \caption{Histograms (left), Q-Q plots (middle) and empirical distributions (right) of $\sqrt{n}((\mathbb{Q}_{\mathbb{XX}})_{ij}-({\bf{\Sigma}}_0)_{ij})$
    for $i\leq j$ and $i,j=1,\cdots,6$. The red lines are theoretical curves.}
    \label{Qfigurenon2}
\end{figure}
\newpage
\begin{figure}[h]
    \ \\ \ \\ \ \\
    \includegraphics[width=0.3\columnwidth]{./files/histtheta1.png}
    \includegraphics[width=0.3\columnwidth]{./files/QQtheta1.png}
    \includegraphics[width=0.3\columnwidth]{./files/ecdftheta1.png}
    \\
    \includegraphics[width=0.3\columnwidth]{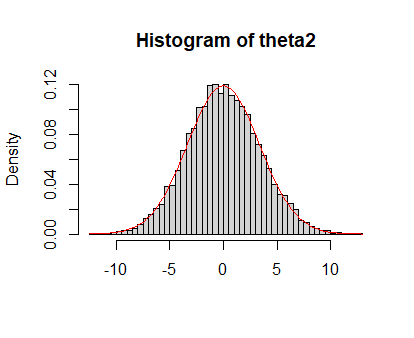}
    \includegraphics[width=0.3\columnwidth]{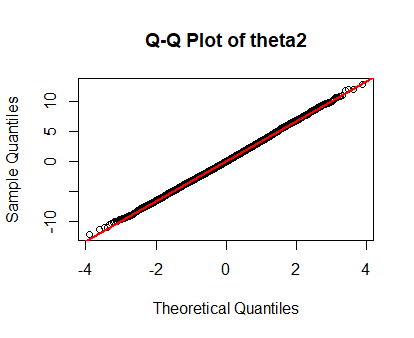}
    \includegraphics[width=0.3\columnwidth]{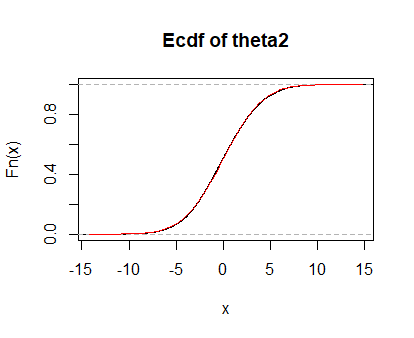}
    \\ 
    \includegraphics[width=0.3\columnwidth]{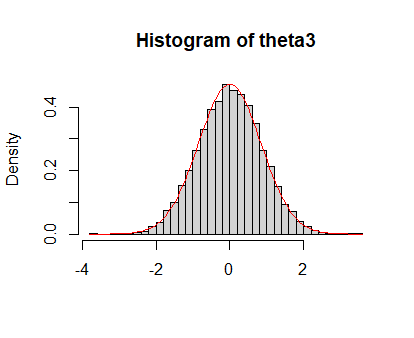}
    \includegraphics[width=0.3\columnwidth]{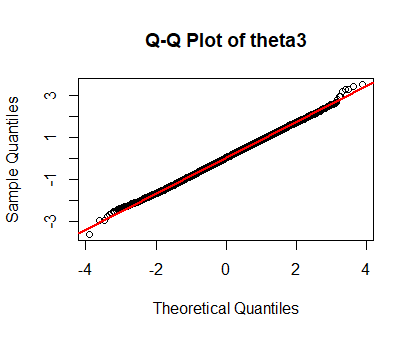}
    \includegraphics[width=0.3\columnwidth]{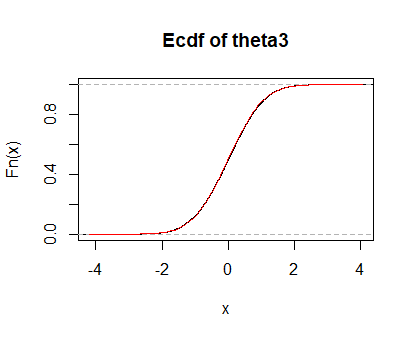}
    \\ 
    \includegraphics[width=0.3\columnwidth]{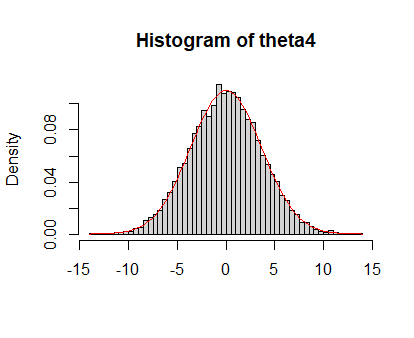}
    \includegraphics[width=0.3\columnwidth]{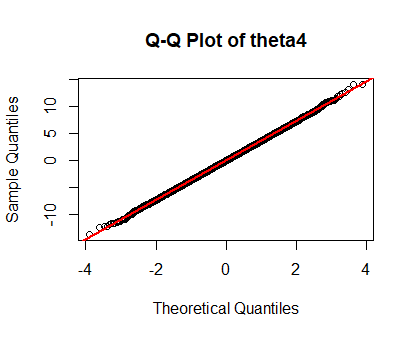}
    \includegraphics[width=0.3\columnwidth]{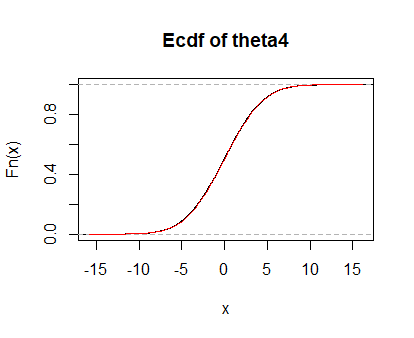}
    \\
    \includegraphics[width=0.3\columnwidth]{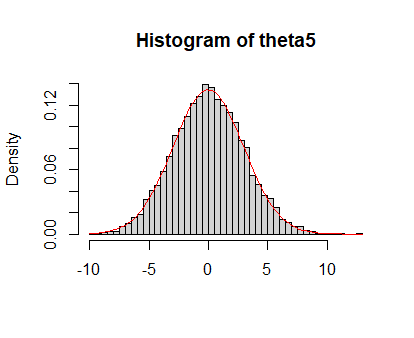}
    \includegraphics[width=0.3\columnwidth]{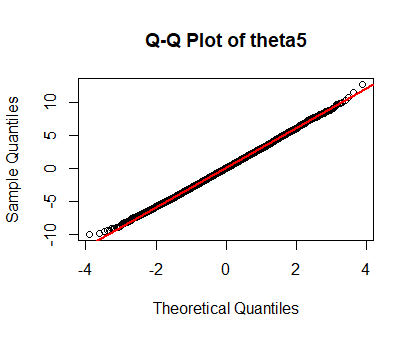}
    \includegraphics[width=0.3\columnwidth]{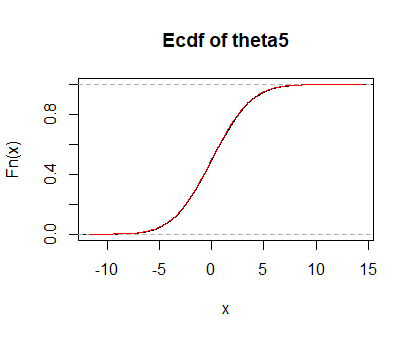}
\end{figure}
\newpage
\begin{figure}[h]
    \ \\ \ \\ \ \\
    \includegraphics[width=0.3\columnwidth]{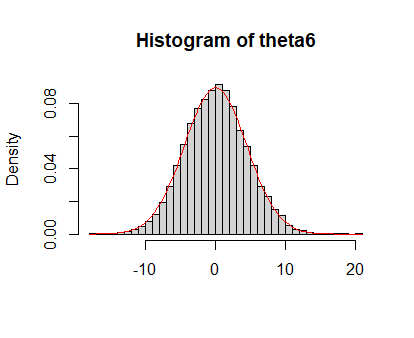}
    \includegraphics[width=0.3\columnwidth]{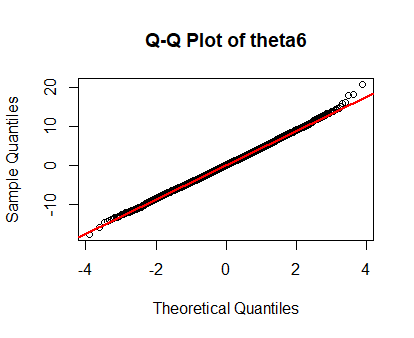}
    \includegraphics[width=0.3\columnwidth]{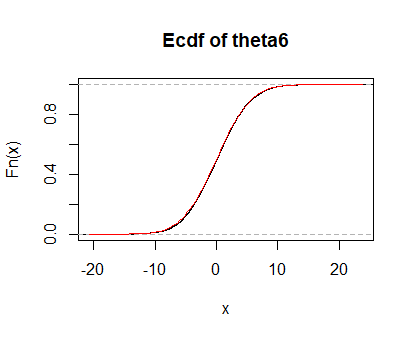}
    \\
    \includegraphics[width=0.3\columnwidth]{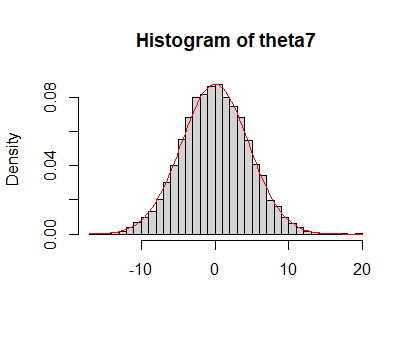}
    \includegraphics[width=0.3\columnwidth]{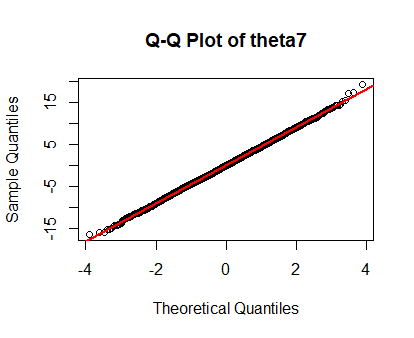}
    \includegraphics[width=0.3\columnwidth]{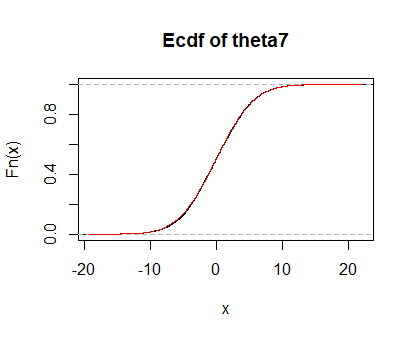}
    \\
    \includegraphics[width=0.3\columnwidth]{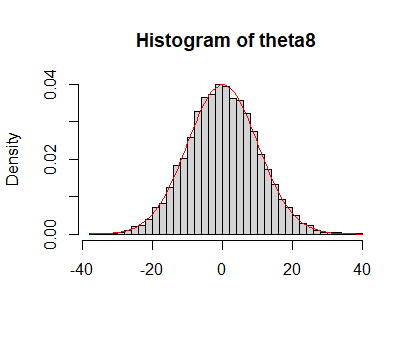}
    \includegraphics[width=0.3\columnwidth]{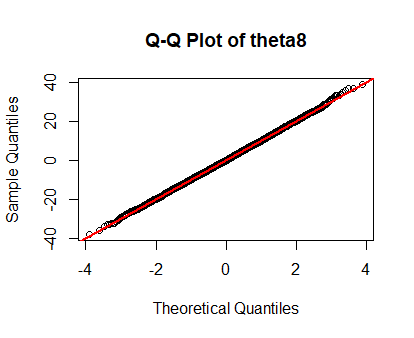}
    \includegraphics[width=0.3\columnwidth]{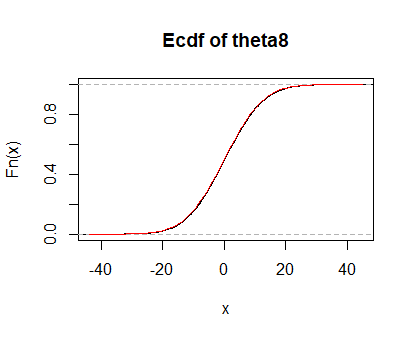}
    \\
    \includegraphics[width=0.3\columnwidth]{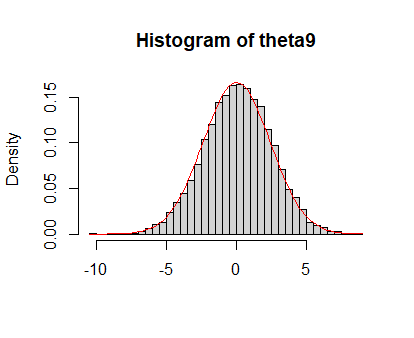}
    \includegraphics[width=0.3\columnwidth]{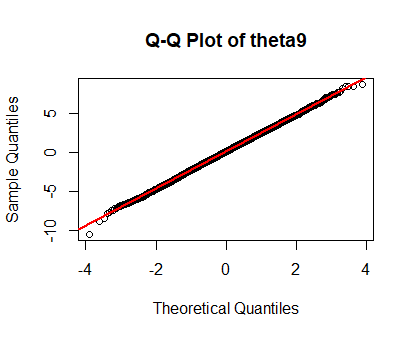}
    \includegraphics[width=0.3\columnwidth]{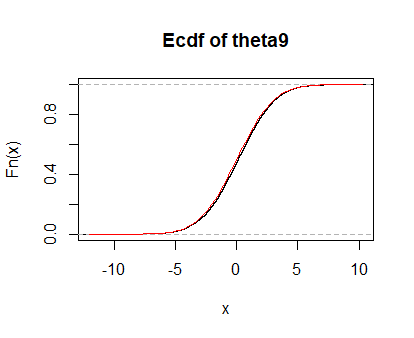}
    \\
    \includegraphics[width=0.3\columnwidth]{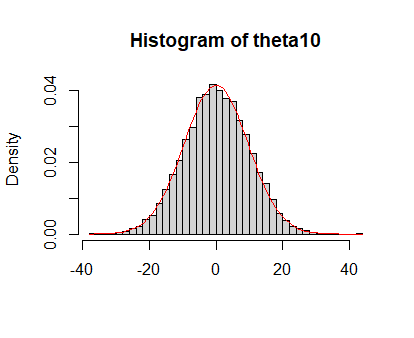}
    \includegraphics[width=0.3\columnwidth]{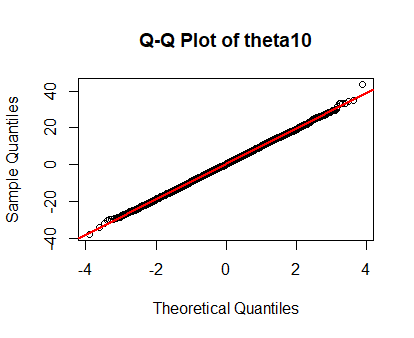}
    \includegraphics[width=0.3\columnwidth]{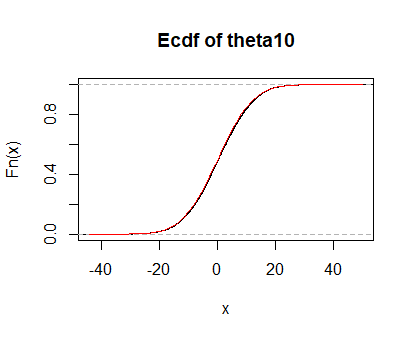}
\end{figure}
\newpage
\begin{figure}[h]
    \ \\ \ \\
    \includegraphics[width=0.3\columnwidth]{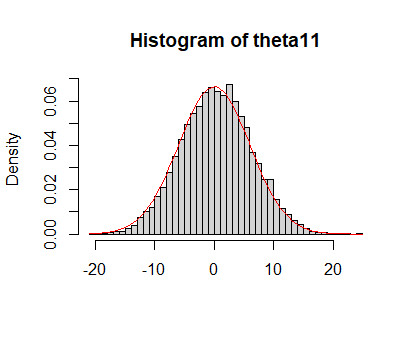}
    \includegraphics[width=0.3\columnwidth]{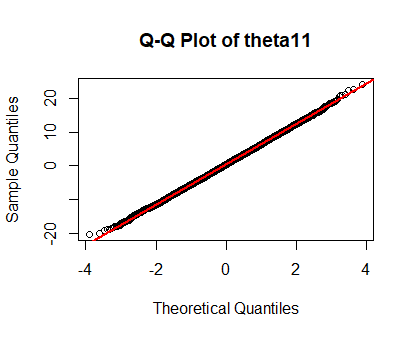}
    \includegraphics[width=0.3\columnwidth]{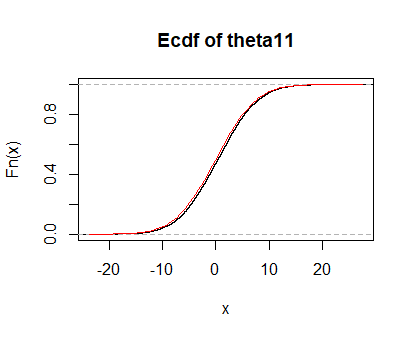}
    \\
    \includegraphics[width=0.3\columnwidth]{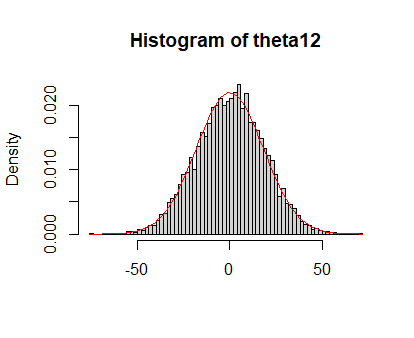}
    \includegraphics[width=0.3\columnwidth]{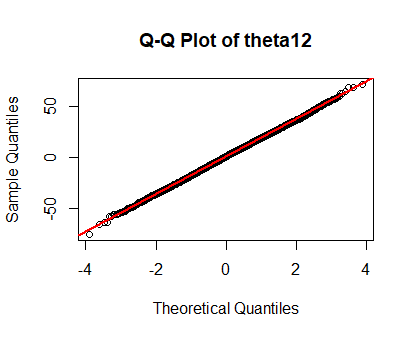}
    \includegraphics[width=0.3\columnwidth]{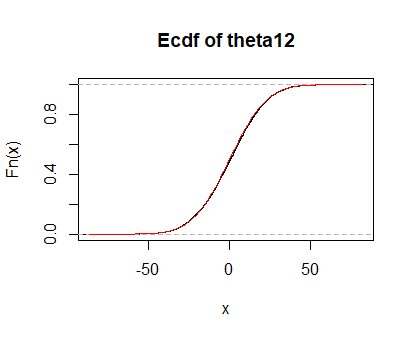}
    \\
    \includegraphics[width=0.3\columnwidth]{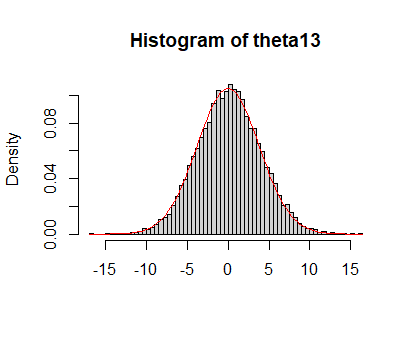}
    \includegraphics[width=0.3\columnwidth]{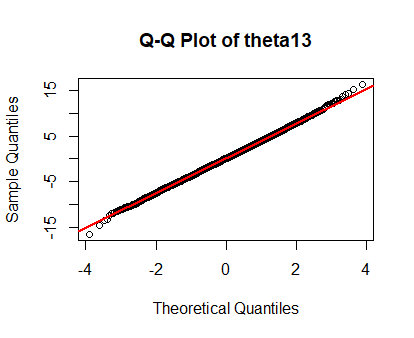}
    \includegraphics[width=0.3\columnwidth]{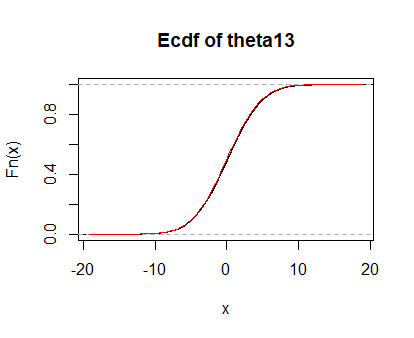}
    \\
    \includegraphics[width=0.3\columnwidth]{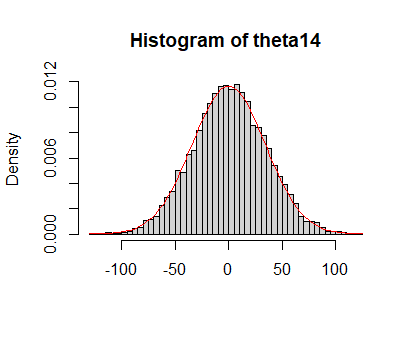}
    \includegraphics[width=0.3\columnwidth]{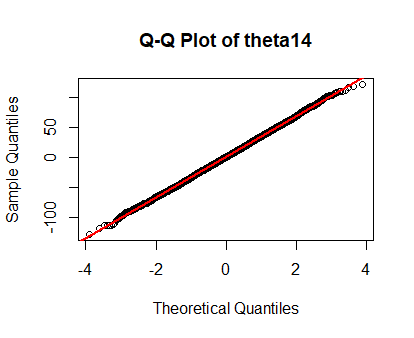}
    \includegraphics[width=0.3\columnwidth]{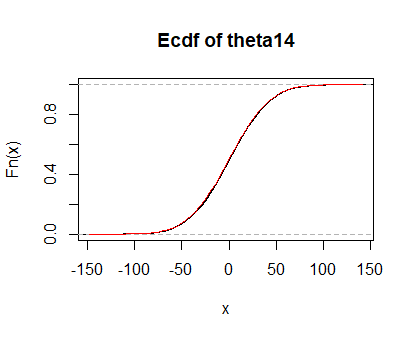}
    \\
    \includegraphics[width=0.3\columnwidth]{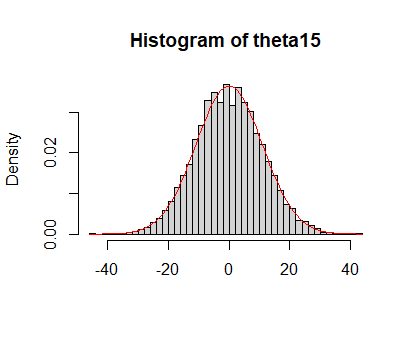}
    \includegraphics[width=0.3\columnwidth]{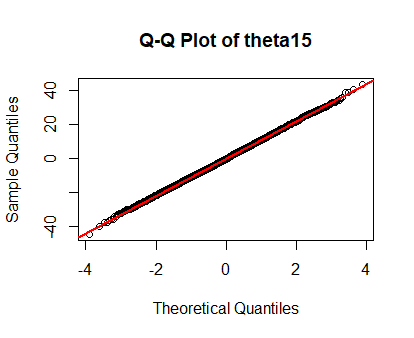}
    \includegraphics[width=0.3\columnwidth]{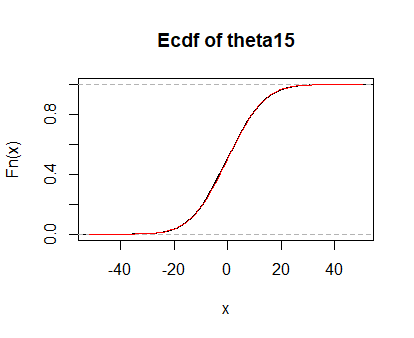}
    \caption{Histograms (left), Q-Q plots (middle) and empirical distributions (right) of $\sqrt{n}(\hat{\theta}_{n}^{(i)}-\theta_{0}^{(i)})$ for $i=1,\cdots,15$. The red lines are theoretical curves.} \label{thetafigurenon2}
\end{figure}
\end{document}